\documentclass[a4paper,12pt,reqno]{amsart}

\usepackage[utf8]{inputenc}
\usepackage[top=25truemm,bottom=25truemm,left=20truemm,right=20truemm]{geometry}

\usepackage{amsmath}
\usepackage{amsthm,amsfonts}
\usepackage{amssymb}
\usepackage{stmaryrd}
\usepackage{mathtools}
\usepackage{mathrsfs} 
\usepackage[bbgreekl]{mathbbol} 
\usepackage{upgreek}
\usepackage{tikz-cd}
\usetikzlibrary{decorations.markings,decorations.pathmorphing, shapes.multipart}
\usepackage{thmtools} 
\usepackage{enumitem}
\usepackage[nolist,smaller]{acronym}
\usepackage{mathpartir}

\DeclareSymbolFontAlphabet{\mathbb}{AMSb}
\DeclareSymbolFontAlphabet{\mathbbl}{bbold}
\DeclareEmphSequence{\bfseries\itshape}

\renewcommand{\epsilon}{\varepsilon}
\renewcommand{\phi}{\varphi}

\setlist{nosep}
\setlist*[enumerate,1]{label*=(\roman*)}

\makeatletter
\newcommand{\labeleditem}[1]{
\item[\text{#1}]\protected@edef\@currentlabel{\text{#1}}\phantomsection
}
\makeatother

\usepackage[%
    backend=biber,%
    style=alphabetic,%
    sorting=nyt,%
    giveninits=true,
    backref=true,%
    isbn=false,%
    url=true,%
]{biblatex}
\makeatletter
\renewrobustcmd*{\mkbibemph}{\mkbibitalic}
\protected\long\def\blx@imc@mkbibemph#1{\blx@imc@mkbibitalic{#1}}
\makeatother

\definecolor{darkblue}{rgb}{0.2,0,0.6}
\definecolor{darkgreen}{rgb}{0.2,0.5,0.2}
\usepackage[%
    colorlinks=true,%
    linkcolor=darkblue,%
    citecolor=green,%
    urlcolor=darkgreen,%
    hypertexnames=false,%
    naturalnames=true,
    pdfusetitle=true,
]{hyperref}
\usepackage[%
    nameinlink,%
]{cleveref}

\crefname{equation}{}{}
\crefname{enumi}{}{}
\crefname{section}{Section}{Sections}
\crefname{appendix}{Appendix}{Appendices}
\crefname{figure}{Figure}{Figures}

\numberwithin{equation}{section}

\declaretheoremstyle[
    spaceabove=0pt,%
    spacebelow=5pt,%
    headfont={\normalfont\bfseries},
    notefont={\normalfont},%
    bodyfont={\normalfont},%
    postheadspace=.5em,%
    headpunct={.},%
]{claim}
\declaretheoremstyle[
    spaceabove=-15pt,%
    spacebelow=5pt,%
    headfont={\normalfont},%
    notefont={\normalfont},%
    bodyfont={\addtolength{\leftskip}{.8em}\normalfont},%
    postheadspace=.5em,%
    headpunct={},%
    postfoothook={\ignorespacesafterend\noindent},%
]{since}

\declaretheorem[%
    style=definition,%
    numberwithin=section,%
    qed=$\lrcorner$,%
]{definition}
\declaretheorem[%
    style=definition,%
    numberwithin=section,%
    sibling=definition,%
    qed=$\lrcorner$,%
]{remark,example,notation,fact,construction}
\declaretheorem[%
    style=definition,%
    numberwithin=section,%
    sibling=definition,%
    postfoothook={\setcounter{claim}{0}},
]{theorem,lemma,corollary,proposition}
\declaretheorem[%
    style=claim,%
]{claim}
\declaretheorem[%
    style=since,%
    numbered=no,%
    qed=$\lozenge$,%
    name=$\because$,%
]{since}

\crefname{theorem}{Theorem}{Theorems}
\crefname{definition}{Definition}{Definitions}
\crefname{lemma}{Lemma}{Lemmas}
\crefname{corollary}{Corollary}{Corollaries}
\crefname{proposition}{Proposition}{Propositions}
\crefname{remark}{Remark}{Remarks}
\crefname{example}{Example}{Examples}
\crefname{notation}{Notation}{Notations}
\crefname{construction}{Construction}{Construction}
\crefname{claim}{Claim}{Claims}
\crefname{fact}{Fact}{Facts}

\usepackage{xpatch}
\makeatletter   
\xpatchcmd{\@tocline}
{\hfil\hbox to\@pnumwidth{\@tocpagenum{#7}}\par}
{\ifnum#1<2\hfill\else\dotfill\fi\hbox to\@pnumwidth{\@tocpagenum{#7}}\par}
{}{}
\def\l@subsection{\@tocline{2}{0pt}{2pc}{6pc}{}}
\makeatother    

\tikzcdset{equal/.append style={double distance=1.0pt,-}}

\tikzcdset{%
    arrow style=tikz,
    diagrams={>={Straight Barb[scale=0.8]}}
}

\tikzstyle{tri} = [row sep={1.6em}, column sep={0.8em}]
\tikzstyle{tinytri} = [row sep={0.8em}, column sep={0.1em}]

\tikzstyle{huge} = [row sep={3.6em}, column sep={3.6em}]
\tikzstyle{large} = [row sep={2.7em}, column sep={2.7em}]
\tikzstyle{normal} = [row sep={1.8em}, column sep={1.8em}]
\tikzstyle{scriptsize} = [row sep={1.35em}, column sep={1.35em}]
\tikzstyle{small} = [row sep={0.9em}, column sep={0.9em}]
\tikzstyle{tiny} = [row sep={0.45em}, column sep={0.45em}]

\tikzstyle{hugecolumn} = [column sep={3.6em}]
\tikzstyle{largecolumn} = [column sep={2.7em}]
\tikzstyle{normalcolumn} = [column sep={1.8em}]
\tikzstyle{scriptsizecolumn} = [column sep={1.35em}]
\tikzstyle{smallcolumn} = [column sep={0.9em}]
\tikzstyle{tinycolumn} = [column sep={0.45em}]

\tikzstyle{hugerow} = [row sep={3.6em}]
\tikzstyle{largerow} = [row sep={2.7em}]
\tikzstyle{normalrow} = [row sep={1.8em}]
\tikzstyle{scriptsizerow} = [row sep={1.35em}]
\tikzstyle{smallrow} = [row sep={0.9em}]
\tikzstyle{tinyrow} = [row sep={0.45em}]
\renewcommand{\-}{\mathchar`-}

\newcommand{\A}{\mathscr{A}}

\newcommand{\C}{\mathscr{C}}
\newcommand{\D}{\mathscr{D}}
\newcommand{\E}{\mathscr{E}}

\newcommand{\J}{\mathscr{J}}

\newcommand{\X}{\mathscr{X}}

\newcommand{\bN}{\mathbb{N}}

\newcommand{\bT}{\mathbb{T}}

\newcommand{\bfLambda}{\mathbf{\Lambda}}
\newcommand{\bfDelta}{\mathbf{\Delta}}
\newcommand{\bfE}{\mathbf{E}}
\newcommand{\bfM}{\mathbf{M}}

\newcommand{\Set}{\mathbf{Set}}
\newcommand{\Cat}{\mathbf{Cat}}
\newcommand{\CAT}{\mathbf{CAT}}
\newcommand{\Cattwo}{\mathcal{C}\mathrm{at}}
\newcommand{\CATtwo}{\mathcal{C}\mathrm{AT}}
\newcommand{\Top}{\mathbf{Top}}
\newcommand{\Pos}{\mathbf{Pos}}
\newcommand{\omegaCat}{\omega\-\mathbf{Cat}}
\newcommand{\nCat}{n\-\mathbf{Cat}}
\newcommand{\fourCat}{4\-\mathbf{Cat}}
\newcommand{\threeCat}{3\-\mathbf{Cat}}
\newcommand{\twoCat}{2\-\mathbf{Cat}}
\newcommand{\twoCATtwo}{2\-\mathcal{C}\mathrm{AT}}
\newcommand{\Field}{\mathbf{Field}}
\newcommand{\Grp}{\mathbf{Grp}}

\newcommand{\initialcat}{\mathbf{0}}
\newcommand{\terminalcat}{\mathbf{1}}
\newcommand{\walkmorcat}{\mathbf{2}}

\newcommand{\op}{\mathrm{op}}

\newcommand{\Ob}{\mathrm{Ob}}

\newcommand{\Mor}{\mathrm{Mor}}

\newcommand{\id}{\mathsf{id}}

\newcommand{\inv}{^{-1}}

\newcommand{\Colim}{\mathop{\mathrm{Colim}}}
\newcommand{\El}{{\textstyle\int}}

\newcommand{\tup}[1]{\vec{#1}}

\newcommand{\abs}[1]{{|#1|}}
\newcommand{\const}[1]{\ulcorner #1\urcorner}

\newcommand{\incat}[1]{\hspace{1em}\text{in }{#1}}
\NewDocumentCommand{\proofdirection}{O{\implies} m m}{%
    \noindent\textbf{[#2$#1$#3]}%
}

\newcommand{\parr}{\rightrightarrows}
\newcommand{\pto}{\rightharpoonup} 

\renewcommand{\emptyset}{\varnothing}

\NewDocumentCommand{\arr}{D(){} O{} O{1}}{\mathrel{%
    \begin{tikzcd}[column sep={#3em}, ampersand replacement=\&]
        \hspace{-2ex} \& \hspace{-2ex}
        \arrow[from=1-1, to=1-2, "{%
            \IfValueTF{#1}{%
                \raisebox{0pt}[3pt][\depth]{$\scriptstyle #1$}%
            }{%
                \scriptstyle #1%
            }%
        }", #2]
    \end{tikzcd}
}}
\NewDocumentCommand{\rra}{D(){} O{} O{1}}{\mathrel{%
    \begin{tikzcd}[column sep={#3em}, ampersand replacement=\&]
        \hspace{-2ex} \& \hspace{-2ex}
        \arrow[from=1-2, to=1-1, "{%
            \IfValueTF{#1}{%
                \raisebox{0pt}[3pt][\depth]{$\scriptstyle #1$}%
            }{%
                \scriptstyle #1%
            }%
        }"', #2]
    \end{tikzcd}
}}

\NewDocumentCommand{\adjoint}{D(){\perp} O{1} m m m m }{
    \begin{tikzcd}[column sep={#2em},ampersand replacement=\&]
        {#3}\ar[rr,"{#5}",shift left=4pt,bend left=10] \& {#1} \& {#4}\ar[ll,"{#6}",shift left=4pt,bend left=10]
    \end{tikzcd}
}

\NewDocumentCommand{\cellsymb}{D(){} O{} m m}{
    \arrow[from=#3,to=#4,phantom,"{\scriptstyle #1}"{#2}]
}

\NewDocumentCommand{\pullbackcorner}{O{rd}}{%
    \ar[
        "\lrcorner"{description, near start, inner sep=0mm},
        phantom, #1,
        start anchor={[xshift=0ex, yshift=0ex]center},
        end anchor={[xshift=0ex, yshift=0ex]center},
        ]
}
\NewDocumentCommand{\pushoutcorner}{O{lu}}{%
    \ar[
        "\ulcorner"{description, near start, inner sep=0mm},
        phantom, #1,
        start anchor={[xshift=0ex, yshift=0ex]center},
        end anchor={[xshift=0ex, yshift=0ex]center},
        ]
}

\newcommand{\veq}{%
    \mathord{
        \rotatebox{90}{$\scriptstyle =$}
    }
}

\newcommand{\vcong}{\rotatebox{90}{$\scriptstyle\cong$}}

\newcommand{\colonequiv}{%
    \mathrel{\vcentcolon\equiv}
}

\NewDocumentCommand{\utwocell}{D(){} O{right=0pt} m m O{}}{
    \arrow[from=#3, to=#4, phantom, "{\scriptstyle #1}"{#2}, "\rotatebox{90}{$\Rightarrow$}", #5]
}
\NewDocumentCommand{\dtwocell}{D(){} O{right=0pt} m m O{}}{
    \arrow[from=#3, to=#4, phantom, "{\scriptstyle #1}"{#2}, "\rotatebox{90}{$\Leftarrow$}", #5]
}
\NewDocumentCommand{\ltwocell}{D(){} O{above=2pt} m m O{}}{
    \arrow[from=#3, to=#4, phantom, "{\scriptstyle #1}"{#2}, "\Leftarrow", #5]
}
\NewDocumentCommand{\rtwocell}{D(){} O{above=2pt} m m O{}}{
    \arrow[from=#3, to=#4, phantom, "{\scriptstyle #1}"{#2}, "\Rightarrow", #5]
}
\NewDocumentCommand{\urtwocell}{D(){} O{above left=-1pt} m m O{}}{
    \arrow[from=#3, to=#4, phantom, "{\scriptstyle #1}"{#2}, "\rotatebox{45}{$\Rightarrow$}", #5]
}
\NewDocumentCommand{\rutwocell}{D(){} O{above left=-1pt} m m O{}}{
    \arrow[from=#3, to=#4, phantom, "{\scriptstyle #1}"{#2}, "\rotatebox{45}{$\Rightarrow$}", #5]
}
\NewDocumentCommand{\drtwocell}{D(){} O{above right=0pt} m m O{}}{
    \arrow[from=#3, to=#4, phantom, "{\scriptstyle #1}"{#2}, "\rotatebox{-45}{$\Rightarrow$}", #5]
}
\NewDocumentCommand{\rdtwocell}{D(){} O{above right=0pt} m m O{}}{
    \arrow[from=#3, to=#4, phantom, "{\scriptstyle #1}"{#2}, "\rotatebox{-45}{$\Rightarrow$}", #5]
}
\NewDocumentCommand{\ultwocell}{D(){} O{above right=-1pt} m m O{}}{
    \arrow[from=#3, to=#4, phantom, "{\scriptstyle #1}"{#2}, "\rotatebox{-45}{$\Leftarrow$}", #5]
}
\NewDocumentCommand{\lutwocell}{D(){} O{above right=-1pt} m m O{}}{
    \arrow[from=#3, to=#4, phantom, "{\scriptstyle #1}"{#2}, "\rotatebox{-45}{$\Leftarrow$}", #5]
}
\NewDocumentCommand{\dltwocell}{D(){} O{above left=0pt} m m O{}}{
    \arrow[from=#3, to=#4, phantom, "{\scriptstyle #1}"{#2}, "\rotatebox{45}{$\Leftarrow$}", #5]
}
\NewDocumentCommand{\ldtwocell}{D(){} O{above left=0pt} m m O{}}{
    \arrow[from=#3, to=#4, phantom, "{\scriptstyle #1}"{#2}, "\rotatebox{45}{$\Leftarrow$}", #5]
}

\newcommand{\vx}{\syn{x}}
\newcommand{\vy}{\syn{y}}
\newcommand{\vz}{\syn{z}}

\newcommand{\ttau}{\syn{\tau}}
\newcommand{\tsigma}{\syn{\sigma}}
\newcommand{\trho}{\syn{\rho}}
\newcommand{\tnu}{\syn{\nu}}
\newcommand{\fphi}{\syn{\phi}}
\newcommand{\fpsi}{\syn{\psi}}

\newcommand{\defined}{\mathord{\downarrow}}
\newcommand{\ofsort}{\mathord{:}}
\newcommand{\subst}[2]{\lbrack {#1}/{#2} \rbrack}
\newcommand{\intpn}[2]{\mathord{\left\llbracket #1 \right\rrbracket_{#2}}}

\newcommand{\seq}[1]{\mathrel{
\tikz\draw[|-] (0,0) -- node[above=1.8pt,inner sep=0pt] {\small$#1$} (1,0);
}}
\newcommand{\biseq}[1]{\mathrel{
\tikz\draw[|-|] (0,0) -- node[above=1.8pt,inner sep=0pt] {\small$#1$} (1,0);
}}
\newcommand{\longseq}[1]{\mathrel{
\tikz\draw[|-] (0,0) -- node[above=1.8pt,inner sep=0pt] {\small$#1$} (1.5,0);
}}

\newcommand{\PStr}{\mathop{\mathbf{PStr}}}
\newcommand{\PMod}{\mathop{\mathbf{PMod}}}
\newcommand{\repn}[1]{\mathord{\left\langle #1 \right\rangle}}
\newcommand{\Term}{\mathrm{Term}}

\NewDocumentCommand{\termeq}{m O{}}{%
    \mathrel{%
        {
            \tikz\draw[draw=none] (0,0) -- node[above=-2pt,inner sep=0pt] {$\approx$} node[above=4.5pt,inner sep=0pt] {\tiny$#1$} (0.4,0);
        }_{#2}
    }%
}
\NewDocumentCommand{\termle}{m O{}}{%
    \mathrel{%
        {
            \tikz\draw[draw=none] (0,0) -- node[inner sep=0pt] {$\preccurlyeq$} node[above=3pt,inner sep=0pt] {\tiny$#1$} (0.4,0);
        }_{#2}
    }%
}
\NewDocumentCommand{\termge}{m O{}}{%
    \mathrel{%
        {
            \tikz\draw[draw=none] (0,0) -- node[inner sep=0pt] {$\succcurlyeq$} node[above=3pt,inner sep=0pt] {\tiny$#1$} (0.4,0);
        }_{#2}
    }%
}

\newcommand{\rorth}{\mathbf{r}}
\newcommand{\lorth}{\mathbf{l}}

\NewDocumentCommand{\orth}{O{}}{\mathrel{
    \bot^{\mkern-4mu{#1}}
}}

\newcommand{\epclass}{\mathbf{\Lambda}}
\newcommand{\Reg}[1][]{\mathord{\mathbf{Reg}_{#1}}}
\newcommand{\Stg}[1][]{\mathord{\mathbf{Stg}_{#1}}}

\newcommand{\Epi}[1][]{\mathord{\mathbf{Epi}_{#1}}}

\NewDocumentCommand{\kercond}{O{\epclass} O{q}}{%
    {%
        \hyperref[kernel_condition]{%
            (\mathrm{K}_{#1,#2})%
        }%
    }%
}

\newcommand{\Met}{\mathbf{Met}}

\newcommand{\inacc}{\upsilon}

\NewDocumentCommand{\decnum}{O{} m}{\mathord{\delta_{{#1}}(#2)}}
\NewDocumentCommand{\cdecnum}{O{} m}{\mathord{\sigma_{{#1}}(#2)}}
\NewDocumentCommand{\Predec}{O{}}{\mathord{\mathbf{PDec}_{{#1}}}}

\newcommand{\CATradj}{\mathbf{CAT}_\mathrm{radj}}
\newcommand{\ladj}[1]{{{#1}^*}}
\newcommand{\unit}[1]{{\eta^{#1}}}
\newcommand{\counit}[1]{{\epsilon^{#1}}}

\newcommand{\lelm}[1]{\overline{#1}}
\newcommand{\relm}[1]{\widetilde{#1}}
\newcommand{\bottom}{\mathord{\perp}}
\tikzset{%
    monotoneheader/.style={%
        label={[%
            rectangle, fill=white, node font=\ttfamily, anchor=base,name=\tikzlastnode-header]north:{#1}
            }
    }
}
\tikzstyle{monotone-example} = [row sep={-0.3em}, column sep={-0.5em}]

\newcommand{\height}{\sharp}

\newcommand{\pos}{\mathrm{pos}}
\newcommand{\phtpos}{{\theory{T}_\pos}}

\newcommand{\ncat}{{n\mathrm{cat}}}

\newcommand{\ccomp}[1][]{\syn{\circ}_{#1}}
\newcommand{\bdry}{\syn{\partial}}

\tikzcdset{%
    dom/.code = {\tikzcdset{dashed,no head,start anchor=west,end anchor=east}},%
    cod/.code = {\tikzcdset{no head,start anchor=west,end anchor=east};},%
}
\tikzstyle{globular} = [row sep={-0.1em}]

\newcommand{\ann}[1]{\textsf{#1}}
\newcommand{\type}{\ \ann{type}}
\newcommand{\ctx}{\ \ann{ctx}}

\newcommand{\sort}{_{\mathrm{so}}}
\newcommand{\oper}{_{\mathrm{op}}}

\newcommand{\drv}{{\,\triangleright\ }}

\newcommand{\Fix}[2]{#1{[#2]}}

\tikzcdset{%
    disp/.style = {%
        > = {Triangle[open]}
    }
}
\NewDocumentCommand{\darr}{D(){} O{} O{1}}{\mathrel{
    \begin{tikzcd}[column sep={#3em}, ampersand replacement=\&, every label/.append style={font=\small}]
        \hspace{-2ex}\ar[r, "{#1}", #2, disp] \& \hspace{-2ex}
    \end{tikzcd}
}}
\NewDocumentCommand{\Disp}{o}{\IfValueTF{#1}{\mathop{\mathbf{Disp}}_{#1}}{\mathbf{\Delta}}}
\newcommand{\Clan}{\mathcal{C}\mathrm{lan}}
\newcommand{\CLAN}{\mathcal{C}\mathrm{LAN}}
\newcommand{\Cland}{\mathcal{C}\mathrm{lan}_{\mathrm{d}}}
\newcommand{\CLANd}{\mathcal{C}\mathrm{LAN}_{\mathrm{d}}}
\newcommand{\ClanMod}{\mathop{\mathbf{Mod}}}
\newcommand{\Dopfib}{\mathbf{DOpFib}}
\newcommand{\ClanDopfib}{\mathbf{ClanDOpFib}}
\newcommand{\discard}[1]{|_{\le #1}}

\newcommand{\StrMonCat}{\mathbf{StrMonCat}}
\newcommand{\MultiCat}{\mathbf{MultiCat}}
\newcommand{\DblCat}{\mathbf{DblCat}}

\newcommand{\depran}[1]{\mathord{\mathrm{dr}}(#1)}

\newcommand{\clfy}[1]{\mathbf{Cl}({#1})}

\newcommand{\isocom}{\mathbin{\downarrow^{\cong}}}

\newcommand{\pocomp}[1]{c_{#1}}

\newcommand{\setg}{{\mathrm{set}}}
\newcommand{\catg}{{\mathrm{cat}}}
\newcommand{\twocatg}{{2\mathrm{cat}}}
\newcommand{\ncatg}{{n\mathrm{cat}}}
\newcommand{\moncatg}{{\mathrm{moncat}}}
\newcommand{\multicatg}{{\mathrm{multicat}}}
\newcommand{\dblcatg}{{\mathrm{dblcat}}}

\newcommand{\MCat}{\mathcal{M}\mathcal{C}\mathrm{at}}
\newcommand{\MCatIso}{\mathcal{M}\mathcal{C}{\mathrm{at}}_{\mathrm{iso}}}
\newcommand{\MCatPbTer}{\mathcal{M}\mathcal{C}{\mathrm{at}}_{\mathrm{mpb,1}}}
\newcommand{\PreClan}{\mathcal{P}\mkern-1mu\mathrm{re}\mkern1mu\mathcal{C}\mathrm{lan}}
\newcommand{\arrowtwocat}{\mathcal{A}\mathrm{rr}}
\newcommand{\Psd}{\mathop{\mathcal{P}\mathrm{sd}}}
\newcommand{\Cosp}{\mathcal{C}\mathrm{osp}}
\newcommand{\Comp}{\mathcal{C}\mathrm{omp}}
\newcommand{\freeAdj}{\mathcal{A}\mathrm{dj}}
\newcommand{\WLP}{\mathcal{W}\mathrm{LP}}
\newcommand{\chosen}{^{\mathrm{*}\mkern-2mu}}

\newcommand{\mcosp}{\mathbf{MCosp}}

\newcommand{\arrowcat}[1]{{#1}^\to}
\newcommand{\marrowcat}[1]{{#1}^\rightarrowtriangle}

\tikzcdset{%
    efs/.style={row sep=0.6em, column sep=0.8em}
}
\NewDocumentCommand{\psq}{D(){1-1} D(){3-3} m O{}}{%
    \arrow[from = #1, to = #2, "\rotatebox{-90}{$\Longrightarrow$}"{description}, phantom, "\rotatebox{90}{\small$\cong$}"{left}, "#3"{right = 1pt, #4}]
}
\NewDocumentCommand{\lsq}{D(){1-1} D(){3-3} m O{}}{%
    \arrow[from = #1, to = #2, "\rotatebox{-90}{$\Longrightarrow$}"{description}, phantom, "#3"{right = 1pt, #4}]
}
\NewDocumentCommand{\olsq}{D(){1-1} D(){3-3} m O{}}{%
    \arrow[from = #1, to = #2, "\rotatebox{90}{$\Longrightarrow$}"{description}, phantom, "#3"{right = 1pt, #4}]%
}
\NewDocumentCommand{\eqsq}{D(){1-1} D(){3-3}}{%
    \arrow[from = #1, to = #2, "\rotatebox{-90}{$=$}"{description}, phantom]
}
\NewDocumentCommand{\pefs}{D(){1-1} D(){3-3} D(){2-2} O{} O{} s}{%
    \arrow[from = #1, to = #3, "\rotatebox{-90}{$\IfBooleanTF{#6}{=}{\Rightarrow}$}"{description}, phantom, "\rotatebox{90}{$\scriptstyle \IfBooleanF{#6}{\cong}$}"{left}, "\scriptstyle #4"{right = 2pt, #5}]%
    \arrow[from = #3, to = #2, "\rotatebox{90}{\small$=$}"{description}, phantom]
}

\newcommand{\theory}[1]{\mathbb{#1}}
\newcommand{\Var}{\mathrm{Var}}
\newcommand{\fv}{\mathrm{fv}}

\makeatletter

\ExplSyntaxOn
\NewDocumentCommand{\strip@backslash}{m}{
    \cs_to_str:N #1
}
\ExplSyntaxOff

\NewDocumentCommand{\syn}{m}{%
    \@ifundefined{syn@\strip@backslash{#1}}%
    {\mathsf{#1}}
    {\csname syn@\strip@backslash{#1}\endcsname}
}

\newcounter{derivation}

\crefname{derivation}{Rule}{Rules}
\Crefname{derivation}{Rule}{Rules}

\newcommand{\derule}[3]{%
  \refstepcounter{derivation}%
  \label{#1}%
  \inferrule*[right={\cref*{#1}}]{#2}{#3}%
}

\def\syn@alpha{\upalpha}
\def\syn@beta{\upbeta}
\def\syn@gamma{\upgamma}
\def\syn@delta{\updelta}
\def\syn@epsilon{\upepsilon}
\def\syn@zeta{\upzeta}
\def\syn@eta{\upeta}
\def\syn@theta{\uptheta}
\def\syn@iota{\upiota}
\def\syn@kappa{\upkappa}
\def\syn@lambda{\uplambda}
\def\syn@mu{\upmu}
\def\syn@nu{\upnu}
\def\syn@xi{\upxi}
\def\syn@pi{\uppi}
\def\syn@rho{\uprho}
\def\syn@sigma{\upsigma}
\def\syn@tau{\uptau}
\def\syn@upsilon{\upupsilon}
\def\syn@phi{\upphi}
\def\syn@chi{\upchi}
\def\syn@psi{\uppsi}
\def\syn@omega{\upomega}

\makeatother

\newcommand*{\lon}{%
    \nobreak
    \mskip6mu plus1mu
    \mathpunct{}%
    \nonscript
    \mkern-\thinmuskip
    {:}%
    \mskip2mu
    \relax
}

\newcommand*{\adjl}{%
    \hspace{-1.5truemm}
    \begin{tikzcd}[column sep=small, ampersand replacement=\&]\!%
        \ar[r, phantom, "\mbox{\tiny $\bot$}"{marking, yshift=0.1em}]\ar[r, no head, white, shift left=3.0pt, line width=1.5pt]\ar[r, shift left]%
        \&\!\ar[l, shift left=3.1pt]%
    \end{tikzcd}
    \hspace{-1.5truemm}
}

\newcommand*{\universe}{\mathsf{U}}

\tikzset{%
    header/.style={%
        label={[%
            rectangle, fill=white, draw, node font=\ttfamily, anchor=center,name=\tikzlastnode-header]north west:{#1}
            }
    }
}

\acrodef{PHT}{partial Horn theory}
\acrodefplural{PHT}[PHTs]{partial Horn theories}

\acrodef{PHL}{partial Horn logic}

\acrodef{GAT}{generalized algebraic theory}
\acrodefplural{GAT}[GATs]{generalized algebraic theories}

\acrodef{ZFC}{Zermelo--Fraenkel set theory with the axiom of choice}

\title{On the decomposition of a strong epimorphism into regular epimorphisms}

\author{Yuto Kawase}
\address{Research Institute for Mathematical Sciences, Kyoto University, Kyoto 606-8502, Japan}
\email{ykawase@kurims.kyoto-u.ac.jp}

\author{Hayato Nasu}
\address{Department of Mathematics and Statistics, Dalhousie University, Halifax NS\,B3H\,4R2, Canada}
\email{hnasu@dal.ca}

\date{\today}

\begin{document}
\begin{abstract}
    Strong epimorphisms and regular epimorphisms are two important classes of morphisms, and they do not coincide in general.
    Yet, in a locally presentable category, it is known that any strong epimorphism can be decomposed into a transfinite composite of regular epimorphisms.
    In this paper, we provide two syntactic methods to determine how many regular epimorphisms are needed in such a decomposition, using partial Horn theory and generalized algebraic theory.
    We start by discussing a general problem of decomposing a morphism into a transfinite composite of morphisms in a given class, which also covers the decomposition of an adjoint functor into monadic functors.
\end{abstract}

\maketitle
\tableofcontents
\section{Introduction}
The fundamental theorem of homomorphisms plays a central role in abstract algebra.
It states that a homomorphism induces a canonical isomorphism between its image and the quotient of its domain by its kernel.
Universal algebra generalizes this by extending the notion of kernels using congruence (equivalence) relations.
In category theory, the fundamental theorem of homomorphisms can be reformulated as a categorical property called regularity, which, in particular, implies that every morphism can be decomposed into a regular epimorphism followed by a monomorphism.\footnote{Regular categories in the modern sense usually require more than just the existence of such factorizations.}
Algebraic categories, i.e., categories of models of finite product theories (equational theories, or Lawvere theories), are known to be regular, which ensures that the fundamental theorem of homomorphisms holds in a broad range of algebraic structures.

As the notion of algebras is generalized, regularity of the corresponding categories tends to break down.
For instance, several categories of generalized algebras, such as ordered algebras, metric algebras, and partial algebras, are not regular.
One categorical framework for dealing with such generalized algebras is locally presentable categories, which are known to be equivalent to the categories of models of finite limit theories, as opposed to finite product theories for algebraic categories.
In this paper, we will figure out how close a given category of generalized algebras is to being regular in a quantitative way.

The \textit{decomposition number} or \textit{regular length} studied in \cite{GabrielUlmer1971lokal,MacdonaldStone1982tower,Borger1991making} is one such measure.
Intuitively, the decomposition number of a morphism measures how many regular epimorphisms are required in order to reach its image from its domain.
Here, by the image of the morphism, we mean the smallest subobject that the morphism factors through. 
In a fortunate case, for instance in algebraic categories, the decomposition number of a morphism therein is $1$ or $0$, and the fundamental theorem of homomorphisms holds exactly.
However, in general, a category may have a morphism that collapses too much of the ``algebra structure'' for its kernel (pair) to retain enough information of how the structure is collapsed.
If so, the quotient by the kernel is still uncompressed, indicating that taking the quotient just once is not enough to reach the image.
For example, consider the following morphism $\Phi$ in $\Cat$, the category of small categories and functors:
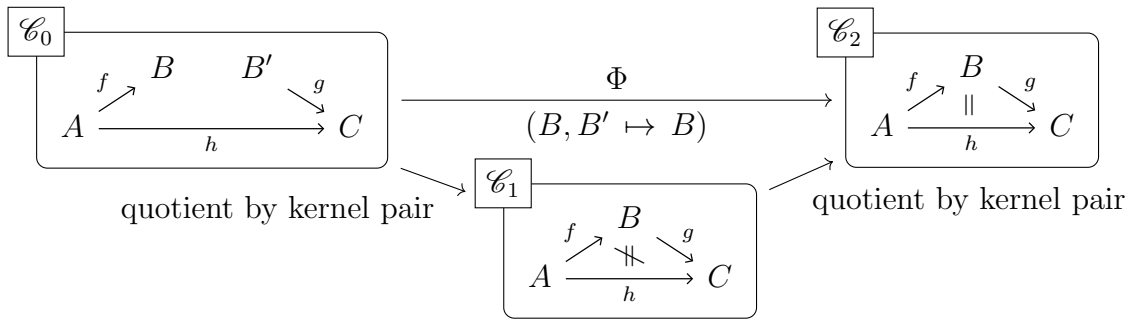
\begin{equation*}
    \begin{tikzpicture}
        \node[draw, rounded corners, header=$\C_0$] (Cat1) at (0,3) {
            \begin{tikzcd}[row sep=tiny, column sep=small]
                &
                B
                &
                B'
                \ar[rd,"g"]
                \\
                A
                \ar[ru,"f"]
                \ar[rrr,"h"']
                &
                &
                &
                C
            \end{tikzcd}
        };
        \node[draw, rounded corners, header=$\C_1$] (Cat2) at (5.5,1) {
            \begin{tikzcd}[row sep=tiny, column sep=small]
                &
                {B}
                \ar[rd,"g"]
                \\
                A
                \ar[ru,"{f}"]
                \ar[rr,"{h}"']
                \ar[rr, shift left=3, "\rotatebox{90}{$\neq$}", phantom]
                &
                &
                C
            \end{tikzcd}
        };
        \node[draw, rounded corners, header=$\C_2$] (Cat3) at (10,3)
        {
            \begin{tikzcd}[row sep=tiny, column sep=small]
                &
                {B}
                \ar[rd,"g"]
                \\
                A
                \ar[ru,"{f}"]
                \ar[rr,"{h}"']
                \ar[rr, shift left=3, "\rotatebox{90}{$=$}", phantom]
                &
                &
                C
            \end{tikzcd}
        };
        \draw [->,shorten <=5pt, shorten >=5pt] (Cat1) to node[above] {$\Phi$} node[below] {($B,B'\,\mapsto\,B$)} (Cat3);
        \draw [->,shorten <=5pt, shorten >=15pt] (Cat1) to node[below left] {quotient by kernel pair} (Cat2);
        \draw [->,shorten <=5pt, shorten >=5pt] (Cat2) to node[below right] {quotient by kernel pair} (Cat3);
    \end{tikzpicture}
\end{equation*}
The kernel pair of the morphism identifies $B$ and $B'$, but there is no way of identifying $g\circ f$ and $h$ inside the domain category $\C_0$ since $g\circ f$ does not even make sense there.
In short, the morphism collapses beyond what the domain category can observe.
Once $B$ and $B'$ get identified in $\C_1$, $g\circ f$ makes sense, and taking another quotient by the equation $g\circ f = h$ realizes its image $\C_2$.
Thus, the decomposition number of $\Phi$ is $2$, as we will formally define later.
Furthermore, we are interested in how large this number can be in a given category, which leads to the notion of the \textit{global decomposition number} that we will introduce.
In the case of $\Cat$, any morphism is decomposed into two regular epimorphisms and a monomorphism, as observed in \cite{BednarczykBorzyszkowskiPawlowski1999epi}.

Although the basic properties of decomposition numbers have been studied, a method for determining the global decomposition numbers of categories has yet to be established, even for some familiar classes of categories.
The idea of factorizing strong epimorphisms into infinite sequences of regular epimorphisms by taking quotients appeared in \cite{Kelly1969mono}, and the paper \cite{GabrielUlmer1971lokal} showed that every locally presentable category has a small global decomposition number.
Subsequent papers \cite{MacdonaldStone1982tower,MacdonaldTholen1982decomposition,Borger1991making} studied the problem in greater generality, but the question of how to compute them has remained open.

We will present two concrete methods to determine the global decomposition numbers of locally presentable categories using their syntactic presentations.
When a given category is presented as a category of models of some theory, the process of taking quotients is understood as imposing new equations on the models, which should reflect some characteristic features of the theory.
It is then natural to expect that the decomposition number can be computed by its syntactic presentation.
The two proposed methods will focus on partial Horn theories \cite{PalmgrenVickers2007partial} and generalized algebraic theories \cite{Cartmell1986generalised}, respectively, both of which are known to present all locally finitely presentable categories.
Infinitary versions of the former have been considered by several authors \cite{Parker2022covariant,TsukadaAsada2022linear,Kawase2026relativized} to cover locally presentable categories in general.

The core ideas of the two approaches can be illustrated with the above example in $\Cat$.
The first approach, using partial Horn theories, is quite direct given the intuition explained above.
One way to present small categories via partial Horn theories is to prepare two sorts, $\syn{Obj}$ and $\syn{Mor}$, for objects and morphisms, respectively, together with partial operations for domain, codomain, identity morphisms, and composition of morphisms.\footnote{In \cref{sec:the_partial-Horn-theoretic_approach}, we will use a different but equivalent presentation that has only one sort.}
Let us look at the functor $\Phi$ again, as in \cref{fig:the_partial-Horn-theoretic_approach}.
It forces the equations $B = B'$ and $g\circ f = h$, but the definedness of $g\circ f$ is preceded by $B = B'$, which is why it takes two steps of quotients.
This observation leads us to the first method of finding the global decomposition number by stratifying all terms according to the order of their definedness.
To formalize this, we will introduce the notion of \textit{gauge} for a partial Horn theory, which assigns an ordinal number to each term.
For instance, we have a gauge $\height$ for the partial Horn theory for small categories, where $\height B = \height B' = 0$ and $\height(g\circ f)=1$.
We will discuss this approach in \cref{sec:the_partial-Horn-theoretic_approach} in more detail.
\begin{figure}[htbp]
    \tikzcdset{every label/.append style={font=\scriptsize}}
    \begin{tikzpicture}
        \node[draw, rounded corners, minimum width=2cm, header=$\C_0$] (Cat1) at (0,0) {\small
            \begin{tabular}{c}
                $B\neq B'$ \\
                $g\circ f\,\uparrow$ \\
            \end{tabular}
        };
        \node[draw, rounded corners, header=$\C_1$] (Cat2) at (5,0) 
        {\small
            \begin{tabular}{c}
                $B= B'$ \\
                $g\circ f\,\downarrow$ \\
                $g\circ f \neq h$ \\
            \end{tabular}
        };
        \node[draw, rounded corners, header=$\C_2$] (Cat3) at (10,0)
        {\small
            \begin{tabular}{c}
                $B= B'$ \\
                $g\circ f\,\downarrow$ \\
                $g\circ f = h$ \\
            \end{tabular}
        };
        \draw[->,shorten <=5pt, shorten >=5pt] (Cat1) to node[below] {\small $B=B'$} (Cat2);
        \draw[->,shorten <=5pt, shorten >=5pt]  
        (Cat2) to node[below] {\small $g\circ f = h$} (Cat3);
        \draw[->,shorten <=15pt, shorten >=15pt, bend left=25] (Cat1) to node[above] {\small $\Phi$} (Cat3);
    \end{tikzpicture}
    \caption{Illustration of the approach via partial Horn logic}    
    \label{fig:the_partial-Horn-theoretic_approach}
\end{figure}
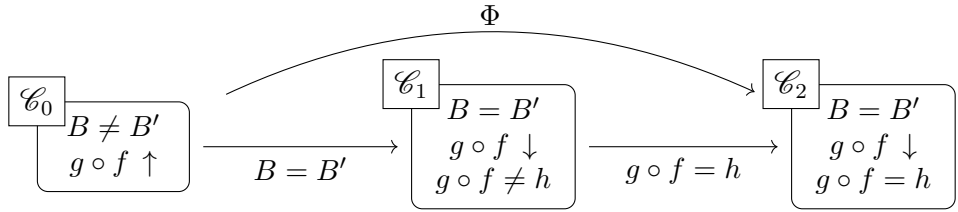

The second approach, using generalized algebraic theories, focuses on how sorts depend on each other. 
In the generalized algebraic theory of small categories, the sort of morphisms $\syn{Mor}(\syn{x},\syn{y})$ depends on the sort of objects $\syn{Obj}$.
As in \cref{fig:the_generalized_algebraic-theoretic_approach}, the first step of taking quotients is to collapse the sort of objects, and then that of morphisms.
Observe that $g\circ f$ shows up after the sort of objects is collapsed, enabling the next step to identify $g\circ f$ and $h$.
This example might lead the reader to expect that identifying elements of sorts from the less dependent to the more dependent would provide the regular decomposition in general.
However, this method is only applicable to a limited case where the output of every function symbol and the terms appearing in every equational axiom should not be of less dependent sorts than their contexts; otherwise, a step involving the more dependent sorts may influence the less dependent ones that have already been taken care of, requiring some additional steps to reach the image.
When this condition is satisfied, we can imitate the above example to give an upper bound of the decomposition number based on the dependency of the sorts.
We will delve into this approach in \cref{sec:for_models_of_dependent_algebraic_theories} using the framework of clans \cite{Joyal2017notes}.
\begin{figure}[htbp]
    \tikzcdset{every label/.append style={font=\scriptsize}}
    \begin{tikzpicture}
        \node[rectangle split, rectangle split parts=2, rounded corners]
        (Cat0) at (-2.5,0) {\small
            $\mathsf{Mor}$
            \nodepart{second}
            $\mathsf{Obj}$
        };
        \node[rectangle split, rectangle split parts=2, draw, rounded corners, minimum width=2cm, header=$\C_0$] (Cat1) at (0,0) {\small
            $f, g, h$
            \nodepart{second}
            $A, B, B', C$
        };
        \node[rectangle split, rectangle split parts=2, draw, rounded corners, minimum width=3cm,header=$\C_1$] (Cat2) at (5,0) 
        {\small
            $f, g, g\circ f, h$
            \nodepart{second}
            $A, B=B', C$
        };
        \node[rectangle split, rectangle split parts=2, draw, rounded corners, minimum width=3.5cm, header=$\C_2$] (Cat3) at (11,0)
        {\small
            $f, g, g\circ f = h$
            \nodepart{second}
            $A, B=B', C$
        };
        \draw[->,shorten <=5pt, shorten >=5pt] (Cat1) to node[below] {\small $B=B'$} (Cat2);
        \draw[->,shorten <=5pt, shorten >=5pt]  
        (Cat2) to node[below] {\small $g\circ f = h$} (Cat3);
        \draw[->,shorten <=25pt, shorten >=25pt, bend left=20] (Cat1) to node[above] {\small $\Phi$} (Cat3);
    \end{tikzpicture}
    \caption{Illustration of the approach via generalized algebraic theory}
    \label{fig:the_generalized_algebraic-theoretic_approach}
\end{figure}
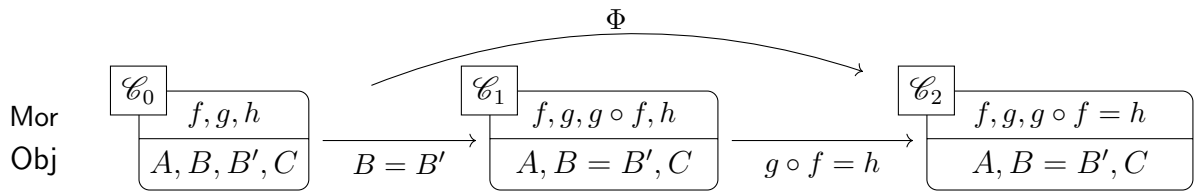

The two methods have their own merits and drawbacks. 
The first approach is quite general and can be applied to any partial Horn theory, even for infinitary ones, while it requires some effort to stratify all terms.
The second approach has a limited range of application because of the restriction we mentioned earlier, but once the assumption is met, it instantly yields the upper bound on the global decomposition number by simply looking at the dependency of sorts.

The notion of decomposition number can further be generalized from regular epimorphisms to other classes of morphisms that we think of as one-step collapses of structures.
This treatment captures several known concepts, such as strict localizations for categories and monotone quotient maps for topological spaces.
A key observation is that a morphism is a regular epimorphism if and only if it is universal among the morphisms that coequalize every pair coequalized by it.
In a similar manner, we will define generalized regular epimorphisms to encompass more notions of collapsing beyond coequalizing, as explained in \cref{sec:generalized_regularity}.

We also prove that, while a strong epimorphism can be decomposed into regular epimorphisms in different ways, taking quotients by the kernel (pair) at each step gives the shortest possible decomposition.
This means that for any morphism, the number of quotients required to reach its image (the \textit{canonical decomposition number}) coincides with the minimum number of regular epimorphisms that are composed to give its image (the \textit{minimum decomposition number}).
The former notion appears in the literature \cite{GabrielUlmer1971lokal,MacdonaldStone1982tower,Borger1991making}, but in a slightly different fashion: see \cref{rem:why_strict_supremum} for more details.
Furthermore, we will show the coincidence of the two notions of decomposition numbers for generalized regular epimorphisms mentioned above, and even more generally for a class of morphisms that admits a suitable notion of canonical decomposition based on the framework of \textit{locally orthogonal factorizations} introduced in \cite{Tholen1979semitopological}.
This generalization covers the decomposition of adjoint functors into monadic functors, which are studied in \cite{ApplegateTierney1970iterated,MacdonaldStone1982tower,MacdonaldTholen1982decomposition,AdamekHerrlichTholen1989monadic,Yanovski2022monadic}.
Applying our result to this case, we will show that repeatedly taking the Eilenberg--Moore categories is the most efficient way to present a given adjoint functor as a composition of monadic functors under suitable assumptions.

\vspace{1em}
\paragraph{\textbf{Organization of the paper}}
In \cref{sec:decomposition_number}, we will study basic properties of decomposition numbers, in particular comparing the minimum decomposition with the canonical decomposition.
In \cref{sec:generalized_regularity}, we will introduce the generalized notion of regular epimorphisms and study the decomposition numbers arising from them.
In \cref{sec:the_partial-Horn-theoretic_approach}, we will present a syntactic method for computing the global decomposition number of the category of models of a partial Horn theory.
In \cref{sec:for_models_of_dependent_algebraic_theories}, we will show that the global decomposition number can be bounded above by the maximal depth of dependencies among sorts for certain generalized algebraic theories.

It is worth noting that each section gets increasingly specialized in scope, starting from categories admitting locally orthogonal factorizations for a class $\mathbf{E}$ of morphisms in \cref{sec:decomposition_number}, those with $\mathbf{E}$ contained in the class of epimorphisms in \cref{sec:generalized_regularity}, locally presentable categories in \cref{sec:the_partial-Horn-theoretic_approach}, and locally finitely presentable categories in \cref{sec:for_models_of_dependent_algebraic_theories}.

\vspace{1em}
\paragraph{\textbf{Conventions and notations}}
\begin{remark}[Set-theoretic foundation]
    We fix three universes $\universe_1\in\universe_2\in\universe_3$ in this paper, whatever set-theoretic foundation we take.
    We will call elements of $\universe_1$ \emph{small sets}, those of $\universe_2$ \emph{large sets}, those of $\universe_3$ \emph{very large sets}, and \emph{classes} refer to collections with no restriction on size.
    Accordingly, we will call mathematical objects such as ordinals, cardinals, and categories in $\universe_1$ \emph{small}, those in $\universe_2$ \emph{large}, and those in $\universe_3$ \emph{very large}.
    Unless otherwise specified, categories are assumed to be large by default, and 2-categories are assumed to be very large.
\end{remark}

\begin{notation}[The smallest large ordinal]\label{note:smallest_large_ordinal}
    We define $\inacc$ (upsilon) to be the smallest ordinal not in $\universe_1$, which is automatically a cardinal.
    This cardinal coincides with the cardinality of the first universe $\universe_1$ in some set-theoretic foundations, for instance, when we take \acs{ZFC} as our foundation and $\universe_1$ to be a Grothendieck universe within it; see \cite[Theorem 2]{Williams1969grothendieck}.
\end{notation}

\begin{notation}[The class of epimorphisms]
    For a category $\C$, we write $\Epi[\C]$ for the class of all epimorphisms in $\C$, or simply $\Epi$ when $\C$ is clear from context.
\end{notation}

\begin{definition}
    A class $\bfE$ of morphisms in a category is called \emph{iso-closed} if it contains all isomorphisms and is closed under composition with isomorphisms.
\end{definition}

\begin{notation}
    For ordinal numbers $\alpha$ and $\beta$, we write $[\alpha,\beta]$ for the class of all ordinal numbers $\gamma$ such that $\alpha\le\gamma\le\beta$, $(\alpha,\beta)$ for the class of all ordinal numbers $\gamma$ such that $\alpha<\gamma<\beta$.
    The classes $(\alpha,\beta]$ and $[\alpha,\beta)$ are also defined similarly.
\end{notation}

\begin{notation}
    For an ordinal number $\alpha$, we will use the blackboard form $\bbalpha$ for the poset $\bbalpha=\{\beta\mid\beta<\alpha\}$ with the usual order.
    For instance, for $\beta,\gamma,\upsilon$, the corresponding posets are written $\bbbeta,\bbgamma,\bbupsilon$.
    The posets $\bbalpha,\bbbeta,\bbgamma,\bbupsilon$ are often regarded as categories.
\end{notation}

\begin{notation}
    Let $\alpha$ be an ordinal.
    For a functor $A_\bullet\colon\bbalpha\arr\C$ from $\bbalpha$ to a category $\C$ and ordinals $\beta\le\beta'<\alpha$, we write $A_{\beta,\beta'}$ for the image under $A_\bullet$ of the unique morphism $\beta\arr\beta'$ in $\bbalpha$.
    \begin{equation*}
        A_0\arr(A_{0,1})[][2]A_1\arr(A_{1,2})[][2]A_2\arr\cdots~ A_\beta\arr(A_{\beta,\beta'})[][2]A_{\beta'}\arr\cdots
        \incat{\C}
    \end{equation*}
\end{notation}
\section{The decomposition number}\label{sec:decomposition_number}
\subsection{The minimum length of decompositions of a morphism}\label{subsec:minimum_length}
We will introduce the notion of transfinite decompositions of a morphism, and will investigate the minimum length of them.
We will use the concept of orthogonality for morphisms; for notation and conventions, see \cref{sec:local_orthogonality}.
\begin{definition}[Transfinite sequences]
    Let $\alpha$ be an ordinal.
    A \emph{transfinite ($\alpha$-)sequence} in a category $\C$ is a functor $A_\bullet\colon\bbalpha\arr\C$ such that for every non-zero limit ordinal $\gamma\in(0,\alpha)$, $A_\gamma\cong\Colim_{\beta\in\bbgamma}A_\beta$ holds.
    For a transfinite $\alpha$-sequence $A_\bullet$ with $\alpha=\alpha'+1$, the morphism $A_{0,\alpha'}$ is called the \emph{(transfinite) composite} of $A_\bullet$.
\end{definition}

\begin{definition}[$\bfE$-transfinite sequences]
    Let $\bfE$ be a class of morphisms in a category $\C$.
    Let $\alpha$ be an ordinal.
    An \emph{$\bfE$-transfinite ($\alpha$-)sequence} in $\C$ is a transfinite $\alpha$-sequence such that for every ordinal $\beta$ with $\beta+1<\alpha$, $A_{\beta,\beta+1}\in\bfE$ holds.
\end{definition}

\begin{remark}
    For any $\bfE$-transfinite $\alpha$-sequence $A_\bullet$, the morphism $A_{\beta,\beta'}$ always belongs to $\lorth\rorth(\bfE)$ for any $\beta\le\beta'<\alpha$, because the class $\lorth\rorth(\bfE)$ is closed under composition and colimit.
\end{remark}

\begin{definition}[Predecompositions and decompositions]
    Let $\bfE$ be a class of morphisms in a category $\C$.
    Let $\gamma\in (0,\inacc]$ be an ordinal.
    \begin{enumerate}
        \item
            An \emph{$(\bfE,\gamma)$-predecomposition} $(A_\bullet,f_\bullet,X)$ in $\C$ consists of a cocone $f_\bullet = (f_\alpha\colon A_\alpha\arr X)_\alpha$ over an $\bfE$-transfinite $\gamma$-sequence $A_\bullet\colon\bbgamma\arr \C$.
            When $\gamma=\inacc$, we simply refer to it as an \emph{$\bfE$-predecomposition}.
            When $A_\bullet$ and $X$ are clear from context, the predecomposition $(A_\bullet,f_\bullet,X)$ is simply denoted by $f_\bullet$.
        \item
            We say that an $(\bfE,\gamma)$-predecomposition $(A_\bullet,f_\bullet,X)$ \emph{stabilizes} at an ordinal $\alpha<\gamma$ if $\alpha$ satisfies the equivalent conditions in \cref{lem:termination_predecomposition}.
            An $(\bfE,\gamma)$-predecomposition $f_\bullet$ is called an \emph{$(\bfE,\gamma)$-decomposition} if it stabilizes at some ordinal less than $\gamma$; the smallest such ordinal is called the \emph{length} of $f_\bullet$ and is denoted by $\abs{f_\bullet}$.
            \qedhere
    \end{enumerate}
\end{definition}

\begin{lemma}\label{lem:termination_predecomposition}
    Let $\bfE$ be a class of morphisms in a category $\C$.
    Let $(A_\bullet,f_\bullet,X)$ be an $(\bfE,\gamma)$-predecomposition.
    Then, the following conditions are equivalent for an ordinal $\alpha<\gamma$:
    \begin{enumerate}
        \item\label{lem:termination_predecomposition-1}
            For every $\beta\in [\alpha,\gamma)$, $\bfE\orth f_\beta$ holds.
        \item\label{lem:termination_predecomposition-2}
            $\bfE\orth f_\alpha$ holds, and
            for $\alpha\le\beta\le\beta'<\gamma$, the morphisms $A_{\beta,\beta'}$ are isomorphisms.
    \end{enumerate}
\end{lemma}
\begin{proof}
    \proofdirection{\cref{lem:termination_predecomposition-2}}{\cref{lem:termination_predecomposition-1}}
    Trivial.

    \proofdirection{\cref{lem:termination_predecomposition-1}}{\cref{lem:termination_predecomposition-2}}
    By $A_{0,\beta},A_{0,\beta'}\in\lorth\rorth(\bfE)$ and the cancellation property of $\lorth\rorth(\bfE)$, we have $A_{\beta,\beta'}\in\lorth\rorth(\bfE)$.
    On the other hand, by $f_\beta,f_{\beta'}\in\rorth(\bfE)$ and the cancellation property of $\rorth(\bfE)$, we have $A_{\beta,\beta'}\in\rorth(\bfE)$.
    Thus, $A_{\beta,\beta'}$ becomes an isomorphism.
\end{proof}

\begin{remark}
    By \cref{lem:termination_predecomposition}\cref{lem:termination_predecomposition-2}, an $(\bfE,\gamma)$-decomposition can naturally extend to an $(\bfE,\gamma')$-decomposition for any $\gamma'\in [\gamma,\inacc]$.
    Therefore, the particular choice of $\gamma$ is immaterial when considering decompositions.
\end{remark}

\begin{definition}[The category of predecompositions]
    Let $\bfE$ be a class of morphisms in a category $\C$, and let $\gamma\in (0,\inacc]$ be an ordinal.
    Let $\Predec[\bfE,\gamma]$ denote the category of $(\bfE,\gamma)$-predecompositions in $\C$, whose morphism $\phi\colon (A_\bullet,f_\bullet,X)\arr[][1] (B_\bullet,g_\bullet,Y)$ consists of a natural transformation $(\phi_\alpha)_\alpha \colon A_\bullet\arr[Rightarrow] B_\bullet$ and a morphism $\phi_\ast\colon X\arr Y$ such that the following diagram commutes for every $\alpha<\gamma$:
    \begin{equation*}
        \begin{tikzcd}
            A_\alpha\ar[d,"\phi_\alpha"']\ar[r,"f_\alpha"] & X\ar[d,"\phi_\ast"] \\
            B_\alpha\ar[r,"g_\alpha"'] & Y
        \end{tikzcd}\incat{\C}.
    \end{equation*}
    When $\gamma=\inacc$, we simply write $\Predec[\bfE]\coloneq\Predec[\bfE,\inacc]$.
\end{definition}

\begin{remark}
    We will refer to an $(\bfE,\gamma)$-(pre)decomposition $(A_\bullet,f_\bullet,X)$
    \begin{equation*}
        \begin{tikzcd}
            A_0\ar[dr,"A_{0,1}"']\ar[rrr,"f~(=f_0)"] &[-30pt] & &[30pt] X \\
            & A_1\ar[r,"A_{1,2}"']\ar[rru,"f_1"] & A_2\ar[ru,"f_2"']\ar[r,phantom,"\cdots"{pos=0.2}] & {}
        \end{tikzcd}
    \end{equation*}
    as an \emph{$(\bfE,\gamma)$-(pre)decomposition of $f$} where $f=f_0$.
\end{remark}

\begin{definition}[Minimum decomposition number]
    For a class $\bfE$ of morphisms in $\C$,
    \begin{enumerate}
        \item
            The \emph{minimum $\bfE$-decomposition number} of a morphism $f$ in $\C$, denoted by $\decnum[\bfE]{f}$, is the smallest ordinal $\alpha$ such that $f$ has an $\bfE$-decomposition of length $\alpha$.
            If $f$ has no $\bfE$-decompositions, its minimum $\bfE$-decomposition number is undefined.
        \item
            Suppose that every morphism $f$ in $\C$ has an $\bfE$-decomposition.
            Then, the \emph{global minimum $\bfE$-decomposition number} of the category $\C$, denoted by $\decnum[\bfE]{\C}$, is the smallest ordinal $\alpha$ such that $\decnum[\bfE]{f}<\alpha$ holds for every morphism $f$ in $\C$.\qedhere
    \end{enumerate}
\end{definition}

\begin{remark}\label{rem:why_strict_supremum}
    In the definition of the global minimum decomposition number $\decnum[\bfE]{\C}$, we adopt the least strict upper bound rather than the supremum.
    This choice allows us to distinguish between the case in which there is a morphism $f$ with $\decnum[\bfE]{f}=\alpha$ and it is the supremum, and the case in which there are morphisms having arbitrarily large minimum $\bfE$-decomposition number less than $\alpha$ but no morphism has decomposition number $\alpha$, where $\alpha$ is a limit ordinal.
    An example with $\decnum[\bfE]{\C}=\omega$ will be given in \cref{eg:global_decnum_is_omega}, while examples with $\decnum[\bfE]{\C}=\omega+1$ will appear in \cref{eg:global_minimum_decnum}\cref{eg:global_minimum_decnum-localizations,eg:global_minimum_decnum-antiLipschitz}, \cref{eg:global_decnum_is_omegaplusone}, and \cref{rem:global_decnum_of_omegaCat}.
\end{remark}

Our motivating example is the case where the class $\bfE$ is taken to be that of all regular epimorphisms.
However, there are different definitions of regular (and strong) epimorphisms in the literature (e.g., \cite{Kelly1969mono}), which are not necessarily equivalent in general, so we clarify below which one we adopt:
\begin{definition}\quad
    \begin{enumerate}
        \item
            A morphism in a category is called a \emph{regular epimorphism} if it is the coequalizer of some parallel pair of morphisms.
        \item
            A morphism $f$ in a category is called a \emph{strong epimorphism}\footnote{Strong epimorphisms may fail to be epic unless the category has binary powers.} if $f\orth m$ holds for every monomorphism $m$.\qedhere
    \end{enumerate} 
\end{definition}

\begin{definition}[Regular-decomposition number]
    When $\bfE$ is the class of all regular epimorphisms, we simply write $\decnum{f}\coloneq\decnum[\bfE]{f}$ and $\decnum{\C}\coloneq\decnum[\bfE]{\C}$ and called them the \emph{(global) minimum regular-decomposition number}.
    Similarly, we also use the terms \emph{regular transfinite sequences} and \emph{regular (pre)decompositions}.
\end{definition}

The following is well-known:
\begin{lemma}
    Let $\bfE$ be the class of all regular epimorphisms in a category $\C$.
    Then, every monomorphism in $\C$ belongs to $\rorth(\bfE)$.
    The converse also holds whenever $\C$ has binary copowers.
\end{lemma}

Therefore, whenever we consider a category with binary copowers, a regular decomposition of a morphism is precisely a decomposition into a transfinite composite of regular epimorphisms followed by a monomorphism.
\begin{example}\label{eg:global_minimum_decnum_for_regular_epi}
    We present several examples of the global minimum regular-decomposition number of categories:
    \begin{enumerate}
        \item
            The empty category $\initialcat$ is the only example with $\decnum{\initialcat} = 0$.
        \item
            For a non-empty category in which every morphism is monic, its global minimum regular-decomposition number is $1$.
            Thus, we have $\decnum{\terminalcat}=1$ and $\decnum{\Field}=1$, where $\terminalcat$ denotes the terminal category, and where $\Field$ denotes the category of (small) fields.
        \item
            For the categories of (small) sets, groups, and copresheaves, we have $\decnum{\Set}=2$, $\decnum{\Grp}=2$, and $\decnum{\Set^\D}\le 2$, since all of these are regular categories.
        \item
            More generally, even without assuming that the category is regular, for any category in which every morphism factors as a regular epimorphism followed by a monomorphism, its global minimum regular-decomposition number is less than or equal to $2$.
            For instance, although the category $\Pos$ of posets is not regular, it satisfies $\decnum{\Pos}=2$.
        \item
            For the category $\Cat$ of (small) categories, we have $\decnum{\Cat}=3$, which means that every functor factors as the composite of two regular epifunctors followed by a subcategory inclusion \cite[4.5.\ Proposition]{BednarczykBorzyszkowskiPawlowski1999epi}, and that there exists a functor that requires two regular epifunctors to decompose itself.
            In fact, for the category $\nCat$ of (small) strict $n$-categories and strict $n$-functors, we have $\decnum{\nCat}=n+2$ in general, as we see later in \cref{thm:global_decnum_of_nCat} and \cref{eg:global_decnum_by_clanapproach}\cref{eg:global_decnum_by_clanapproach-nCat}.\qedhere
    \end{enumerate}
\end{example}

\begin{example}\label{eg:global_minimum_decnum}
    The following are examples of the global minimum decomposition number in cases where $\bfE$ is not the class of regular epimorphisms:
    \begin{enumerate}
        \item\label{eg:global_minimum_decnum-surjections}
            Let $\bfE$ be the class of all surjections in $\Pos$, the category of posets.
            Then, the class $\rorth(\bfE)$ coincides with the class of order-reflecting injections in $\Pos$, and we have $\decnum[\bfE]{\Pos}=2$.
            That is, every morphism in $\Pos$ can be decomposed into a surjection followed by an order-reflecting subposet.
            This can immediately be generalized to an arbitrary category with an orthogonal factorization system $(\bfE,\rorth(\bfE))$.
        \item\label{eg:global_minimum_decnum-localizations}
            Let $\bfE$ be the class of all \textit{strict localizations}, functors that arise as strict localizations of some class of morphisms in the domain category.
            Then, the category $\CAT$ of (large) categories admits an orthogonal factorization system $(\lorth\rorth(\bfE),\rorth(\bfE))$, where $\rorth(\bfE)$ coincides with the class of all conservative functors.
            In fact, the left class $\lorth\rorth(\bfE)$ coincides with the class of \textit{iterated strict localizations}, functors obtained by the colimit of a sequence of strict localizations, and we have $\decnum[\bfE]{\CAT}=\omega+1$.
            See \cref{eg:iterated_localization_decnum_in_Cat} for the proof.
        \item\label{eg:global_minimum_decnum-antiLipschitz}
            Let $\Met$ denote the category of (usual) metric spaces and non-expansive maps.
            For $0<k<1$, we refer to a non-expansive map $f\colon X\arr Y$ between metric spaces as an \emph{anti-$k$-Lipschitz quotient map} if it is bijective and $k\cdot d_X(x,x')\le d_Y(f(x),f(x'))$ holds for any $x,x'\in X$, and let $\bfE$ be the class of all such morphisms in $\Met$.
            Then, the class $\rorth(\bfE)$ coincides with that of all isometries.
            Since every non-expansive map can be factorized as a surjective one followed by an isometry, and since every surjective non-expansive map can be presented as a colimit of a countable sequence of anti-$k$-Lipschitz quotient maps, we have $\decnum[\bfE]{\Met}=\omega+1$.
        \item\label{eg:global_minimum_decnum-monotone}
            Let $\Top$ be the category of topological spaces and continuous maps.
            A quotient map such that every fiber is connected is called a \textit{monotone quotient map}, and a continuous map such that every fiber is totally disconnected is called a \textit{light map} in the literature.
            Let $\bfE$ be the class of all monotone quotient maps.
            Then, the class $\rorth(\bfE)$ coincides with that of all light maps, and we have $\decnum[\bfE]{\Top}=\inacc$, the smallest large ordinal (cf.\ \cref{note:smallest_large_ordinal}).
            See \cref{eg:monotone_example_by_koizumi} for the proof.\qedhere
    \end{enumerate}
\end{example}

We now explore fundamental properties for the minimum decomposition numbers.
\begin{lemma}\label{lem:decomposition_of_leftclass}
    Let $\bfE$ be a class of morphisms in a category $\C$.
    Let $(A_\bullet,f_\bullet,X)$ be an $\bfE$-decomposition of a morphism $f\in\lorth\rorth(\bfE)$.
    Then, for every $\alpha\in [\abs{f_\bullet},\inacc)$, the morphism $f_\alpha$ is an isomorphism.
\end{lemma}
\begin{proof}
    Let $\alpha\in [\abs{f_\bullet},\inacc)$ be an ordinal.
    Then, we have $f_\alpha\in\rorth(\bfE)$ by the definition of the length $\abs{f_\bullet}$.
    Since the class $\lorth\rorth(\bfE)$ is closed under transfinite composition, the morphism $A_{0,\alpha}$ belongs to $\lorth\rorth(\bfE)$.
    By the cancellation property of $\lorth\rorth(\bfE)$, we have $f_\alpha\in\lorth\rorth(\bfE)$, which implies that $f_\alpha$ is an isomorphism.
\end{proof}

The following proposition is useful when one tries to determine the global minimum decomposition numbers.
It shows that it often suffices to consider morphisms in the class $\lorth\rorth(\bfE)$, because postcomposing a morphism in $\rorth(\bfE)$ does not affect the minimum decomposition number so this class does not contribute to the global minimum decomposition number.
\begin{proposition}\label{prop:decnum_of_composite}
    Let $\bfE$ be a class of morphisms in a category $\C$.
    Consider a composable pair of morphisms
    \begin{equation*}
        \begin{tikzcd}
            X\ar[r,"f"] & Y\ar[r,"g"] & Z
        \end{tikzcd}\incat{\C}.
    \end{equation*}
    \begin{enumerate}
        \item\label{prop:decnum_of_composite-precompose}
            Suppose $f\in\lorth\rorth(\bfE)$.
            If $\decnum[\bfE]{f}$ and $\decnum[\bfE]{g}$ are defined, then $\decnum[\bfE]{g\circ f}$ is defined and $\decnum[\bfE]{g\circ f}\le \decnum[\bfE]{f}+\decnum[\bfE]{g}$ holds.
            Here, the symbol $+$ denotes the addition of ordinals.
        \item\label{prop:decnum_of_composite-postcompose}
            Suppose $g\in\rorth(\bfE)$.
            Then, $\decnum[\bfE]{f}$ is defined if and only if $\decnum[\bfE]{g\circ f}$ is defined, and whenever they are defined, $\decnum[\bfE]{f}=\decnum[\bfE]{g\circ f}$ holds.
    \end{enumerate}
\end{proposition}
\begin{proof}\quad
    \begin{enumerate}
        \item
            Let $(A_\bullet, f_\bullet, Y)$ and $(B_\bullet, g_\bullet, Z)$ be arbitrary $\bfE$-decompositions of $f$ and $g$, respectively.
            Since the morphisms $f_\alpha$ for $\alpha\in [\abs{f_\bullet},\inacc)$ are isomorphisms by \cref{lem:decomposition_of_leftclass}, they can be supposed to be identities without loss of generality.
            Then, the following yields an $\bfE$-decomposition $(C_\bullet, h_\bullet, Z)$ of the composite $h\coloneq gf$:
            \begin{equation*}
                C_\alpha\coloneq
                \begin{cases}
                    A_\alpha & \text{if $\alpha < \abs{f_\bullet}$} \\
                    B_{\alpha - \abs{f_\bullet}} & \text{if $\alpha \ge \abs{f_\bullet}$}
                \end{cases}
                \qquad
                h_\alpha\coloneq
                \begin{cases}
                    gf_\alpha & \text{if $\alpha < \abs{f_\bullet}$} \\
                    g_{\alpha - \abs{f_\bullet}} & \text{if $\alpha \ge \abs{f_\bullet}$}
                \end{cases}
            \end{equation*}
            Here, $\alpha - \abs{f_\bullet}$ denotes a unique ordinal $\gamma$ satisfying $\abs{f_\bullet}+\gamma = \alpha$.
            Then, $\decnum[\bfE]{h} \le \abs{h_\bullet} = \abs{f_\bullet}+\abs{g_\bullet}$ follows, which finishes the proof.
        \item
            Suppose that $\decnum[\bfE]{f}$ is defined, and take an $\bfE$-decomposition $(A_\bullet, f_\bullet, Y)$ of $f$.
            Then, it is easy to see that $(A_\bullet, gf_\bullet, Z)$ is still an $\bfE$-decomposition of $h\coloneq gf$.

            For converse, suppose that $\decnum[\bfE]{gf}$ is defined, and take an $\bfE$-decomposition $(C_\bullet, h_\bullet, Z)$ of $h=gf$.
            Since $\bfE\orth g$, there exists a unique morphism $k_\alpha$ making the following diagram commute for each $\alpha<\inacc$:
            \begin{equation*}
                \begin{tikzcd}
                    X\ar[d,"C_{0,\alpha}"']\ar[r,"f"] & Y\ar[d,"g"] \\
                    C_\alpha\ar[r,"h_\alpha"']\ar[ru,dotted,"k_\alpha"] & Z
                \end{tikzcd}
            \end{equation*}
            Then, $(C_\bullet, k_\bullet, Y)$ becomes an $\bfE$-decomposition of $f$, which finishes the proof.\qedhere
    \end{enumerate}
\end{proof}

The following two lemmas assert that the minimum decomposition number of a given morphism can be, to some extent, inferred from that of another morphism:
\begin{lemma}\label{lem:descending_decnum}
    Let $\bfE$ be a class of morphisms in a category $\C$ such that $\bfE\subseteq\Epi[\C]$.
    Let $\phi\colon (A_\bullet,f_\bullet,X) \arr[] (B_\bullet,g_\bullet,Y)$ be a morphism of $(\bfE,\gamma)$-predecompositions such that $\phi_0\in \lorth\rorth(\bfE)$ and $\phi_\ast\in\rorth(\bfE)$.
    If $f_\bullet$ stabilizes, then so does $g_\bullet$, and $\abs{f_\bullet}\ge\abs{g_\bullet}$ holds.
\end{lemma}
\begin{proof}
    Note that the assumption $\bfE\subseteq\Epi$ leads to the strong cancellation property of $\rorth(\bfE)$ (cf.\ \cref{lem:leftclass_epic}).
    Suppose that $f_\bullet$ is an $\bfE$-decomposition, and let $\alpha_0<\gamma$ be the length of $f_\bullet$.
    Take $\beta\in [\alpha_0,\gamma)$ arbitrarily.
    By $g_{\beta}\circ \phi_\beta = \phi_\ast \circ f_{\beta}\in\rorth(\bfE)$ and the strong cancellation property of $\rorth(\bfE)$, we have $\phi_\beta\in\rorth(\bfE)$.
    On the other hand, we can easily show $\phi_\alpha\in\lorth\rorth(\bfE)$ for every $\alpha<\gamma$ by transfinite induction using the cancellation property of $\lorth\rorth(\bfE)$.
    In particular, $\phi_\beta$ is an isomorphism.
    Thus, we have $g_{\beta}=\phi_\ast\circ f_{\beta}\circ \phi_\beta^{-1}\in\rorth(\bfE)$, which finishes the proof.
\end{proof}

\begin{lemma}\label{lem:pushing_of_decomposition}
    Let $\C$ be a category with pushouts.
    Let $\bfE$ be a class of morphisms in $\C$ with $\bfE\subseteq\Epi[\C]$ that is stable under pushouts.
    Suppose that we are given the following commutative square with $e\in\lorth\rorth(\bfE)$ and $m\in\rorth(\bfE)$:
    \begin{equation*}
        \begin{tikzcd}[scriptsize]
            A\ar[d,"e"']\ar[r,"f"] & X\ar[d,"m"] \\
            B\ar[r,"g"'] & Y
        \end{tikzcd}\incat{\C}.
    \end{equation*}
    If $\decnum[\bfE]{f}$ is defined, then $\decnum[\bfE]{g}$ is also defined and $\decnum[\bfE]{f}\ge\decnum[\bfE]{g}$ holds.
\end{lemma}
\begin{proof}
    By $e\orth m$, we can take a unique morphism $h\colon B\arr X$ such that $he=f$ and $mh=g$ hold.
    By \cref{prop:decnum_of_composite}\cref{prop:decnum_of_composite-postcompose}, it suffices to show that $\decnum[\bfE]{h}$ is defined and $\decnum[\bfE]{f}\ge\decnum[\bfE]{h}$ holds.

    Let $(A_\bullet,f_\bullet,X)$ be an $\bfE$-decomposition of $f$ such that $\abs{f_\bullet}=\decnum[\bfE]{f}$.
    Taking pushouts of $f$ and morphisms in the sequence $A_\bullet$ along $e$, we obtain an $\bfE$-predecomposition $(B_\bullet,g_\bullet,Z)$ and a morphism $\phi\colon f_\bullet\arr g_\bullet$ in $\Predec[\bfE]$ as follows:
    \begin{equation*}
        \begin{tikzcd}
            (A=) &[-35] A_0\ar[d,"(e=)~\phi_0"']\ar[rr,"f",bend left=20]\ar[r,"A_{0,\alpha}"'] & A_\alpha\ar[d,"\phi_\alpha"]\ar[r,"f_\alpha"'] & X\ar[d,"\phi_\ast"']\ar[ddr,equal,bend left=15] & \\
            (B=) & B_0\ar[r,"B_{0,\alpha}"]\ar[drrr,"h"',bend right=15] & B_\alpha\ar[r,"g_\alpha"]\pushoutcorner & Z\ar[rd,dotted,"r"']\pushoutcorner & \\
            &&&& X
        \end{tikzcd}\incat{\C}\quad (0\le\alpha<\inacc).
    \end{equation*}
    Here, $r$ denotes a unique morphism making the diagram commute.
    Then, $\phi_\ast$ belongs to $\rorth(\bfE)$ as it is a split monomorphism (see \cref{lem:leftclass_epic}).
    On the other hand, since $e=\phi_0$ belongs to $\lorth\rorth(\bfE)$, so does its pushout $\phi_\ast$.
    This implies that $\phi_\ast$ and $r$ are mutually inverses.
    Thus, without loss of generality, we can assume that both $\phi_\ast$ and $r$ are the identities, and hence $X=Z$.
    Then, $g_\bullet$ becomes an $\bfE$-predecomposition of $h$.
    Applying \cref{lem:descending_decnum} for $\phi$, we can show that $g_\bullet$ stabilizes and $\decnum[\bfE]{f}=\abs{f_\bullet}\ge\abs{g_\bullet}\ge\decnum[\bfE]{h}$ holds, which finishes the proof.
\end{proof}

\subsection{The canonical decomposition by locally orthogonal factorizations}\label{subsec:The_canonical_decomposition_by_locally_orthogonal_factorizations}
In this subsection, we will focus on a ``greedy method'' for decomposing a morphism, and compare the length of the resulting decomposition with the minimum decomposition number.
We will use the notion of \textit{locally orthogonal factorizations} studied in \cite{Tholen1979semitopological,MacdonaldTholen1982decomposition, Tholen1983factorizations}.
For notation and conventions, we refer the reader to the minimal summary in \cref{sec:local_orthogonality}.
\begin{definition}[Canonical predecompositions]
    Let $\bfE$ be a class of morphisms in a category $\C$, and let $\gamma\in (0,\inacc]$ be an ordinal.
    An $(\bfE,\gamma)$-predecomposition $(A_\bullet,f_\bullet,X)$ is called \emph{canonical} if for every $\alpha$ such that $\alpha+1<\gamma$, $\bfE\orth[A_{\alpha,\alpha+1}] f_{\alpha+1}$ holds (cf.\ \cref{note:local_orthogonality}).
\end{definition}

\begin{example}[Coequalizers of kernel pairs]\label{eg:canonical_predecomposition_for_regular_epis}
    A regular predecomposition $(A_\bullet,f_\bullet,X)$ in a category with kernel pairs is canonical if and only if for every small ordinal $\alpha$, the morphism $A_{\alpha,\alpha+1}$ exhibits the coequalizer of the kernel pair of $f_\alpha$.
    This equivalence will later be generalized, see \cref{rem:canonical_predecomp_for_GRegEpi}.
\end{example}

\begin{example}[Monotone quotient--light factorizations]\label{eg:canonical_predecomposition_for_monotone_maps}
    Let $\bfE$ be the class of all monotone quotient maps between topological spaces as in \cref{eg:global_minimum_decnum}\cref{eg:global_minimum_decnum-monotone}.
    According to \cite[{}2.2]{MacdonaldTholen1982decomposition}, an $\bfE$-predecomposition $(A_\bullet,f_\bullet,X)$ in $\Top$ is canonical if and only if for every small ordinal $\alpha$, the morphism $A_{\alpha,\alpha+1}$ is the quotient map under the equivalence relation $\sim$ on $A_\alpha$ defined as follows: $x\sim y$ in $A_\alpha$ if $f_\alpha(x)=f_\alpha(y)$ holds and $x,y$ are contained in a common connected subspace $K\subseteq A_\alpha$ such that $f_\alpha(K)$ is a constant.
\end{example}

\begin{remark}
    Canonical $(\bfE,\gamma)$-predecompositions are already studied in \cite{MacdonaldTholen1982decomposition} under the name \textit{locally orthogonal $(\bfE,\gamma)$-factorizations}.
\end{remark}

A canonical predecomposition has the following universal property, which also appears in \cite[Section~3]{MacdonaldTholen1982decomposition}.
\begin{proposition}[The universal property of a canonical predecomposition]\label{prop:canonical_decomp_adjointprop}
    Let $\bfE$ be a class of morphisms in a category $\C$, and let $(A_\bullet,f_\bullet,X)$ and $(B_\bullet,g_\bullet,Y)$ be $(\bfE,\gamma)$-predecompositions in $\C$.
    If $g_\bullet$ is canonical, the following data correspond bijectively:
    \begin{enumerate}
        \item
            A morphism $f_0\arr g_0$ in the arrow category $\C^\to$.
        \item
            A morphism $f_\bullet\arr g_\bullet$ in $\Predec[\bfE,\gamma]$.
    \end{enumerate}
\end{proposition}
\begin{proof}
    For any morphism $\phi\colon f_\bullet\arr g_\bullet$, each $\phi_\alpha$ is uniquely determined by $\phi_\beta$ for $\beta<\alpha$ and $\phi_\ast$.
    Indeed, if $\alpha=\beta+1$ is a successor, the local orthogonality $A_{\beta,\beta+1}\orth[B_{\beta,\beta+1}] g_{\beta+1}$ uniquely determines $\phi_\alpha=\phi_{\beta+1}$; if $\alpha$ is a limit ordinal, the universal property of $A_\alpha$ as $\Colim_{\beta\in\bbalpha}A_\beta$ uniquely determines $\phi_\alpha$.
    This shows that every morphism $(\phi_0,\phi_\ast)\colon f_0\arr g_0$ in $\C^\to$ uniquely extends to a morphism $f_\bullet\arr g_\bullet$ in $\Predec[\bfE,\gamma]$.
\end{proof}

The proposition above says that canonical decompositions, if they exist, yield a right adjoint of the composition functor $\Predec[\bfE,\gamma]\arr \C^\to$ as in \cref{cor:cdec_is_right_adjoint}.
With that in mind, the following corollary can also be considered as a consequence of the uniqueness of right adjoints:
\begin{corollary}
    Let $\bfE$ be a class of morphisms in a category $\C$.
    Then, for every morphism in $\C$, its canonical $(\bfE,\gamma)$-predecomposition is unique up to isomorphism if exists.
\end{corollary}
\begin{proof}
    By \cref{prop:canonical_decomp_adjointprop}, the identity on $f$ in $\C^\to$ uniquely extends an isomorphism between canonical predecompositions of $f$.
\end{proof}

\begin{definition}[Canonical decomposition number]
    Let $\bfE$ be a class of morphisms in a category $\C$.
    \begin{enumerate}
        \item
            The \emph{canonical $\bfE$-decomposition number} of a morphism $f$ in $\C$, denoted by $\cdecnum[\bfE]{f}$, is the length of the canonical $\bfE$-decomposition of $f$ if it exists.
        \item
            Suppose that every morphism $f$ in $\C$ has the canonical $\bfE$-decomposition.
            Then, the \emph{global canonical $\bfE$-decomposition number} of the category $\C$, denoted by $\cdecnum[\bfE]{\C}$, is the smallest ordinal $\alpha$ such that $\cdecnum[\bfE]{f}<\alpha$ holds for every morphism $f$ in $\C$.\qedhere
    \end{enumerate}
\end{definition}

\begin{remark}
    If $\cdecnum[\bfE]{f}$ is defined, then $\decnum[\bfE]{f}$ is defined and $\decnum[\bfE]{f}\le\cdecnum[\bfE]{f}$ holds, because $\decnum[\bfE]{f}$ is defined as the smallest length of $\bfE$-decompositions.
    In fact, if every morphism in $\bfE$ is epic, the two ordinals $\decnum[\bfE]{f}$ and $\cdecnum[\bfE]{f}$ coincide (\cref{thm:correspondence_decnum_and_cdecnum}).
    In such a case, we simply refer to these common ordinals $\decnum[\bfE]{f}$ and $\cdecnum[\bfE]{f}$ as the \emph{$\bfE$-decomposition number} of $f$.
\end{remark}

\begin{remark}\label{rem:decnum_in_literature}
    The canonical $\bfE$-decomposition numbers in the case where $\bfE$ is the class of regular epimorphisms have been studied under the name \textit{decomposition number} or \textit{regular length} in the literature (e.g., \cite{GabrielUlmer1971lokal,MacdonaldStone1982tower,Borger1991making}).
    However, the definitions adopted there differ among themselves, and we also slightly modify them in our definition.
    One of the modifications lies in the definition of the global canonical decomposition number.
    Whereas the existing definitions take $\cdecnum[\bfE]{\C}\coloneq\sup_f \cdecnum[\bfE]{f}$ with the usual supremum, we define the global canonical decomposition number so that $\cdecnum[\bfE]{\C}=\sup_f ( \cdecnum[\bfE]{f}+1)$ holds, for the same reason as what is mentioned in \cref{rem:why_strict_supremum}.
\end{remark}

\begin{definition}[$\bfE$-admissible categories]
    Let $\bfE$ be an iso-closed class of morphisms in a category $\C$.
    The category $\C$ is called \emph{$\bfE$-admissible} if:
    \begin{itemize}
        \item
            $\C$ admits locally orthogonal $\bfE$-factorizations (cf.\ \cref{def:locally_orthogonal_factorization}).
        \item
            Every $\bfE$-transfinite sequence $\bbalpha\arr\C$ admits a colimit for any small ordinal $\alpha$.\qedhere
    \end{itemize}
\end{definition}

\begin{proposition}[{\cite[Theorem~3.3]{MacdonaldTholen1982decomposition}}]\label{prop:canonical_predecomposition}
    Let $\C$ be an $\bfE$-admissible category for an iso-closed class $\bfE$ of morphisms in $\C$.
    Then, every morphism in $\C$ has the canonical $\bfE$-predecomposition.
\end{proposition}
\begin{proof}
    The desired predecomposition can be constructed by transfinite recursion.
    Take an arbitrary morphism in $\C$ and write it as $f_0\colon A_0\arr X$.
    In each successor step, we can take $(A_{\alpha+1},A_{\alpha,\alpha+1},f_{\alpha+1})$ as the locally orthogonal $\bfE$-factorization of $f_\alpha$:
    \begin{equation*}
        \begin{tikzcd}[small]
            A_\alpha\ar[dr,"A_{\alpha,\alpha+1}"']\ar[rr,"f_\alpha"] && X \\
            & A_{\alpha+1}\ar[ru,"f_{\alpha+1}"'] &
        \end{tikzcd}\incat{\C},
        \quad
        A_{\alpha,\alpha+1}\in\bfE,
        \quad
        \bfE\orth[A_{\alpha,\alpha+1}]f_{\alpha+1}.
    \end{equation*}
    In each limit step, the desired data is given by the universal property of the colimit.
\end{proof}

\begin{corollary}\label{cor:cdec_is_right_adjoint}
    Let $\C$ be an $\bfE$-admissible category for an iso-closed class $\bfE$ of morphisms in $\C$.
    Then, there is an adjunction:
    \begin{equation*}
        \begin{tikzcd}
            \Predec[\bfE]\ar[rr,shift left=3,"\text{composition}"] & \perp & \C^\to\ar[ll,shift left=3,"\text{canonical predecomposition}"]
        \end{tikzcd}
    \end{equation*}
    Here, the left adjoint sends an $\bfE$-predecomposition $(A_\bullet,f_\bullet,X)$ to the morphism $f_0$; the right adjoint sends a morphism in $\C$ to its canonical $\bfE$-predecomposition.
\end{corollary}

The following lemma provides a simple criterion (\cref{cor:simple_criterion_of_cdec_stabilizing}) for when a canonical predecomposition stabilizes:
\begin{lemma}\label{lem:termination_canonical_predecomp}
    Let $\bfE$ be a class of morphisms in a category $\C$, and let $\gamma\in (0,\inacc]$ be an ordinal.
    Let $(A_\bullet,f_\bullet,X)$ be a canonical $(\bfE,\gamma)$-predecomposition.
    \begin{enumerate}
        \item\label{lem:termination_canonical_predecomp-1}
            Let $\alpha$ be an ordinal such that $\alpha+1<\gamma$.
            Then, the morphism $A_{\alpha,\alpha+1}$ is an isomorphism if and only if $\bfE\orth f_\alpha$ holds.
        \item\label{lem:termination_canonical_predecomp-2}
            Let $\alpha\le\beta<\gamma$ be ordinals.
            Then, $\bfE\orth f_\alpha$ implies $\bfE\orth f_\beta$.\qedhere
    \end{enumerate}
\end{lemma}\pagebreak
\begin{proof}\quad
    \begin{enumerate}
        \item
            Suppose $A_{\alpha,\alpha+1}$ is an isomorphism.
            Then, $\bfE\orth f_\alpha$ directly follows from $\bfE\orth[A_{\alpha,\alpha+1}]f_{\alpha+1}$.
            Conversely, let us suppose $\bfE\orth f_\alpha$.
            By $A_{\alpha,\alpha+1}\orth f_\alpha$, there is a unique morphism $h$ such that the diagram on the left below commutes.
            Then, by $A_{\alpha,\alpha+1}\orth[A_{\alpha,\alpha+1}] f_{\alpha+1}$ and the commutative diagram on the right below, we have $A_{\alpha,\alpha+1}\circ h=\id$, which shows that $A_{\alpha,\alpha+1}$ and $h$ are mutually inverse.\vspace{-1.2em}
            \begin{equation*}
                \begin{tikzcd}[large]
                    A_\alpha\ar[d,"A_{\alpha,\alpha+1}"']\ar[r,equal] & A_\alpha\ar[d,"f_\alpha"] \\
                    A_{\alpha+1}\ar[ru,dotted,"h"]\ar[r,"f_{\alpha+1}"'] & X
                \end{tikzcd}
                \quad
                \begin{tikzcd}
                    A_\alpha\ar[dd,"A_{\alpha,\alpha+1}"']\ar[rd,equal]\ar[rr,"A_{\alpha,\alpha+1}"] && A_{\alpha+1}\ar[dd,"f_{\alpha+1}"] \\
                    & A_\alpha\ar[ru,"A_{\alpha,\alpha+1}"{description}]\ar[rd,"f_\alpha"{description}] & \\
                    A_{\alpha+1}\ar[ru,"h"]\ar[rr,"f_{\alpha+1}"'] && X
                \end{tikzcd}
                \incat{\C}
            \end{equation*}
        \item
            Suppose $\bfE\orth f_\alpha$.
            It suffices to show that $A_{\alpha,\beta}$ is an isomorphism for every $\beta\in [\alpha,\gamma)$.
            Indeed, if $A_{\alpha,\beta}$ is an isomorphism, then $f_\beta =f_\alpha\circ A_{\alpha,\beta}^{-1}\in\rorth(\bfE)$.
            This follows from transfinite induction on $\beta$.
            The successor step follows from \cref{lem:termination_canonical_predecomp-1} straightforwardly, and the limit step follows from the universal property of the colimit $A_\beta\cong\Colim_{\beta'\in\bbbeta}A_{\beta'}$.\qedhere
    \end{enumerate}
\end{proof}

\begin{corollary}[Criterion for stabilization]\label{cor:simple_criterion_of_cdec_stabilizing}
    Let $\bfE$ be a class of morphisms in a category, and let $\gamma\in (0,\inacc]$ be an ordinal.
    The following are equivalent for a canonical $(\bfE,\gamma)$-predecomposition $(A_\bullet,f_\bullet,X)$ and an ordinal $\alpha$ such that $\alpha+1<\gamma$:
    \begin{enumerate}
        \item
            $f_\bullet$ stabilizes at $\alpha$.
        \item
            The morphism $A_{\alpha,\alpha+1}$ is an isomorphism.
        \item
            The orthogonality $\bfE\orth f_\alpha$ holds.
    \end{enumerate}
\end{corollary}

The following theorem guarantees that the canonical predecomposition is the most efficient way to decompose a morphism into an $\bfE$-transfinite sequence:\vspace{-0.4em}
\begin{theorem}[Minimum vs canonical I]\label{thm:correspondence_decnum_and_cdecnum}
    Let $\bfE$ be a class of morphisms in a category with $\bfE\subseteq\Epi$.
    Let $f$ be a morphism whose canonical $\bfE$-predecomposition exists.
    \begin{enumerate}
        \item
            $\decnum[\bfE]{f}$ is defined if and only if $\cdecnum[\bfE]{f}$ is defined.
        \item
            $\decnum[\bfE]{f}=\cdecnum[\bfE]{f}$ holds whenever they are defined.
    \end{enumerate}
\end{theorem}
\begin{proof}
    Let $(A_\bullet,f_\bullet,X)$ be an $\bfE$-decomposition of $f$, and let $(B_\bullet,g_\bullet,X)$ be the canonical $\bfE$-predecomposition of $f$.
    By \cref{prop:canonical_decomp_adjointprop}, we have a unique morphism $\phi\colon f_\bullet\arr g_\bullet$ such that $\phi_0$ and $\phi_\ast$ are the identities.
    Then, \cref{lem:descending_decnum} shows that $\cdecnum[\bfE]{f}$ is defined and $\cdecnum[\bfE]{f}\le\abs{f_\bullet}$.
    The decomposition $f_\bullet$ is arbitrary, so we find that $\cdecnum[\bfE]{f}=\decnum[\bfE]{f}$.
\end{proof}

\begin{example}[The global decomposition number for monotone quotient maps]\label{eg:monotone_example_by_koizumi}
    Consider the class $\bfE$ of all monotone quotient maps as in \cref{eg:global_minimum_decnum}\cref{eg:global_minimum_decnum-monotone}.
    Then, the category $\Top$ of topological spaces is $\bfE$-admissible, because it is cocomplete and locally orthogonal $\bfE$-factorizations are given by taking quotient under the equivalence relation described in \cref{eg:canonical_predecomposition_for_monotone_maps}.
    Hence, \cref{prop:canonical_predecomposition} implies that every morphism in $\Top$ has the canonical $\bfE$-predecomposition, which in fact stabilizes at some small ordinal because of the co-wellpoweredness of $\Top$.
    Since every monotone quotient map is epic in $\Top$, \cref{thm:correspondence_decnum_and_cdecnum} can be applied, and there is no difference between the minimum $\bfE$-decomposition number and the canonical one in this case.

    In what follows, we construct a continuous map $f$ with $\decnum[\bfE]{f}=\cdecnum[\bfE]{f}=2\lambda$ for an arbitrary small ordinal $\lambda$, which shows that the global $\bfE$-decomposition number of $\Top$ is $\inacc$, the smallest large ordinal (cf.\ \cref{note:smallest_large_ordinal}).
    The construction given here was communicated to us by Junnosuke Koizumi during the preparation of this paper.
    
    We extend the set of ordinals by adding a new element $-1$ that is smaller than every ordinal just for here, and use the interval notation such as $[-1,\lambda)$ to denote sets of ordinals possibly containing $-1$.
    Let $\{a,b\}$ be the two-element set, and let $X\coloneq [-1,\lambda)\times \{a,b\}$ be the product set, in which we write $\lelm{\alpha}$ for the ordered pair of $\alpha$ and $a$, and $\relm{\alpha}$ for that of $\alpha$ and $b$.
    We can equip $X$ with a topology in the following way: a subset $U\subseteq X$ is open if and only if the following conditions are satisfied.
    \begin{itemize}
        \item
            The condition $\lelm{-1}\in U$ is equivalent to $\relm{-1}\in U$.
        \item
            If $\lelm{\alpha}\in U$ for some $\alpha\in [0,\lambda)$, then there exists $\beta\in [-1,\alpha)$ such that $[\beta,\alpha]\times \{b\}\subseteq U$ holds.
    \end{itemize}
    Let $Y\coloneq\{a,b\}$ be the two-element indiscrete space, and let $(X_\bullet,f_\bullet,Y)$ be the canonical $\bfE$-decomposition of the second projection $f\colon X\arr Y$.
    
    Let $\alpha\le 2\lambda$ be an ordinal, and let $\beta$ and $k$ be the unique pair of ordinals satisfying $\alpha=2\beta + k$ with $k=0,1$.
    Let $\sim_\alpha$ be the kernel equivalence relation of the quotient map $X_{0,\alpha}\colon X\arr X_\alpha$, i.e., $x\sim_\alpha x'$ if and only if $X_{0,\alpha}(x)=X_{0,\alpha}(x')$.
    This relation can completely be determined as follows:
    \begin{claim}
        Both $[-1,\beta + k)\times\{a\}$ and $[-1,\beta)\times\{b\}$ are equivalence classes under $\sim_\alpha$, and the other equivalence classes are singletons.
    \end{claim}
    \begin{since}
        We show this by transfinite induction on $\alpha=2\beta + k$.
        The initial case is immediate since $\beta=k=0$.
        If $\alpha$ is a limit ordinal, $k$ must be $0$ and $\sim_\alpha$ is the union of $(\sim_{\alpha'})_{\alpha'<\alpha}$ as a subset of $X\times X$, which shows the limit case.

        We now show the successor case.
        Suppose that the claim holds on $\alpha=2\beta + k$.
        In the following, the equivalence classes $[-1,\beta + k)\times\{a\}$ and $[-1,\beta)\times\{b\}$ under $\sim_\alpha$ are denoted by $\lelm{\bottom}$ and $\relm{\bottom}$ respectively, and the other one-element equivalence classes are identified with their unique element in terms of notation.
        We proceed by cases on the value of $k$.

        \noindent\textbf{Case I: $k=0$.}
        We first show that the subset $\{\lelm{\bottom},\lelm{\beta}\}\subseteq X_\alpha$ is connected.
        Let $V\subseteq X_\alpha$ be an open set containing $\lelm{\beta}$.
        Then, the open set $X_{0,\alpha}^{-1}(V)\subseteq X$ must contain $\relm{\gamma}$ for some $\gamma\in [-1,\beta)$, and $V$ contains $X_{0,\alpha}(\relm{\gamma})=\relm{\bottom}$.
        Then, the open set $X_{0,\alpha}^{-1}(V)$ must contain $\lelm{-1}$, and $\lelm{\bottom}\in V$ holds, which implies the desired connectedness.
        Therefore, the quotient map $X_{\alpha,\alpha+1}$ collapses $\lelm{\bottom}$ and $\lelm{\beta}\in X_\alpha$ to the same point.
        On the other hand, one can see that every $x\in X_\alpha \backslash \{\lelm{\bottom},\lelm{\beta}\}$ is a clopen point in the fiber to which $x$ belongs, that is, the singleton $\{x\}$ is closed and open in the subspace $f_\alpha^{-1}(f_\alpha(x))\subseteq X_\alpha$.
        This shows that the map $X_{\alpha,\alpha+1}$ does not collapse any other points.

        \noindent\textbf{Case II: $k=1$.}
        In this case, the subset $\{\relm{\bottom},\relm{\beta}\}\subseteq X_\alpha$ is connected.
        Indeed, we can similarly show that in $X_\alpha$, every open set containing $\relm{\bottom}$ must contain $\relm{\beta}$.
        Moreover, one can see that every other point $x\in X_\alpha \backslash \{\relm{\bottom},\relm{\beta}\}$ is a clopen point in the fiber containing it.
        This shows that the quotient map $X_{\alpha,\alpha+1}$ only collapses $\relm{\bottom}$ and $\relm{\beta}$.
    \end{since}
    The claim shows that the canonical $\bfE$-decomposition $f_\bullet$ stabilizes exactly at $2\lambda$.
    Since $\lambda$ is arbitrary, we conclude $\decnum[\bfE]{\Top}=\cdecnum[\bfE]{\Top}=\inacc$.
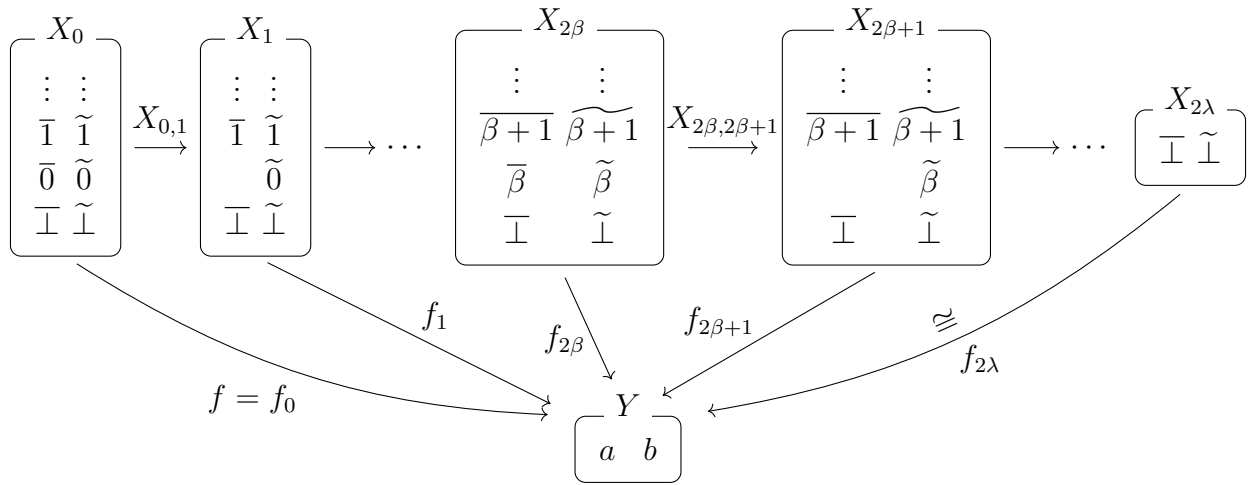
\vspace{-0.6em}
    \begin{equation*}
        \begin{tikzpicture}
            \node[draw, rounded corners, monotoneheader=$X_0$] (X0) at (0,0) {
                \begin{tikzcd}[monotone-example]
                    \vdots & \vdots \\
                    \lelm{1} & \relm{1} \\
                    \lelm{0} & \relm{0} \\
                    \lelm{\bottom} & \relm{\bottom}
                \end{tikzcd}
            };
            \node[draw, rounded corners, monotoneheader=$X_1$] (X1) at (2.5,0) {
                \begin{tikzcd}[monotone-example]
                    \vdots & \vdots \\
                    \lelm{1} & \relm{1} \\
                    & \relm{0} \\
                    \lelm{\bottom} & \relm{\bottom}
                \end{tikzcd}
            };
            \node (X2) at (4.5,0) {
                $\cdots$
            };
            \node[draw, rounded corners, monotoneheader=$X_{2\beta}$] (Xbeta) at (6.5,0) {
                \begin{tikzcd}[monotone-example]
                    \vdots & \vdots \\
                    \lelm{\beta+1} & \relm{\beta+1} \\
                    \lelm{\beta} & \relm{\beta} \\
                    \lelm{\bottom} & \relm{\bottom}
                \end{tikzcd}
            };
            \node[draw, rounded corners, monotoneheader=$X_{2\beta+1}$] (Xbetaone) at (10.8,0) {
                \begin{tikzcd}[monotone-example]
                    \vdots & \vdots \\
                    \lelm{\beta+1} & \relm{\beta+1} \\
                    & \relm{\beta} \\
                    \lelm{\bottom} & \relm{\bottom}
                \end{tikzcd}
            };
            \node (Xbetatwo) at (13.5,0) {
                $\cdots$
            };
            \node[draw, rounded corners, monotoneheader=$X_{2\lambda}$] (Xlambda) at (14.8,0) {
                \begin{tikzcd}[monotone-example]
                    \lelm{\bottom} & \relm{\bottom}
                \end{tikzcd}
            };
            \draw[->,shorten <=5pt,shorten >=5pt] (X0) to node[above] {$X_{0,1}$} (X1);
            \draw[->,shorten <=5pt,shorten >=0pt] (X1) to (X2);
            \draw[->,shorten <=9pt,shorten >=9pt] (Xbeta) to node[above] {$X_{2\beta,2\beta+1}$} (Xbetaone);
            \draw[->,shorten <=5pt,shorten >=0pt] (Xbetaone) to (Xbetatwo);
            \node[draw, rounded corners, monotoneheader=$Y$] (Y) at (7.4,-4) {
                \begin{tikzcd}[row sep=-0.3em, column sep=0em]
                    a & b
                \end{tikzcd}
            };
            \draw[->,shorten <=5pt,shorten >=10pt] (X0.south) to[bend right=15] node[below left] {$f=f_0$} (Y.north west);
            \draw[->,shorten <=5pt,shorten >=10pt] (X1.south) to node[above right=-2] {$f_1$} (Y.north west);
            \draw[->,shorten <=5pt,shorten >=15pt] (Xbeta.south) to[bend right=0] node[left=-1] {$f_{2\beta}$} (Y.north);
            \draw[->,shorten <=5pt,shorten >=15pt] (Xbetaone.south) to[bend left=0] node[above left=-3] {$f_{2\beta+1}$} (Y.north);
            \draw[->,shorten <=5pt,shorten >=10pt] (Xlambda.south) to[bend left=15] node[below right=-2] {$f_{2\lambda}$} node[above=-1,sloped] {$\cong$} (Y.north east);
        \end{tikzpicture}\vspace{-1.3em}
    \end{equation*}
\end{example}

\begin{remark}
    We do not know whether the assumption $\bfE\subseteq\Epi$ can be omitted from \cref{thm:correspondence_decnum_and_cdecnum}.
    In particular, we do not know whether $\decnum[\bfE]{f}=\cdecnum[\bfE]{f}$ holds in the category described in \cref{eg:monadic_decomposition} below, but a partial answer will be given in \cref{cor:new_result_for_monadic_decomposition}.
\end{remark}

\begin{example}[Monadic towers]\label{eg:monadic_decomposition}
    The \textit{monadic tower}, originally introduced in \cite{ApplegateTierney1970iterated}, is a decomposition of a right adjoint functor into co-transfinite composite of strictly monadic functors which is given by iterated application of the Eilenberg--Moore construction.
    Although the monadic tower of a right adjoint functor does not always exist, several sufficient conditions for existence are studied in \cite{AdamekHerrlichTholen1989monadic}.

    Let $\CATradj$ denote the (very large) category whose objects are (large) categories and whose morphisms are right adjoint functors between them.
    Consider the opposite category and let $\bfE\subseteq\Mor(\CATradj^\op)$ be the class consisting of all strictly monadic functors.
    Note that there is no reason for morphisms in $\bfE$ being epic in $\CATradj^\op$.
    MacDonald and Tholen showed in \cite[{}2.3]{MacdonaldTholen1982decomposition} that an $(\bfE,\omega)$-predecomposition $(A_\bullet,f_\bullet,X)$ in $\CATradj^\op$ is canonical if it exhibits a monadic tower (up to isomorphism) and for every $n\in (0,\omega)$, the category $A_n$ has reflexive coequalizers.\footnote{%
        This need not hold for $\gamma>\omega$ because the limit of a chain in $\CAT$, the category of (large) categories and functors, is not necessarily a limit in $\CATradj$ (a colimit in $\CATradj^\op$).%
    }
    Here, the existence of reflexive coequalizers is required to show that the diagonal fillers arising in local orthogonality are right adjoints.
    Then, the canonical $\bfE$-decomposition number coincides with what is called the \textit{monadic length} in the literature whenever it is finite.
\end{example}

The following lemma allows us to compare the two ordinals $\decnum[\bfE]{f}$ and $\cdecnum[\bfE]{f}$ in the general case where morphisms in $\bfE$ are not necessarily epic:
\begin{lemma}\label{lem:comparison_decnum_and_cdecnumII}
    Let $\bfE$ be a class of morphisms in a category $\C$, and let $\gamma\in (0,\inacc]$ be a limit ordinal.
    Let $(A_\bullet,f_\bullet,X)$ be an $(\bfE,\gamma)$-decomposition, and let $(B_\bullet,g_\bullet,Y)$ be a canonical $(\bfE,\gamma)$-predecomposition.
    Let $(\phi_0,\phi_\ast)\colon f_0\arr g_0$ be a morphism in $\C^\to$ such that $\phi_0\in\bfE$ and $\phi_\ast\in\rorth(\bfE)$.
    Then,
    \begin{enumerate}
        \item\label{lem:comparison_decnum_and_cdecnumII-1}
            $g_\bullet$ stabilizes.
        \item
            If $\abs{f_\bullet}$ is a successor, $\abs{f_\bullet}\ge\abs{g_\bullet}$ holds.
        \item
            If $\abs{f_\bullet}$ is not a successor, $\abs{f_\bullet}+1\ge\abs{g_\bullet}$ holds.\qedhere
    \end{enumerate}
\end{lemma}
\begin{proof}
    Let $\phi\colon f_\bullet\arr g_\bullet$ in $\Predec[\bfE,\gamma]$ be the unique extension of $(\phi_0,\phi_\ast)$ as in \cref{prop:canonical_decomp_adjointprop}.
    Let us define an ordinal $\alpha$ as follows:
    \begin{equation*}
        \alpha\coloneq
        \begin{cases}
            \alpha' & \text{if $\abs{f_\bullet}=\alpha'+1$} \\
            \abs{f_\bullet} & \text{otherwise}
        \end{cases}
    \end{equation*}
    In what follows, we show the statement \cref{lem:comparison_decnum_and_cdecnumII-1} and $\alpha+1\ge\abs{g_\bullet}$.
    
    By $\phi_{\alpha+1}\circ A_{0,\alpha+1}=B_{0,\alpha+1}\circ\phi_0\in\lorth\rorth(\bfE)$, $A_{0,\alpha+1}\in\lorth\rorth(\bfE)$, and the cancellation property, we have $\phi_{\alpha+1}\in\lorth\rorth(\bfE)$.
    Then, by $\phi_\ast\circ f_{\alpha+1}\in\rorth(\bfE)$, there is a unique morphism $r$ making the following commute:
    \begin{equation*}
        \begin{tikzcd}
            A_{\alpha+1}\ar[d,"\phi_{\alpha+1}"']\ar[r,equal] & A_{\alpha+1}\ar[r,"f_{\alpha+1}"] & X\ar[d,"\phi_\ast"] \\
            B_{\alpha+1}\ar[ur,dotted,"r"]\ar[rr,"g_{\alpha+1}"'] & & Y
        \end{tikzcd}\incat{\C}.
    \end{equation*}
    To prove that $r$ is the inverse of $\phi_{\alpha+1}$, we show the following by transfinite induction on $\beta\le\alpha+1$:
    \begin{claim}\label{claim:r_is_inverse}
        The following diagram commutes for every $\beta\le\alpha+1$:
        \begin{equation*}
            \begin{tikzcd}
                B_\beta\ar[d,"B_{\beta,\alpha+1}"']\ar[r,"B_{\beta,\alpha+1}"] & B_{\alpha+1} \\
                B_{\alpha+1}\ar[r,"r"'] & A_{\alpha+1}\ar[u,"\phi_{\alpha+1}"']
            \end{tikzcd}\incat{\C}.
        \end{equation*}
    \end{claim}
    \begin{since}
        \noindent\textbf{Initial step:}
        By $\phi_0\in\bfE$ and $\bfE\orth[B_{\alpha,\alpha+1}] g_{\alpha+1}$, we have $\phi_0\orth[B_{\alpha,\alpha+1}] g_{\alpha+1}$, hence two diagonal fillers below must coincide.
        This completes the initial step.
        \begin{equation*}
            \begin{tikzcd}[scriptsizecolumn]
                A_0\ar[ddd,"\phi_0"']\ar[r,"A_{0,\alpha}"] & A_\alpha\ar[d,"A_{\alpha,\alpha+1}"]\ar[r,"\phi_\alpha"] & B_\alpha\ar[rr,"B_{\alpha,\alpha+1}"] & & B_{\alpha+1}\ar[ddd,"g_{\alpha+1}"] \\
                & A_{\alpha+1}\ar[d,"\phi_{\alpha+1}"']\ar[r,equal] & A_{\alpha+1}\ar[rd,"f_{\alpha+1}"]\ar[rru,"\phi_{\alpha+1}"'] & & \\
                & B_{\alpha+1}\ar[drrr,"g_{\alpha+1}"{description}]\ar[ru,"r"'] & & X\ar[dr,"\phi_\ast"] & \\
                B_0\ar[rrrr,"g_0"']\ar[ru,"B_{0,\alpha+1}"{description}] & & & & Y
            \end{tikzcd}
            \quad
            \begin{tikzcd}[hugerow]
                A_0\ar[d,"\phi_0"']\ar[r,"A_{0,\alpha}"] & A_\alpha\ar[r,"\phi_\alpha"] & B_\alpha\ar[r,"B_{\alpha,\alpha+1}"] & B_{\alpha+1}\ar[d,"g_{\alpha+1}"] \\
                B_0\ar[urr,"B_{0,\alpha}"]\ar[urrr,"B_{0,\alpha+1}"']\ar[rrr,"g_0"'] &&& Y
            \end{tikzcd}
        \end{equation*}

        \noindent\textbf{Successor step:}
        Suppose that the statement follows for an ordinal $\beta<\alpha+1$.
        Then, we have two commutative diagrams below.
        Note that the upper triangle on the left below commutes by the induction hypothesis.
        Then, by $B_{\beta,\beta+1}\in\bfE$ and $\bfE\orth[B_{\alpha,\alpha+1}] g_{\alpha+1}$, two diagonal fillers below must coincide.
        This completes the successor step.
        \begin{equation*}
            \begin{tikzcd}[scriptsizecolumn]
                B_\beta\ar[ddd,"B_{\beta,\beta+1}"']\ar[r,rrr,"B_{\beta,\alpha}"] & & & B_\alpha\ar[r,"B_{\alpha,\alpha+1}"] & B_{\alpha+1}\ar[ddd,"g_{\alpha+1}"] \\
                & & A_{\alpha+1}\ar[rd,"f_{\alpha+1}"]\ar[rru,"\phi_{\alpha+1}"'] & & \\
                & B_{\alpha+1}\ar[drrr,"g_{\alpha+1}"{description}]\ar[ru,"r"'] & & X\ar[dr,"\phi_\ast"] & \\
                B_{\beta+1}\ar[rrrr,"g_{\beta+1}"']\ar[ru,"B_{\beta+1,\alpha+1}"{description}] & & & & Y
            \end{tikzcd}
            \quad
            \begin{tikzcd}
                B_\beta\ar[d,"B_{\beta,\beta+1}"']\ar[r,"B_{\beta,\alpha}"] & B_\alpha\ar[r,"B_{\alpha,\alpha+1}"] & B_{\alpha+1}\ar[d,"g_{\alpha+1}"] \\
                B_{\beta+1}\ar[urr,"B_{\beta+1,\alpha+1}"{description}]\ar[rr,"g_{\beta+1}"'] & & Y
            \end{tikzcd}
        \end{equation*}

        \noindent\textbf{Limit step:}
        This is trivial since $B_\beta=\Colim_{\beta'\in\bbbeta}B_{\beta'}$ for every limit ordinal $\beta$.
    \end{since}
    By taking $\beta\coloneq\alpha+1$ in \cref{claim:r_is_inverse}, it follows that $\phi_{\alpha+1}$ is an isomorphism.
    Consequently, we have $g_{\alpha+1}=\phi_\ast\circ f_{\alpha+1}\circ \phi^{-1}_{\alpha+1}\in\rorth(\bfE)$, hence $\abs{g_\bullet}\le\alpha+1$.
    This finishes the proof.
\end{proof}

\begin{theorem}[Minimum vs canonical II]\label{thm:correspondence_decnum_and_cdecnumII}
    Let $\bfE$ be a class of morphisms in a category $\C$.
    Let $f$ be a morphism in $\C$ that has the canonical $\bfE$-predecomposition.
    \begin{enumerate}
        \item
            $\decnum[\bfE]{f}$ is defined if and only if $\cdecnum[\bfE]{f}$ is defined.
        \item
            $\decnum[\bfE]{f}\le\cdecnum[\bfE]{f}\le\decnum[\bfE]{f}+1$ holds whenever they are defined.
        \item
            If $\decnum[\bfE]{f}$ is a successor, $\decnum[\bfE]{f}=\cdecnum[\bfE]{f}$ holds.
    \end{enumerate}
\end{theorem}
\begin{proof}
    This directly follows from \cref{lem:comparison_decnum_and_cdecnumII} by taking $\phi_0$ and $\phi_\ast$ as the identities.
\end{proof}

\begin{definition}
    A functor is called a \emph{coreflection} if it has a fully faithful left adjoint.
\end{definition}

\begin{lemma}\label{lem:coreflection_orthogonal_to_monadic}
    Let $R$ be a coreflection and $U$ a strictly monadic functor between large categories.
    Then, the orthogonality $R\orth U$ holds in the category $\CATradj$ as in \cref{eg:monadic_decomposition}.
\end{lemma}
\begin{proof}
    For a morphism $G$ in $\CATradj$, we write $\ladj{G}$ for the left adjoint, $\unit{G}$ and $\counit{G}$ for the unit and counit of the adjunction $\ladj{G}\dashv G$.
    Let $R\colon\D\arr \C$ be a coreflection.
    Let $T=(T,\eta^T,\mu^T)$ be a monad on a (large) category $\A$, and let $U\colon\A^T\arr \A$ and $\xi\colon TU\arr[Rightarrow] U$ be the forgetful functor and the natural transformation associated with the Eilenberg--Moore category $\A^T$.
    Let $V$ and $W$ be right adjoints making the outermost square of the following diagram commute.
    \begin{equation*}
        \begin{tikzcd}
            \D\ar[d,"R"']\ar[r,"W"]
                &
                \A^T\ar[d,"U"]
            \\
            \C\ar[r,"V"']\ar[ru,dotted,"K"]
                &
                \A
        \end{tikzcd}\incat{\CATradj}
    \end{equation*}
    If a right adjoint $K$ makes the diagram above commute, then the following equality shows that $K$ must be induced by the 2-cell exhibited by the rightmost pasting diagram.
    \begin{equation*}
        \begin{tikzcd}[tri,ampersand replacement=\&]
                \&
                \C\ar[d,"K"]
                    \&
            \\
                \&
                \A^T\ar[dl,"U"']\ar[dr,"U"]
                    \&
            \\
            \A\ar[rr,"T"']
                \&
                {}
                    \&
                    \A
            \urtwocell(\xi){2-2}{3-2}
        \end{tikzcd}
        =
        \begin{tikzcd}[tri,ampersand replacement=\&]
                \&
                \C\ar[dd,bend right=70,shift right=2,equal,""{name=LEQ}]\ar[d,"\ladj{R}"{description}]\ar[dd,bend left=70,shift left=2,equal,""{name=REQ}]
                    \&
            \\
                \&
                \D\ar[d,"R"{description}]
                    \&
            \\
                \&
                \C\ar[d,"K"]
                    \&
            \\
                \&
                \A^T\ar[dl,"U"']\ar[dr,"U"]
                    \&
            \\
            \A\ar[rr,"T"']
                \&
                {}
                    \&
                    \A
            \rtwocell(\unit{R}){LEQ}{2-2}
            \rtwocell((\unit{R})^{-1})[above=2pt,xshift=-4pt]{2-2}{REQ}
            \urtwocell(\xi){4-2}{5-2}
        \end{tikzcd}
        =
        \begin{tikzcd}[tri,ampersand replacement=\&]
                \&
                \C\ar[ddl,equal,bend right=50,""{below,name=LEQ}]\ar[d,"\ladj{R}"{description}]\ar[ddr,equal,bend left=50,""{below,name=REQ}]
                    \&
            \\
                \&
                \D\ar[dl,"R"{description}]\ar[d,"W"{description}]\ar[dr,"R"{description}]
                    \&
            \\
            \C\ar[d,"V"']
                \&
                \A^T\ar[dl,"U"{description}]\ar[dr,"U"{description}]
                    \&
                    \C\ar[d,"V"]
            \\
            \A\ar[rr,"T"']
                \&
                {}
                    \&
                    \A
            \rtwocell(\unit{R}){LEQ}{2-2}
            \rtwocell((\unit{R})^{-1})[above=2pt,xshift=-3pt]{REQ}{2-2}
            \urtwocell(\xi){3-2}{4-2}
            \cellsymb(=){3-1}{3-2}
            \cellsymb(=){3-3}{3-2}
        \end{tikzcd}
    \end{equation*}
    Conversely, the universal property of $\A^T$ yields a functor $K$ such that $KR=W$ and $UK=V$.

    What remains to show is $K$ being a right adjoint.
    Take an arbitrary object $X\in\A^T$.
    Since $\ladj{R}$ is fully faithful, it suffices to show that $\ladj{W}(X)$ lies in $\D_0$, the essential image of $\ladj{R}$.
    Now, the $T$-algebra $X$ can be expressed as a coequalizer of a parallel pair of morphisms between free $T$-algebras:
    \begin{equation*}
        \begin{tikzcd}[scriptsize]
            \ladj{U}(A_0)\ar[r,shift left=1]\ar[r,shift right=1]
                &
                \ladj{U}(A_1)\ar[r]
                    &
                    X
        \end{tikzcd}\incat{\A^T}.
    \end{equation*}
    This coequalizer is preserved by $\ladj{W}$ because it is a left adjoint.
    Then, the objects $\ladj{W}\ladj{U}(A_i)$ $(i=0,1)$ belong to $\D_0$ by $UW=VR$.
    Since any coreflective full subcategory is closed under colimits, the coequalizer $\ladj{W}(X)$ also belongs to $\D_0$, hence $K$ has a left adjoint.
\end{proof}

The following corollary seems a new result:
\begin{corollary}\label{cor:new_result_for_monadic_decomposition}
    Let $G\colon\X\arr \A$ be a right adjoint functor whose monadic tower exists for any finite step and all the Eilenberg--Moore categories appearing in the tower have reflexive coequalizers.
    If $G$ can be decomposed into a coreflection followed by a composite of an $n$-tuple of strictly monadic functors, then the monadic tower of $G$ stabilizes at (or before) $n$.
\end{corollary}
\begin{proof}
    Consider the category $\CATradj^\op$ and the class $\bfE$ as in \cref{eg:monadic_decomposition}.
    By assumption, the morphism $G$ in $\CATradj^\op$ has the canonical $(\bfE,\omega)$-predecomposition.
    Suppose that $G$ can be decomposed into a coreflection followed by a composite of $n$ strictly monadic functors, which yields an $\bfE$-decomposition of $G$ by \cref{lem:coreflection_orthogonal_to_monadic}.
    Then, the statement follows from \cref{lem:comparison_decnum_and_cdecnumII} by taking $\phi_0$ and $\phi_\ast$ as the identities.
\end{proof}
\section{Generalized regularity and the decomposition numbers}\label{sec:generalized_regularity}
We will extend the notions of regular epimorphisms and strong epimorphisms, and investigate how the decomposition numbers behave with respect to this generalized notion of regular epimorphisms.
\subsection{A generalization of regular and strong epimorphisms}\label{subsec:generalized_regular_and_strong_epi}
\begin{definition}[$\epclass$-regular epimorphisms]
    Let $\epclass$ be a class of epimorphisms in a category $\C$.
    A morphism $q\colon A\arr B$ in $\C$ is called \emph{$\epclass$-regularly epic} if for any morphism $f\colon A\arr X$ satisfying the following condition $\kercond$, there uniquely exists a morphism $k\colon B\arr X$ such that $f=kq$.
    \begin{itemize}
        \labeleditem{($\mathrm{K}_{\epclass,q}$)}\label{kernel_condition}
            If we are given morphisms $\Gamma\arr(e)\Delta$ in $\epclass$ and $\Gamma\arr(u) A$ such that $qu$ factors through $e$, then $fu$ also factors through $e$.
            \begin{equation*}
                \begin{tikzcd}
                    \Gamma\ar[d,"e"']\ar[r,"u"] & A\ar[d,"q"']\ar[rdd,"f"] & \\
                    \Delta\ar[r,dotted]\ar[rrd,dotted] & B & \\
                    && X
                \end{tikzcd}
            \end{equation*}
    \end{itemize}
    The above condition $\kercond$ is called the \emph{$\epclass$-kernel condition (for $f$) associated with $q$}.
    We write $\Reg[\C](\epclass)$ for the class of all $\epclass$-regular epimorphisms in $\C$.
    When $\C$ is clear from context, we simply write $\Reg(\epclass)$.
\end{definition}

\begin{definition}[$\epclass$-strong epimorphisms]
    Let $\epclass$ be a class of epimorphisms in a category $\C$.
    Morphisms that belong to $\lorth\rorth(\epclass)$ are called \emph{$\epclass$-strong epimorphisms}.
    We write $\Stg[\C](\epclass)\coloneq\lorth\rorth(\epclass)$ for the class of all $\epclass$-strong epimorphisms in $\C$.
    When $\C$ is clear from context, we simply write $\Stg(\epclass)$.
\end{definition}

\begin{remark}
    It will be proved in \cref{prop:fundamental_prop_of_GREpi}\cref{prop:fundamental_prop_of_GREpi-epimorphicity} that $\epclass$-regular epimorphisms are always epic.
    In particular, there is no difference between the minimum decomposition numbers $\decnum[\Reg(\epclass)]{f}$ and the canonical decomposition numbers $\cdecnum[\Reg(\epclass)]{f}$ by \cref{thm:correspondence_decnum_and_cdecnum}.
    On the other hand, $\epclass$-strong epimorphisms are not epic in general unless the category has binary powers.
\end{remark}

\begin{example}[Ordinary regular epimorphisms]\label{eg:regular_epi}
    Let $\C$ be a category with binary copowers.
    Let $\epclass$ be the class of all codiagonals:
    \begin{equation*}
        \epclass\coloneq\{ \nabla_X\colon X+X\arr((\id,\id))[][3] X \}_{X\in\C}.
    \end{equation*}
    Then, a morphism $f$ in $\C$ satisfies the $\epclass$-kernel condition $\kercond$ associated with a morphism $q$, whose domain is the same as $f$, if and only if every parallel pair of morphisms coequalized by $q$ is also coequalized by $f$.
    Therefore, if the kernel pair of $q$ exists, the morphism $q$ is a $\epclass$-regular epimorphism if and only if it is a regular epimorphism in the usual sense.
    Moreover, the $\epclass$-strong epimorphisms coincide with the usual strong epimorphisms.
\end{example}

One can easily verify that the codiagonals in $\epclass$ can reduce to ones on objects from a generator:
\begin{proposition}\label{eg:regular_epi_with_generator}
    Let $\C$ be a category with binary copowers, kernel pairs, and a generator $G\subseteq\Ob\C$.
    Let $\epclass$ be the class of all codiagonals on objects from $G$:
    \begin{equation*}
        \epclass\coloneq\{ \nabla_X\colon X+X\arr((\id,\id))[][3] X \}_{X\in G}.
    \end{equation*}
    Then, a morphism is a $\epclass$-regular epimorphism if and only if it is a regular epimorphism; a morphism is a $\epclass$-strong epimorphism if and only if it is a strong epimorphism.
\end{proposition}

\begin{example}[Surjections in $\Pos$]\label{eg:surjections_in_Pos}
    Let $\Pos$ be the category of posets.
    Let $\epclass\coloneq\{\iota\}$ be the class consisting of only the following surjective inclusion:
    \begin{equation*}
        \{0,1\}\arr(\iota)[hook][2]\{0<1\}\incat{\Pos}.
    \end{equation*}
    Then, the following are equivalent for a morphism in $\Pos$: it is $\epclass$-regularly epic; it is $\epclass$-strongly epic; it is epic; it is surjective.
\end{example}

\begin{example}[Strict localizations in $\Cat$]\label{eg:strict_localizations_in_Cat}
    Let $\Cat$ be the category of small categories.
    Let $\epclass\coloneq\{L\}$ be the class consisting of only the following identity-on-object functor:
    \begin{equation*}
        \{0\to 1\} \arr(L)[][2] \{0\cong 1\} \incat{\Cat}.
    \end{equation*}
    Then, a functor between small categories is $\epclass$-regularly epic if and only if it is a strict localization of some class of morphisms.
    Moreover, the $\epclass$-strong epimorphisms are precisely the iterated strict localizations, cf.\ \cref{eg:global_minimum_decnum}\cref{eg:global_minimum_decnum-localizations}.
\end{example}

\begin{example}[Anti-$k$-Lipschitz quotient maps in $\Met$]\label{eg:antiLipschitz_in_Met}
    Let $\Met$ be the category of metric spaces and non-expansive maps, and let $0<k<1$.
    For a positive real number $r$, let $\Omega_r\coloneq\{0,1\}$ be the two-element metric space with $d(0,1)\coloneq r$.
    Let $\epclass\coloneq\{ e_p \mid 0<p\in\mathbb{Q} \}$ be the class of the morphisms $e_p\colon \Omega_p\arr \Omega_{kp}$ in $\Met$ with $e_p(0)=0$ and $e_p(1)=1$.
    Then, a morphism in $\Met$ is $\epclass$-regularly epic if and only if it is an anti-$k$-Lipschitz quotient map in the sense of \cref{eg:global_minimum_decnum}\cref{eg:global_minimum_decnum-antiLipschitz}.
    The $\epclass$-strong epimorphisms are precisely the surjective non-expansive maps.
\end{example}

\begin{example}[Monotone quotient maps in $\Top$]
    Let $\Top$ be the category of topological spaces and continuous maps.
    Let $\epclass$ be the class of all unique continuous maps $K\arr 1$ from a connected space $K$ to the terminal space.
    Then, a morphism in $\Top$ is $\epclass$-regularly epic if and only if it is a quotient map such that every fiber is connected, i.e., a monotone quotient map.
\end{example}


Here, we summarize a few basic properties of the $\epclass$-kernel condition that one can easily observe.
\begin{lemma}
    Let $\epclass$ be a class of epimorphisms in a category $\C$.
    Consider morphisms $f,g,q,r,a$ of the following form:
    \begin{equation*}
        \begin{tikzcd}[small]
            A'\ar[r,"a"] & A\ar[d,"q"']\ar[r,"f"] & X\ar[r,"g"] & Y \\
            & B\ar[d,"r"'] && \\
            & C &&
        \end{tikzcd}\incat{\C}.
    \end{equation*}
    Then, the following statements hold:
    \begin{enumerate}
        \labeleditem{(A0)}\label{axiom:kercond-self}
            $q$ always satisfies $\kercond[\epclass][q]$.
        \labeleditem{(A1)}\label{axiom:kercond-f}
            If $f$ satisfies $\kercond[\epclass][q]$, then $gf$ satisfies $\kercond[\epclass][q]$.
            The converse also holds whenever $g\in\rorth(\epclass)$.
        \labeleditem{(A2)}\label{axiom:kercond-q}
            If $f$ satisfies $\kercond[\epclass][rq]$, then $f$ satisfies $\kercond[\epclass][q]$.
            The converse also holds whenever $r\in\rorth(\epclass)$.
        \labeleditem{(A3)}\label{axiom:kercond-a}
            If $f$ satisfies $\kercond[\epclass][q]$, then $fa$ satisfies $\kercond[\epclass][qa]$.
    \end{enumerate}
\end{lemma}

\begin{definition}\label{def:multiple_and_wide_pushout}
    A \emph{multiple pushout} in a category is a colimit for the diagram $(A, \Gamma_i, \Delta_i, p_i,u_i)_{i\in I}$ consisting of a (not necessarily small) family of spans $(p_i,u_i)_{i\in I}$ with a common codomain of the morphisms $u_i$.
    Note that if the index class $I$ is empty, the diagram still contains the object\nolinebreak{} $A$.
    \begin{equation*}
        \begin{tikzcd}[small]
            \Gamma_{i'}\ar[dd,"p_{i'}"']\ar[rrr,"u_{i'}"] &&& A \\
            {} & \Gamma_i\ar[dd,"p_i"]\ar[rru,"u_i"'] && \\
            \Delta_{i'} & {} && \\
            & \Delta_i &&
            \arrow[from=1-1,to=2-2,phantom,"\ddots"]
            \arrow[from=2-1,to=3-2,phantom,"\ddots"{above=-4}]
            \arrow[from=3-1,to=4-2,phantom,"\ddots"]
            \arrow[from=1-1,to=1-4,phantom,"\ddots"{pos=0.6,below=-6}]
        \end{tikzcd}
        \quad
        (i,i'\in I)
    \end{equation*}
    As a special case, when all the $u_i$ are identities, the corresponding multiple pushout is called a \emph{wide pushout}.
\end{definition}

We now show several elementary properties of generalized regular epimorphisms.
\begin{proposition}\label{prop:fundamental_prop_of_GREpi}
    Let $\epclass$ be a class of epimorphisms in a category $\C$.
    \begin{enumerate}
        \item\label{prop:fundamental_prop_of_GREpi-inclusions}
            $\epclass\subseteq\Reg[\C](\epclass)\subseteq\Stg[\C](\epclass)$ holds.
        \item\label{prop:fundamental_prop_of_GREpi-epimorphicity}
            $\Reg[\C](\epclass)\subseteq\Epi[\C]$ holds.
            If $\C$ has binary powers, $\Stg[\C](\epclass)\subseteq\Epi[\C]$ also holds.
        \item\label{prop:fundamental_prop_of_GREpi-multiple_pushout}
            The class of $\epclass$-regular epimorphisms is stable under existing (not necessarily small) multiple pushout.
            That is, if we are given a multiple pushout of the following diagram $(A,A_i,B_i,q_i,a_i)_i$, $q_i\in\Reg[\C](\epclass)$ for all $i$ implies $q\in\Reg[\C](\epclass)$.
            \begin{equation*}
                \begin{tikzcd}
                    A_i\ar[d,"q_i"']\ar[r,"a_i"] & A\ar[d,"q"] \\
                    B_i\ar[r,"b_i"'] & B
                \end{tikzcd}\incat{\C}
            \end{equation*}
        \item\label{prop:fundamental_prop_of_GREpi-closure}
           $\Reg[\C](\epclass)$ is the smallest class of morphisms satisfying the following:
            \begin{itemize}
                \item
                    It contains $\epclass$;
                \item
                    It is stable under existing (not necessarily small) multiple pushout.
            \end{itemize}
    \end{enumerate}
\end{proposition}
\begin{proof}\quad
    \begin{enumerate}
        \item
            We first show $\epclass\subseteq\Reg(\epclass)$.
            Let $q\colon A\arr B$ be a morphism in $\epclass$.
            Let $f\colon A\arr X$ be an arbitrary morphism satisfying the condition $\kercond$.
            Since $q$ clearly factors through itself, the condition $\kercond$ implies that the morphism $f$ also factors through $q$ by a morphism $k$ below:
            \begin{equation*}
                \begin{tikzcd}
                    A\ar[d,"q"']\ar[r,equal] & A\ar[d,"q"']\ar[rdd,"f"] & \\
                    B\ar[r,equal]\ar[rrd,dotted,"k"'] & B & \\
                    && X
                \end{tikzcd}\incat{\C}.
            \end{equation*}
            Since $q$ is epic, the morphism $k$ above is unique, which proves $q\in\Reg(\epclass)$.

            We next show the second inclusion $\Reg(\epclass)\subseteq\Stg(\epclass)$.
            Take $q\colon A\arr B$ from $\Reg(\epclass)$.
            Suppose that we are given the following commutative square with $m\in\rorth(\epclass)$:
            \begin{equation}\label{eq:commutative_square_q_m}
                \begin{tikzcd}
                    A\ar[d,"q"']\ar[r,"f"] & X\ar[d,"m"] \\
                    B\ar[r,"g"'] & Y
                \end{tikzcd}\incat{\C}.
            \end{equation}
            Take arbitrary morphisms $u\colon\Gamma\arr A$ and $e\colon\Gamma\arr\Delta$ such that $e\in\epclass$ and the composite $qu$ factors through $e$.
            Then, the orthogonal property of $m$ induces a (unique) diagonal filler of the following commutative square:
            \begin{equation*}
                \begin{tikzcd}
                    \Gamma\ar[d,"e"']\ar[r,"u"] & A\ar[d,"q"'{pos=0.2}]\ar[r,"f"] & X\ar[d,"m"] \\
                    \Delta\ar[r] & B\ar[r,"g"'] & Y
                    \arrow[from=2-1,to=1-3,dotted,crossing over]
                \end{tikzcd}\incat{\C},
            \end{equation*}
            which implies the morphism $f$ satisfies $\kercond$.
            Then, by $q\in\Reg(\epclass)$, there exists a (unique) morphism $k$ such that $f=kq$.
            Since $q$ is epic, $k$ becomes a unique diagonal filler for \cref{eq:commutative_square_q_m}, which proves $q\in\lorth\rorth(\epclass)=\Stg(\epclass)$.
        \item
            To show $\Reg(\epclass)\subseteq\Epi$, take a morphism $q\colon A\arr B$ from $\Reg(\epclass)$.
            Let $f,g\colon B\parr C$ be a parallel pair of morphisms that $q$ equalizes.
            Then, the composite $fq(=gq)$ satisfies $\kercond$ by \cref{axiom:kercond-self,axiom:kercond-f}.
            Since $q$ is $\epclass$-regularly epic, the composite $fq(=gq)$ uniquely factors through $q$, hence $f=g$.
            This proves that $q$ is epic.
            The latter statement follows from \cref{lem:leftclass_epic}.
        \item
            Suppose $q_i\in\Reg(\epclass)$ for all $i$.
            Let $f\colon A\arr X$ be a morphism satisfying the condition $\kercond[\epclass][q]$.
            By \cref{prop:fundamental_prop_of_GREpi-epimorphicity}, all $q_i$ are epic, hence so is $q$.
            Thus, a factorization of $f$ through $q$ is unique if it exists.
            By \ref{axiom:kercond-a}, each composite $fa_i$ satisfies $\kercond[\epclass][qa_i]$.
            Then, as follows from $qa_i=b_iq_i$ and \ref{axiom:kercond-q}, $fa_i$ satisfies $\kercond[\epclass][q_i]$. 
            Since $q_i\in\Reg(\epclass)$, each composite $fa_i$ (uniquely) factors through $q_i$.
            Then, the universal property of the multiple pushout implies that $f$ factors through $q$, which proves $q\in\Reg(\epclass)$.
        \item
            Let $q\colon A\arr B$ be a $\epclass$-regular epimorphism.
            Consider all data $(\Gamma_i,\Delta_i,p_i,u_i,v_i)$ of the following form such that $p_i$ belongs to $\epclass$ and $qu_i=v_ip_i$, and suppose that they are indexed by $i\in I$:
            \begin{equation*}
                \begin{tikzcd}
                    \Gamma_i\ar[d,"p_i"']\ar[r,"u_i"] & A\ar[d,"q"] \\
                    \Delta_i\ar[r,"v_i"'] & B
                \end{tikzcd}\incat{\C}.
            \end{equation*}
            Then, $q$ being $\epclass$-regularly epic implies that $B$ is the multiple pushout of the diagram above.
            Indeed, since $p_i$ is epic, a cocone over the diagram $(A,\Gamma_i,\Delta_i,p_i,u_i)_i$ is precisely a morphism from $A$ satisfying the condition $\kercond[\epclass][q]$, and hence $q$ exhibits $B$ as the multiple pushout if and only if $q$ is $\epclass$-regularly epic.
            Therefore, $q$ belongs to the closure of $\epclass$ under existing multiple pushouts, which finishes the proof.\qedhere
    \end{enumerate}
\end{proof}

\begin{remark}
    According to \cite{Tholen1983factorizations}, in an arbitrary co-wellpowered category with small colimits and products, an iso-closed class $\bfE$ of epimorphisms is stable under existing multiple pushout if and only if the category admits locally orthogonal $\bfE$-factorizations.
    In particular, such a class of morphisms can always be regarded as a class of generalized regular epimorphisms by \cref{prop:fundamental_prop_of_GREpi}\cref{prop:fundamental_prop_of_GREpi-closure}.
\end{remark}

\begin{corollary}\label{cor:criterion_local_orth_for_GEEpi}
    Let $\epclass$ be a class of epimorphisms in a category $\C$.
    Then, for a composable pair $(e,f)$ of morphisms in $\C$, the condition $\Reg[\C](\epclass)\orth[e] f$ is equivalent to $\epclass\orth[e] f$.
    In particular, $\rorth(\Reg[\C](\epclass))=\rorth(\epclass)$.
\end{corollary}
\begin{proof}
    The class of all morphisms $p$ such that $p\orth[e] f$ is clearly stable under multiple pushout.
    Combining this with \cref{prop:fundamental_prop_of_GREpi}\cref{prop:fundamental_prop_of_GREpi-closure}, we can observe that $\epclass\orth[e] f$ implies $\Reg[\C](\epclass)\orth[e] f$. 
    The converse implication is trivial.
\end{proof}

The following proposition is a slight generalization from results in \cite[Section 2]{Kelly1969mono}:
\begin{proposition}\label{prop:composable_pair_GREpi}
    Let $\epclass$ be a class of epimorphisms in a category $\C$.
    Consider a composable pair of morphisms
    \begin{equation*}
        \begin{tikzcd}[small]
            A\ar[dr,"p"']\ar[rr,"r"] && C \\
            & B\ar[ru,"q"'] &
        \end{tikzcd}\incat{\C}.
    \end{equation*}
    \begin{enumerate}
        \item\label{prop:composable_pair_GREpi-cancellability}
            If $p$ is epic and $r$ is $\epclass$-regularly epic, then $q$ is $\epclass$-regularly epic.
        \item\label{prop:composable_pair_GREpi-composability}
            If $p$ is split epic and $\epclass$-regularly epic and $q$ is $\epclass$-regularly epic, then $r$ is $\epclass$-regularly epic.
    \end{enumerate}
\end{proposition}
\begin{proof}\quad
    \begin{enumerate}
        \item
            Suppose that $p$ is epic and $r$ is $\epclass$-regularly epic.
            Let $f\colon B\arr\cdot$ be a morphism satisfying the condition $\kercond[\epclass][q]$.
            Then, \ref{axiom:kercond-a} implies that the composite $fp$ satisfies $\kercond[\epclass][r]$.
            Since $r$ is $\epclass$-regularly epic, $fp$ uniquely factors through $r=qp$.
            Then, since $p$ is epic, $f$ factors through $q$, which proves that $q$ is $\epclass$-regularly epic.
        \item
            Let $s\colon B\arr A$ be a section of $p$, and suppose that both $p$ and $q$ are $\epclass$-regularly epic.
            Let $f\colon A\arr X$ be a morphism satisfying the condition $\kercond[\epclass][r]$.
            Then, by \ref{axiom:kercond-q}, $f$ also satisfies $\kercond[\epclass][p]$.
            Since $p$ is $\epclass$-regularly epic, there is a unique morphism $g\colon B\arr X$ such that $f=gp$.
            Since $f$ satisfies $\kercond[\epclass][r]$, \ref{axiom:kercond-a} now implies that $fs$ satisfies $\kercond[\epclass][rs]$.
            By $fs=gps=g$ and $rs=qps=q$, this is equivalent to $g$ satisfying $\kercond[\epclass][q]$.
            Since $q$ is $\epclass$-regularly epic, $g$ uniquely factors through $q$.
            Consequently, we have shown that $f$ uniquely factors through $r=qp$, and $r$ is $\epclass$-regularly epic.\qedhere
    \end{enumerate}
\end{proof}

\begin{theorem}\label{thm:GREpi_admissibility}
    Let $\epclass$ be a small class of epimorphisms in a locally small category $\C$ with small colimits.
    Then, the category $\C$ is $\Reg[\C](\epclass)$-admissible.
\end{theorem}
\begin{proof}
    It suffices to show that the category $\C$ admits locally orthogonal $\Reg[\C](\epclass)$-factorizations.
    Let $f\colon A\arr X$ be an arbitrary morphism in $\C$.
    Consider all data $(\Gamma_i,\Delta_i,p_i,u_i)$ of the following form such that $p_i$ belongs to $\epclass$ and $fu_i$ factors through $p_i$, and suppose that they are indexed by $i\in I$.
    \begin{equation*}
        \begin{tikzcd}
            \Gamma_i\ar[d,"p_i"']\ar[r,"u_i"] & A\ar[r,"f"] & X \\
            \Delta_i\ar[rru,dotted]
        \end{tikzcd}\incat{\C}
    \end{equation*}
    Since $\epclass$ is small and $\C$ is locally small, the index class $I$ is small.
    Thus, we can take the multiple pushout $(B,q,\iota_i)_i$ of the diagram $(A,\Gamma_i,\Delta_i,p_i,u_i)_i$, and we obtain a unique morphism $g$ by the universal property.
    \begin{equation*}
        \begin{tikzcd}[scriptsize]
            \Gamma_{i'}\ar[dd,"p_{i'}"']\ar[rrr,"u_{i'}"] &&& A\ar[dd,"q"]\ar[rrrr,"f"] &&&& X \\
            {} & \Gamma_i\ar[rru,"u_i"'] &&&&&& \\
            \Delta_{i'}\ar[rrr,"\iota_{i'}"{pos=0.7}] & {} && B\ar[rrrruu,dotted,"g"'] &&&& \\
            & \Delta_i\ar[rru,"\iota_i"'] &&&&&&
            \arrow[from=2-2,to=4-2,crossing over,"p_i"{pos=0.2}]
            \arrow[from=1-1,to=2-2,phantom,"\ddots"]
            \arrow[from=2-1,to=3-2,phantom,"\ddots"{above=-2}]
            \arrow[from=3-1,to=4-2,phantom,"\ddots"]
            \arrow[from=1-1,to=1-4,phantom,"\ddots"{pos=0.6,below=-4}]
            \arrow[from=3-1,to=3-4,phantom,"\ddots"{pos=0.6,below=-4}]
        \end{tikzcd}\incat{\C}
    \end{equation*}
    We now have $p_i\in\epclass\subseteq\Reg(\epclass)$ by \cref{prop:fundamental_prop_of_GREpi}\cref{prop:fundamental_prop_of_GREpi-inclusions}.
    Then, \cref{prop:fundamental_prop_of_GREpi}\cref{prop:fundamental_prop_of_GREpi-multiple_pushout} implies $q\in\Reg(\epclass)$.

    It remains to show $\Reg(\epclass)\orth[q]g$, but this is equivalent to $\epclass\orth[q]g$ by \cref{cor:criterion_local_orth_for_GEEpi}, which we will prove below.
    Suppose that we are given the following data $(\Gamma,\Delta,p,u,v)$ with $p\in\epclass$ and $gqu=vp$:
    \begin{equation}\label{eq:local_orth_q_g}
        \begin{tikzcd}
            \Gamma\ar[d,"p"']\ar[r,"u"] & A\ar[r,"q"] & B\ar[d,"g"] \\
            \Delta\ar[rr,"v"'] && X
        \end{tikzcd}\incat{\C}.
    \end{equation}
    Since $gq=f$, the data $(\Gamma,\Delta,p,u)$ is already indexed as $(\Gamma_{i_0},\Delta_{i_0},p_{i_0},u_{i_0})$ for some $i_0\in I$.
    Then, since $p\in\epclass$ is epic, the copojection $\iota_{i_0}$ becomes a unique diagonal filler for the square \cref{eq:local_orth_q_g}, which shows $\epclass\orth[q]g$.
\end{proof}

\begin{remark}\label{rem:canonical_predecomp_for_GRegEpi}
    From the proof of \cref{thm:GREpi_admissibility} above, we obtain a concrete construction of the canonical predecomposition in the case of generalized regular epimorphisms, which particularly clarifies why the canonical regular predecomposition is given by iteratively taking the coequalizers of the kernel pairs.
\end{remark}

\begin{example}\label{eg:canonical_predecomposition_for_strict_localizations}
    Let $\bfE$ be the class of all strict localizations in the category $\Cat$ of small categories.
    An $\bfE$-predecomposition $(\A_\bullet,F_\bullet,\X)$ in $\Cat$ is canonical if and only if for every small ordinal $\alpha$, the functor $\A_{\alpha,\alpha+1}$ is the universal one from $\A_\alpha$ that inverts all morphisms $f$ in $\A_\alpha$ such that $F_\alpha(f)$ is an isomorphism in $\X$.
\end{example}

\subsection{The decomposition numbers and adjunctions}\label{subsec:decnum_and_adj}
In this subsection, we investigate the relationship between adjunctions and the decomposition numbers arising from the generalized regular epimorphisms.
Under certain conditions, we show that for a given right adjoint, the global decomposition number of the domain category can be bounded from above by that of the codomain up to some difference.
In particular, \cref{cor:adj_and_decnum_regularepi} will play an important role in \cref{sec:for_models_of_dependent_algebraic_theories} later.
\begin{lemma}\label{lem:counit_cotaut}
    Let $F\colon\C\adjl\D\lon G$ be an adjunction whose counit $\epsilon$ is $F\epclass$-regularly\footnote{Here, $F\epclass$ denotes the class of morphisms in $\D$ obtained by applying $F$ to all morphisms in $\epclass$.} epic componentwise for a class $\epclass$ of epimorphisms in $\C$.
    Let $f\colon X\arr Y$ be a morphism in $\D$ such that $GFGf$ is epic in $\C$.
    Then, the naturality square of the counit at $f$ is a pushout in $\D$.
    \begin{equation*}
        \begin{tikzcd}
            FGX\ar[d,"\epsilon_X"']\ar[r,"FGf"] & FGY\ar[d,"\epsilon_Y"] \\
            X\ar[r,"f"'] & Y\pushoutcorner
        \end{tikzcd}\incat{\D}
    \end{equation*}
\end{lemma}
\begin{proof}
    Let $Z$ be an arbitrary object in $\D$, and let $u\colon X\arr Z$ and $v\colon FGY\arr Z$ be morphisms in $\D$ satisfying $v\circ FGf = u\circ \epsilon_X$ as in the diagram on the left below.
    In what follows, we will show that there is a unique morphism $k\colon Y\arr Z$ such that $k\circ\epsilon_Y = v$, which is sufficient because $\epsilon_X$ is epic.

    For the morphism $\hat{v}$ corresponding to $v$ by the adjunction, the diagram on the right below commutes:
    \begin{equation*}
        \begin{tikzcd}
            FGX\ar[d,"\epsilon_X"']\ar[r,"FGf"] & FGY\ar[d,"v"] \\
            X\ar[r,"u"'] & Z
        \end{tikzcd}\incat{\D},
        \qquad\qquad
        \begin{tikzcd}
            GX\ar[dr,bend right=15,"Gu"']\ar[r,"Gf"] & GY\ar[d,"\hat{v}"] \\
            & GZ
        \end{tikzcd}\incat{\C}.
    \end{equation*}
    Then, both the left-hand square and the outer rectangle commute in the following diagram:
    \begin{equation}\label{eq:square_commute_by_epi}
        \begin{tikzcd}
            GFGX\ar[d,"G\epsilon_X"']\ar[r,"GFGf"']\ar[rr,shift left=3,"GFGu"]
                &
                GFGY\ar[d,"G\epsilon_Y"]\ar[r,"GF\hat{v}"']
                    &
                    GFGZ\ar[d,"G\epsilon_Z"]
            \\
            GX\ar[r,"Gf"]\ar[rr,shift right=3,"Gu"']
                &
                GY\ar[r,"\hat{v}"]
                    &
                    GZ
        \end{tikzcd}\incat{\C}.
    \end{equation}
    Since $GFGf$ is supposed to be epic, the right-hand square also commutes.

    It immediately follows from the adjunction that the morphism $v$ satisfies the condition $\kercond[F\epclass][\epsilon_Y]$ if and only if the morphism $Gv$ satisfies $\kercond[\epclass][G\epsilon_Y]$, and the latter holds because $Gv=(G\epsilon_Z)\circ GF\hat{v}$ factors through $G\epsilon_X$, as in \cref{eq:square_commute_by_epi}.
    Since $\epsilon_Y$ is $F\epclass$-regularly epic by assumption, there is a unique morphism $k\colon Y\arr Z$ such that $k\circ\epsilon_Y = v$, which finishes the proof.
\end{proof}

\begin{notation}\label{note:pushout_component_c_f}
    Let $F\colon\C\adjl\D\lon G$ be an adjunction with $\D$ having pushouts.
    For a morphism $f\colon X\arr Y$ in $\D$, we write $\pocomp{f}$ for the unique morphism from the pushout $C_f$ in $\D$ such that the following diagram commutes:
    \begin{equation*}
        \begin{tikzcd}
            FGX\ar[r,"FGf"]\ar[d,"\epsilon_X"']
                &
                FGY\ar[rd,"\epsilon_Y"]\ar[d]
                    &
            \\
            X\ar[r]\ar[rr,"f"',bend right=20]
                &
                C_f\ar[r,"{\pocomp{f}}",pos=0.3]\pushoutcorner
                    &
                    Y
        \end{tikzcd}\incat{\D},
    \end{equation*}
    where $\epsilon$ denotes the counit of the adjunction.
\end{notation}

\begin{proposition}\label{prop:adj_and_decnum}
    Let $F\colon\C\adjl\D\lon G$ be an adjunction with $\D$ having pushouts.
    Let $\epclass_0$ and $\epclass_1$ be classes of epimorphisms in $\C$ and $\D$, respectively.
    Let $f$ be a morphism in $\D$.
    Assume the following conditions:
    \begin{itemize}
        \item
            $F\epclass_0 \subseteq \Reg(\epclass_1)$ holds, i.e., every morphism in $\epclass_0$ is sent to a $\epclass_1$-regular epimorphism by $F$.
        \item
            $Gf$ is $\epclass_0$-strongly epic, and $\decnum[\Reg(\epclass_0)]{Gf}$ is defined.
    \end{itemize}
    Then, the following statements hold:
    \begin{enumerate}
        \item\label{prop:adj_and_decnum-definability}
            $\decnum[\Reg(\epclass_1)]{f}$ is defined if and only if $\decnum[\Reg(\epclass_1)]{\pocomp{f}}$ is defined.
        \item\label{prop:adj_and_decnum-general}
            Whenever $\decnum[\Reg(\epclass_1)]{f}$ and $\decnum[\Reg(\epclass_1)]{\pocomp{f}}$ are defined, we have
            \begin{equation*}
                \decnum[\Reg(\epclass_1)]{f} \le \decnum[\Reg(\epclass_0)]{Gf} + \decnum[\Reg(\epclass_1)]{\pocomp{f}}.
            \end{equation*}
        \item\label{prop:adj_and_decnum-plusone}
            If the counit of the adjunction is $\epclass_1$-regularly epic componentwise, then $\decnum[\Reg(\epclass_1)]{f}$ is defined, and we have
            \begin{equation*}
                \decnum[\Reg(\epclass_1)]{f} \le \decnum[\Reg(\epclass_0)]{Gf} + 1.
            \end{equation*}
        \item\label{prop:adj_and_decnum-pluszero}
            If the counit of the adjunction is $F\epclass_0$-regularly epic componentwise and $GFGf$ is epic, then $\decnum[\Reg(\epclass_1)]{f}$ is defined, and we have
            \begin{equation*}
                \decnum[\Reg(\epclass_1)]{f} \le \decnum[\Reg(\epclass_0)]{Gf}.
            \end{equation*}\qedhere
    \end{enumerate}
\end{proposition}
\begin{proof}
    Let $f\colon X\arr Y$ be a morphism in $\D$.
    Let $(A_\bullet,h_\bullet,GY)$ be an arbitrary $\Reg(\epclass_0)$-decomposition of $Gf$, and let $\alpha\coloneq\abs{h_\bullet}$.
    By $Gf\in\Stg(\epclass_0)$ and \cref{lem:decomposition_of_leftclass}, we can assume that $h_\beta$ is the identity for every $\beta\in[\alpha,\inacc)$, and hence $A_\alpha=GY$, without loss of generality.
    Since $\Reg(\epclass_0)$ is the closure of $\epclass_0$ under multiple pushout (\cref{prop:fundamental_prop_of_GREpi}\cref{prop:fundamental_prop_of_GREpi-closure}) and the left adjoint $F$ preserves multiple pushouts, we have $F\Reg(\epclass_0) \subseteq \Reg(\epclass_1)$, i.e., every $\epclass_0$-regular epimorphism is sent to a $\epclass_1$-regular epimorphism by $F$.
    Since the left adjoint $F$ also preserves colimits of chains, the image of $(A_\bullet,h_\bullet,GY)$ under $F$, denoted by $(FA_\bullet,Fh_\bullet,FGY)$, becomes a $\Reg(\epclass_1)$-decomposition of $FGf$ with $FA_\alpha = FGY$.
    Since the class $\Reg(\epclass_1)$ is closed under pushout (\cref{prop:fundamental_prop_of_GREpi}\cref{prop:fundamental_prop_of_GREpi-multiple_pushout}), taking pushouts along $\epsilon_X$, where $\epsilon$ denotes the counit of the adjunction $F\dashv G$, we obtain a $\Reg(\epclass_1)$-decomposition $(B_\bullet,k_\bullet,B_\alpha)$ and a morphism $\phi\colon (FA_\bullet,Fh_\bullet,FGY) \arr[] (B_\bullet,k_\bullet,B_\alpha)$ below with $\abs{k_\bullet}\le\alpha$, $B_0=X$, and $\phi_0=\epsilon_X$:
    \begin{equation*}
        \begin{tikzcd}
            (FGX=) &[-35] FA_0\ar[rr,bend left=20,"FGf"]\ar[d,"(\epsilon_X=)~\phi_0"']\ar[r,"FA_{0,\beta}"'] & FA_\beta\ar[d,"\phi_\beta"]\ar[r,"FA_{\beta,\alpha}"'] & FA_\alpha\ar[d,"\phi_\alpha"]\ar[rrd,"\epsilon_Y"] &[-35] (=FGY) & \\
            (X=) & B_0\ar[rrrr,bend right=15,"f"']\ar[r,"B_{0,\beta}"]\ & B_\beta\ar[r,"B_{\beta,\alpha}"]\pushoutcorner & B_\alpha\ar[rr,"\pocomp{f}"]\pushoutcorner && Y
        \end{tikzcd}\incat{\D}\quad (\beta\le\alpha).
    \end{equation*}
    Suppose that $\decnum[\Reg(\epclass_1)]{f}$ is defined.
    Then, applying \cref{lem:pushing_of_decomposition} with $e\coloneq B_{0,\alpha}\in\lorth\rorth(\bfE)$ and $m\coloneq\id_Y$, we can show that $\decnum[\Reg(\epclass_1)]{\pocomp{f}}$ is defined.

    For converse, let us suppose that $\decnum[\Reg(\epclass_1)]{\pocomp{f}}$ is defined.
    Then, \cref{prop:decnum_of_composite}\cref{prop:decnum_of_composite-precompose} shows that $\decnum[\Reg(\epclass_1)]{f}$ is defined and
    \[
        \decnum[\Reg(\epclass_1)]{f}
        \le
        \decnum[\Reg(\epclass_1)]{B_{0,\alpha}} + \decnum[\Reg(\epclass_1)]{\pocomp{f}}
        \le
        \abs{k_\bullet} + \decnum[\Reg(\epclass_1)]{\pocomp{f}}
        \le
        \alpha + \decnum[\Reg(\epclass_1)]{\pocomp{f}}
    \]
    holds, which shows \cref{prop:adj_and_decnum-general}.

    To show \cref{prop:adj_and_decnum-plusone}, suppose that both $\epsilon_X$ and $\epsilon_Y$ are $\epclass_1$-regularly epic.
    Since $\epsilon_X$ is epic (\cref{prop:fundamental_prop_of_GREpi}\cref{prop:fundamental_prop_of_GREpi-epimorphicity}), the morphism $\phi_\alpha$, given by the pushout, is also epic.
    Then, since $\epsilon_Y=\pocomp{f}\circ \phi_\alpha$ is $\epclass_1$-regularly epic, \cref{prop:composable_pair_GREpi}\cref{prop:composable_pair_GREpi-cancellability} implies that $\pocomp{f}$ is $\epclass_1$-regularly epic.
    As a result, $\decnum[\Reg(\epclass_1)]{\pocomp{f}}$ is defined and less than or equal to $1$.
    Then, \cref{prop:adj_and_decnum-plusone} follows from \cref{prop:adj_and_decnum-general}.

    It remains to prove \cref{prop:adj_and_decnum-pluszero}.
    By assumption, \cref{lem:counit_cotaut} works, and hence the morphism $\pocomp{f}$ becomes an isomorphism.
    Therefore, $\decnum[\Reg(\epclass_1)]{\pocomp{f}}$ is defined and equal to $0$, and hence \cref{prop:adj_and_decnum-pluszero} follows from \cref{prop:adj_and_decnum-general}.
\end{proof}

Applying \cref{prop:adj_and_decnum} for the case of ordinary regular epimorphisms, we obtain the following:
\begin{corollary}\label{cor:adj_and_decnum_regularepi}
    Let $F\colon\C\adjl\D\lon G$ be an adjunction such that $\D$ has pushouts and both $\C$ and $\D$ have binary copowers and kernel pairs.
    Let $f$ be a morphism in $\D$ such that the regular-decomposition number $\decnum{Gf}$ is defined and $Gf$ is strongly epic.
    Then, the following hold:
    \begin{enumerate}
        \item\label{cor:adj_and_decnum_regularepi-definability}
            $\decnum{f}$ is defined if and only if $\decnum{\pocomp{f}}$ is defined.
        \item\label{cor:adj_and_decnum_regularepi-general}
            Whenever $\decnum{f}$ and $\decnum{\pocomp{f}}$ are defined, we have
            \begin{equation*}
                \decnum{f} \le \decnum{Gf} + \decnum{\pocomp{f}}.
            \end{equation*}
        \item\label{cor:adj_and_decnum_regularepi-plusone}
            Suppose that the adjunction $F\dashv G$ is of descent type, i.e., its counit is regularly epic componentwise.
            Then $\decnum{f}$ is defined, and we have
            \begin{equation*}
                \decnum{f} \le \decnum{Gf} + 1.
            \end{equation*}
        \item\label{cor:adj_and_decnum_regularepi-pluszero}
            Suppose that the adjunction $F\dashv G$ is of descent type and the morphism $GFGf$ is epic.
            Then $\decnum{f}$ is defined, and we have
            \begin{equation*}
                \decnum{f} \le \decnum{Gf}.
            \end{equation*}
    \end{enumerate}
\end{corollary}
\begin{proof}
    Let $\epclass_0$ and $\epclass_1$ be the classes of all codiagonals on all objects in $\C$ and $\D$, respectively.
    As we saw in \cref{eg:regular_epi}, the classes $\Reg(\epclass_0)$ and $\Reg(\epclass_1)$ are precisely the classes of all regular epimorphisms in $\C$ and $\D$, respectively.
    Since the left adjoint $F$ preserves colimits, the inclusions $F\epclass_0\subseteq\epclass_1\subseteq\Reg(\epclass_1)$ hold.
    Then, the statements \cref{cor:adj_and_decnum_regularepi-definability,cor:adj_and_decnum_regularepi-general,cor:adj_and_decnum_regularepi-plusone} directly follow from \cref{prop:adj_and_decnum}.

    We now show \cref{cor:adj_and_decnum_regularepi-pluszero}.
    By the characterization of adjunctions of descent type \cite[Theorem 2.4]{KellyPower1993adjunctions}, each component of the counit is regularly epic, in particular, epic.
    Since the latter is equivalent to the class $\{FX\}_{X\in\C}$ being a generator of $\D$, \cref{eg:regular_epi_with_generator} implies that $F\epclass_0$-regular epimorphisms are precisely regular epimorphisms in $\D$.
    Then, \cref{prop:adj_and_decnum}\cref{prop:adj_and_decnum-pluszero} works, and we obtain the desired inequality.
\end{proof}

\begin{corollary}\label{cor:adj_and_globaldecnum}
    Let $\C$ be a category equipped with a class of epimorphisms $\epclass_0$ such that $\decnum[\Reg(\epclass_0)]{\C}$ is defined.
    Let $\D$ be a category with pushouts that admits an orthogonal factorization system $(\Stg(\epclass_1),\rorth(\epclass_1))$ for a class of epimorphisms $\epclass_1$.
    Let $F\colon\C\adjl\D\lon G$ be an adjunction with $F\epclass_0\subseteq\Reg(\epclass_1)$.
    \begin{enumerate}
        \item\label{cor:adj_and_globaldecnum-plusone}
            Under the following assumptions:
            \begin{itemize}
                \item
                    the counit of the adjunction is $\epclass_1$-regularly epic componentwise,
                \item
                    for every $\epclass_1$-strong epimorphism $e$ in $\D$, $Ge$ is $\epclass_0$-strongly epic in $\C$
            \end{itemize}
            then $\decnum[\Reg(\epclass_1)]{\D}$ is defined, and we have
            \[
                \decnum[\Reg(\epclass_1)]{\D} \le \decnum[\Reg(\epclass_0)]{\C} + 1.
            \]
        \item\label{cor:adj_and_globaldecnum-pluszero}
            Under the following assumptions:
            \begin{itemize}
                \item
                    the counit of the adjunction is $F\epclass_0$-regularly epic componentwise,
                \item
                    for every $\epclass_1$-strong epimorphism $e$ in $\D$, $Ge$ is $\epclass_0$-strongly epic, and $GFGe$ is epic in $\C$
            \end{itemize}
            then $\decnum[\Reg(\epclass_1)]{\D}$ is defined, and we have
            \[
                \decnum[\Reg(\epclass_1)]{\D} \le \decnum[\Reg(\epclass_0)]{\C}.
            \]
    \end{enumerate}
\end{corollary}
\begin{proof}
    The proof of \cref{cor:adj_and_globaldecnum-pluszero} is similar to that of \cref{cor:adj_and_globaldecnum-plusone}, so we prove only \cref{cor:adj_and_globaldecnum-plusone}.
    Let $f$ be an arbitrary morphism in $\D$, and take its factorization $f=me$ into a $\epclass_1$-strong epimorphism $e$ and a morphism $m\in\rorth(\epclass_1)$.
    By assumption, $\decnum[\Reg(\epclass_0)]{Ge}$ is defined, and $Ge$ is $\epclass_0$-strongly epic.
    Then, by \cref{prop:adj_and_decnum}\cref{prop:adj_and_decnum-plusone}, $\decnum[\Reg(\epclass_1)]{e}$ is defined, and we have
    \[
        \decnum[\Reg(\epclass_1)]{e} \le \decnum[\Reg(\epclass_0)]{Ge} + 1 < \decnum[\Reg(\epclass_0)]{\C} + 1.
    \]
    In fact, $\decnum[\Reg(\epclass_1)]{f}$ is also defined and equal to $\decnum[\Reg(\epclass_1)]{e}$ by \cref{prop:decnum_of_composite}\cref{prop:decnum_of_composite-postcompose}, which finishes the proof.
\end{proof}

\begin{corollary}\label{cor:adj_and_globaldecnum_regularepi}
    Let $\C$ and $\D$ be categories with binary copowers and kernel pairs.
    Suppose that $\decnum{\C}$ is defined, and that $\D$ has pushouts and admits the (strong epi, mono)-factorization system.
    Let $F\colon\C\adjl\D\lon G$ be an adjunction of descent type.
    \begin{enumerate}
        \item\label{cor:adj_and_globaldecnum_regularepi-plusone}
            Suppose that for every strong epimorphism $e$ in $\D$, $Ge$ is strongly epic in $\C$.
            Then $\decnum{\D}$ is defined, and we have
            \[
                \decnum{\D} \le \decnum{\C} + 1.
            \]
        \item\label{cor:adj_and_globaldecnum_regularepi-pluszero}
            Suppose that for every strong epimorphism $e$ in $\D$, $Ge$ is strongly epic and $GFGe$ is epic in $\C$.
            Then $\decnum{\D}$ is defined, and we have
            \[
                \decnum{\D} \le \decnum{\C}.
            \]
    \end{enumerate}
\end{corollary}
\begin{proof}
    These statements follow from \cref{cor:adj_and_decnum_regularepi} by a similar argument to the proof of \cref{cor:adj_and_globaldecnum}.
\end{proof}

\begin{example}[Monadic categories over $\Set$]
    Let $F\colon\Set\adjl\C\lon U$ be a monadic adjunction, where $\Set$ denotes the category of (small) sets.
    Then, for every strong epimorphism $e$ in $\C$, the underlying map $Ue$ is surjective.
    In fact, $FUFe$ is also surjective since any surjection is a split epimorphism in $\Set$.
    Then, \cref{cor:adj_and_globaldecnum_regularepi}\cref{cor:adj_and_globaldecnum_regularepi-pluszero} implies $\decnum{\C}\le\decnum{\Set}=2$.
    This inequality can also be derived from the fact that every monadic category over $\Set$ is regular.
\end{example}

The following example shows that the condition that a right adjoint preserves strong epimorphisms cannot be removed from \cref{cor:adj_and_globaldecnum_regularepi}.
We use the framework of \textit{\acp{PHT}}, introduced in \cite{PalmgrenVickers2007partial}.
We refer the reader to \cref{sec:partial_horn_theories} for a self-contained summary of the notation and conventions used in the paper.
\begin{example}\label{eg:monadic_adj_plustwo}
    We give an example for a monadic adjunction that increases the global regular-decomposition number by $2$.
    Let $\Sigma\coloneq\{\syn{a},\syn{b},\syn{c}\}$ and $\Sigma'\coloneq\{\syn{a},\syn{b},\syn{c},\syn{d}\}$ be one-sorted signatures consisting of (partial) constants.
    Define finitary partial Horn theories $\theory{S}$ over $\Sigma$ and $\theory{T}$ over $\Sigma'$ as follows:
    \begin{gather*}
        \theory{S}\coloneq
        \left\{
            \top
                \seq{()}
                    (\syn{a}=\syn{a})
                    \wedge
                    (\syn{b}=\syn{b})
        \right\}
        \\
        \theory{T}\coloneq\theory{S}+
        \left\{
            \begin{gathered}
                \syn{a}=\syn{b}
                    \biseq{()}
                        \syn{c}=\syn{c},
                \\
                \syn{a}=\syn{c}
                    \biseq{()}
                        \syn{d}=\syn{d}
            \end{gathered}
        \right\}
    \end{gather*}
    Then, the theory extension $\theory{S}\subseteq\theory{T}$ induces the (finitary) monadic adjunction.
    \[
        \adjoint{\PMod\theory{S}}{\PMod\theory{T}}{F}{U}
    \]
    The monadicity directly follows from the framework of \textit{relative algebraic theories} introduced in \cite{Kawase2026relativized}.
    Indeed, the definability of the new constant $\syn{d}$ is characterized by the equation $\syn{a}=\syn{c}$ from $\theory{S}$, and hence $\theory{T}$ is an $\theory{S}$-relative algebraic theory.

    We also note that the category $\PMod\theory{S}$ of models of $\theory{S}$ is regular because $\theory{S}$ can be rewritten as a two-sorted equational theory by regarding the partial constant $\syn{c}$ as a unary function from a new sort $\syn{t}$ and by adding a new equation $\top\seq{\vx{,}\vy\ofsort\syn{t}} \vx=\vy$.
    Since the other constants $\syn{a},\syn{b}$ are already total, no modification to them is required.
    As a result, we have $\decnum{\PMod\theory{S}}=2$.

    However, if we consider a two-element model $M\coloneq\{a,b\}\in\PMod\theory{T}$ in which neither $c$ nor $d$ is undefined, we find $\decnum{!_M}=3$ for the unique morphism $!_M$ from $M$ to the terminal model $T\coloneq\{a=b=c=d\}$ in $\PMod\theory{T}$.
    Here, the letters $a,b,c,d$ are meant to be the interpretations of the constants $\syn{a},\syn{b},\syn{c},\syn{d}$.
    Indeed, the canonical regular decomposition of $!_M$ can be illustrated as follows:
    \begin{equation*}
        \begin{tikzcd}
            M= &[-35] \{a,b\}\ar[dr,two heads]\ar[rrr,"!_M"] &[-30] & &[-30] \{a=b=c=d\} &[-35] =T \\
            & & \{a=b,c\}\ar[r,two heads] & \{a=b=c,d\}\ar[ru,two heads] & &
        \end{tikzcd}\incat{\PMod\theory{T}}.
    \end{equation*}
    Therefore, we obtain $\decnum{\PMod\theory{T}} \ge 4$ (in fact, even the equality holds, as shown later in \cref{eg:toyexample_of_gauge}).
    Note that the morphism $!_M$ is strongly epic in $\PMod\theory{T}$, but the underlying morphism $U(!_M)$ is not.
    Indeed, in the category $\PMod\theory{S}$, the morphism $U(!_M)$ factors through a non-trivial submodel $\{a=b\}\subsetneq\{a=b=c\}=U(T)$.
\end{example}

\subsection{The decomposition numbers for locally presentable categories}\label{subsec:decnum_of_lp}
We now focus on a special case of locally presentable categories and provide a ``loose'' upper bound for the global $\Reg(\epclass)$-decomposition numbers under the assumption that $\epclass$ is small.
\begin{theorem}\label{thm:GREpi_decomposition_in_lp}
    Let $\kappa\le\mu$ be small regular cardinals.
    Let $\C$ be a locally $\kappa$-presentable category.
    Let $\epclass$ be a small class of epimorphisms in $\C$ whose domain and codomain are $\mu$-presentable in $\C$.
    Then, every morphism $f$ in $\C$ has a $\Reg[\C](\epclass)$-decomposition whose length is less than or equal to $\mu$, hence $\decnum[{\Reg[\C](\epclass)}]{f}\le\mu$ and $\decnum[{\Reg[\C](\epclass)}]{\C}\le \mu+1$.
\end{theorem}
\begin{proof}
    This follows from the well-known small object argument, but for completeness, we present a direct proof here.
    By \cref{thm:GREpi_admissibility}, the category $\C$ is $\Reg(\epclass)$-admissible.
    Thus, by \cref{prop:canonical_predecomposition}, every morphism $f$ in $\C$ has the canonical $\Reg(\epclass)$-predecomposition $(A_\bullet,f_\bullet,X)$.
    Since $\decnum[\Reg(\epclass)]{f}=\cdecnum[\Reg(\epclass)]{f}$ (\cref{thm:correspondence_decnum_and_cdecnum}), we will prove $\cdecnum[\Reg(\epclass)]{f}\le\mu$, that is, $f_\bullet$ stabilizes at $\mu$.

    By \cref{cor:simple_criterion_of_cdec_stabilizing}, it suffices to show $f_\mu\in\rorth(\epclass)$.
    Suppose that we are given a commutative square with $e\in\epclass$ below:
    \begin{equation*}
        \begin{tikzcd}
            \Gamma\ar[d,"e"']\ar[r,"u"] & A_\mu\ar[d,"f_\mu"] \\
            \Delta\ar[r,"v"'] & X
        \end{tikzcd}\incat{\C}.
    \end{equation*}
    Since $\Gamma$ is $\mu$-presentable and $A_\mu$ is the $\mu$-filtered colimit of its predecessors, the morphism $u$ factors through some coprojection $A_{\alpha,\mu}\colon A_\alpha\arr A_\mu$.
    Then, $e\orth[A_{\alpha,\alpha+1}] f_{\alpha+1}$ implies that there is a (unique) morphism $h$ making the following commute:
    \begin{equation*}
        \begin{tikzcd}[largerow]
            \Gamma\ar[rrr,"u",shift left=4]\ar[d,"e"']\ar[r] & A_\alpha\ar[r,"A_{\alpha,\alpha+1}"{below=-2}] & A_{\alpha+1}\ar[r,"A_{\alpha+1,\mu}"'] & A_\mu\ar[d,"f_\mu"] \\
            \Delta\ar[rru,dotted,"h"']\ar[rrr,"v"'] &&& X
        \end{tikzcd}\incat{\C}.
    \end{equation*}
    Since $e$ is epic, the existence of $h$ implies the desired property $e\orth f_\mu$, hence $f_\mu\in\rorth(\epclass)$.
\end{proof}

\begin{corollary}\label{cor:GSEpi_is_transfinite_composite_of_GREpi}
    Let $\kappa\le\mu$ be small regular cardinals.
    Let $\C$ be a locally $\kappa$-presentable category.
    Let $\epclass$ be a small class of epimorphisms in $\C$ whose domain and codomain are $\mu$-presentable in $\C$.
    Then, a morphism in $\C$ is $\epclass$-strongly epic if and only if it is a composite of a $\Reg[\C](\epclass)$-transfinite $(\mu+1)$-sequence.
\end{corollary}
\begin{proof}
    Since the class $\lorth\rorth(\epclass)$ is closed under composition and wide pushout, a composite of a $\Reg(\epclass)$-transfinite $(\mu+1)$-sequence is $\epclass$-strongly epic.
    Suppose that we are given a $\epclass$-strong epimorphism $e$ conversely.
    By \cref{thm:GREpi_decomposition_in_lp}, $e$ has a $\Reg(\epclass)$-decomposition $(A_\bullet,e_\bullet,X)$ with $\abs{e_\bullet}\le\mu$.
    Then, \cref{lem:decomposition_of_leftclass} shows that the morphism $e_\mu$ is an isomorphism, which shows that $e$ is a composite of a $\Reg(\epclass)$-transfinite $(\mu+1)$-sequence.
\end{proof}

\begin{example}\label{eg:regular_epi_is_smallgenerating}
    Let $\C$ be a locally $\kappa$-presentable category.
    Let $G\subseteq\Ob\C$ be a skeleton of the class of all $\kappa$-presentable objects in $\C$, and let $\epclass$ be the class of the codiagonals on objects from $G$.
    Since $G$ becomes a generator of the category $\C$, $\epclass$-regular epimorphisms are precisely regular epimorphisms (\cref{eg:regular_epi_with_generator}).
    As already mentioned in \cref{eg:canonical_predecomposition_for_regular_epis}, the canonical regular (pre)decomposition of a morphism in $\C$ is obtained by iterated application of the coequalizer of the kernel pair.
\end{example}

If we consider the class $\epclass$ in the above example, each of \cref{thm:GREpi_decomposition_in_lp,cor:GSEpi_is_transfinite_composite_of_GREpi} gives rise to a corollary:
\begin{corollary}[{\cite[{}6.6]{GabrielUlmer1971lokal}}]\label{cor:upperbound_for_global_regdecnum_of_lp}
    Let $\C$ be a locally $\kappa$-presentable category.
    Then, for every morphism $f$ in $\C$, we have $\decnum{f}\le\kappa$, hence $\decnum{\C}\le\kappa+1$.
\end{corollary}

\begin{corollary}
    Let $\C$ be a locally $\kappa$-presentable category.
    Then, a morphism in $\C$ is strongly epic if and only if it is a composite of a regular transfinite $(\kappa+1)$-sequence.
\end{corollary}

\begin{example}[A category with $\decnum{\C}=\omega+1$]\label{eg:global_decnum_is_omegaplusone}
    We present an example for a locally finitely presentable category $\C$ such that $\decnum{\C}=\omega+1$, which is the largest possible number by \cref{cor:upperbound_for_global_regdecnum_of_lp}.
    We will define a finitary \ac{PHT}, and will consider the category of models of it as the category $\C$.
    Let $\Sigma\coloneq\{ \syn{b},\syn{c}_0,\syn{c}_1,\syn{c}_2,\dots \}$ be a one-sorted signature consisting of countably many (partial) constant symbols, and let $\theory{T}$ be a finitary \ac{PHT} over $\Sigma$ consisting of the following Horn sequents:
    \begin{gather*}
        \left(
        \syn{c}_{n} = \syn{b}
            \seq{()}
                \syn{c}_{n+1} = \syn{c}_{n+1}
        \right)_{n\ge 0}
    \end{gather*}
    For $n\ge 0$, we define a model $A_n$ of $\theory{T}$ as follows: the underlying set is the two-element set $\{0,1\}$ with $0\neq 1$; the partial constants $\syn{b}$ and $\syn{c}_i$ $(0\le i < n)$ are defined and interpreted as $0$; $\syn{c}_n$ is also defined and interpreted as $1$; the other constants $\syn{c}_i$ $(i>n)$ are undefined.
    Let $T$ denote the terminal model of $\theory{T}$, whose underlying set is the singleton, and in which all constants are defined.
    Then, for every $n\ge 0$, the coequalizer of the kernel pairs of the unique morphism $!_n\colon A_n\arr T$ in $\PMod\theory{T}$ is given by a morphism $A_n\arr A_{n+1}$ that sends all elements to $0\in A_{n+1}$.
    This means that the canonical regular decomposition of the morphism $!_0\colon A_0\arr T$ does not stabilize in any finite steps, so it stabilizes exactly at $\omega$ by \cref{cor:upperbound_for_global_regdecnum_of_lp}.
    Consequently, we obtain $\decnum{\PMod\theory{T}} = \omega+1$.
\end{example}

\begin{example}[A category with $\decnum{\C}=\omega$]\label{eg:global_decnum_is_omega}
    We now present an example for a locally finitely presentable category $\C$ such that $\decnum{\C}=\omega$.
    As well as the previous example, we will use the framework of \acp{PHT}.
    Let $\Sigma$ be the same signature as in \cref{eg:global_decnum_is_omegaplusone}, and let $\theory{T}$ be a finitary \ac{PHT} over $\Sigma$ consisting of the following Horn sequents:
    \begin{gather*}
        \top
            \seq{()}
                \syn{b} = \syn{b}
        \\
        \left(
        \syn{c}_n = \syn{c}_n
            \seq{()}
                \syn{c}_{n+1} = \syn{c}_{n+1}
        \right)_{n\ge 0}
        \\
        \left(
        \syn{c}_{n+1} = \syn{b}
            \seq{()}
                \syn{c}_n = \syn{c}_n
        \right)_{n\ge 0}
    \end{gather*}
    Then, a model of $\theory{T}$ is precisely the following data $(M,b,\alpha,c_\bullet)$:
    \begin{itemize}
        \item
            a set $M$;
        \item
            an element $b\in M$, which is regarded as $\intpn{\syn{b}}{M}$;
        \item
            an ordinal $\alpha\in [0,\omega]$ and a family $(c_n\in M)_{\alpha\le n <\omega}$ of elements, which are regarded as $\intpn{\syn{c}_n}{M}$
    \end{itemize}
    such that if $0<\alpha<\omega$, then $c_\alpha\neq b$ holds.
    Note that a morphism
    \[
        M=(M,b,\alpha,c_\bullet)\arr[] (M',b',\alpha',c'_\bullet)=M'
        \incat{\PMod\theory{T}}
    \]
    exists only if $\alpha\ge\alpha'$, and it is precisely a map $f\colon M\arr M'$ such that $f(b)=b'$ and $f(c_n)=c'_n$ hold for $n\ge\alpha$.
    \begin{claim}\label{claim:decnum_f_finite}
        For every morphism $f$ in $\PMod\theory{T}$, $\decnum{f}<\omega$ holds.
    \end{claim}
    \begin{since}
        In what follows, we will calculate the coequalizer $N$ of the kernel pair of an arbitrary morphism $f\colon M=(M,b,\alpha,c_\bullet)\arr[] (M',b',\alpha',c'_\bullet)=M'$ in $\PMod\theory{T}$:
        \begin{equation*}
            \begin{tikzcd}[small]
                M\ar[dr,two heads,"q"']\ar[rr,"f"] && M' \\
                & N\ar[ru,"f_1"'] &
            \end{tikzcd}\incat{\PMod\theory{T}}.
        \end{equation*}
        As noted in \cref{rem:canonical_predecomp_for_GRegEpi}, the canonical regular predecomposition of $f$ is obtained by repeating this factorization.
        The calculation is split into two cases:

        \noindent\textbf{Case I: either $\alpha=0$, $\alpha=\omega$, or both $0<\alpha<\omega$ and $c'_\alpha\neq b'$ hold.}
        Let $f(M)\subseteq M'$ be the image of $M$ under the map $f$.
        Then, a tuple $N=(f(M),b',\alpha,(c'_n)_{\alpha\le n<\omega})$ becomes a model of $\theory{T}$, and the morphisms $q$ and $f_1$ are induced by $f$ naturally.
        Since the morphism $f_1$ is monic, $\decnum{f}=\cdecnum{f}\le 1$ holds in this case.

        \noindent\textbf{Case II: both $0<\alpha<\omega$ and $c'_\alpha = b'$ hold.}
        Let $f(M)\subseteq M'$ be the image of $M$ under the map $f$ as well.
        In this case, the same construction as in the previous case does not work, because in order for $f(M)$ to be a model of $\theory{T}$, the constant $\syn{c}_{\alpha-1}$ must be defined in $f(M)$.
        Thus, we add a new element $\bar{c}_{\alpha-1}$ to the set $f(M)$, and then we find that $N = (f(M)+\{\bar{c}_{\alpha-1}\}, b', \alpha-1, \bar{c}_\bullet)$ becomes a model of $\theory{T}$, where $\bar{c}_n\coloneq c'_n$ for $\alpha\le n<\omega$.
        Then, the morphism $f$ factors through $N$.
        Indeed, the assumption $c'_\alpha=b'$ implies that $c'_{\alpha-1}$ is already defined in $M'$, and hence the inclusion $f(M)\arr[hook]M'$ uniquely extends to the morphism $f_1\colon N\arr M'$ in $\PMod\theory{T}$.
        The morphism $q$ is given by restricting $f$.

        Since the ordinal associated with $N$ is strictly smaller than that associated with $M$ in the case II, iterated application of the factorization  that described above stops after at most finitely many steps, which shows $\decnum{f}=\cdecnum{f}<\omega$.
    \end{since}

    \begin{claim}\label{claim:decnum_f_unbounded}
        For every natural number $n$, there is a morphism $f$ in $\PMod\theory{T}$ such that $n\le\decnum{f}$.
    \end{claim}
    \begin{since}
        For a natural number $m\ge 0$, let $M_m\coloneq(\{0,1\}, 0, m, c^m_\bullet)$ be a model of $\theory{T}$ defined as follows: $\{0,1\}$ is a two-element set; $c^m_m\coloneq 1$ and $c^m_{l}\coloneq 0$ for $l\ge m+1$.
        Let $T$ denote the terminal model of $\theory{T}$, that is, $T\coloneq(\{\ast\}, \ast, 0, (\ast)_{n})$.
        Let $!_m$ denote the unique morphism $M_m\arr T$ in $\PMod\theory{T}$.
        Then, the following diagram exhibits the canonical regular decomposition of $!_n$:
        \begin{equation*}
            \begin{tikzcd}
                M_n\ar[d,two heads,"q_n"']\ar[rrrrr,"!_n"]
                    &
                        &
                            &
                                &
                                    &
                                    T
                \\
                M_{n-1}\ar[rd,two heads,"q_{n-1}"']\ar[rrrrru,"!_{n-1}"{description}]
                    &
                        &
                            &
                                &
                                    &
                \\
                    &
                    M_{n-2}\ar[r,two heads,"q_{n-2}"']\ar[rrrruu,"!_{n-2}"{description}]
                        &
                        \cdots\ar[r,two heads,"q_2"']
                            &
                            M_1\ar[r,two heads,"q_1"']\ar[rruu,"!_1"{description}]
                                &
                                M_0\ar[ruu,two heads,"!_0"']
                                    &
            \end{tikzcd}\incat{\PMod\theory{T}}.
        \end{equation*}
        Here, each $q_m$ is the coequalizer of the kernel pair of $!_m$, which sends all elements to $0$.
        Since the last morphism $!_0$ is actually regularly epic, it follows that $\decnum{!_n}=\cdecnum{!_n}=n+1$.
    \end{since}
    By \cref{claim:decnum_f_finite,claim:decnum_f_unbounded}, we now conclude that $\decnum{\PMod\theory{T}}=\omega$.
\end{example}

\begin{example}\label{eg:iterated_localization_decnum_in_Cat}
    We show that the category $\Cat$ of small categories has the global $\bfE$-decomposition number $\decnum[\bfE]{\Cat}=\omega+1$ where $\bfE$ is the class of all strict localizations.
    Strict localizations are precisely $\epclass$-regular epimorphisms for the class $\epclass$ as in \cref{eg:strict_localizations_in_Cat}.
    Since $\Cat$ is locally finitely presentable and the unique morphism in $\epclass$ has a finitely presentable domain and codomain, we have $\decnum[\bfE]{\Cat}\le\omega+1$ by \cref{thm:GREpi_decomposition_in_lp}.
    In what follows, we give an example of a morphism $f$ in $\Cat$ with $\decnum[\bfE]{f}=\omega$.
    Let $\C_0$ and $\D$ be the categories as below. 
        \begin{equation*}
        \C_0\coloneq\left(
        \begin{tikzcd}
            A_{-1}
            \ar[d,"h_{-1}"']
                &
                A_0
                \ar[dl,"s_0"{above left=-1,pos=0.1}]
                    &
                    A_1
                    \ar[dl,"s_1"{above left=-1,pos=0.1}]
                        &
                        \cdots
                        \ar[dl,"s_2"{above left=-1,pos=0.1}]
            \\
            B_{-1}
                &
                B_0
                    &
                    B_1
                        &
                        \cdots
            \arrow[from=1-1,to=2-2,crossing over,"r_0"{below left=-1,pos=0.9}]
            \arrow[from=1-2,to=2-3,crossing over,"r_1"{below left=-1,pos=0.9}]
            \arrow[from=1-3,to=2-4,crossing over,"r_2"{below left=-1,pos=0.9}]
        \end{tikzcd}
        \right)
    \end{equation*}
    \begin{equation*}
        \D\coloneq\left(
        \begin{tikzcd}
            X_{-1}
            \ar[r,"r_0",shift left=1]
                &
                X_0
                \ar[r,"r_1",shift left=1]
                \ar[l,"s_0",shift left=1]
                    &
                    X_1
                    \ar[r,"r_2",shift left=1]
                    \ar[l,"s_1",shift left=1]
                        &
                        \cdots
                        \ar[l,"s_2",shift left=1]
        \end{tikzcd}
        \right)
        \quad
        \text{with $r_n\circ s_n = \id$ $(n\ge 0)$}
    \end{equation*}
    Let $F_0\colon \C_0\arr \D$ be the functor that sends $A_k$ and $B_k$ to $X_k$, $h_{-1}$ to the identity on $X_{-1}$, and is identity on the others.
    Then, the canonical $\bfE$-decomposition $F_\bullet$ of $F_0$ is given as follows:
    \begin{equation*}
        \begin{tikzcd}
            \C_0
            \ar[r, "\C_{0,1}"]
            \ar[rrrrd, "F_0"{below left=-2}, bend right=15, near start]
                &
                \C_1
                \ar[r, "\C_{1,2}"]
                \ar[rrrd, "F_1"{below left=-2}, bend right=8, near start]
                    &
                    \C_2
                    \ar[r, "\C_{2,3}"]
                    \ar[rrd, "F_2"{below left=-2,pos=0.4}]
                        &
                        \cdots
                            &[-30pt]
                            \C_\omega
                            \ar[d,"F_\omega"]
            \\
                &
                    &
                        &
                            &
                            \D
        \end{tikzcd}
    \end{equation*}
    Here, each $\C_\alpha$ $(1\le \alpha\le \omega)$ is the category obtained by adding morphisms $h_k\colon A_k\arr B_k$ $(0\le k<\alpha)$ to $\C_0$, subject to the conditions that for every $0\le k<\alpha$, $h_{k-1}$ is an isomorphism and $r_k\circ h_{k-1}^{-1} \circ s_k = h_k$ holds.
    For instance, each category $\C_n$ for $n<\omega$ can be illustrated as below.
    \begin{gather*}
        \C_n\coloneq\left(
        \begin{tikzcd}[hugecolumn,largerow,ampersand replacement=\&]
            A_{-1}
            \ar[d,"h_{-1}"',"\cong"{sloped}]
                \&
                \cdots
                \ar[dl,"s_0"{above left=-1,pos=0.1}]
                    \&
                    A_{n-2}
                    \ar[dl,"s_{n-2}"{above=4,pos=0.2}]
                    \ar[d,"h_{n-2}","\cong"'{sloped}]
                        \&
                        A_{n-1}
                        \ar[d,"h_{n-1}"{right=-1}]
                        \ar[dl,"s_{n-1}"{above=4,pos=0.2}]
                            \&
                            A_n
                            \ar[dl,"s_n"{above left=-1,pos=0.1}]
                                \&
                                \cdots
                                \ar[dl,"s_{n+1}"{above=4,pos=0.2}]
            \\
            B_{-1}
                \&
                \cdots
                    \&
                    B_{n-2}
                        \&
                        B_{n-1}
                            \&
                            B_n
                                \&
                                \cdots
            \arrow[from=1-1,to=2-2,crossing over,"r_0"{below left=-1,pos=0.9}]
            \arrow[from=1-2,to=2-3,crossing over,"r_{n-2}"{below=4,pos=0.8}]
            \arrow[from=1-3,to=2-4,crossing over,"r_{n-1}"{below=4,pos=0.8}]
            \arrow[from=1-4,to=2-5,crossing over,"r_n"{below left=-1,pos=0.9}]
            \arrow[from=1-5,to=2-6,crossing over,"r_{n+1}"{below=4,pos=0.8}]
        \end{tikzcd}
        \right)
    \end{gather*}
    The functors $F_\alpha$ and the strict localizations $\C_{n,n+1}$ are defined accordingly.
    Since $\C_\omega$ is the colimit of the chain $(\C_n)_{n<\omega}$ and $F_\omega$ is conservative, $F_\bullet$ is an $\bfE$-decomposition that stabilizes exactly at $\omega$.
    Moreover, this is canonical by \cref{eg:canonical_predecomposition_for_strict_localizations}.
    Indeed, the morphism $h_{n-1}$ is the only non-isomorphism in $\C_n$ that is inverted by $F_n$, and hence the localization of $\C_n$ by $h_{n-1}$ becomes the next step in the $\bfE$-decomposition, which coincides with $\C_{n,n+1}$.
    Therefore, $\decnum[\bfE]{F_0}=\cdecnum[\bfE]{F_0}=\omega$ follows.
\end{example}
\section{An approach via partial Horn theories}\label{sec:the_partial-Horn-theoretic_approach}
We now present the first syntactic approach to determine the global decomposition numbers of locally presentable categories.
This approach is based on the framework of \ac{PHT}, for which the reader can find a self-contained summary in \cref{sec:partial_horn_theories}.
\subsection{The stratification of a partial Horn theory}\label{subsec:stratification_partialhorn}
We introduce a parameter that specifies a class of generalized regular epimorphisms:
\begin{definition}[Scales]
    Let $\Sigma =(S,\Sigma)$ be a $\kappa$-ary signature.
    A \emph{scale} over $\Sigma$ consists of the following data:
    \begin{itemize}
        \item
            a small set $\Lambda$;
        \item
            for each $\lambda\in\Lambda$, a finitary context $\tup{\vz}^\lambda\ofsort\tup{\syn{s}}^\lambda = (\vz^\lambda_1\ofsort\syn{s}^\lambda_1,\dots,\vz^\lambda_{n^\lambda}\ofsort\syn{s}^\lambda_{n^\lambda})$ over $S$;
        \item
            for each $\lambda\in\Lambda$, a Horn formula $\tup{\vz}^\lambda.\syn{\zeta}^\lambda$ over $\Sigma$ in the context $\tup{\vz}^\lambda$.
    \end{itemize}
    As long as there is no risk of confusion, we simply write $\Lambda$ for a scale.
\end{definition}

\begin{example}[The equational scale]
    Let $\Sigma =(S,\Sigma)$ be a $\kappa$-ary signature.
    There is a scale over $\Sigma$ consisting of the following:
    \begin{itemize}
        \item
            $\Lambda\coloneq S$;
        \item
            for each $\syn{s}\in S$, the context $(\vz^\syn{s}_1\ofsort\syn{s},\vz^\syn{s}_2\ofsort\syn{s})$;
        \item
            for each $\syn{s}\in S$, the Horn formula $(\vz^\syn{s}_1,\vz^\syn{s}_2).\vz^\syn{s}_1 = \vz^\syn{s}_2$.
    \end{itemize}
    This is called the \emph{equational scale}.
\end{example}

A scale is identified with the following small set of epimorphisms:
\begin{notation}\label{note:morphism_class_from_scale}
    Let $\theory{T}=(S,\Sigma,\theory{T})$ be a $\kappa$-ary \ac{PHT}.
    For a scale $\Lambda$ over $\Sigma$, we write its bold font, $\bfLambda$, for the following small set of epimorphisms in $\PMod\theory{T}$:
    \begin{equation*}
        \bfLambda\coloneq
        \left(
        \begin{tikzcd}
            \repn{\tup{\vz}^\lambda.\top}_\theory{T} \ar[r,"\repn{\tup{\vz}^\lambda}_\theory{T}"]
                & \repn{\tup{\vz}^\lambda.\syn{\zeta}^\lambda}_\theory{T}
        \end{tikzcd}
        \right)_{\lambda\in\Lambda}
    \end{equation*}
\end{notation}


\begin{remark}
    Given a scale $\Lambda$, we immediately obtain the class $\Reg(\bfLambda)$ of $\bfLambda$-regular epimorphisms.
    Therefore, a scale can be regarded as a parameter that specifies a class of generalized regular epimorphisms.
    However, not all the classes of generalized regular epimorphisms can be specified by some scale.
    For instance, there is no natural way to present the class $\epclass$ defined in \cref{eg:antiLipschitz_in_Met} in the form of \cref{note:morphism_class_from_scale}.
\end{remark}

\begin{example}[The equational scale specifies the class of regular epimorphisms]
    For a \ac{PHT} $\theory{T}=(S,\Sigma,\theory{T})$, the equational scale is identified with the following class of epimorphisms in $\PMod\theory{T}$:
    \begin{equation*}
        \bfLambda=
        \left(
        \begin{tikzcd}
            \repn{(\vz^\syn{s}_1,\vz^\syn{s}_2).\top}_\theory{T} \ar[r,"\repn{\vz^\syn{s}_1,\vz^\syn{s}_2}_\theory{T}"]
                &
                \repn{(\vz^\syn{s}_1,\vz^\syn{s}_2).\vz^\syn{s}_1 = \vz^\syn{s}_2}_\theory{T}
        \end{tikzcd}
        \right)_{\syn{s}\in S}
    \end{equation*}
    Each morphism in $\bfLambda$ is isomorphic to the codiagonal on the object $\repn{\vz^\syn{s}_1.\top}_\theory{T}$.
    Since the objects $\repn{\vz^\syn{s}_1.\top}_\theory{T}$ $(\syn{s}\in S)$ form a generator in $\PMod\theory{T}$, \cref{eg:regular_epi_with_generator} implies that the $\bfLambda$-regular epimorphisms are precisely the regular epimorphisms in $\PMod\theory{T}$.
\end{example}

\begin{example}\label{eg:scale_specifying_surjections_in_Pos}
    Let $\phtpos=(\{\ast\},\Sigma_\pos,\phtpos)$ be the one-sorted \ac{PHT} for posets, i.e., the signature $\Sigma_\pos$ only contains the binary relation symbol $\le$, and $\phtpos$ consists of the Horn sequents that represent reflexivity, antisymmetry, and transitivity.
    Let $\Lambda$ denote the scale over $\Sigma_\pos$ given by the single Horn formula $(\vz_1,\vz_2).\vz_1\le\vz_2$.
    Then, the induced class $\bfLambda$ coincides with the one considered in \cref{eg:surjections_in_Pos}, and hence the scale $\Lambda$ specifies the class of surjections in $\Pos\cong\PMod\phtpos$.
\end{example}

\begin{definition}[The minimum context]
    Let $\Sigma =(S,\Sigma)$ be a $\kappa$-ary signature.
    Let $\ttau$ be a raw term over $\Sigma$.
    Then, the set $\fv(\ttau)$ of all free variables appearing in $\ttau$ can be regarded as a $\kappa$-ary context $\tup{\vx}=\fv(\ttau)$; hence, $\ttau$ becomes a $\kappa$-ary term $\tup{\vx}.\ttau$.
    The context $\tup{\vx}$ is called \emph{the minimum context of $\ttau$}.
\end{definition}

We now introduce the key concept, called \textit{gauges}, which measures the stratification of partial terms and calculates an upper bound for the global decomposition numbers.
\begin{definition}[$\Lambda$-gauges]
    Let $\theory{T}=(S,\Sigma,\theory{T})$ be a $\kappa$-ary \ac{PHT}, and let $\Lambda$ be a scale over $\Sigma$.
    A \emph{$\Lambda$-gauge for $\theory{T}$} consists of:
    \begin{itemize}
        \item
            for each raw term $\ttau$ over $\Sigma$, a small ordinal $\height\ttau$;
        \item
            for each raw term $\ttau$ over $\Sigma$, a small set $T(\ttau)$;
        \item
            for each $t\in T(\ttau)$, an element $\lambda(t)\in\Lambda$;
        \item
            for each $t\in T(\ttau)$, raw terms $\tup{\tsigma}^t\ofsort \tup{\syn{s}}^{\lambda(t)} = (\tsigma^t_k\ofsort \syn{s}^{\lambda(t)}_k)_{1\le k\le n^{\lambda(t)}}$ in the minimum context of $\ttau$
    \end{itemize}
    such that:
    \begin{itemize}
        \item
            For any $t\in T(\ttau)$ and $1\le k\le n^{\lambda(t)}$, the inequality $\height\tsigma^t_k < \height\ttau$ holds.
        \item
            For any raw term $\ttau$, the following Horn sequent is derivable from $\theory{T}$:
            \begin{equation*}
                \ttau\defined
                \biseq{\tup{\vx}}
                \bigwedge_{t\in T(\ttau)}
                \left(
                    \syn{\zeta}^{\lambda(t)} \subst{\tup{\tsigma}^t}{\tup{\vz}^{\lambda(t)}}
                    \wedge
                    \bigwedge_{1\le k\le n^{\lambda(t)}} \tsigma^t_k \defined
                \right)
            \end{equation*}
            Here, $\tup{\vx}$ is the minimum context of $\ttau$.\qedhere
    \end{itemize}
\end{definition}

\begin{remark}[Equational gauges]
    In the case when $\Lambda$ is the equational scale, we refer to a $\Lambda$-gauge simply as a \emph{gauge}.
    Then, the data consists of the following:
    \begin{itemize}
        \item
            for each raw term $\ttau$ over $\Sigma$, a small ordinal $\height\ttau$;
        \item
            for each raw term $\ttau$ over $\Sigma$, a small set $T(\ttau)$;
        \item
            for each $t\in T(\ttau)$, two raw terms $(\tsigma^t_1, \tsigma^t_2)$ of the same sort $\syn{s}(t)$ in the minimum context of $\ttau$
    \end{itemize}
    such that:
    \begin{itemize}
        \item
            For each $t\in T(\ttau)$, the inequality $\height\tsigma^t_1, \height\tsigma^t_2 < \height\ttau$ holds.
        \item
            For any raw term $\ttau$, the following Horn sequent is derivable from $\theory{T}$:
            \begin{equation*}
                \ttau\defined
                \biseq{\tup{\vx}}
                \bigwedge_{t\in T(\ttau)}
                \tsigma^t_1 = \tsigma^t_2
            \end{equation*}
            Here, $\tup{\vx}$ is the minimum context of $\ttau$.\qedhere
    \end{itemize}
\end{remark}

\begin{remark}
    If we are given a gauge, the equality $\height\ttau = 0$ forces the (raw) term $\ttau$ to be total, since the small set $T(\ttau)$ can only be the empty set.
    Intuitively, terms with $\height\ttau = 1$ are those whose definability can be characterized by equations between total terms, whereas terms with $\height\ttau = 2$ are those whose definability can be characterized by equations between terms with $\height\ttau \le 1$, and so on.
    In this way, gauges bring stratifications in the terms of a \ac{PHT}.
\end{remark}

We now state the main theorem in this section:
\begin{theorem}[Upper bound theorem by gauges]\label{thm:gauge}
    Let $\theory{T}=(S,\Sigma,\theory{T})$ be a $\kappa$-ary \ac{PHT}, and let $\Lambda$ be a scale over $\Sigma$.
    Suppose that we are given a $\Lambda$-gauge for $\theory{T}$ such that there is a small ordinal $\gamma$ satisfying $\height\ttau<\gamma$ for any raw term $\ttau$ over $\Sigma$.
    Then, we have $\decnum[\Reg(\bfLambda)]{f}\le\gamma$ for every morphism $f$ in $\PMod\theory{T}$; therefore, $\decnum[\Reg(\bfLambda)]{\PMod\theory{T}}\le\gamma+1$ also holds.
\end{theorem}
\begin{proof}
    Let $f$ be a morphism in $\PMod\theory{T}$.
    Since $\PMod\theory{T}$ admits an orthogonal factorization system $(\Stg(\bfLambda),\rorth(\bfLambda))$, $f$ can be decomposed as $f=me$, where $e$ is a $\bfLambda$-strong epimorphism and $m\in\rorth(\bfLambda)$.
    Since \cref{prop:decnum_of_composite}\cref{prop:decnum_of_composite-postcompose} implies $\decnum[\Reg(\bfLambda)]{f}=\decnum[\Reg(\bfLambda)]{e}$, it suffices to show $\decnum[\Reg(\bfLambda)]{e}\le\gamma$.

    We now consider another orthogonal factorization system on $\PMod\theory{T}$ consisting of \textit{$\theory{T}$-dense morphisms} and \textit{$\theory{T}$-closed monomorphisms} that are studied in \cite{Kawase2026relativized}.
    Since each morphism $\repn{\tup{\vz}^\lambda}$ in $\bfLambda$ as in \cref{note:morphism_class_from_scale} factors through no proper $\theory{T}$-closed submodel of its codomain, every morphism in $\bfLambda$ is $\theory{T}$-dense.
    Thus, every $\bfLambda$-strong epimorphism is also $\theory{T}$-dense, in particular, so is $e$.
    For a sufficiently large regular cardinal $\mu\in[\kappa,\inacc)$, the domain and codomain of $e$ are $\mu$-presentable in $\PMod\theory{T}$.
    Then, by \cite[3.17.\ Lemma]{Kawase2026relativized}, the morphism $e$ is isomorphic to a morphism of the following form:
    \begin{equation*}
        \repn{\tup{\vx}.\fphi}_\theory{T}
        \arr(\repn{\tup{\vx}}_\theory{T})[][2]
        \repn{\tup{\vx}.\fpsi}_\theory{T}
        \incat{\PMod\theory{T}},
    \end{equation*}
    where $\tup{\vx}$, $\fphi$, and $\fpsi$ are all $\mu$-ary.
    In what follows, we will regard $e$ as the morphism above.

    For each $\lambda\in\Lambda$, by simultaneous recursion, we define a family $(K^\lambda_\alpha)_{\alpha<\inacc}$ of small sets of terms that are in the context $\tup{\vx}$ and whose sorts are compatible with the finitary context $\tup{\vz}^\lambda$:
    \begin{gather}
        \notag
        K^\lambda_0\coloneq\emptyset;
        \\
        \notag
        K^\lambda_\alpha\coloneq\bigcup_{\alpha'<\alpha}K^\lambda_{\alpha'}
            \quad (\alpha\colon\text{of limit type});
        \\
        \notag
        K^\lambda_{\alpha+1}\coloneq K^\lambda_\alpha\cup
            \left\{
                (\tup{\vx}.\ttau_i)_{1\le i\le n^\lambda}
                    \mid
                        \text{The following \cref{eq:first_sequent_of_K} and \cref{eq:second_sequent_of_K} are derivable from $\theory{T}.$}
            \right\}
        \\
        \label{eq:first_sequent_of_K}
        \fphi
        \wedge
        \syn{\epsilon}_\alpha
            \seq{\tup{\vx}}
                \bigwedge_{1\le i\le n^\lambda} \ttau_i \defined
        \\
        \label{eq:second_sequent_of_K}
        \fpsi
            \seq{\tup{\vx}}
                \syn{\zeta}^\lambda\subst{\ttau_i}{\vz^\lambda_i}_{1\le i\le n^\lambda}
    \end{gather}
    Here, $\syn{\epsilon}_\alpha$ denotes the following Horn formula in the context $\tup{\vx}$:
    \begin{equation*}
        \bigwedge_{\lambda'\in\Lambda}\bigwedge_{(\tsigma_i)_i\in K^{\lambda'}_\alpha} \syn{\zeta}^{\lambda'}\subst{\tsigma_i}{\vz^{\lambda'}_i}_{1\le i\le n^{\lambda'}}
    \end{equation*}
    We now consider the objects $A_\alpha\coloneq\repn{\tup{\vx}.\fphi\wedge\syn{\epsilon}_\alpha}_\theory{T}\in\PMod\theory{T}$.
    Then, for every small ordinals $\alpha\le\beta$, the inclusions $K^\lambda_\alpha\subseteq K^\lambda_\beta$ for $\lambda$ implies that the Horn sequent
    \[
        \fphi
        \wedge
        \syn{\epsilon}_\beta
            \seq{\tup{\vx}}
                \fphi
                \wedge
                \syn{\epsilon}_\alpha
    \]
    is derivable from $\theory{T}$.
    Thus, we obtain the morphisms $A_{\alpha,\beta}\coloneq\repn{\tup{\vx}}\colon A_\alpha\arr A_\beta$.
    In fact, for a limit ordinal $\alpha$, $A_\alpha$ is the colimit of chains $(A_{\alpha'})_{\alpha'\in\bbalpha}$ because the set $K^\lambda_\alpha$ is defined as the union of $(K^\lambda_{\alpha'})_{\alpha'<\alpha}$.
    Thus, we obtain a transfinite sequence $A_\bullet=(A_\alpha,A_{\alpha,\beta})_{\alpha\le\beta}$.

    Let $X\coloneq\repn{\tup{\vx}.\fpsi}_\theory{T}$.
    By transfinite induction, we will show that for every small ordinal $\alpha$, the morphism $e_\alpha\coloneq\repn{\tup{\vx}}_\theory{T}\colon A_\alpha\arr X$ is well-defined, that is, the Horn sequent
    \[
        \fpsi
            \seq{\tup{\vx}}
                \fphi
                \wedge
                \syn{\epsilon}_\alpha
    \]
    is derivable from $\theory{T}$.
    The case when $\alpha$ is $0$ or a limit ordinal is trivial.
    Let $\alpha$ be an arbitrary small ordinal, and suppose that $e_\alpha$ is defined.
    Then, for each index $\lambda\in\Lambda$ of the scale, an element of the set $K^\lambda_{\alpha+1}\backslash K^\lambda_\alpha$ corresponds to a morphism $\repn{\tup{\vz}^\lambda.\top}_\theory{T} \arr A_\alpha$ whose composite with $e_\alpha$ factors through the morphism $\repn{\tup{\vz}^\lambda.\top}_\theory{T} \arr \repn{\tup{\vz}^\lambda.\syn{\zeta}^\lambda}_\theory{T}$ in $\bfLambda$ described in \cref{note:morphism_class_from_scale}.
    Indeed, the first Horn sequent \cref{eq:first_sequent_of_K} ensures that $\tup{\ttau}=(\tup{\vx}.\ttau_i)_i \in K^\lambda_{\alpha+1}\backslash K^\lambda_\alpha$ exhibits a morphism $\repn{\tup{\vz}^\lambda.\top}_\theory{T} \arr A_\alpha$, while the second one \cref{eq:second_sequent_of_K} is precisely the condition about factorization.
    \begin{equation*}
        \begin{tikzcd}
            &
            A_\alpha\ar[d,phantom,sloped,"="]
            &[15]
            X\ar[d,phantom,sloped,"="]
            \\[-15]
            \repn{\tup{\vz}^\lambda.\top}_\theory{T} \ar[d,"\repn{\tup{\vz}^\lambda}_\theory{T}"']\ar[r,"\repn{\tup{\ttau}}_\theory{T}"]
            &
            \repn{\tup{\vx}.\fphi\wedge\syn{\epsilon}_\alpha}_\theory{T} \ar[r,"e_\alpha=\repn{\tup{\vx}}_\theory{T}"]
            &
            \repn{\tup{\vx}.\fpsi}_\theory{T}
            \\
            \repn{\tup{\vz}^\lambda.\syn{\zeta}^\lambda}_\theory{T} \ar[rru,dotted,"\repn{\tup{\ttau}}_\theory{T}"']
            &
            &
        \end{tikzcd}\incat{\PMod\theory{T}}
    \end{equation*}
    Then, the object $A_{\alpha+1}$ is the multiple pushout of the morphisms $\repn{\tup{\vz}^\lambda}_\theory{T}$ in $\bfLambda$ along the morphisms $\repn{\tup{\ttau}}_\theory{T}$ exhibited by elements of $K^\lambda_{\alpha+1}\backslash K^\lambda_\alpha$, and hence the morphism $e_{\alpha+1}$ is induced by the universal property of the multiple pushout:
    \begin{equation*}
        \begin{tikzcd}
            \repn{\tup{\vz}^\lambda.\top}_\theory{T} \ar[d,"\repn{\tup{\vz}^\lambda}_\theory{T}"']\ar[r,"\repn{\tup{\ttau}}_\theory{T}"]
                &
                A_\alpha \ar[r,"e_\alpha"]\ar[d,"A_{\alpha,\alpha+1}"']
                &[15]
                X
                \\
                \repn{\tup{\vz}^\lambda.\syn{\zeta}^\lambda}_\theory{T} \ar[r]
                &
                A_{\alpha+1}\ar[ru,dotted,"e_{\alpha+1}"']
                &
        \end{tikzcd}\incat{\PMod\theory{T}}.
    \end{equation*}
    This shows that $e_{\alpha+1}$ is well-defined.

    According to the construction in the proof of \cref{thm:GREpi_admissibility}, the above argument also shows that $A_{\alpha,\alpha+1}\in\Reg(\bfLambda)$ and $\Reg(\bfLambda)\orth[A_{\alpha,\alpha+1}] e_{\alpha+1}$.
    Therefore, the tuple $(A_\bullet,e_\bullet,X)$ becomes a canonical $\Reg(\bfLambda)$-predecomposition of the morphism $e=e_0$.
    To determine where $(A_\bullet,e_\bullet,X)$ stabilizes, we now show the following claims:
    \begin{claim}\label{claim:gauge_mainthm_firstclaim}
        Let $\tup{\vx}.\tup{\ttau}\in K^\lambda_\alpha$.
        Then, the sequent \cref{eq:second_sequent_of_K} for $\tup{\vx}.\tup{\ttau}$ is derivable from $\theory{T}$, and there is an ordinal $\alpha'<\alpha$ such that the sequent \cref{eq:first_sequent_of_K} at $\alpha'$ for $\tup{\vx}.\tup{\ttau}$ is derivable from $\theory{T}$.
    \end{claim}
    \begin{since}
        This follows from straightforward transfinite induction.
    \end{since}
    
    \begin{claim}\label{claim:gauge_mainthm_secondclaim}
        Let $\lambda\in\Lambda$, and let $\tup{\vx}.\tup{\ttau}\coloneq (\tup{\vx}.\ttau_i)_i$ be a family of terms such that $\tup{\vx}.\tup{\ttau}\in K^\lambda_\alpha$ for some $\alpha<\inacc$.
        If there is an ordinal $\beta<\inacc$ such that $\height\ttau_i<\beta$ holds for every $1\le i\le n^\lambda$, then $\tup{\vx}.\tup{\ttau}\in K^\lambda_\beta$ holds.
    \end{claim}
    \begin{since}
        We show this by transfinite induction on $\beta$.
        Fix an ordinal $\beta<\inacc$, and suppose that the statement holds for any ordinal less than $\beta$.
        Let $\lambda\in\Lambda$, and let $\tup{\vx}.\tup{\ttau}\coloneq (\tup{\vx}.\ttau_i)_i$ be a family of terms satisfying $\tup{\vx}.\tup{\ttau}\in K^\lambda_\alpha$ for some $\alpha$ and $\height\ttau_i<\beta$ for all $i$.
        Since the range of the index $i$ is finite, there is an ordinal $\beta'<\beta$ such that for every $i$, $\height\ttau_i\le\beta'$ holds.
        By $\tup{\vx}.\tup{\ttau}\in K^\lambda_\alpha$ and \cref{claim:gauge_mainthm_firstclaim}, regardless of whether $\alpha$ is a successor or not, there is an ordinal $\alpha'<\alpha$ such that the following Horn sequent, obtained by applying \cref{eq:first_sequent_of_K} to $\alpha'$, is derivable from $\theory{T}$.
        \begin{equation*}
            \label{eq:first_sequent_of_K_alpha-dash}
            \fphi
            \wedge
            \syn{\epsilon}_{\alpha'}
                \seq{\tup{\vx}}
                    \bigwedge_{1\le i\le n^\lambda} \ttau_i \defined
        \end{equation*}
        Then, by the definition of $\Lambda$-gauges and the context weakening, the following is derivable from $\theory{T}$:
        \begin{equation*}
            \fphi
            \wedge
            \syn{\epsilon}_{\alpha'}
                \seq{\tup{\vx}}
                    \bigwedge_{1\le i\le n^\lambda}
                    \bigwedge_{t\in T(\ttau_i)}
                    \left(
                        \syn{\zeta}^{\lambda(t)} \subst{\tup{\tsigma}^t}{\tup{\vz}^{\lambda(t)}}
                        \wedge
                        \bigwedge_{1\le k\le n^{\lambda(t)}} \tsigma^t_k \defined
                    \right)
        \end{equation*}
        Since it has already shown that $\fpsi\seq{\tup{\vx}}\fphi\wedge\syn{\epsilon}_{\alpha'}$ is derivable from $\theory{T}$, so are the following sequents for each $1\le i\le n^\lambda$ and $t\in T(\ttau_i)$:
        \begin{gather*}
            \fphi\wedge\syn{\epsilon}_{\alpha'}
                \seq{\tup{\vx}}
                    \bigwedge_{1\le k\le n^{\lambda(t)}}
                    \tsigma^t_k\defined
            \\
            \fpsi
                \seq{\tup{\vx}}
                    \syn{\zeta}^{\lambda(t)} \subst{\tup{\tsigma}^t}{\tup{\vz}^{\lambda(t)}}
        \end{gather*}
        These imply that for every $1\le i\le n^\lambda$ and $t\in T(\ttau_i)$, the tuple $\tup{\vx}.\tup{\tsigma}^t \coloneq (\tup{\vx}.\tsigma^t_k)_{1\le k\le n^{\lambda(t)}}$ belongs to $K^{\lambda(t)}_{\alpha'+1}$.
        Then, by $\height\tsigma^t_k < \height\ttau_i \le \beta'$ and the induction hypothesis, $\tup{\vx}.\tup{\tsigma}^t$ belongs to $K^{\lambda(t)}_{\beta'}$
        This implies that the first sequent below is derivable from $\theory{T}$, and so is the second one by the definition of $\syn{\epsilon}_{\beta'}$.
        \begin{gather*}
            \fphi\wedge\syn{\epsilon}_{\beta'}
                \seq{\tup{\vx}}
                    \bigwedge_{1\le i\le n^\lambda}
                    \bigwedge_{t\in T(\ttau_i)}
                    \bigwedge_{1\le k\le n^{\lambda(t)}}
                    \tsigma^t_k\defined
            \\
            \fphi\wedge\syn{\epsilon}_{\beta'}
                \seq{\tup{\vx}}
                    \bigwedge_{1\le i\le n^\lambda}
                    \bigwedge_{t\in T(\ttau_i)}
                    \syn{\zeta}^{\lambda(t)} \subst{\tup{\sigma}^t}{\tup{\vz}^{\lambda(t)}}
        \end{gather*}
        Thus, by the definition of $\Lambda$-gauges and the context weakening again, the following is derivable from $\theory{T}$:
        \begin{equation}\label{eq:first_sequent_of_K_beta-dash}
            \fphi\wedge\syn{\epsilon}_{\beta'}
                \seq{\tup{\vx}}
                    \bigwedge_{1\le i\le n^\lambda}
                    \ttau_i\defined
        \end{equation}
        On the other hand, the sequent \cref{eq:second_sequent_of_K} for $(\tup{\vx}.\ttau_i)_i$ is also derivable from $\theory{T}$ as stated in \cref{claim:gauge_mainthm_firstclaim}.
        Combining this with \cref{eq:first_sequent_of_K_beta-dash} implies $\tup{\vx}.\tup{\ttau} \in K^\lambda_{\beta'+1}\subseteq K^\lambda_\beta$, which proves the claim.
    \end{since}
    By \cref{claim:gauge_mainthm_secondclaim}, the sequences $(K^\lambda_\alpha)_{\alpha<\inacc}$ stabilize at (or before) the ordinal $\gamma$, which also means that the predecomposition $(A_\bullet,e_\bullet,X)$ stabilizes at $\gamma$.
    Thus, we have $\decnum[\Reg(\bfLambda)]{f}=\cdecnum[\Reg(\bfLambda)]{f}=\abs{e_\bullet}\le\gamma$.
\end{proof}

\begin{example}\label{eg:toyexample_of_gauge}
    Consider the finitary \ac{PHT} $\theory{T}$ defined in \cref{eg:monadic_adj_plustwo}.
    Since its signature $\Sigma'=\{\syn{a},\syn{b},\syn{c},\syn{d}\}$ contains no symbols other than constants, the raw terms over $\Sigma'$ are exactly $\syn{a}$, $\syn{b}$, $\syn{c}$, $\syn{d}$, and variables $\vx$.
    Then, we can define a gauge for $\theory{T}$ as follows:
    \begin{gather*}
        \height\syn{a}=\height\syn{b}=\height\vx\coloneq 0,
        \quad
        \height\syn{c}\coloneq 1,
        \quad
        \height\syn{d}\coloneq 2;
        \\
        T(\syn{a})=T(\syn{b})\coloneq\emptyset,
        \quad
        T(\syn{c})\coloneq\{i\},
        \quad
        T(\syn{d})\coloneq\{j\};
        \\
        \tsigma^i_1 = \tsigma^j_1 \coloneq \syn{a},\quad
        \tsigma^i_2 \coloneq \syn{b},\quad
        \tsigma^j_2 \coloneq \syn{c}.
    \end{gather*}
    Using \cref{thm:gauge}, we obtain $\decnum{\PMod\theory{T}} \le 3+1 = 4$.
    Combining this with the argument in \cref{eg:monadic_adj_plustwo}, we conclude $\decnum{\PMod\theory{T}} = 4$.
\end{example}

\begin{example}
    Consider the scale $\Lambda$ that specifies the class of surjections in $\Pos$ as in \cref{eg:scale_specifying_surjections_in_Pos}.
    Since there are no non-trivial terms over the signature $\Sigma_\pos$ for posets, letting $\height\ttau\coloneq 0$ and $T(\ttau)\coloneq\emptyset$ yields a $\Lambda$-gauge for the \ac{PHT} for posets.
    Then, \cref{thm:gauge} implies $\decnum[\bfE]{\Pos}\le 2$ for the class $\bfE$ of surjections in $\Pos$.
    Since equality is immediate, we actually have $\decnum[\bfE]{\Pos} = 2$, as already mentioned in \cref{eg:global_minimum_decnum}\cref{eg:global_minimum_decnum-surjections}.
\end{example}

\begin{remark}
    For a given locally presentable category, the existence of gauges depends on the choice of a \ac{PHT} whose models form that category.
    Indeed, the finitary \ac{PHT} $\theory{S}$ considered in \cref{eg:monadic_adj_plustwo} does not admit gauges because the definability of the partial constant $\syn{c}$ cannot be characterized by any equation between other terms.
    However, as explained in \cref{eg:monadic_adj_plustwo}, $\theory{S}$ can be rewritten into a two-sorted equational theory without changing its models, and the latter clearly admits a gauge because all terms are totally defined there.
\end{remark}

\subsection{Example: strict $n$-categories}\label{subsec:Example_strict_n-categories}
This subsection is devoted to illustrating the use of gauges to determine the global regular-decomposition number, taking the category of strict $n$-categories as an example.
We will revisit this result from a different perspective in \cref{sec:for_models_of_dependent_algebraic_theories}.
\begin{notation}
    For terms $\tup{\vx}.\ttau_0,\dots,\tup{\vx}.\ttau_n$ of the same sort, the (raw) Horn formula $\bigwedge_{0\le i<n}\ttau_i=\ttau_{i+1}$ is also denoted by $\ttau_0=\ttau_1=\cdots=\ttau_n$.
\end{notation}

We recall the one-sort definition of strict $n$-categories due to Street \cite{Street1987oriented}:
\begin{example}[Strict $n$-categories]
    Let $\theory{T}_\ncat = (\{\syn{*}\},\Sigma_\ncat,\theory{T}_\ncat)$ be the one-sorted finitary \ac{PHT} defined as follows:
    \begin{equation*}
        \Sigma_\ncat \coloneq
        \left\{
        \begin{gathered}
            \syn{d}_k,\syn{c}_k\colon\syn{*}\to\syn{*},
            \quad
            \ccomp[k]\colon \syn{*}\sqcap\syn{*}\to\syn{*}
        \end{gathered}
        \right\}_{1\le k\le n}.
    \end{equation*}
    The reader may find it helpful to see the sort $\syn{*}$ as the sort of $n$-cells, including lower-dimensional cells as identities, and $\syn{d}_k$ and $\syn{c}_k$ as the $(k-1)$-th domain and codomain operators, respectively.
    In this viewpoint, $\vx\ccomp[k]\vy$ represents, if it exists, the composite of $\vx$ and $\vy$ in the direction of $k$-cells along the common $(k-1)$-cell $\syn{d}_k(\vx)=\syn{c}_k(\vy)$.
    The \ac{PHT} $\theory{T}_\ncat$ consists of the following Horn sequents for $1\le k\le n$:
    \begin{gather}
        \label{eq:ncataxiom-rightmostboundary_d}
        \top
            \seq{\vx}
                \syn{d}_k\syn{d}_k(\vx) = \syn{d}_k(\vx) = \syn{c}_k\syn{d}_k(\vx)
        \\
        \label{eq:ncataxiom-rightmostboundary_c}
        \top
            \seq{\vx}
                \syn{c}_k\syn{c}_k(\vx) = \syn{c}_k(\vx) = \syn{d}_k\syn{c}_k(\vx)
        \\
        \label{eq:ncataxiom-definability_of_composition}
        \syn{d}_k(\vx)=\syn{c}_k(\vy)
            \biseq{\vx,\vy}
                \vx\ccomp[k]\vy\defined
        \\
        \label{eq:ncataxiom-boundary_of_composite}
        \syn{d}_k(\vx)=\syn{c}_k(\vy)
            \seq{\vx,\vy}
                \syn{d}_k(\vx\ccomp[k]\vy) = \syn{d}_k(\vy)
                \wedge
                \syn{c}_k(\vx\ccomp[k]\vy) = \syn{c}_k(\vx)
        \\
        \label{eq:ncataxiom-identity_law}
        \top
            \seq{\vx}
                \vx\ccomp[k]\syn{d}_k(\vx) = \vx = \syn{c}_k(\vx)\ccomp[k]\vx
        \\
        \label{eq:ncataxiom-associativity}
        \syn{d}_k(\vx)=\syn{c}_k(\vy)
        \wedge
        \syn{d}_k(\vy)=\syn{c}_k(\vz)
            \seq{\vx,\vy,\vz}
                (\vx\ccomp[k]\vy)\ccomp[k]\vz = \vx\ccomp[k](\vy\ccomp[k]\vz)
    \end{gather}
    and the following Horn sequents for $1\le i<j\le n$:
    \begin{gather}
        \label{eq:ncataxiom-lowerboundary_d}
        \top
            \seq{\vx}
                \syn{d}_j\syn{d}_i(\vx)
                =\syn{d}_i\syn{d}_j(\vx)
                =\syn{d}_i\syn{c}_j(\vx)
                =\syn{c}_j\syn{d}_i(\vx)
                =\syn{d}_i(\vx)
        \\
        \label{eq:ncataxiom-lowerboundary_c}
        \top
            \seq{\vx}
                \syn{c}_j\syn{c}_i(\vx)
                =\syn{c}_i\syn{c}_j(\vx)
                =\syn{c}_i\syn{d}_j(\vx)
                =\syn{d}_j\syn{c}_i(\vx)
                =\syn{c}_i(\vx)
        \\
        \label{eq:ncataxiom-functoriality_of_boundary}
        \syn{d}_i(\vx)=\syn{c}_i(\vy)
            \seq{\vx,\vy}
                \syn{d}_j(\vx\ccomp[i]\vy)
                =
                \syn{d}_j(\vx)\ccomp[i]\syn{d}_j(\vy)
                \wedge
                \syn{c}_j(\vx\ccomp[i]\vy)
                =
                \syn{c}_j(\vx)\ccomp[i]\syn{c}_j(\vy)
    \end{gather}
    \begin{multline}\label{eq:ncataxiom-interchange_law}
        \syn{d}_j(\vx)=\syn{c}_j(\vy)
        \wedge
        \syn{d}_j(\vx')=\syn{c}_j(\vy')
        \wedge
        \syn{d}_i(\vx)=\syn{c}_i(\vx') \\
            \longseq{\vx,\vx',\vy,\vy'}
                (\vx\ccomp[j]\vy)\ccomp[i](\vx'\ccomp[j]\vy')
                    =
                (\vx\ccomp[i]\vx')\ccomp[j](\vy\ccomp[i]\vy').
    \end{multline}
    Then, $\PMod\theory{T}_\ncat\simeq\nCat$ follows, where $\nCat$ denotes the category of strict $n$-categories and strict $n$-functors.
\end{example}

Our immediate goal is to give a gauge for $\theory{T}_\ncat$.
To this end, in \cref{lem:normalization_term_ncat}, we will show that every term can be rewritten in a form that represents the result of iterated composition, starting from lower to higher dimensions, or in other words, from less vertical to more vertical compositions.
We use the symbols $\termle{}$ and $\termeq{}$ defined in \cref{subsec:reduction_of_partialterm}.
\begin{lemma}\label{lem:termrewriting_ncat}\quad
    \begin{enumerate}
        \item\label{lem:termrewriting_ncat-1}
            Let $\bdry_{k_i}$ denote either $\syn{d}_{k_i}$ or $\syn{c}_{k_i}$ for $1\le i\le m$.
            Let $j\in\{1,\dots,m\}$ be the largest one such that $k_j=\min\{k_1,\dots,k_m\}$.
            Then, we have $\bdry_{k_1}\bdry_{k_2}\cdots\bdry_{k_m}(\vx)\termeq{\vx}[\theory{T}_\ncat]\bdry_{k_j}(\vx)$.
        \item\label{lem:termrewriting_ncat-2}
            $(\vx\ccomp[k]\vy)\ccomp[k]\vz \termeq{\vx,\vy,\vz}[\theory{T}_\ncat] \vx\ccomp[k](\vy\ccomp[k]\vz)$.
        \item\label{lem:termrewriting_ncat-3}
            If $i\le j$, $\syn{d}_i(\vx\ccomp[j]\vy)\termle{\vx,\vy}[\theory{T}_\ncat]\syn{d}_i(\vy)$ and $\syn{c}_i(\vx\ccomp[j]\vy)\termle{\vx,\vy}[\theory{T}_\ncat]\syn{c}_i(\vx)$ hold.
        \item\label{lem:termrewriting_ncat-4}
            If $i>j$, $\syn{d}_i(\vx\ccomp[j]\vy)\termeq{\vx,\vy}[\theory{T}_\ncat]\syn{d}_i(\vx)\ccomp[j]\syn{d}_i(\vy)$ and $\syn{c}_i(\vx\ccomp[j]\vy)\termeq{\vx,\vy}[\theory{T}_\ncat]\syn{c}_i(\vx)\ccomp[j]\syn{c}_i(\vy)$ hold.
        \item\label{lem:termrewriting_ncat-5}
            If $i<j$,
            $
            (\vx\ccomp[j]\vy)\ccomp[i]\vz
                \termle{\vx,\vy,\vz}[\theory{T}_\ncat]
                    (\vx\ccomp[i]\vz)\ccomp[j](\vy\ccomp[i]\syn{d}_j(\vz))
            $ and
            $
            \vx\ccomp[i](\vy\ccomp[j]\vz)
                \termle{\vx,\vy,\vz}[\theory{T}_\ncat]
                    (\syn{c}_j(\vx)\ccomp[i]\vy)\ccomp[j](\vx\ccomp[i]\vz)
            $ hold.
    \end{enumerate}
\end{lemma}
\begin{proof}
    The statement \cref{lem:termrewriting_ncat-1} follows from basic axioms about boundary operators \cref{eq:ncataxiom-rightmostboundary_d,eq:ncataxiom-rightmostboundary_c,eq:ncataxiom-lowerboundary_d,eq:ncataxiom-lowerboundary_c}.
    \cref{lem:termrewriting_ncat-2} follows from \cref{eq:ncataxiom-definability_of_composition,eq:ncataxiom-associativity}.
    For \cref{lem:termrewriting_ncat-3,lem:termrewriting_ncat-4,lem:termrewriting_ncat-5}, it suffices to prove only the left-hand statements by symmetry.
    \cref{lem:termrewriting_ncat-3} can be shown by the following calculation:
    \[
        \syn{d}_i(\vx\ccomp[j]\vy)
            \termeq{\vx,\vy}[\theory{T}_\ncat] \syn{d}_i\syn{d}_j(\vx\ccomp[j]\vy)
            \termle{\vx,\vy}[\theory{T}_\ncat] \syn{d}_i\syn{d}_j(\vy)
            \termeq{\vx,\vy}[\theory{T}_\ncat] \syn{d}_i(\vy).
    \]
    Here, the first and the last reductions follow from \cref{lem:termrewriting_ncat-1}, and the second one follows from the axiom \cref{eq:ncataxiom-boundary_of_composite}.
    The statement \cref{lem:termrewriting_ncat-4} follows directly from \cref{eq:ncataxiom-functoriality_of_boundary}.
    Then, \cref{lem:termrewriting_ncat-5} can be shown as follows:
    \[
        (\vx\ccomp[j]\vy)\ccomp[i]\vz
            \termeq{\vx,\vy,\vz}[\theory{T}_\ncat] (\vx\ccomp[j]\vy) \ccomp[i] (\vz\ccomp[j]\syn{d}_j(\vz))
            \termle{\vx,\vy,\vz}[\theory{T}_\ncat] (\vx\ccomp[i]\vz)\ccomp[j](\vy\ccomp[i]\syn{d}_j(\vz)).
    \]
    Here, the first reduction is by \cref{eq:ncataxiom-identity_law}, and the last one follows from the interchange law \cref{eq:ncataxiom-interchange_law}.
\end{proof}

\begin{remark}
    We can apply \cref{lem:substitution_into_reduction} to all relations $\termle{}$ obtained in \cref{lem:termrewriting_ncat}, because no new variable appears when going from the left-hand side to the right-hand side in those relations.
\end{remark}

\begin{definition}\label{def:gauge_for_ncat}
    For a raw term $\ttau$ over $\Sigma_\ncat$, we define $\height\ttau$ as the largest natural number $k$ such that the function symbol $\ccomp[k]$ appears in $\ttau$.
    If $\ttau$ does not contain any $\ccomp[k]$, $\height\ttau$ is defined as $0$.
\end{definition}

\begin{definition}
    A raw term over $\Sigma_\ncat$ is called \emph{normal} if it can be constructed by the following inductive rules:
    \begin{itemize}
        \item
            For every $k\le n$ and every variable $\vx$, the raw terms $\vx,\syn{d}_k(\vx),\syn{c}_k(\vx)$ are normal.
        \item
            Let $1\le k\le n$ and $l\ge 1$ be natural numbers.
            Let $\ttau_0,\ttau_1,\dots,\ttau_l$ be normal raw terms such that $\height\ttau_i < k$ for $0\le i\le l$.
            Then, the raw term
            \begin{equation*}
                ( \cdots ((\ttau_0 \ccomp[k] \ttau_1)\ccomp[k] \ttau_2) \cdots ) \ccomp[k] \ttau_l
            \end{equation*}
            is normal.
    \end{itemize}
    A term $\tup{\vx}.\ttau$ is called \emph{normal} if so is $\ttau$.
\end{definition}

\begin{definition}\label{def:normalizable_terms}
    A term $\tup{\vx}.\ttau$ over $\Sigma_\ncat$ is called \emph{normalizable} if there exists a normal term $\tup{\vx}.\tnu$ such that $\ttau\termle{\tup{\vx}}[\theory{T}_\ncat] \tnu$ and $\height\ttau\ge\height\tnu$ hold.
\end{definition}

\begin{lemma}\label{lem:normalization_term_ncat}
    Every term over $\Sigma_\ncat$ is normalizable.
\end{lemma}
\begin{proof}
    We fix a context $\tup{\vx}$ throughout the proof.
    \begin{claim}\label{claim:normalization-d}
        If $\tup{\vx}.\tnu$ is a normal term, then for every $i\ge 1$, $\tup{\vx}.\syn{d}_i(\tnu)$ is normalizable.
    \end{claim}
    \begin{since}
        We show this by induction on $\height\tnu$.
        If $\height\tnu=0$, the terms $\tup{\vx}.\syn{d}_i(\tnu)$ $(i\ge 1)$ are clearly normalizable by \cref{lem:termrewriting_ncat}\cref{lem:termrewriting_ncat-1}.

        Suppose $k\coloneq\height\tnu\ge 1$.
        Then, by the definition of normal terms, $\tnu$ can be written as
        $
            ( \cdots ((\ttau_0 \ccomp[k] \ttau_1)\ccomp[k] \ttau_2) \cdots ) \ccomp[k] \ttau_l
        $
        with $l\ge 1$ and normal raw terms $\ttau_0,\dots,\ttau_l$ such that $\height\ttau_s<k$ for $0\le s\le l$.
        The proof is now split into two cases.

        \noindent\textbf{Case I: $i\le k$.}
        In this case, using \cref{lem:termrewriting_ncat}\cref{lem:termrewriting_ncat-3}, we have $\syn{d}_i(\tnu) \termle{\tup{\vx}}[\theory{T}_\ncat] \syn{d}_i(\ttau_l)$.
        Since $\height\ttau_l<k$, we can apply the induction hypothesis for the latter term, which shows that $\tup{\vx}.\syn{d}_i(\tnu)$ is normalizable.

        \noindent\textbf{Case II: $i>k$.}
        Using \cref{lem:termrewriting_ncat}\cref{lem:termrewriting_ncat-4} repeatedly, we have
        \begin{equation}\label{eq:swapping_d}
            \syn{d}_i(\tnu)
            \termeq{\tup{\vx}}[\theory{T}_\ncat]
            ( \cdots ((\syn{d}_i(\ttau_0) \ccomp[k] \syn{d}_i(\ttau_1))\ccomp[k] \syn{d}_i(\ttau_2)) \cdots ) \ccomp[k] \syn{d}_i(\ttau_l).
        \end{equation}
        Applying the induction hypothesis to the (normal) terms $\ttau_s$ $(0\le s\le l)$, for each $s$, we obtain a normal term $\tup{\vx}.\tnu_s$ such that $\syn{d}_i(\ttau_s) \termle{\tup{\vx}}[\theory{T}_\ncat] \tnu_s$ and $\height\ttau_s \ge \height\tnu_s$ hold.
        Combining this with \cref{eq:swapping_d}, we have
        \[
            \syn{d}_i(\tnu)
            \termle{\tup{\vx}}[\theory{T}_\ncat]
            ( \cdots ((\tnu_0 \ccomp[k] \tnu_1)\ccomp[k] \tnu_2) \cdots ) \ccomp[k] \tnu_l.
        \]
        Then, by $\height\tnu_s\le\height\ttau_s<k$, the right-hand term above is normal, which proves the claim.
    \end{since}

    \begin{claim}\label{claim:normalization-c}
        If $\tup{\vx}.\tnu$ is a normal term, then for every $i\ge 1$, $\tup{\vx}.\syn{c}_i(\tnu)$ is normalizable.
    \end{claim}
    \begin{since}
        This can be proved in the same way as \cref{claim:normalization-d}.
    \end{since}

    \begin{claim}\label{claim:normalization-comp}
        If $\tup{\vx}.\tnu$ and $\tup{\vx}.\tnu'$ are normal terms, then for every $i\ge 1$, $\tup{\vx}.(\tnu\ccomp[i]\tnu')$ is normalizable.
    \end{claim}
    \begin{since}
        We show this by induction on the sum $\height\tnu + \height\tnu'$.
        The case $\height\tnu=\height\tnu'=0$ is trivial since $\tnu\ccomp[i]\tnu'$ is already normal.

        Let $k\coloneq\height\tnu$ and $k'\coloneq\height\tnu'$, and suppose $k+k'\ge 1$.
        The proof is split into three cases.

        \noindent\textbf{Case I: $\max\{k,k'\} < i$.}
        This case is trivial since $\tnu\ccomp[i]\tnu'$ is already normal.

        \noindent\textbf{Case II: $\max\{k,k'\} = i$.}
        In this case, the associativity of the operator $\ccomp[k]$ (\cref{lem:termrewriting_ncat}\cref{lem:termrewriting_ncat-2}) shows that $\tnu\ccomp[i]\tnu'$ is normalizable.

        \noindent\textbf{Case III: $\max\{k,k'\} > i$.}
        In what follows, we only give the proof for the case $k\ge k'$; the other case is analogous by symmetry and will be omitted.
        Since $\tnu$ is normal, it can be written as
        $
            ( \cdots ((\ttau_0 \ccomp[k] \ttau_1)\ccomp[k] \ttau_2) \cdots ) \ccomp[k] \ttau_l
        $
        with $l\ge 1$ and normal raw terms $\ttau_0,\dots,\ttau_l$ such that $\height\ttau_s<k$ for $0\le s\le l$.
        Since $\height\ttau_0 + \height\tnu' < k + k'$, applying the induction hypothesis to $\ttau_0\ccomp[i]\tnu'$, we obtain a normal term $\tup{\vx}.\tnu_0$ such that $\ttau_0\ccomp[i]\tnu' \termle{\tup{\vx}}[\theory{T}_\ncat] \tnu_0$ and $\height(\ttau_0\ccomp[i]\tnu') \ge \height\tnu_0$ hold.
        By \cref{claim:normalization-d}, the term $\tup{\vx}.\syn{d}_k(\tnu')$ is normalizable; hence there is a normal term $\tup{\vx}.\trho$ such that $\syn{d}_k(\tnu') \termle{\tup{\vx}}[\theory{T}_\ncat] \trho$ and $k'=\height\syn{d}_k(\tnu') \ge \height\trho$ hold.
        Since $\height\ttau_s + \height\trho < k+k'$, applying the induction hypothesis to each $\ttau_s\ccomp[i]\trho$ for $1\le s\le l$, we also obtain normal terms $\tup{\vx}.\tnu_s$ such that $\ttau_s\ccomp[i]\trho \termle{\tup{\vx}}[\theory{T}_\ncat] \tnu_s$ and $\height(\ttau_s\ccomp[i]\trho) \ge \height\tnu_s$ hold.
        Then, we have the following.
        Note that we omit parentheses for the operation $\ccomp[k]$, which is to be understood as the left-associated expression.
        \begin{align*}
            \tnu\ccomp[i]\tnu'
                & \termle{\tup{\vx}}[\theory{T}_\ncat]
                    ((\ttau_0\ccomp[k]\ttau_1\ccomp[k]\cdots\ccomp[k]\ttau_{l-1}) \ccomp[i] \tnu')
                    \ccomp[k]
                    (\ttau_l\ccomp[i]\syn{d}_k(\tnu'))
                \\
                & \termle{\tup{\vx}}[\theory{T}_\ncat] \cdots \termle{\tup{\vx}}[\theory{T}_\ncat]
                    (\ttau_0\ccomp[i]\tnu')
                    \ccomp[k]
                    (\ttau_1\ccomp[i]\syn{d}_k(\tnu'))
                    \ccomp[k]\cdots\ccomp[k]
                    (\ttau_l\ccomp[i]\syn{d}_k(\tnu'))
                \\
                & \termle{\tup{\vx}}[\theory{T}_\ncat]
                    (\ttau_0\ccomp[i]\tnu')
                    \ccomp[k]
                    (\ttau_1\ccomp[i]\trho)
                    \ccomp[k]\cdots\ccomp[k]
                    (\ttau_l\ccomp[i]\trho)
                \\
                & \termle{\tup{\vx}}[\theory{T}_\ncat]
                    \tnu_0\ccomp[k]\tnu_1\ccomp[k]\cdots\ccomp[k]\tnu_l
        \end{align*}
        Here, the first and second rows follow from \cref{lem:termrewriting_ncat}\cref{lem:termrewriting_ncat-5}, and the third and last rows follow from the choice of $\trho$ and $\tnu_0,\dots,\tnu_l$.
        We now have $\height\tnu_0 \le \height(\ttau_0\ccomp[i]\tnu') = \max\{\height\ttau_0, i, \height\tnu'\} \le k$ and $\height\tnu_s \le \height(\ttau_s\ccomp[i]\trho) = \max\{\height\ttau_s, i , \height\trho\} \le k$ for $1\le s\le l$.
        Therefore, the associativity of $\ccomp[k]$ (\cref{lem:termrewriting_ncat}\cref{lem:termrewriting_ncat-2}) shows that the term $\tup{\vx}.(\tnu_0\ccomp[k]\tnu_1\ccomp[k]\cdots\ccomp[k]\tnu_l)$ is normalizable, which proves the claim.
    \end{since}
    Using \cref{claim:normalization-d,claim:normalization-c,claim:normalization-comp}, we can show that for any normalizable terms $\tup{\vx}.\ttau, \tup{\vx}.\ttau'$ and any $1\le i\le n$, the terms $\tup{\vx}.\syn{d}_i(\ttau)$, $\tup{\vx}.\syn{c}_i(\ttau')$, and $\tup{\vx}.(\ttau\ccomp[i]\ttau')$ are normalizable, which implies that every term is normalizable.
\end{proof}

\begin{theorem}\label{thm:verification_gauge_for_ncat}
    The function $\height$ in \cref{def:gauge_for_ncat} gives a gauge for $\theory{T}_\ncat$.
\end{theorem}
\begin{proof}
    By induction on the structure of raw terms $\ttau$, we show existence of the ``defining sets'' $T(\ttau)$ in the definition of gauges.
    For a variable $\vx$, the empty set $T(\syn{\vx})\coloneq\emptyset$ clearly becomes a defining set of $\vx$.
    For a raw term $\ttau$ with $T(\ttau)$, the raw term $\syn{d}_k(\ttau)$ also admits a defining set by $T(\syn{d}_k(\ttau))=T(\ttau)$, since $\height\syn{d}_k(\ttau)=\height\ttau$ and $\syn{d}_k(\ttau)\defined\biseq{\tup{\vx}}\ttau\defined$ is derivable from $\theory{T}_\ncat$, where $\tup{\vx}$ is the minimum context of $\ttau$.
    We can also construct $T(\syn{c}_k(\ttau))$ similarly.

    Let $\ttau_1,\ttau_2$ be raw terms with $T(\ttau_1),T(\ttau_2)$.
    In what follows, we show that the raw term $\ttau_1\ccomp[k]\ttau_2$ also admits a defining set.
    Let $\tup{\vx}\coloneq\fv(\ttau_1)\cup\fv(\ttau_2)$.
    By \cref{lem:normalization_term_ncat}, we can take normal terms $\tup{\vx}.\tnu_1$ and $\tup{\vx}.\tnu_2$ for $\tup{\vx}.\ttau_1$ and $\tup{\vx}.\ttau_2$, respectively.
    Then, the following are derivable from $\theory{T}_\ncat$:
    \begin{align}
        \notag
        (\ttau_1\ccomp[k]\ttau_2)\defined
            &\biseq{\tup{\vx}}
                \syn{d}_k(\ttau_1)=\syn{c}_k(\ttau_2)
        \\
        \notag
        \syn{d}_k(\ttau_1) = \syn{c}_k(\ttau_2)
            &\biseq{\tup{\vx}}
                \ttau_1\defined
                \wedge
                \ttau_2\defined
                \wedge
                \syn{d}_k(\tnu_1)=\syn{c}_k(\tnu_2)
    \end{align}
    Then, by the defining sets $T(\ttau_1)$ and $T(\ttau_2)$, the following is derivable from $\theory{T}_\ncat$:
    \begin{equation}\label{eq:domain_of_tau1tau2}
        (\ttau_1\ccomp[k]\ttau_2)\defined
            \biseq{\tup{\vx}}
                \left(
                \bigwedge_{t\in T(\ttau_1)}
                \tsigma^t_1 = \tsigma^t_2
                \right)
                \wedge
                \left(
                \bigwedge_{t'\in T(\ttau_2)}
                \tsigma^{t'}_1 = \tsigma^{t'}_2
                \right)
                \wedge
                \syn{d}_k(\tnu_1)=\syn{c}_k(\tnu_2)
    \end{equation}
    Here, we have $\tsigma^t_1,\tsigma^t_2 < \height\ttau_1$ and $\tsigma^{t'}_1,\tsigma^{t'}_2 < \height\ttau_2$.
    Now, we split the proof into three cases.
    
    \noindent\textbf{Case I: $\max\{\height\tnu_1,\height\tnu_2\}< k$.}
    In this case, the Horn formula on the right in \cref{eq:domain_of_tau1tau2} directly yields the desired set $T(\ttau_1\ccomp[k]\ttau_2)$ since $\height\syn{d}_k(\tnu_1)=\height\tnu_1 < k$, $\height\syn{c}_k(\tnu_2)=\height\tnu_2 < k$, and $\height(\ttau_1\ccomp[k]\ttau_1)=\max\{\height\ttau_1,\height\ttau_2,k\}$.

    \noindent\textbf{Case II: $k\le\max\{\height\tnu_1,\height\tnu_2\}$ and $\height\tnu_1\neq\height\tnu_2$.}
    By symmetry, we can additionally assume $\height\tnu_1 < \height\tnu_2$.
    By $k\le\height\tnu_2$ and \cref{lem:termrewriting_ncat}\cref{lem:termrewriting_ncat-3}, when we apply $\syn{c}_k$ to $\tnu_2$, we can eliminate all outermost $\ccomp[\height\tnu_2]$.
    Therefore, there is a new term $\tup{\vx}.\trho_2$ such that $\syn{c}_k(\tnu_2)\termle{\tup{\vx}}[\theory{T}_\ncat]\trho_2$ and $\height\tnu_2>\height\trho_2$ hold.
    Then, by \cref{eq:domain_of_tau1tau2}, the bisequent below is derivable from $\theory{T}_\ncat$.
    Indeed, the sequent from left to right is obvious; moreover, the right-hand Horn formula implies that $\ttau_2$ is defined, and so is $\syn{c}_k(\tnu_2)$, which implies the right-hand Horn formula in \cref{eq:domain_of_tau1tau2}, hence the left-hand one is derived.
    \begin{equation*}
        (\ttau_1\ccomp[k]\ttau_2)\defined
            \biseq{\tup{\vx}}
                \left(
                \bigwedge_{t\in T(\ttau_1)}
                \tsigma^t_1 = \tsigma^t_2
                \right)
                \wedge
                \left(
                \bigwedge_{t'\in T(\ttau_2)}
                \tsigma^{t'}_1 = \tsigma^{t'}_2
                \right)
                \wedge
                \syn{d}_k(\tnu_1)=\trho_2
    \end{equation*}
    This yields the desired set $T(\ttau_1\ccomp[k]\ttau_2)$ since $\height\syn{d}_k(\tnu_1)=\height\tnu_1 < \height\tnu_2 \le \height\ttau_2$, $\height\trho_2 < \height\tnu_2 \le \height\ttau_2$, and $\height(\ttau_1\ccomp[k]\ttau_2)=\max\{\height\ttau_1,\height\ttau_2\}$.

    \noindent\textbf{Case III: $k\le\height\tnu_1=\height\tnu_2$.}
    As well as the previous case, there are new terms $\tup{\vx}.\trho_i$ $(i=1,2)$ such that $\syn{d}_k(\tnu_1)\termle{\tup{\vx}}[\theory{T}_\ncat]\trho_1$, $\syn{c}_k(\tnu_2)\termle{\tup{\vx}}[\theory{T}_\ncat]\trho_2$, and $\height\tnu_i>\height\trho_i$ hold.
    Then, by $\cref{eq:domain_of_tau1tau2}$, the following is also derivable from $\theory{T}_\ncat$:
    \begin{equation*}
        (\ttau_1\ccomp[k]\ttau_2)\defined
            \biseq{\tup{\vx}}
                \left(
                \bigwedge_{t\in T(\ttau_1)}
                \tsigma^t_1 = \tsigma^t_2
                \right)
                \wedge
                \left(
                \bigwedge_{t'\in T(\ttau_2)}
                \tsigma^{t'}_1 = \tsigma^{t'}_2
                \right)
                \wedge
                \trho_1=\trho_2
    \end{equation*}
    This yields the desired set $T(\ttau_1\ccomp[k]\ttau_2)$ since $\height\trho_i < \height\tnu_i \le \height\ttau_i$ $(i=1,2)$, and $\height(\ttau_1\ccomp[k]\ttau_2)=\max\{\height\ttau_1,\height\ttau_2\}$.
\end{proof}

\begin{example}[A lower bound for $\decnum{\Cat}$]\label{eg:functor_whose_decnum_is_2}
    We present an example for a functor that has regular-decomposition number $2$ in $\Cat$.
    Consider the following functor $\Phi\colon\C_0\arr \C_2$:
    \begin{equation*}
        \C_0\coloneq\left(
        \begin{tikzcd}[tiny]
                &
                B
                    &
                    B'
                    \ar[rd,"g"]
                        &
            \\
            A
            \ar[ru,"f"]
            \ar[rrr,"h"']
                &
                    &
                        &
                        C
        \end{tikzcd}
        \right)
        \arr(\Phi)[][2]
        \left(
        \begin{tikzcd}[row sep=tiny, column sep=small]
                &
                B
                \ar[rd,"g"]
                    &
            \\
            A
            \ar[ru,"{f}"]
            \ar[rr,"{h}"']
            \ar[rr,shift left=3,"\rotatebox{90}{$=$}",phantom]
                &
                    &
                    C
        \end{tikzcd}
        \right)
        \eqcolon\C_2
        \quad
        \incat{\Cat}.
    \end{equation*}
    Here, the functor $\Phi$ sends $B'$ to $B$ and is identity on the other objects and morphisms.
    As mentioned in the introduction, the coequalizer of the kernel pair of $\Phi$ can be exhibited as below:
    \begin{equation*}
        \C_1\coloneq
        \left(
        \begin{tikzcd}[row sep=tiny, column sep=small]
                &
                B
                \ar[rd,"g"]
                    &
            \\
            A
            \ar[ru,"{f}"]
            \ar[rr,"{h}"']
            \ar[rr,shift left=3,"\rotatebox{90}{$\neq$}",phantom]
                &
                    &
                    C
        \end{tikzcd}
        \right)
    \end{equation*}
    Then, the canonical functor $\C_1\arr \C_2$ given by the universal property of the coequalizer is itself a regular epimorphism in $\Cat$, which implies $\decnum{\Phi}=\cdecnum{\Phi}=2$.
\end{example}

\begin{example}[A lower bound for $\decnum{\nCat}$]\label{eg:nfunctor_whose_decnum_is_nplus1}
    For $n\ge 2$, we present an example for a strict $n$-functor whose regular-decomposition number equals to $n+1$.
    In what follows, we identify strict $n$-categories with models of the partial Horn theory $\theory{T}_\ncat$, in which the interpretations of $\syn{d}_k,\syn{c}_k,\ccomp[k]$ are denoted by $d_k,c_k,\ccomp[k]$, and elements with a subscript $k$ are supposed to be \textit{$k$-cells}; for example, an element denoted by $x_k$ must satisfy the equality $d_i(x_k)=c_i(x_k)= x_k$ for $k<i\le n$.

    We first define a strict $n$-category $\X^{(n)}_m$ for each $0\le m\le n-1$ as follows:
    \begin{multline*}
        \X^{(n)}_m\coloneq
            \{ x_i \mid 2\le i \le n\}
            \cup
            \{ y^0_1, y^+_1 \mid 1\le i \le n-1 \}
            \cup
            \{ z^-_i, z^0_i, z^+_i \mid 0\le i\le n-1\}
            \cup
            \{ y^0_n, z^0_n\}
            \\
            \cup
            \{ t_i \mid 1\le i\le m \}
            \cup
            \{ w^0_m, w^+_m, u_{m+1} \};
    \end{multline*}
    \begin{gather}
        \label{eq:boundary_of_x}
        d_i(x_i)\coloneq y^0_{i-1},
            \quad
            c_i(x_i)\coloneq y^+_{i-1},
        \\
        \label{eq:boundary_of_y_z_t}
        d_i(y^j_i)=d_i(z^j_i)\coloneq z^-_{i-1},
            \quad
            c_i(y^j_i)=d_i(t_i)\coloneq z^0_{i-1},
                \quad
                c_i(z^j_i)=c_i(t_i)\coloneq z^+_{i-1},
        \\
        \label{eq:boundary_of_w}
        d_m(w^j_m)\coloneq z^-_{m-1},
            \quad
            c_m(w^j_m)\coloneq z^+_{m-1},
        \\
        \notag
        d_{m+1}(u_{m+1})\coloneq w^0_m,
            \quad
            c_{m+1}(u_{m+1})\coloneq w^+_m,
        \\
        \notag
        t_i\ccomp[i]y^j_i\coloneq
        \begin{cases}
            z^j_i & \text{if $1\le i<m$} \\
            w^j_i & \text{if $i=m$},
        \end{cases}
        \qquad
        t_i\ccomp[i]x_{i+1}\coloneq
        \begin{cases}
            t_{i+1} & \text{if $1\le i<m$} \\
            u_{m+1} & \text{if $i=m$}.
        \end{cases}
    \end{gather}
    We further define strict $n$-categories $\X^{(n)}_n$ and $\X^{(n)}_{n+1}$ as follows:
    \begin{gather*}
        \X^{(n)}_{n+1}\coloneq
            \{ x_i \mid 2\le i \le n\}
            \cup
            \{ y^0_1, y^+_1 \mid 1\le i \le n-1 \}
            \cup
            \{ z^-_i, z^0_i, z^+_i \mid 0\le i\le n-1\}
            \cup
            \{ y^0_n, z^0_n\}
            \\
            \cup
            \{ t_i \mid 1\le i\le n \},
        \\
        \X^{(n)}_n\coloneq \X^{(n)}_{n+1}\cup\{ w^0_n \}.
    \end{gather*}
    In both $\X^{(n)}_n$ and $\X^{(n)}_{n+1}$, the values of the boundary operators are defined by the equations of the same form as \cref{eq:boundary_of_x,eq:boundary_of_y_z_t,eq:boundary_of_w}.
    In $\X^{(n)}_n$, the values of the composition operators are defined by
    \begin{gather*}
        t_i\ccomp[i]y^j_i\coloneq
        \begin{cases}
            z^j_i & \text{if $1\le i\le n-1$} \\
            w^j_i & \text{if $i=n$},
        \end{cases}
        \quad
        t_{i-1}\ccomp[i-1]x_i\coloneq t_i.
    \end{gather*}
    On the other hand, in $\X^{(n)}_{n+1}$, the values of the composition operators are defined by
    \begin{gather*}
        t_i\ccomp[i]y^j_i\coloneq z^j_i,
        \quad
            t_{i-1}\ccomp[i-1]x_i\coloneq t_i.
    \end{gather*}

    The values of the boundary operators in $\X^{(n)}_m$ $(0\le m\le n+1)$ can be illustrated as below.
    Here, dashed and solid lines indicate the domain and the codomain, respectively.
    \begin{equation*}
        \X^{(n)}_0\colon\quad
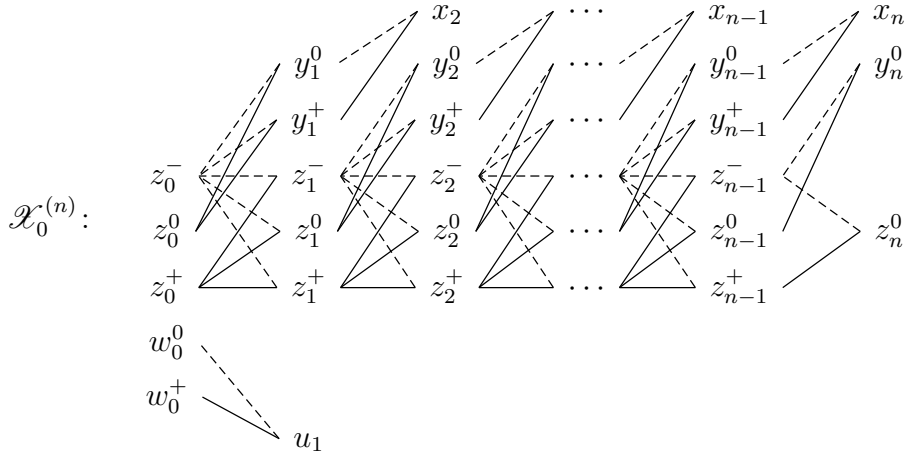
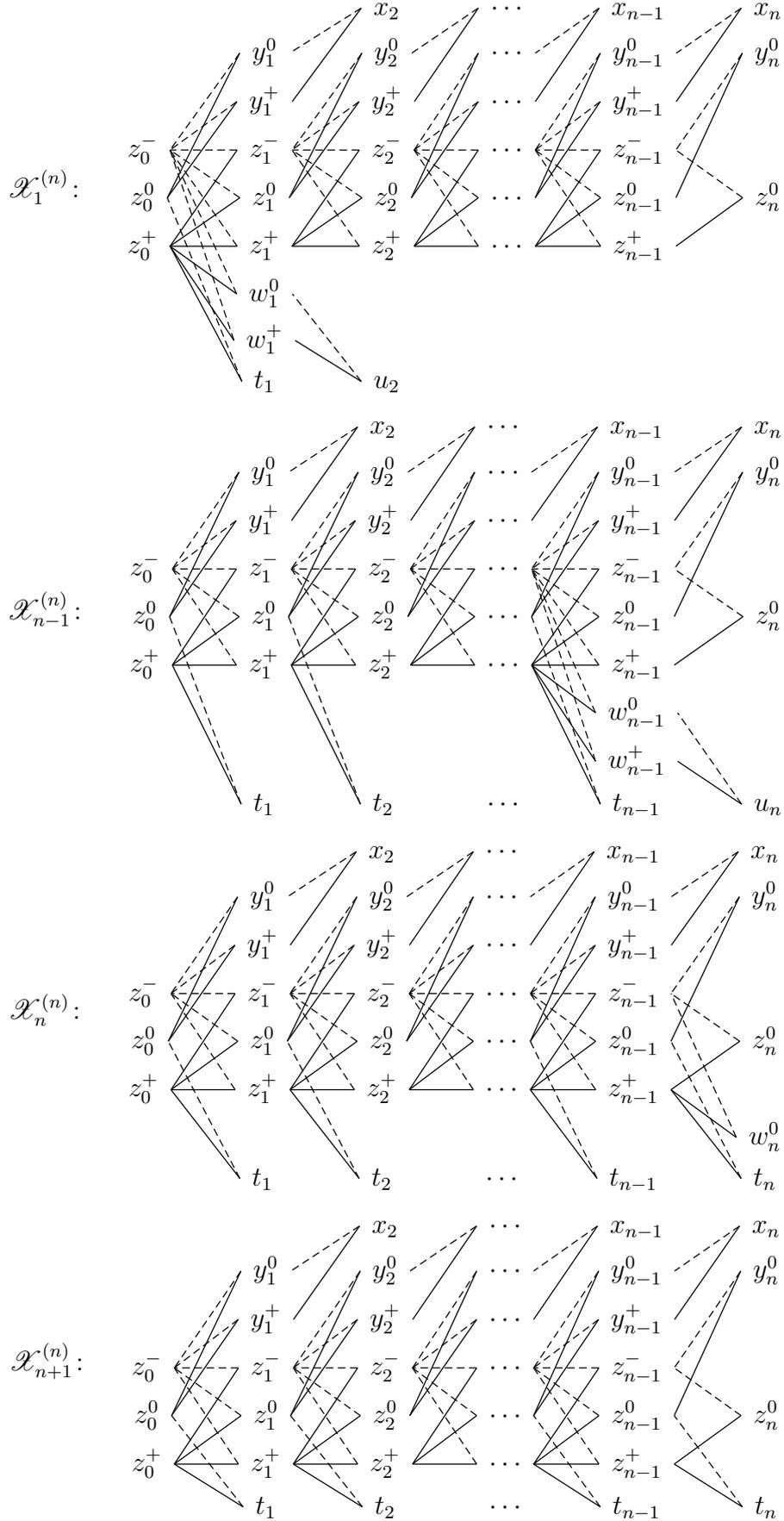
        \begin{tikzcd}[globular]
                &
                    &
                    x_2\ar[dl,dom]\ar[ddl,cod]
                        &
                        \cdots\ar[dl,dom]\ar[ddl,cod]
                            &
                            x_{n-1}\ar[dl,dom]\ar[ddl,cod]
                                &
                                x_n\ar[dl,dom]\ar[ddl,cod]
            \\
                &
                y^0_1\ar[ddl,dom]\ar[dddl,cod]
                    &
                    y^0_2\ar[ddl,dom]\ar[dddl,cod]
                        &
                        \cdots\ar[ddl,dom]\ar[dddl,cod]
                            &
                            y^0_{n-1}\ar[ddl,dom]\ar[dddl,cod]
                                &
                                y^0_n\ar[ddl,dom]\ar[dddl,cod]
            \\
                &
                y^+_1\ar[dl,dom]\ar[ddl,cod]
                    &
                    y^+_2\ar[dl,dom]\ar[ddl,cod]
                        &
                        \cdots\ar[dl,dom]\ar[ddl,cod]
                            &
                            y^+_{n-1}\ar[dl,dom]\ar[ddl,cod]
                                &
            \\
            z^-_0
                &
                z^-_1\ar[l,dom]\ar[ddl,cod]
                    &
                    z^-_2\ar[l,dom]\ar[ddl,cod]
                        &
                        \cdots\ar[l,dom]\ar[ddl,cod]
                            &
                            z^-_{n-1}\ar[l,dom]\ar[ddl,cod]
                                &
            \\
            z^0_0
                &
                z^0_1\ar[ul,dom]\ar[dl,cod]
                    &
                    z^0_2\ar[ul,dom]\ar[dl,cod]
                        &
                        \cdots\ar[ul,dom]\ar[dl,cod]
                            &
                            z^0_{n-1}\ar[ul,dom]\ar[dl,cod]
                                &
                                z^0_n\ar[ul,dom]\ar[dl,cod]
            \\
            z^+_0
                &
                z^+_1\ar[uul,dom]\ar[l,cod]
                    &
                    z^+_2\ar[uul,dom]\ar[l,cod]
                        &
                        \cdots\ar[uul,dom]\ar[l,cod]
                            &
                            z^+_{n-1}\ar[uul,dom]\ar[l,cod]
                                &
            \\
            w^0_0
            \\
            w^+_0
            \\
                &
                u_1\ar[uul,dom]\ar[ul,cod]
        \end{tikzcd}
    \end{equation*}
    \begin{equation*}
        \X^{(n)}_1\colon\quad
        \begin{tikzcd}[globular]
                &
                    &
                    x_2\ar[dl,dom]\ar[ddl,cod]
                        &
                        \cdots\ar[dl,dom]\ar[ddl,cod]
                            &
                            x_{n-1}\ar[dl,dom]\ar[ddl,cod]
                                &
                                x_n\ar[dl,dom]\ar[ddl,cod]
            \\
                &
                y^0_1\ar[ddl,dom]\ar[dddl,cod]
                    &
                    y^0_2\ar[ddl,dom]\ar[dddl,cod]
                        &
                        \cdots\ar[ddl,dom]\ar[dddl,cod]
                            &
                            y^0_{n-1}\ar[ddl,dom]\ar[dddl,cod]
                                &
                                y^0_n\ar[ddl,dom]\ar[dddl,cod]
            \\
                &
                y^+_1\ar[dl,dom]\ar[ddl,cod]
                    &
                    y^+_2\ar[dl,dom]\ar[ddl,cod]
                        &
                        \cdots\ar[dl,dom]\ar[ddl,cod]
                            &
                            y^+_{n-1}\ar[dl,dom]\ar[ddl,cod]
                                &
            \\
            z^-_0
                &
                z^-_1\ar[l,dom]\ar[ddl,cod]
                    &
                    z^-_2\ar[l,dom]\ar[ddl,cod]
                        &
                        \cdots\ar[l,dom]\ar[ddl,cod]
                            &
                            z^-_{n-1}\ar[l,dom]\ar[ddl,cod]
                                &
            \\
            z^0_0
                &
                z^0_1\ar[ul,dom]\ar[dl,cod]
                    &
                    z^0_2\ar[ul,dom]\ar[dl,cod]
                        &
                        \cdots\ar[ul,dom]\ar[dl,cod]
                            &
                            z^0_{n-1}\ar[ul,dom]\ar[dl,cod]
                                &
                                z^0_n\ar[ul,dom]\ar[dl,cod]
            \\
            z^+_0
                &
                z^+_1\ar[uul,dom]\ar[l,cod]
                    &
                    z^+_2\ar[uul,dom]\ar[l,cod]
                        &
                        \cdots\ar[uul,dom]\ar[l,cod]
                            &
                            z^+_{n-1}\ar[uul,dom]\ar[l,cod]
                                &
            \\
                &
                w^0_1\ar[uuul,dom]\ar[ul,cod]
            \\
                &
                w^+_1\ar[uuuul,dom]\ar[uul,cod]
            \\
                &
                t_1\ar[uuuul,dom]\ar[uuul,cod]
                    &
                    u_2\ar[uul,dom]\ar[ul,cod]
        \end{tikzcd}
    \end{equation*}
    \begin{equation*}
        \X^{(n)}_{n-1}\colon\quad
        \begin{tikzcd}[globular]
                &
                    &
                    x_2\ar[dl,dom]\ar[ddl,cod]
                        &
                        \cdots\ar[dl,dom]\ar[ddl,cod]
                            &
                            x_{n-1}\ar[dl,dom]\ar[ddl,cod]
                                &
                                x_n\ar[dl,dom]\ar[ddl,cod]
            \\
                &
                y^0_1\ar[ddl,dom]\ar[dddl,cod]
                    &
                    y^0_2\ar[ddl,dom]\ar[dddl,cod]
                        &
                        \cdots\ar[ddl,dom]\ar[dddl,cod]
                            &
                            y^0_{n-1}\ar[ddl,dom]\ar[dddl,cod]
                                &
                                y^0_n\ar[ddl,dom]\ar[dddl,cod]
            \\
                &
                y^+_1\ar[dl,dom]\ar[ddl,cod]
                    &
                    y^+_2\ar[dl,dom]\ar[ddl,cod]
                        &
                        \cdots\ar[dl,dom]\ar[ddl,cod]
                            &
                            y^+_{n-1}\ar[dl,dom]\ar[ddl,cod]
                                &
            \\
            z^-_0
                &
                z^-_1\ar[l,dom]\ar[ddl,cod]
                    &
                    z^-_2\ar[l,dom]\ar[ddl,cod]
                        &
                        \cdots\ar[l,dom]\ar[ddl,cod]
                            &
                            z^-_{n-1}\ar[l,dom]\ar[ddl,cod]
                                &
            \\
            z^0_0
                &
                z^0_1\ar[ul,dom]\ar[dl,cod]
                    &
                    z^0_2\ar[ul,dom]\ar[dl,cod]
                        &
                        \cdots\ar[ul,dom]\ar[dl,cod]
                            &
                            z^0_{n-1}\ar[ul,dom]\ar[dl,cod]
                                &
                                z^0_n\ar[ul,dom]\ar[dl,cod]
            \\
            z^+_0
                &
                z^+_1\ar[uul,dom]\ar[l,cod]
                    &
                    z^+_2\ar[uul,dom]\ar[l,cod]
                        &
                        \cdots\ar[uul,dom]\ar[l,cod]
                            &
                            z^+_{n-1}\ar[uul,dom]\ar[l,cod]
                                &
            \\
                &
                    &
                        &
                            &
                            w^0_{n-1}\ar[uuul,dom]\ar[ul,cod]
                                &
            \\
                &
                    &
                        &
                            &
                            w^+_{n-1}\ar[uuuul,dom]\ar[uul,cod]
                                &
            \\
                &
                t_1\ar[uuuul,dom]\ar[uuul,cod]
                    &
                    t_2\ar[uuuul,dom]\ar[uuul,cod]
                        &
                        \cdots
                            &
                            t_{n-1}\ar[uuuul,dom]\ar[uuul,cod]
                                &
                                u_n\ar[uul,dom]\ar[ul,cod]
        \end{tikzcd}
    \end{equation*}
    \begin{equation*}
        \X^{(n)}_n\colon\quad
        \begin{tikzcd}[globular]
                &
                    &
                    x_2\ar[dl,dom]\ar[ddl,cod]
                        &
                        \cdots\ar[dl,dom]\ar[ddl,cod]
                            &
                            x_{n-1}\ar[dl,dom]\ar[ddl,cod]
                                &
                                x_n\ar[dl,dom]\ar[ddl,cod]
            \\
                &
                y^0_1\ar[ddl,dom]\ar[dddl,cod]
                    &
                    y^0_2\ar[ddl,dom]\ar[dddl,cod]
                        &
                        \cdots\ar[ddl,dom]\ar[dddl,cod]
                            &
                            y^0_{n-1}\ar[ddl,dom]\ar[dddl,cod]
                                &
                                y^0_n\ar[ddl,dom]\ar[dddl,cod]
            \\
                &
                y^+_1\ar[dl,dom]\ar[ddl,cod]
                    &
                    y^+_2\ar[dl,dom]\ar[ddl,cod]
                        &
                        \cdots\ar[dl,dom]\ar[ddl,cod]
                            &
                            y^+_{n-1}\ar[dl,dom]\ar[ddl,cod]
                                &
            \\
            z^-_0
                &
                z^-_1\ar[l,dom]\ar[ddl,cod]
                    &
                    z^-_2\ar[l,dom]\ar[ddl,cod]
                        &
                        \cdots\ar[l,dom]\ar[ddl,cod]
                            &
                            z^-_{n-1}\ar[l,dom]\ar[ddl,cod]
                                &
            \\
            z^0_0
                &
                z^0_1\ar[ul,dom]\ar[dl,cod]
                    &
                    z^0_2\ar[ul,dom]\ar[dl,cod]
                        &
                        \cdots\ar[ul,dom]\ar[dl,cod]
                            &
                            z^0_{n-1}\ar[ul,dom]\ar[dl,cod]
                                &
                                z^0_n\ar[ul,dom]\ar[dl,cod]
            \\
            z^+_0
                &
                z^+_1\ar[uul,dom]\ar[l,cod]
                    &
                    z^+_2\ar[uul,dom]\ar[l,cod]
                        &
                        \cdots\ar[uul,dom]\ar[l,cod]
                            &
                            z^+_{n-1}\ar[uul,dom]\ar[l,cod]
                                &
            \\
                &
                    &
                        &
                            &
                                &
                                w^0_n\ar[uuul,dom]\ar[ul,cod]
            \\
                &
                t_1\ar[uuul,dom]\ar[uul,cod]
                    &
                    t_2\ar[uuul,dom]\ar[uul,cod]
                        &
                        \cdots
                            &
                            t_{n-1}\ar[uuul,dom]\ar[uul,cod]
                                &
                                t_n\ar[uuul,dom]\ar[uul,cod]
        \end{tikzcd}
    \end{equation*}
    \begin{equation*}
        \X^{(n)}_{n+1}\colon\quad
        \begin{tikzcd}[globular]
                &
                    &
                    x_2\ar[dl,dom]\ar[ddl,cod]
                        &
                        \cdots\ar[dl,dom]\ar[ddl,cod]
                            &
                            x_{n-1}\ar[dl,dom]\ar[ddl,cod]
                                &
                                x_n\ar[dl,dom]\ar[ddl,cod]
            \\
                &
                y^0_1\ar[ddl,dom]\ar[dddl,cod]
                    &
                    y^0_2\ar[ddl,dom]\ar[dddl,cod]
                        &
                        \cdots\ar[ddl,dom]\ar[dddl,cod]
                            &
                            y^0_{n-1}\ar[ddl,dom]\ar[dddl,cod]
                                &
                                y^0_n\ar[ddl,dom]\ar[dddl,cod]
            \\
                &
                y^+_1\ar[dl,dom]\ar[ddl,cod]
                    &
                    y^+_2\ar[dl,dom]\ar[ddl,cod]
                        &
                        \cdots\ar[dl,dom]\ar[ddl,cod]
                            &
                            y^+_{n-1}\ar[dl,dom]\ar[ddl,cod]
                                &
            \\
            z^-_0
                &
                z^-_1\ar[l,dom]\ar[ddl,cod]
                    &
                    z^-_2\ar[l,dom]\ar[ddl,cod]
                        &
                        \cdots\ar[l,dom]\ar[ddl,cod]
                            &
                            z^-_{n-1}\ar[l,dom]\ar[ddl,cod]
                                &
            \\
            z^0_0
                &
                z^0_1\ar[ul,dom]\ar[dl,cod]
                    &
                    z^0_2\ar[ul,dom]\ar[dl,cod]
                        &
                        \cdots\ar[ul,dom]\ar[dl,cod]
                            &
                            z^0_{n-1}\ar[ul,dom]\ar[dl,cod]
                                &
                                z^0_n\ar[ul,dom]\ar[dl,cod]
            \\
            z^+_0
                &
                z^+_1\ar[uul,dom]\ar[l,cod]
                    &
                    z^+_2\ar[uul,dom]\ar[l,cod]
                        &
                        \cdots\ar[uul,dom]\ar[l,cod]
                            &
                            z^+_{n-1}\ar[uul,dom]\ar[l,cod]
                                &
            \\
                &
                t_1\ar[uul,dom]\ar[ul,cod]
                    &
                    t_2\ar[uul,dom]\ar[ul,cod]
                        &
                        \cdots
                            &
                            t_{n-1}\ar[uul,dom]\ar[ul,cod]
                                &
                                t_n\ar[uul,dom]\ar[ul,cod]
        \end{tikzcd}
    \end{equation*}

    Let $F^{(n)}_m\colon \X^{(n)}_m\arr \X^{(n)}_{n+1}$ be the strict $n$-functor that satisfies $F^{(n)}_m(w^j_m)=z^j_m$ and $F^{(n)}_m(u_{m+1})=t_{m+1}$, and is identity on the other elements.
    We now show that for each $0\le m\le n-2$, $F^{(n)}_{m+1}$ coincides with the canonical morphism from the coequalizer of the kernel pair of $F^{(n)}_m$.
    Since $F^{(n)}_m$ only collapses the pairs $(z^0_m,w^0_m)$ and $(z^+_m,w^+_m)$, the coequalizer is obtained from $\X^{(n)}_m$ by collapsing them.
    However, collapsing $(z^0_m,w^0_m)$ generates three non-trivial composites $w^j_{m+1}=u_{m+1}\ccomp[m+1]y^j_{m+1}$ $(j=0,+)$ and $u_{m+2}=u_{m+1}\ccomp[m+1]x_{m+2}$ as new elements.
    Then, renaming $u_{m+1}$ to $t_{m+1}$, we have shown that the coequalizer is exactly $\X^{(n)}_{m+1}$, and $F^{(n)}_{m+1}$ is the canonical morphism from it.

    Similarly, we can verify that $\X^{(n)}_{n}$ is the coequalizer of the kernel pair of $F^{(n)}_{n-1}$, and $F^{(n)}_n$ coincides with the canonical morphism from it.
    Moreover, $F^{(n)}_n$ is itself a regular epimorphism, because collapsing $z^0_n$ and $w^0_n$ in $\X^{(n)}_n$ does not cause any non-trivial composition.
    This shows that the canonical regular decomposition of $F^{(n)}_0$ stabilizes exactly at $n+1$, and $\decnum{F^{(n)}_0}=\cdecnum{F^{(n)}_0}=n+1$.

    For convenience, we illustrate $\X^{(n)}_m$ for the case $n=2$ below.
    \begin{equation*}
        \X^{(2)}_0=\left(
        \begin{tikzcd}[huge]
            w^+_0
                &[-30pt]
                z^+_0
                    &
                    z^-_0
                    \ar[d,bend left=30,"y^+_1"{right}]
                    \ar[d,bend right=30,"y^0_1"{left}]
                    \ar[l,bend right=50,"z^+_1"{above}]
                    \ar[l,"z^-_1"{description}]
                    \ar[l,bend left=50,"z^0_1"{below}]
            \\
            w^0_0\ar[u,"u_1"]
                &
                    &
                    z^0_0
            \rtwocell(x_2){1-3}{2-3}
            \dtwocell(y^0_2){1-2}{1-3}[shift right=4]
            \utwocell(z^0_2){1-2}{1-3}[shift left=4]
        \end{tikzcd}
        \right)
        \qquad
        \X^{(2)}_1=\left(
        \begin{tikzcd}[huge]
            z^+_0
                &
                z^-_0
                \ar[d,bend left=30,"y^+_1"{right}]
                \ar[d,bend right=30,"y^0_1"{left}]
                \ar[l,bend right=50,"z^+_1"{above}]
                \ar[l,"z^-_1"{description}]
                \ar[l,bend left=50,"z^0_1"{below}]
            \\
                &
                z^0_0
                \ar[ul,bend left=50,"t_1"]
            \rtwocell(x_2){1-2}{2-2}
            \dtwocell(y^0_2){1-1}{1-2}[shift right=4]
            \utwocell(z^0_2){1-1}{1-2}[shift left=4]
        \end{tikzcd}
        \right)
    \end{equation*}
    \begin{equation*}
        \X^{(2)}_2=\left(
        \begin{tikzcd}[smallrow]
            z^+_0
                &[-20pt]
                    &
                    {}
                        &
                        z^-_0
                        \ar[lll,bend right=40,"z^-_1"{above},""{name=Z}]
                        \ar[lld,bend right=10,"y^0_1"{above,pos=0.2},""{name=Y0}]
                        \ar[lld,bend left=30,"y^+_1",""{name=Y1,pos=0.8}]
            \\
                &
                z^0_0\ar[ul,bend left=10,"t_1"]
                    &
                    {}
                        &
            \arrow[from=Z,to=Y0,Rightarrow,"y^0_2",shorten=3]
            \arrow[from=Y0,to=Y1,Rightarrow,"x_2"{right},shorten=3]
            \arrow[from=Z,to=Y1,crossing over,Rightarrow,bend right=30,"z^0_2"{left},shorten=3]
        \end{tikzcd}
        \right)
        \qquad
        \X^{(2)}_3=\left(
        \begin{tikzcd}[smallrow]
            z^+_0
                &[-20pt]
                    &
                    {}
                        &
                        z^-_0
                        \ar[lll,bend right=40,"z^-_1"{above},""{name=Z}]
                        \ar[lld,bend right=10,"y^0_1"{above,pos=0.2},""{name=Y0}]
                        \ar[lld,bend left=30,"y^+_1",""{name=Y1,pos=0.8}]
            \\
                &
                z^0_0\ar[ul,bend left=10,"t_1"]
                    &
                    {}
                        &
            \arrow[from=Z,to=Y0,Rightarrow,"y^0_2",shorten=3]
            \arrow[from=Y0,to=Y1,Rightarrow,"x_2"{right},shorten=3]
        \end{tikzcd}
        \right)
    \end{equation*}
    Here, in the 2-category $\X^{(2)}_1$, the composite 1-cell $t_1\ccomp[1]y^j_1$ does not equal to $z^j_1$ for each $j=0,+$; in the 2-category $\X^{(2)}_2$, the composite 2-cell $(t_1\ccomp[1]x_2)\ccomp[2]y^0_2$ does not equal to $z^0_2$.
\end{example}

\begin{theorem}\label{thm:global_decnum_of_nCat}
    $\decnum{\nCat}=n+2$.
\end{theorem}
\begin{proof}
    Since $\height\ttau < n+1$ holds for every raw term $\ttau$ over $\Sigma_\ncat$, \cref{thm:verification_gauge_for_ncat,thm:gauge} imply that for every strict $n$-functor $f$, $\decnum{f}\le n+1$ holds.
    Since we have constructed an example with $\decnum{f}=n+1$ in \cref{eg:functor_whose_decnum_is_2,eg:nfunctor_whose_decnum_is_nplus1}, it follows that $n+1$ is the maximum.
\end{proof}

\begin{remark}\label{rem:global_decnum_of_omegaCat}
    Let $\omegaCat$ be the category of small strict $\omega$-categories and strict $\omega$-functors.
    According to \cite[Lemma 2.1.0.2]{Goldthorpe2024higher}, $\omegaCat$ is the limit of the sequence $\twoCat\rra(\pi_{\le 2})\threeCat\rra(\pi_{\le 3}) \fourCat\rra \cdots$ in the category of large categories, where $\pi_{\le n}$ denotes the functor that sends an $(n+1)$-category to the $n$-category obtained by taking the connected components of $n$-cells.
    Since the strict $(n+1)$-functor $F^{(n+1)}_0$ constructed in \cref{eg:nfunctor_whose_decnum_is_nplus1} is transferred to $F^{(n)}_0$ by $\pi_{\le n}$, the sequence $(F^{(n)}_0)_n$ gives a strict $\omega$-functor $F^{(\omega)}_0$.
    Then, a similar argument to \cref{eg:nfunctor_whose_decnum_is_nplus1} shows $\decnum{F^{(\omega)}_0}=\cdecnum{F^{(\omega)}_0}=\omega$.
    Since $\decnum{\omegaCat}$ is bounded by \cref{cor:upperbound_for_global_regdecnum_of_lp} as $\omegaCat$ is locally finitely presentable, we conclude $\decnum{\omegaCat}=\omega+1$.
\end{remark}
\section{An approach via generalized algebraic theories}\label{sec:for_models_of_dependent_algebraic_theories}
\subsection{Reflective full subclans and regular-decomposition numbers}\label{subsec:reflective_subclans_and_decomposition_numbers}
We start by recalling the notion of clans, following \cite{Joyal2017notes,Frey2025duality}.
\begin{definition}[Clans]\label{def:clan}
    A \emph{clan} is a pair $(\C,\Disp)$ where $\C$ is a category and $\Disp$ is a class of morphisms in $\C$, whose elements are called \emph{display morphisms}, such that the following conditions hold:
    \begin{enumerate} 
        \item\label{def:clan-2} $\Disp$ is closed under composition and contains all identity morphisms.
        \item\label{def:clan-1} $\Disp$ contains all isomorphisms in $\C$.
        \item\label{def:clan-3} $\C$ has pullbacks of any morphism along any morphism in $\Disp$.
        \item\label{def:clan-4} $\Disp$ is closed under pullbacks along a morphism.
        \item\label{def:clan-5} $\C$ has a terminal object $1$.
        \item\label{def:clan-6} $\Disp$ contains the unique morphism from any object to the terminal object $1$.
    \end{enumerate}
    We also call the class $\Disp$ the \emph{clan structure} of the clan.
\end{definition}

\begin{notation}
    Display morphisms in a clan are often denoted by the special arrow $\darr$.
\end{notation}

\begin{example}[{Syntactic clans, cf.\ \cite[Section 14]{Cartmell1986generalised} and \cite[\textsection 10.\ Appendix]{Frey2025duality}}]
    The central example of clans comes from \acfp{GAT}.
    The syntax of \ac{GAT} that we use is presented in \cref{sec:GAT}.
    For a \ac{GAT} $\theory{T}$, we have the \emph{syntactic category} $\clfy{\theory{T}}$ as in \cref{def:syntactic-category} and it has a clan structure which we call the \emph{syntactic clan structure on $\clfy{\theory{T}}$}.
    See \cref{def:syntactic-category,prop:syntactic-category-is-clan} for details.
\end{example}

From now on, we often denote a clan $(\C,\Disp)$ simply by $\C$ where $\Disp$ is clear from context, and refer to its clan structure $\Disp$ as $\Disp[\C]$ when necessary.
\begin{definition}[Clan morphisms and models]
    Given clans $\C$ and $\D$, a \emph{clan morphism} from $\C$ to $\D$ is a functor $\Phi\colon \C \arr \D$ that preserves display morphisms, terminal objects, and pullbacks of any morphisms along display morphisms.
    
    A \emph{model} of a small clan $\C$ is a clan morphism from $\C$ to the (large) clan $\Set$ of small sets and maps with all morphisms being display morphisms.
\end{definition}

\begin{remark}
    The category of models of a small clan $\C$ is locally finitely presentable as it is a model of a finite-limit sketch; see, \cite[Remark 2.9(a)]{Frey2025duality} for details.
\end{remark} 

\begin{definition}[Full subclans]\label{def:full_inclusion_clan}
    A \emph{full inclusion} of clans is a clan morphism $I\colon \C' \arr[hook] \C$ such that the underlying functor is the inclusion functor of a full subcategory $\C'$ of $\C$ and a morphism in $\C'$ is a display morphism if and only if it is a display morphism in $\C$.\footnote{The last condition is not necessary for our general development, but we include it here because it is satisfied in all examples we consider in this paper.}
    A \emph{full subclan} of a clan $\C$ is a clan $\C'$ together with a full inclusion $I\colon \C' \arr[hook] \C$.
\end{definition}

\begin{definition}
    Given clan morphisms $\Phi$ and $\Psi$ from $\C$ to $\D$, a \emph{display transformation} from $\Phi$ to $\Psi$ is a natural transformation whose components are display morphisms.
    By merely saying a natural transformation between clan morphisms, we mean a natural transformation without any restriction on its components.
\end{definition}

\begin{notation}[2-categories of clans]
    We denote the 2-category of small clans, clan morphisms, and natural transformations by $\Clan$, and that of large clans by $\CLAN$.
    For a small clan $\C$, we write $\ClanMod\C$ for the category of models of $\C$ and natural transformations between them, i.e., the hom-category $\CLAN(\C, \Set)$.
    When we restrict 2-cells to display transformations, we denote the resulting 2-categories by $\Cland$ and $\CLANd$, respectively.
\end{notation}

\begin{remark}
    The category of models is equivalently obtained as the hom-category either of $\CLAN$ or $\CLANd$ because $\Set$ has all morphisms as display morphisms.
    They also have the same underlying $(2,1)$-category structure since all natural isomorphisms are componentwise display morphisms.
\end{remark}

\begin{example}\label{exa:models-of-syntactic-clan}
    For a \ac{GAT} $\theory{T}$, the category of models of the syntactic clan $\clfy{\theory{T}}$ is equivalent to the category of models of the \ac{GAT} $\theory{T}$ in the usual sense \cite{Cartmell1986generalised}.
    For instance, the category of models of the syntactic clan $\clfy{\theory{T}_\setg}$ of the \ac{GAT} for sets is equivalent to the category $\Set$ of small sets, and the category of models of the syntactic clan $\clfy{\theory{T}_\catg}$ of the \ac{GAT} for categories is equivalent to the category $\Cat$ of small categories.
    See \cref{eg:set-theory,eg:category-theory} for the definitions of these \acp{GAT}.
\end{example}

\begin{example}[Basic examples of clans and their models]
    \label{prop:finite-product-clan-model}
    \ 
    \begin{enumerate}
    \item There is a unique way to make the terminal category $\terminalcat$ a clan, which we call the \emph{trivial clan}.
    It is also isomorphic to the syntactic clan for the empty \ac{GAT}.
    This clan is the biinitial and 2-terminal object in the 2-category $\Clan$.
    In particular, there is essentially only one model of the trivial clan, i.e., $\ClanMod\terminalcat \simeq \terminalcat$.
    \item\label{prop:finite-product-clan-model-finprod} A clan $\C$ is called a \emph{finite-product clan} \cite[Example 2.3(b)]{Frey2025duality} if its display morphisms are all finite-product projections.
    Given a category with finite products, the finite-product clan structure is the smallest clan structure thereon.
    Models of a small finite-product clan $\C$ are equivalent to finite-product preserving functors from $\C$ to $\Set$, which have been studied classically as models of many-sorted equational (algebraic) theories, see, e.g., \cite[Section 1]{AdamekRosickyVitale2011algebraic}.
    In particular, the category of models of a finite-product clan is known to be regular, and its global regular-decomposition number is at most $2$.
    \item A clan $\C$ is called a \emph{finite-limit clan} \cite[Example 2.3(a)]{Frey2025duality} if all morphisms are display morphisms.
    For example, $\Set$ is a finite-limit clan.
    Given a category with finite limits, the finite-limit clan structure is the largest clan structure thereon.
    Models of a small finite-limit clan $\C$ are equivalent to finite-limit preserving functors from $\C$ to $\Set$.
    \qedhere
    \end{enumerate}
\end{example}

\begin{definition}[Reflective full subclans]\label{def:reflective-full-subclan}
    A \emph{clan reflection} of a clan full inclusion $I\colon \C' \arr[hook] \C$ is a clan morphism $\Lambda\colon \C \arr \C'$ such that the functor $\Lambda$ is a left adjoint of $I$ with the unit being pointwise display morphisms, or equivalently, it is a left adjoint of $I$ in the 2-category $\CLANd$.
    A \emph{reflective full subclan} is a full subclan whose inclusion admits a clan reflection.
\end{definition}

\begin{example}[The trivial clan as a reflective full subclan]\label{ex:trivial-clan-reflective-subclan}
    The canonical clan morphism from the trivial clan $\terminalcat$ to any clan $\C$ is fully faithful, and has a left adjoint given by the unique functor from $\C$ to $\terminalcat$.
    Viewing $\terminalcat$ as a full subclan of $\C$, the left adjoint is a clan reflection.
\end{example}

When a copresheaf $M\colon \C \arr \Set$ on a small clan $\C$ is given, we can consider the discrete opfibration $\pi\colon\El M \arr \C$ corresponding to $M$ through the Grothendieck construction.
Note that a morphism from an object $(C,x)$ to another object $(C',x')$ in the category of elements $\El M$ is given by a morphism $f\colon C \arr C'$ in $\C$ such that $f\cdot x \coloneq M(f)(x) = x'$.
\begin{notation}
    We write $\Dopfib_{\C}$ for the category of discrete opfibrations from a small category to $\C$ and morphisms between them that strictly commute with the discrete opfibrations.
\end{notation}

\begin{remark}
    Starting from a model $M\colon \C \arr \Set$, the domain category $\El M$ of the corresponding discrete opfibration has a clan structure consisting of morphisms whose underlying morphisms are display morphisms in $\C$.
    This is indeed shown to be a clan using \cref{lem:limits-and-discrete-fibrations} below; we can further characterize discrete opfibrations arising from clan models as in the definition following it.
\end{remark}

\begin{lemma}
    \label{lem:limits-and-discrete-fibrations}
    Let $M \colon \C \arr \Set$ be a copresheaf and $\pi\colon \El M \arr \C$ be the corresponding discrete opfibration.
    Take a diagram $D \colon \J \arr \C$ and suppose it has a limit in $\C$.
    Then, the following are equivalent:
    \begin{enumerate}
        \item\label{item:limits-and-discrete-fibrations-1} The functor $M$ preserves the limit of $D$.
        \item\label{item:limits-and-discrete-fibrations-2} The functor $\pi \colon \El M \arr \C$ strictly creates the limit of $D$.
    \end{enumerate}
\end{lemma}
\begin{proof}
    \proofdirection{\cref{item:limits-and-discrete-fibrations-1}}{\cref{item:limits-and-discrete-fibrations-2}}
    This direction is well-known; see, e.g., an exercise in \cite[Section 3.4]{Riehl2016category}.

    \proofdirection{\cref{item:limits-and-discrete-fibrations-2}}{\cref{item:limits-and-discrete-fibrations-1}}
    Let $\lambda_J \colon L \arr D(J)$ be the limit cone in $\C$.
    We prove that the canonical map
    \begin{equation*}
        M(L) \arr \lim_{J \in \J} M(D(J));
        \quad x \mapsto (\lambda_J\cdot x)_{J \in \J}
    \end{equation*}
    is a bijection.
    Giving $y = (y_J)_{J \in \J} \in \lim_{J \in \J} M(D(J))$ is equivalent to giving a diagram $(D, y) \colon \J \arr \El M$ such that $\pi \circ (D,y) = D$.
    By the assumption, the limit of $(D,y)$ is strictly created by $\pi$, so there exists a unique $x \in M(L)$ such that $\lambda_J \cdot x = y_J$ for each $J \in \J$.
\end{proof}

\begin{definition}[Clan discrete opfibrations]
    \label{def:clan-discrete-opfibration}
    A clan morphism $\Phi\colon \C \arr \D$ is said to be a \emph{clan discrete opfibration} if the following conditions hold:
    \begin{enumerate}
        \item It is a discrete opfibration.
        \item It reflects display morphisms, that is, for any morphism $\varphi\colon C \arr C'$ in $\C$, if $\Phi(\varphi)$ is a display morphism in $\D$, then $\varphi$ is a display morphism in $\C$.
        \item It strictly creates terminal objects and pullbacks of any morphisms along display morphisms.
    \end{enumerate}
    Morphisms between clan discrete opfibrations are defined as clan morphisms between the domain clans that strictly commute with the clan discrete opfibrations.
    For a small clan $\D$, we write $\ClanDopfib_{\D}$ for the category consisting of clan discrete opfibrations from a small clan to $\D$ and morphisms between them.
\end{definition}

\begin{remark}\label{rem:clan-discrete-opfibration-closure}
    Clan discrete opfibrations are closed under composition and right cancellative, i.e., for every composable pair of a clan morphism $\Phi$ and a clan discrete opfibration $\Psi$, $\Phi$ is a clan discrete opfibration if and only if so is the composite $\Psi\circ\Phi$.
    Consequently, by \Cref{lem:limits-and-discrete-fibrations}, the category $\ClanDopfib_\D$ becomes a full subcategory of $\Dopfib_\D$.
\end{remark}

\begin{proposition}[The Grothendieck construction for clan models]
    \label{prop:clan-Grothendieck-construction}
    For a small clan $\C$, restricting the equivalence between the copresheaf category $\Set^{\C}$ and the category of discrete opfibrations over $\C$ through the Grothendieck construction, we obtain an equivalence between the category of clan discrete opfibrations over $\C$ and the category of models of the clan $\C$.
    \begin{equation*}
    \begin{tikzcd}
        \ClanMod\C \ar[r,"\int","\simeq"'] 
        \ar[d, hookrightarrow]
        &
        \ClanDopfib_{\C}
        \ar[d, hookrightarrow]
        \\
        \Set^{\C} \ar[r,"\int","\simeq"']
        &
        \Dopfib_{\C}
    \end{tikzcd}
    \end{equation*}
\end{proposition}
\begin{proof}
    The clan structure on the category of elements $\El M$ of a model $M\colon \C \arr \Set$ is uniquely determined to make the projection functor $\pi_M\colon\El M \arr \C$ a clan discrete opfibration: Display morphisms are those underlying display morphisms in $\C$.
    The standard argument on the Grothendieck construction shows that $M$ preserves pullbacks along display morphisms and terminal objects if and only if the projection functor strictly creates them, see \cref{lem:limits-and-discrete-fibrations}.
    Since both the categories are full subcategories, this proves the restricted equivalence.
\end{proof}

\begin{proposition}[Naturality of the Grothendieck construction for clan models]
    \label{prop:naturality-of-clan-Grothendieck-construction}
    For a clan morphism $\Phi\colon \C \arr \D$ between small clans, the following square commutes up to natural isomorphism:
    \begin{equation*}
    \begin{tikzcd}
        \ClanMod\D \ar[r, "\int_\D","\simeq"'] \ar[d, "\Phi^*"']
        &
        \ClanDopfib_{\D} \ar[d, "\Phi^\diamond"]
        \\
        \ClanMod\C \ar[r, "\int_\C","\simeq"']
        &
        \ClanDopfib_{\C}
        \cellsymb(\rotatebox{45}{$\cong$}){1-1}{2-2}
    \end{tikzcd}
    \end{equation*}
    where the left vertical functor $\Phi^*$ is given by precomposition with $\Phi$, and the right vertical functor $\Phi^\diamond$ is given by the pullback (in $\Cat$) of clan discrete opfibrations along $\Phi$.
\end{proposition}
\begin{proof}
    On the level of copresheaves and discrete opfibrations, this is a well-known fact and easily verified.
    For the desired natural isomorphism, it suffices to check that the precomposition of $\Phi$ preserves models of clans, and this follows from the fact that clan morphisms compose.
\end{proof}

\begin{notation}\label{note:Grothendieck-pullback-square}
    By the previous proposition, for a clan morphism $\Phi\colon \C \arr \D$ between small clans and a clan model $N\colon \D \arr \Set$, we have the following pullback square in $\Cat$:
    \begin{equation*}
        \begin{tikzcd}
            \El (\Phi^* N) 
            \ar[r, "\int \Phi"] 
            \ar[d, "\pi_{\Phi^* N}"']
            \ar[rd, "\lrcorner", phantom, very near start]
            &
            \El N \ar[d, "\pi_N"]
            \\
            \C \ar[r, "\Phi"']
            &
            \D
        \end{tikzcd}\incat{\Cat}.
    \end{equation*}
    Note that $\int \Phi$ is a clan morphism.
    This follows from the fact that $\pi_{\Phi^* N}$ and $\pi_N$ are clan discrete opfibrations and $\Phi$ is a clan morphism.
\end{notation}

\begin{lemma}[{\cite[Proposition 3.6]{Frey2025duality}}]\label{lem:slice-clan-model}
    For a clan model $M\colon \C \arr \Set$, we have the following equivalence: 
    \begin{equation*}
        (\ClanMod\C)/M \simeq \ClanMod(\El M).
    \end{equation*}
\end{lemma}
\begin{proof}
    Through the equivalence in \cref{prop:clan-Grothendieck-construction}, the desired equivalence can be reformulated as the following equivalence:
    \begin{equation*}
        \ClanDopfib_{\C}/\pi_M \simeq \ClanDopfib_{\scriptstyle\int M},
    \end{equation*}
    which follows from \cref{rem:clan-discrete-opfibration-closure}.
\end{proof}

\begin{corollary}\label{cor:naturality-of-slice-clan-models}
    For a clan morphism $\Phi\colon \C \arr \D$ between small clans and a clan model $N\colon \D \arr \Set$, we have the following up-to-isomorphism commutative diagram:
    \begin{equation*}
        \begin{tikzcd}
            (\ClanMod\D)/N \ar[r, "\simeq"] 
            \ar[d, "\Phi^*/N"{left}]
            \ar[rd, "\rotatebox{45}{$\cong$}", phantom]
            &
            \ClanMod(\El N) 
            \ar[d, "(\int\Phi)^*"{right}]
            \\
            (\ClanMod\C)/\Phi^* N 
            \ar[r, "\simeq"]
            &
            \ClanMod(\El \Phi^* N),
        \end{tikzcd}
    \end{equation*}
    where $\int \Phi \colon \El\Phi^* N \arr \El N$ is the clan morphism obtained as in \cref{note:Grothendieck-pullback-square}, and the functor $\Phi^*/N$ is induced by applying $\Phi^*$ to morphisms into $N$.
\end{corollary}
\begin{proof}
    Again by \cref{prop:clan-Grothendieck-construction}, this follows from the fact that pullback squares compose.
\end{proof}

We now investigate the way to evaluate the category of models of a clan from that of its reflective full subclan.
This will be useful when we have a sequence of reflective full subclans and want to analyze the category of models of the largest clan step by step.
\begin{lemma}\label{lem:reflective-subclan-adjoint}
    Let $\C'$ be a reflective full subclan of a small clan $\C$ with the inclusion $I\colon \C' \arr[hook] \C$ and a clan reflection $\Lambda\colon \C \arr \C'$.
    Then, we have the following adjunctions between the categories of models:
    \begin{equation*}
        \begin{tikzcd}
            \ClanMod\C' & \ClanMod\C
            \arrow[from=1-1,to=1-2,hook,shift left=6]
            \arrow[from=1-1,to=1-2,phantom,"{\scriptstyle \perp}",shift left=3]
            \arrow[from=1-2,to=1-1,"I^*"{description}]
            \arrow[from=1-1,to=1-2,phantom,"{\scriptstyle \perp}",shift right=3]
            \arrow[from=1-1,to=1-2,hook,"\Lambda^*"',shift right=6]
        \end{tikzcd}
    \end{equation*}
    Furthermore, the right adjoint $\Lambda^*$ of $I^*$ is fully faithful, and hence so is the left adjoint of $I^*$.
\end{lemma}
\begin{proof}
    The adjunction $\Lambda \dashv I$ leads to the adjunction $I^* \dashv \Lambda^*$ between the categories of models, with the counit being an isomorphism, and hence $\Lambda^*$ is fully faithful.
    For the existence of left adjoint of $I^*$, it suffices to check that it preserves all small limits since the categories of models are locally presentable and $I^*$ as a left adjoint preserves all small colimits; see \cite[1.66{}]{AdamekRosicky1994locally}.
    This follows from the fact that limits in $\ClanMod\C$ and $\ClanMod\C'$ are computed pointwise.
\end{proof}

\begin{notation}
    For a pair of functors $F\colon \C \arr \D$ and $G\colon \E \arr \D$, we denote by $F \isocom G$ the full subcategory of the comma category $F \downarrow G$ consisting of isomorphisms from $F(C)$ to $G(E)$ for objects $C$ in $\C$ and $E$ in $\E$.
    We call $F\isocom G$ the \emph{iso-comma category} of $F$ and $G$.
    \begin{equation*}
        \begin{tikzcd}
            F\isocom G\ar[d,"\Pi_1"']\ar[r,"\Pi_2"]
                &
                \E\ar[d,"G"]
            \\
            \C\ar[r,"F"']
                &
                \D
            \rutwocell{2-1}{1-2}["{\rotatebox{45}{$\cong$}}"{below right=-3}]
        \end{tikzcd}
    \end{equation*}
    In particular, when $G$ is the identity functor on $\D$, we write $F \isocom \D \coloneq F \isocom \id_{\D}$, and when $G$ is the functor $\const{D}\colon \terminalcat \arr \D$ selecting an object $D$ in $\D$, we write $F \isocom D$ for the iso-comma category.

    We refer the first and second projection functors from $F\isocom G$ to $\C$ and $\E$ by $\Pi_1$ and $\Pi_2$, respectively.
    In the case of $F\isocom D$ for an object $D$ in $\D$, we simply call the functor $\Pi_1\colon F\isocom D \arr \C$ the \emph{projection}.
\end{notation}

Now, we construct a clan from a model of a reflective full subclan.
In the construction, we use the bipushout in the 2-category $\Clan$, whose proof of existence is postponed to \cref{sec:bicocompletenessofClan}.
\begin{construction}[The clan $\Fix{\C}{M}$]\label{const:fixclan}
    Let $\C'$ be a reflective full subclan of a small clan $\C$ and $M\colon \C' \arr \Set$ be a model.
    Define a small clan $\Fix{\C}{M}$ as follows.
    First, by \cref{prop:naturality-of-clan-Grothendieck-construction}, we have the following pullback of categories:
    \begin{equation*}
        \begin{tikzcd}
            \int M 
            \ar[r,"T"]
            \ar[d,"\pi_M"']
            \pullbackcorner
            & 
            \El \Lambda^*M
            \ar[d,"\pi_{\Lambda^*M}"']
            \ar[r,"\int\Lambda"]
            \pullbackcorner
            &
            \int M  
            \ar[d,"\pi_M"]
            \\
            \C'
            \ar[r,"I"']
            &
            \C
            \ar[r,"\Lambda"']
            &
            \C',
        \end{tikzcd}
    \end{equation*}
    where $\Lambda$ is the left adjoint of the inclusion $I\colon \C' \arr[hook] \C$.
    Taking the bipushout of $T$ along the unique clan morphism $!\colon \El M \arr \terminalcat$ in $\Clan$, we obtain a small clan $\Fix{\C}{M}$.
    \begin{equation*}
        \begin{tikzcd}
            \El M\ar[d,"!"']\ar[r,"T"]\ar[rd, "\rotatebox{45}{$\cong$}", phantom]
                &
                \El \Lambda^*M\ar[d]
            \\
            \terminalcat\ar[r]
                &
                \Fix{\C}{M}\pushoutcorner
        \end{tikzcd}\incat{\Clan}
    \end{equation*}
\end{construction}

\begin{proposition}\label{prop:construct-clan}
    In the situation of \cref{const:fixclan}, we have the following bipullback square in $\CAT$:
    \begin{equation*}
        \begin{tikzcd}
            \ClanMod(\Fix{\C}{M})
            \ar[rr]
            \ar[d]
            \pullbackcorner
            \ar[rrd, "\rotatebox{45}{$\cong$}", phantom]
            &[-30pt]
            &
            \ClanMod(\El \Lambda^*M)
            \ar[d, "T^*"]
            \\
            \terminalcat
            \ar[r, "\simeq"]
            &
            \ClanMod\terminalcat
            \ar[r, "!^*"]
            &
            \ClanMod(\El M)
        \end{tikzcd}
    \end{equation*}
    Moreover, the category $\ClanMod(\Fix{\C}{M})$ is equivalent to the iso-comma category $I^*\isocom M$, where $I^*\colon \ClanMod\C \arr \ClanMod\C'$ is the functor induced by the inclusion $I$.
\end{proposition}
\begin{proof}
    The first part is not immediate from the universal property of the bipushout in $\Clan$ because the universal property is only for small clans, while $\Set$ is a large clan.
    However, since the clans $\El M$, $\El \Lambda^*M$, and $\terminalcat$ are small clans, clan morphisms from them to $\Set$ uniquely factor through a small full subclan $\Set_{<\kappa}$ of sets of cardinality less than a sufficiently large regular cardinal $\kappa$ with its display morphisms being all morphisms.
    Therefore, applying the universal property to this small full subclan, we obtain the desired bipullback square.

    Extending the above bipullback square and using \cref{cor:naturality-of-slice-clan-models}, we have the following:
    \begin{equation*}
        \begin{tikzcd}[row sep=small]
            \ClanMod(\Fix{\C}{M})
            \ar[rr]
            \ar[dd]
            \ar[rrdd, "\lrcorner", phantom, very near start]
            \ar[rrdd, "\rotatebox{45}{$\cong$}", phantom]
            &[-30pt]
            &
            \ClanMod(\El \Lambda^*M)
            \ar[d, "(\int I)^*"]
            \ar[r, "\simeq"]
            \ar[rd, "\rotatebox{45}{$\cong$}", phantom]
            &
            (\ClanMod\C)/\Lambda^*M
            \ar[d, "{I^*/\Lambda^*M}"]
            \ar[r]
            &
            \ClanMod\C
            \ar[dd, "I^*"]
            \\
            &
            &
            \ClanMod(\El I^*\Lambda^*M)
            \ar[d, "\cong"]
            \ar[r, "\simeq"]
            \ar[rd, "\rotatebox{45}{$\cong$}", phantom]
            &
            (\ClanMod\C')/I^*\Lambda^*M
            \ar[d, "\cong"]
            &
            \\
            \terminalcat
            \ar[r, "\simeq"]
            &
            \ClanMod\terminalcat
            \ar[r, "!^*"]
            &
            \ClanMod(\El M)
            \ar[r, "\simeq"]
            &
            (\ClanMod\C')/M
            \ar[r]
            &
            \ClanMod\C'
        \end{tikzcd}
    \end{equation*}
    Here, the bottom middle square is obtained by the counit component $I^*\Lambda^*M \arr M$ which is an isomorphism.
    The rightmost square is also a bipullback square, as it is the pullback square in $\CAT$ and the projection functor $(\ClanMod\C')/M \arr \ClanMod\C'$ is an isofibration.
    Tracing the composite of the functors at the bottom row, which is the functor selecting the object $M$ in $\ClanMod\C'$, we see that the category $\ClanMod(\Fix{\C}{M})$ is equivalent to the iso-comma category $I^*\isocom M$ as it gives the bipullback in $\CATtwo$ \cite[Section 6]{Kelly1989elementary}.
\end{proof}

\begin{lemma}\label{lem:iso-fiber}
    Let $F\colon \C \arr \D$ be a functor and $D$ be an object in $\D$, and take a connected category $\J$.
    When $\C$ has limits (resp.\ colimits) of shape $\J$ and $F$ preserves them, the projection $\Pi_1\colon F\isocom D \arr \C$ strictly creates limits (resp.\ colimits) of shape $\J$.
    In particular, it preserves and reflects monomorphisms, kernel pairs, regular epimorphisms, and colimits of chains, whenever $\C$ has them and $F$ preserves them.
\end{lemma}
\begin{proof}
    The proof is parallel for limits and colimits, so we only show the case of limits.
    Take a diagram $(K, k) \colon \J \arr F\isocom D$ and take the limit cone $(\lambda_J\colon L\arr KJ)_{J\in\J}$ over $K\colon \J \arr \C$.
    If $(\zeta_J\colon (Z,z)\arr[] (KJ, k_J))_J$ is a lift of the cone $(\lambda_J)_J$ along $\Pi_1$, then $Z=L$ and $\zeta_J=\lambda_J$ should hold, and the following diagram must commute for every $J\in\J$:
    \begin{equation}\label{eq:unique_lift_of_limit_cone}
        \begin{tikzcd}
            FL\ar[d,"z"',"\cong"]\ar[r,"F\lambda_J"]
                &
                FKJ\ar[d,"k_J","\cong"'] \\
            D\ar[r,equal]
                &
                D
        \end{tikzcd}\incat{\D}.
    \end{equation}
    This shows that such a lift of $(\lambda_J)_J$ is unique.
        
    We now show the existence of a lift of $(\lambda_J)_J$ along $\Pi_1$.
    Since $F$ preserves limits of shape $\J$, the morphism on the top row in \cref{eq:unique_lift_of_limit_cone} is a leg of a limit cone.
    In addition, since $\J$ is connected, the cone consisting of the identity morphisms exhibits $D$ as a limit of the constant functor $\J\arr \D$ taking the unique value $D$, whose leg we put on the bottom row in \cref{eq:unique_lift_of_limit_cone}.
    Then, the universal property of limits implies that there is a unique isomorphism $z$ such that for every $J\in\J$, the square \cref{eq:unique_lift_of_limit_cone} commutes.
    This shows the existence of the lift.

    Moreover, the lift $(\lambda_J\colon (L,x)\arr[] (KJ, k_J))_J$ is a limit cone of the diagram $(K, k)\colon \J \arr F\isocom D$. 
    Indeed, the uniqueness is clear from the universal property of the limit cone $(\lambda_J)_J$ in $\C$.
    For the existence, take a cone $(\xi_J\colon (X,x)\arr {(KJ,k_J)})_J$, and by the universal property of $L$, we have a unique morphism $f\colon X \arr L$ such that $\lambda_J\circ f = \xi_J$ for each $J \in \J$.
    Then, $f$ defines a morphism in $F\isocom D$ because $z\circ Ff= k_J\circ F\lambda_J \circ Ff = k_J\circ F\xi_J = x$ by taking some $J \in \J$.
    Therefore, the projection $\Pi_1\colon F\isocom D \arr \C$ strictly creates limits of shape $\J$.
\end{proof}

\begin{theorem}[Comparison formula I]\label{thm:bound-of-decomposition-length}
    Let $I\colon\C'\arr[hook]\C$ be a reflective full subclan of a small clan.
    Suppose $\decnum{\ClanMod\C'}\le \alpha+1$ and $\decnum{\ClanMod(\Fix{\C}{I^* N})}\le \beta$ for every model $N\colon\C\arr\Set$ for some small ordinals $\alpha$ and $\beta$, where $\Fix{\C}{I^* N}$ is the clan constructed in \cref{const:fixclan}.
    Then, we have $\decnum{\ClanMod\C} \le \alpha + \beta$.
\end{theorem}
\begin{proof}
    By \cref{lem:reflective-subclan-adjoint}, there is a sequence of adjoints $L \dashv I^* \dashv \Lambda^*$, where $\Lambda$ denotes the clan reflection of $I$.
    Now, the classes of strong epimorphisms and monomorphisms form an orthogonal factorization system on the category $\ClanMod\C$ as it is locally finitely presentable.
    Thus, by \cref{prop:decnum_of_composite}\cref{prop:decnum_of_composite-postcompose}, it suffices to evaluate the regular-decomposition number of an arbitrary strong epimorphism $f\colon N \arr N'$ in $\ClanMod\C$.

    Applying \cref{cor:adj_and_decnum_regularepi}\cref{cor:adj_and_decnum_regularepi-general} to the adjunction $L \dashv I^*$, we have $\decnum{f} \le \decnum{I^*f} + \decnum{\pocomp{f}}$, where $\pocomp{f}$ is the morphism introduced in \cref{note:pushout_component_c_f}.
    Here, note that $I^*f$ is a strong epimorphism because $I^*$ is preserves strong epimorphisms by being a left adjoint, and $\decnum{f}$ and $\decnum{I^*f}$ are defined since $\ClanMod\C$ and $\ClanMod\C'$ are locally finitely presentable.
    \begin{claim}
        $\decnum{\pocomp{f}} < \beta$.
    \end{claim}
    \begin{since}
    Note that we have the equivalence $\ClanMod(\Fix{\C}{I^*N'}) \simeq I^*\isocom I^*N'$ by \cref{prop:construct-clan}.
    Let $\epsilon$ be the counit of the adjunction $L \dashv I^*$.
    Since $I^*$ preserves pushouts and that $I^*\epsilon_N$ and $I^*\epsilon_{N'}$ are isomorphisms because $I^*$ is a reflection of $\Lambda^*$, the morphism $I^*\pocomp{f}$ is also an isomorphism.
    \begin{equation*}
        \begin{tikzcd}
            I^*LI^*N\ar[r,"I^*LI^*f"]\ar[d,"I^*\epsilon_N"',"\cong"]
                &
                I^*LI^*N'\ar[rd,"I^*\epsilon_{N'}"{above right=-3},"\cong"{left=2}]\ar[d,"\cong"']
                    &
            \\
            I^*N\ar[r]\ar[rr,"I^*f"',bend right=20]
                &
                I^*C_f\ar[r,"I^*{\pocomp{f}}"{above=-1},pos=0.3]\pushoutcorner
                    &
                    I^*N'
        \end{tikzcd}\incat{\ClanMod\C'}
    \end{equation*}
    Here, the notations $C_f$ and $\pocomp{f}$ are as in \cref{note:pushout_component_c_f}.
    Note that the two categories are locally finitely presentable because they are the categories of models of clans, and the projection $I^*\isocom I^*N' \arr \ClanMod\C$ (strictly) creates monomorphisms, kernel pairs, regular epimorphisms, and colimits of chains by \cref{lem:iso-fiber} and the fact that $I^*$ preserves them as a left and right adjoint.
    Thus, the projection preserves canonical regular decomposition numbers.
    Meanwhile, the regular-decomposition number of the morphism $\pocomp{f}\colon (C_f, I^*\pocomp{f}) \arr[] (N', \id_{I^*N'})$ in the iso-comma category $I^*\isocom I^*N'$ is less than $\beta$ by assumption and the equivalence $\ClanMod(\Fix{\C}{I^*N'}) \simeq I^*\isocom I^*N'$.
    This shows that $\decnum{\pocomp{f}}=\cdecnum{\pocomp{f}} < \beta$.
    \end{since}
    On the other hand, $\decnum{I^*f} < \alpha + 1$, or equivalently, $\decnum{I^*f} \le \alpha$ by assumption.
    Therefore, we have $\decnum{f} \le \decnum{I^*f} + \decnum{\pocomp{f}} < \alpha + \beta$, as desired.
\end{proof}

\subsection{The generation of clan structures}\label{subsec:generation_of_clan_structures}
In order to measure the regular-decomposition number of the category of models of a clan, in particular that of $\Fix{\C}{M}$ in the previous discussion, we here present some preliminary results on a notion of generation of clan structures.
\begin{lemma}\label{lem:smallest-clan-structure}
    Let $\C$ be a category with finite products and $\bfDelta$ be a class of morphisms that admit pullbacks along arbitrary morphisms.
    Then, for a class of morphisms $\bfDelta'$ including $\bfDelta$, the following conditions are equivalent:
    \begin{enumerate}
        \item\label{lem:smallest-clan-structure-1} $\bfDelta'$ is the smallest clan structure on $\C$ containing $\bfDelta$.
        \item\label{lem:smallest-clan-structure-2} $\bfDelta'$ is a clan structure on $\C$, and for any functor $\Phi\colon \C \arr \D$ into a clan $\D$, if it sends every morphism in $\bfDelta$ to a display morphism in $\D$ and preserves finite products and pullbacks of morphisms in $\bfDelta$ along arbitrary morphisms, then $\Phi$ is a clan morphism when we regard $\C$ as a clan with the clan structure $\bfDelta'$.
        \item\label{lem:smallest-clan-structure-3} $\bfDelta'$ is the class of morphisms that can be presented as the composite of a sequence (of positive finite length) of morphisms obtained by pulling back either an isomorphism, a morphism in $\bfDelta$, or a morphism into the terminal object.
    \end{enumerate}
    \begin{equation}
        \label{eq:relatively-non-dependent}
        \begin{tikzcd}
            X \ar[r,"\phi",disp]
            &
            X'
        \end{tikzcd}
         =
        \begin{tikzcd}
            X_0 \ar[r,"\phi_1",disp]
            &
            \cdots \ar[r,"\phi_n",disp]
            &
            X_n 
        \end{tikzcd}
        (n\ge 1)
        \quad
        \text{s.t.}
        \quad
        \begin{tikzcd}
            X_i \ar[d,"\phi_i"',disp]
            \ar[r]
            \pullbackcorner
            &
            Y_{i}
            \ar[d, disp, "{\in\  \bfDelta_{\mathrm{isoproj}}}","\exists"']
            \\
            X_{i+1}
            \ar[r]
            &
            Y_{i+1}
        \end{tikzcd}
    \end{equation}
    Here, we write $\bfDelta_{\mathrm{isoproj}}$ for the union of the class of all isomorphisms in $\C$, the class $\bfDelta$, and the class of all morphisms into the terminal object $1$ in $\C$.
\end{lemma}
\begin{proof}
    \proofdirection[\iff]{\cref{lem:smallest-clan-structure-1}}{\cref{lem:smallest-clan-structure-3}}
    We show that the class described in \cref{lem:smallest-clan-structure-3} indeed satisfies the axioms of clan structures and contains $\bfDelta$.
    The conditions not involving pullbacks are clearly satisfied.
    If $\phi_i$ in \cref{eq:relatively-non-dependent} is a pullback of a morphism $\psi_i$ along $\alpha_i$, we can present the pullback of $\phi$ along a morphism as the composite of the pullback of the morphisms $\psi_i$ by using the pullback lemma step by step.
    This shows the existence of and stability under pullbacks.

    On the other hand, every morphism in this class has to belong to any clan structure containing $\bfDelta$ by the stability and closure conditions.
    Therefore, this class is the smallest clan structure containing $\bfDelta$.

    \proofdirection{\cref{lem:smallest-clan-structure-3}}{\cref{lem:smallest-clan-structure-2}}
    $\bfDelta'$ being a clan structure is shown above.
    For any functor $\Phi$ satisfying the condition in \cref{lem:smallest-clan-structure-2}, the image of a morphism of the form \cref{eq:relatively-non-dependent} is presented as the composite of a sequence of morphisms arising as pullbacks of display morphisms in $\D$, hence it is a display morphism in $\D$.
    Since pullbacks along morphisms of this form are obtained by composing pullbacks along morphisms in $\bfDelta_{\mathrm{isoproj}}$, which $\Phi$ preserves by assumption, $\Phi$ preserves pullbacks along morphisms in $\bfDelta'$ and hence is a clan morphism.

    \proofdirection{\cref{lem:smallest-clan-structure-2}}{\cref{lem:smallest-clan-structure-1}}
    Take any clan structure $\bfDelta''$ on $\C$ containing $\bfDelta$.
    Applying to the identity functor on $\C$ with the clan structure $\bfDelta''$ on the codomain, we see that $\bfDelta'$ is included in $\bfDelta''$.
\end{proof}

\begin{definition}[Generated clan structure]\label{def:generated-clan-structure}
    For a class $\bfDelta$ of morphisms in a category $\C$ that satisfies the assumption of \cref{lem:smallest-clan-structure}, we write $\overline{\bfDelta}$ for the class $\bfDelta'$ that satisfies the equivalent conditions therein, and call it the \emph{clan structure generated by $\bfDelta$}.
    We simply say a clan $\C$ is \emph{generated} by a class of morphisms $\bfDelta$ if its clan structure $\Disp[\C]$ is generated by $\bfDelta$.
\end{definition}

\begin{definition}[Display density and coarse isomorphism]\label{def:clan-dense}
    \ 
    \begin{enumerate}
    \item
    A clan morphism $\Phi\colon \C \arr \D$ is said to be \emph{display dense} if the clan structure of $\D$ is generated by the images of all display morphisms in $\C$ under $\Phi$, i.e., $\Disp[\D]=\overline{\Phi(\Disp[\C])}$.
    \item
    A clan morphism $\Phi\colon \C \arr \D$ is said to be \emph{a coarse (clan) isomorphism} if its underlying functor $\C \arr \D$ is an isomorphism of categories.\qedhere
    \end{enumerate}
\end{definition}

\begin{remark}\label{rem:display-dense-characterization}
    Because of \cref{lem:smallest-clan-structure}\cref{lem:smallest-clan-structure-2}, a clan morphism $\Phi\colon \C \arr \D$ is display dense if and only if for any clan $\E$ and any functor $\Psi\colon \D \arr \E$, if the composite $\Psi \circ \Phi \colon \C \arr \E$ is a clan morphism and $\Psi$ preserves finite products and pullbacks along morphisms in $\Phi(\Disp[\C])$, then $\Psi$ is also a clan morphism.
    This is because morphisms in $\Phi(\Disp[\C])$ are sent to display morphisms in $\E$ only when $\Psi \circ \Phi$ preserves display morphisms.
\end{remark}

\begin{example}
    The unique clan morphism $\C \arr \terminalcat$ is always display dense.
    The canonical clan morphism $\terminalcat \arr \C$, which is unique up to isomorphism, is display dense if and only if $\C$ is a finite-product clan \mbox{because the finite-product clan structure is the smallest clan structure on $\C$.}
\end{example}

\begin{example}
    For a category with finite limits, if it has a morphism that is not a product projection, then the two clan structures, the finite-product and finite-limit clan structures (\cref{prop:finite-product-clan-model}), are different.
    In this case, the identity functor on the category whose domain is equipped with the finite-product clan structure and whose codomain is equipped with the finite-limit clan structure is a coarse clan isomorphism.
\end{example}

\begin{example}\label{ex:non-display-dense}
    Consider the two \acp{GAT} $\theory{T}_\setg$ and $\theory{T}_\catg$ for sets and categories, respectively, as in \cref{eg:set-theory,eg:category-theory}.
    Restricting the clan $\clfy{\theory{T}_\catg}$ to contexts only involving the sort symbol $\syn{Ob}$, we obtain a full subclan which is isomorphic to $\clfy{\theory{T}_\setg}$.
    However, the inclusion $\clfy{\theory{T}_\setg} \arr[hook] \clfy{\theory{T}_\catg}$ is not display dense because the display morphism\vspace{-0.2em}
    \begin{equation}\label{eq:dependent-morphism-in-CAT}
        \left(\syn{x}:\syn{Ob}, \syn{y}:\syn{Ob}, \syn{f}:\syn{Mor}(\syn{x},\syn{y})\right) \darr \left(\syn{x}:\syn{Ob}, \syn{y}:\syn{Ob}\right)
    \end{equation}
    is not included in the clan structure generated by those coming from $\theory{T}_\setg$.
\end{example}

\begin{proposition}\label{prop:properties-of-generation}
    Let $\C$ be a clan whose clan structure is generated by a class of \mbox{morphisms $\bfDelta$.}
    \begin{enumerate}
        \item\label{prop:properties-of-generation-image} For a display dense clan morphism $\Phi\colon \C \arr \E$, the image $\Phi(\bfDelta)$ generates the clan structure of $\E$.
        \item\label{prop:properties-of-generation-preimage} For a clan discrete opfibration $\Psi\colon \D \arr \C$, the class $\Psi^{-1}(\bfDelta)$ generates the clan structure of $\D$.
    \end{enumerate}
\end{proposition}\vspace{-1.0em}
\begin{proof}
    \
    \begin{enumerate}
        \item The class $\Phi\inv(\overline{\Phi(\bfDelta)})\cap\Disp[\C]$ gives a clan structure on the underlying category of $\C$ and contains $\bfDelta$.
        Indeed, it is closed under composition, contains isomorphisms, and is stable under pullbacks along arbitrary morphisms because $\overline{\Phi(\bfDelta)}$ is a clan structure on $\E$ and $\Phi$ preserves pullbacks along morphisms in $\Disp[\C]$.
        Therefore, by the assumption that $\C$ is generated by $\bfDelta$, we have $\Disp[\C]=\Phi\inv(\overline{\Phi(\bfDelta)})\cap\Disp[\C]$.
        Thus, $\Phi(\Disp[\C])=\Phi(\Phi\inv(\overline{\Phi(\bfDelta)})\cap\Disp[\C]) \subseteq \Phi(\Phi\inv(\overline{\Phi(\bfDelta)})) \subseteq \overline{\Phi(\bfDelta)}$.
        Since $\Phi$ is display dense, this leads to $\Disp[\E]=\overline{\Phi(\Disp[\C])}=\overline{\Phi(\bfDelta)}$.
        \item We make use of the third characterization of clan generation in \cref{lem:smallest-clan-structure}.
        Take an arbitrary display morphism $\phi\colon X \darr Y$ in $\D$.
        Then, since $\Psi(\phi)$ is a display morphism as $\Psi$ is a clan discrete opfibration, we can find a factorization of $\Psi(\phi)$ as a composite of display morphisms, each arising as a pullback of a morphism in $\bfDelta_{\mathrm{isoproj}}$, as in \cref{lem:smallest-clan-structure}.
        Since $\Psi$ is a clan discrete opfibration, we can lift the whole diagram in $\C$ as below along it, and we find that the lifted squares are pullback squares of morphisms in $\Psi^{-1}(\bfDelta_{\mathrm{isoproj}})=\Psi^{-1}(\bfDelta)_{\mathrm{isoproj}}$.
        The equality holds because $\Psi$ is a clan discrete opfibration and thus reflects isomorphisms and strictly creates terminal objects.\vspace{-0.4em}
        \begin{equation*}
            \begin{tikzcd}[column sep=1.2em, row sep=1.3em]
                X
                \ar[rr,"\phi_1"',disp]
                \ar[d]
                \pullbackcorner[rrd]
                \ar[rrrrrrrr,"\phi",disp, bend left=10, shift left=0.5em]
                \ar[rrrrrrrr, phantom, "\rotatebox{90}{$=$}", shift left=1em]
                    &
                        &
                        X_1
                        \ar[rr,disp]
                        \ar[d]
                        \ar[dl]
                        \pullbackcorner[rrd]
                            &
                                &
                                \cdots
                                \ar[rr,disp]
                                \ar[ld]
                                    &
                                        &
                                        X_{n-1}
                                        \ar[rr,"\phi_n"',disp]
                                        \ar[d]
                                        \pullbackcorner[rrd]
                                            &
                                                &
                                                Y
                                                \ar[dl]
                                                    &
                                                    {}
                \\
                Z_1
                \ar[r,disp]
                    &
                    W_1
                        &
                        Z_2
                        \ar[r,disp]
                            &
                            W_2
                                &
                                {}
                                    &
                                        &
                                        Z_n
                                        \ar[r,disp]
                                            &
                                            W_n
                                                &
                                                {}
                                                    &
                                                    {}
                                                    \ar[u,phantom,"\incat{\D}"]
                \\
                \\
                \Psi(X) 
                \ar[rr,"\phi_1'"',disp]
                \ar[d]
                \pullbackcorner[rrd]
                \ar[rrrrrrrr,"\Psi(\phi)",disp, bend left=10, shift left=0.5em]
                \ar[rrrrrrrr, phantom, "\rotatebox{90}{$=$}", shift left=1em]
                    &
                        &
                        X_1'
                        \ar[rr,disp]
                        \ar[d]
                        \ar[dl]
                        \pullbackcorner[rrd]
                            &
                                &
                                \cdots
                                \ar[rr,disp]
                                \ar[ld]
                                    &
                                        &
                                        X_{n-1}'
                                        \ar[rr,"\phi_n'"',disp]
                                        \ar[d]
                                        \pullbackcorner[rrd]
                                            &
                                                &
                                                \Psi(Y)
                                                \ar[dl]
                                                    &
                                                    {}
                \\
                Z_1'
                \ar[%
                    r,%
                    disp,%
                    "\begin{smallmatrix}%
                        \rotatebox{-90}{$\in$}\\%
                        \bfDelta_{\mathrm{isoproj}}%
                    \end{smallmatrix}"'%
                ]
                    &
                    W_1'
                        &
                        Z_2'
                        \ar[%
                            r,%
                            disp,%
                            "\begin{smallmatrix}%
                                \rotatebox{-90}{$\in$}\\%
                                \bfDelta_{\mathrm{isoproj}}%
                            \end{smallmatrix}"'%
                        ]
                            &
                            W_2'
                                &
                                {}
                                    &
                                        &
                                        Z_n'
                                        \ar[%
                                            r,%
                                            disp,%
                                            "\begin{smallmatrix}%
                                                \rotatebox{-90}{$\in$}\\%
                                                \bfDelta_{\mathrm{isoproj}}%
                                            \end{smallmatrix}"'%
                                        ]
                                            &
                                            W_n'
                                                &
                                                {}
                                                    &
                                                    {}
                                                    \ar[u,phantom,"\incat{\C}"]
            \end{tikzcd}
        \end{equation*}
        Thus, the display morphism $\phi$ belongs to the smallest clan structure containing $\Psi^{-1}(\bfDelta)$.\qedhere
    \end{enumerate}
\end{proof}

We will show that the class of display dense morphisms are stable under bipushout in the 2-category of (small) clans.
We focus on the 2-category $\Clan$, but in most parts, the same argument works for $\Cland$ as they share the same iso-2-cells.
\begin{definition}
A 1-cell $f\colon X \arr X'$ in a 2-category has the \emph{left lifting property up to iso-2-cells} to a 1-cell $g\colon Y \arr Y'$ if for any iso-2-cell $\alpha$ on the left below, there exists a filler $h$ with iso-2-cells $\alpha_1$ and $\alpha_2$ making the following equation hold:
        \begin{equation}\label{eq:up_to_iso_lifting_property}
        \begin{tikzcd}
            X
            \ar[r,"u"]
            \ar[d,"f"']
                &
                Y
                \ar[d,"g"]
            \\
            X'
            \ar[r,"u'"']
                &
                Y'
            \dltwocell(\alpha){1-2}{2-1}["\scriptstyle\cong"{sloped,below}]
        \end{tikzcd}
        =
        \begin{tikzcd}[efs]
            X
            \ar[rr,"u"]
            \ar[dd,"f"']
            &&
            Y
            \ar[dd,"g"]
            \\
            &\!
            \\
            X'
            \ar[rr,"u'"']
            \ar[rruu, "h"{below right=-2}, very near start]
            &&
            Y'
            \arrow[from = 1-1, to = 2-2, "\rotatebox{-90}{$\Rightarrow$}"{description}, phantom, "\rotatebox{90}{$\scriptstyle\cong$}"{left}, "\scriptstyle \alpha_1"{right = 2pt}]%
            \arrow[from = 2-2, to = 3-3, pos=0.6, "\rotatebox{-90}{$\Rightarrow$}"{description}, phantom, "\rotatebox{90}{$\scriptstyle\cong$}"{left}, "\scriptstyle \alpha_2"{right = 1pt}]
        \end{tikzcd}
    \end{equation}
\end{definition}

When the morphism on the right has a certain property, the left lifting property up to iso-2-cells reduces to the usual left orthogonal property.
We review some properties on 1-cells in a 2-category we use here.

\begin{definition}
    A functor is called \emph{amnestic} if every isomorphism sent to an identity is itself an identity.
\end{definition}

One can show that an injective-on-objects functor is amnestic if and only if it is faithful on isomorphisms in the domain category.

\begin{definition}
    A 1-cell $f\colon X\arr Y$ in a 2-category $\mathcal{K}$ \emph{representably} satisfies a property of functors if for every object $Z$ in $\mathcal{K}$, the induced functor $\mathcal{K}(Z,X) \arr \mathcal{K}(Z,Y)$ satisfies the property.
\end{definition}

\begin{lemma}\label{lem:reduction_orthogonality}
    Let $f$ and $g$ be 1-cells in a 2-category $\mathcal{K}$.
    If $g$ is representably an isofibration, injective-on-objects, and amnestic, then the following conditions are equivalent:
    \begin{enumerate}
        \item\label{lem:reduction_orthogonality-orth}
            $f\orth g$ holds in the underlying (1-)category $\mathcal{K}_0$.
        \item\label{lem:reduction_orthogonality-uptoisoLlp}
            $f$ has the left lifting property up to iso-2-cells to $g$.
    \end{enumerate}
\end{lemma}
\begin{proof}
    \proofdirection{\cref{lem:reduction_orthogonality-orth}}{\cref{lem:reduction_orthogonality-uptoisoLlp}}
    Suppose we have 1-cells $u,u'$ and an iso-2-cell $\alpha$ as in \cref{eq:up_to_iso_lifting_property}.
    Since $g$ is representably an isofibration, we have a 1-cell $v\colon X\arr Y$ together with an iso-2-cell $\alpha_1\colon u\arr[Rightarrow] v$ such that $g\circ v = u'\circ f$ and $g\circ\alpha_1=\alpha$.
    Applying the orthogonality $f\orth g$ to the commutative square $g\circ v = u'\circ f$, we obtain a filler showing the desired lifting property.

    \proofdirection{\cref{lem:reduction_orthogonality-uptoisoLlp}}{\cref{lem:reduction_orthogonality-orth}}
    Note that the uniqueness in the definition of orthogonality holds automatically when a 1-cell on the right is representably injective-on-objects.
    Suppose that we are given a commutative square on the leftmost below.
    By assumption, we have a filler $h$ with iso-2-cells $\alpha_1$ and $\alpha_2$ as in the second diagram below.
    Since $g$ is representably an isofibration, we have a 1-cell $h'$ together with an iso-2-cell $\widetilde{\alpha}_2$ such that $g\circ h'=u'$ and $g\circ\widetilde{\alpha}_2=\alpha_2$ as in the third diagram below.
    However, $g$ being representably amnestic implies that the iso-2-cell obtained by pasting $\alpha_1$ with $\widetilde{\alpha}_2$ is the identity, hence $h'$ is the desired diagonal filler.
    \begin{gather*}
        \begin{tikzcd}[large,ampersand replacement=\&]
            X\ar[d,"f"']\ar[r,"u"]
                \&
                Y\ar[d,"g"]
            \\
            X'\ar[r,"u'"']
                \&
                Y'
            \cellsymb(\rotatebox{45}{$\scriptstyle =$}){1-1}{2-2}
        \end{tikzcd}
        =
        \begin{tikzcd}[large,ampersand replacement=\&]
            X\ar[d,"f"']\ar[r,"u"]
                \&
                Y\ar[d,"g"]
            \\
            X'\ar[r,"u'"']\ar[ru,"h"{pos=0.15,below right=-2}, ""{pos=0.5,name=MID}]
                \&
                Y'
            \dtwocell(\alpha_1){1-1}{MID}["\vcong"{left}]
            \dtwocell(\alpha_2){MID}{2-2}["\vcong"{left},pos=0.6]
        \end{tikzcd}
        =
        \begin{tikzcd}[large,ampersand replacement=\&]
            X\ar[d,"f"']\ar[r,"u"]
                \&
                Y\ar[d,"g"]
            \\
            X'\ar[r,"u'"']\ar[ru,bend left=20,"h"{above left=-2,pos=0.15},""{pos=0.5,name=MID}]\ar[ru,bend right=20,"h'"{above=-1,pos=0.15},""{pos=0.5,name=MID2}]
                \&
                Y'
            \dtwocell(\alpha_1){1-1}{MID}["\vcong"{left}]
            \drtwocell(\widetilde{\alpha}_2){2-1}{1-2}["\rotatebox{135}{$\scriptstyle\cong$}"{below left=-4}]
            \cellsymb(\veq)[pos=0.6]{MID2}{2-2}
        \end{tikzcd}
        =
        \begin{tikzcd}[large,ampersand replacement=\&]
            X\ar[d,"f"']\ar[r,"u"]
                \&
                Y\ar[d,"g"]
            \\
            X'\ar[r,"u'"']\ar[ru,"h'"{above left=-2},""{name=MID}]
                \&
                Y'
            \cellsymb(\veq){1-1}{MID}
            \cellsymb(\veq)[pos=0.6]{MID}{2-2}
        \end{tikzcd}
    \end{gather*}
\end{proof}

\begin{corollary}\label{cor:leftclass_stable_under_bipushout}
    Let $\mathcal{K}$ be a 2-category.
    Let $(\bfE,\bfM)$ be an orthogonal factorization system on the underlying (1-)category $\mathcal{K}_0$, and suppose that every 1-cell in $\bfM$ is representably an isofibration, injective-on-objects, and amnestic in $\mathcal{K}$.
    Then, the class $\bfE$ is stable under bipushout, if it exists.
\end{corollary}
\begin{proof}
    This follows from \cref{lem:reduction_orthogonality}.
\end{proof}

\begin{lemma}\label{lem:coarse-isomorphism-property}
    Every coarse isomorphism in $\Clan$ is representably an isofibration, injective-on-objects, and fully faithful, hence amnestic.
\end{lemma}
\begin{proof}
    The forgetful 2-functor $\Clan \arr \Cattwo$ is faithful and locally fully faithful; it is also locally an isofibration as the property of being a clan morphism is preserved under isomorphisms.
    Given a coarse isomorphism $\Phi\colon \C \arr \D$ between small clans, the induced functor $\Phi\circ - \colon \Clan(\E,\C) \arr \Clan(\E,\D)$ fits into the following commutative diagram for any small clan $\E$.
    \begin{equation*}
        \begin{tikzcd}
            \Clan(\E,\C)
            \ar[r,"\Phi\circ -"]
            \ar[d, hook]
            &
            \Clan(\E,\D)
            \ar[d, hook]
            \\
            \Cattwo(\E,\C)
            \ar[r,"\Phi\circ -", "\cong"']
            &
            \Cattwo(\E,\D)
        \end{tikzcd}
    \end{equation*}
    Here, the vertical arrows are the functors given by the forgetful 2-functor, and they are isofibrations, injective on objects, and fully faithful by the properties of the forgetful 2-functor.
    Since the bottom horizontal arrow is an isomorphism of categories because $\Phi$ is a coarse isomorphism, the top horizontal arrow is also fully faithful and injective on objects, and it is also an isofibration by the faithfulness.
\end{proof}

\begin{lemma}\label{lem:dispdense_coarseiso_OFS}
    The two classes, display dense morphisms and coarse isomorphisms, give an orthogonal factorization system on the underlying (1-)category of $\Clan$.
\end{lemma}
\begin{proof}
    We use \cref{lem:existence_proper_OFS}.
    Let $\bfM$ be the class of all coarse isomorphisms between small clans.
    In particular, by \cref{lem:coarse-isomorphism-property}, every morphism in $\bfM$ is representably injective-on-objects, hence a monomorphism in the underlying (1-)category of $\Clan$.
    It is straightforward to see that this class is closed under composition and contains all isomorphisms by definition.

    We now show the second condition in \cref{lem:existence_proper_OFS}.
    Let $\Phi\colon\C\arr \D$ be a coarse isomorphism and $\Psi\colon\D'\arr \D$ be an arbitrary clan morphism between small clans.
    One can easily verify that the class $\Disp[\D']\cap\Psi^{-1}(\Phi(\Disp[\C]))$ gives a clan structure on the underlying category of $\D'$, and the resulting clan $\C'$ becomes a pullback of $\Phi$ and $\Psi$.
    Then, the projection $\C'\arr \D'$ is a coarse isomorphism because it's underlying functor is the identity.

    We next show the third condition in \cref{lem:existence_proper_OFS}.
    Let $\Phi_i\colon\C_i \arr \C$ $(i\in I)$ be a (not necessarily small) family of coarse isomorphisms.
    For simplicity, we assume the underlying functors of $\Phi_i$ are the identity on $\C$.
    Then, the intersection $\bigcap_{i\in I} \Disp[\C_i]$ gives a clan structure on the underlying category, and the resulting clan $\D$ actually gives the desired wide pullback.
    Then, the associated clan morphism $\D\arr \C$ is clearly a coarse isomorphism because it's underlying functor is still the identity.
    This completes the proof.
\end{proof}

\begin{remark}\label{rem:dispdense_coarseiso_factorization_explicit}
    For a clan morphism $\Phi\colon\C\arr \D$, we obtain the clan whose underlying category is the same as $\D$ and whose clan structure is generated by the image $\Phi(\Disp[\C])$.
    This clan gives a (display dense, coarse iso)-factorization of $\Phi$.
\end{remark}

\begin{example}
    Taking the factorization of the canonical clan morphism $\terminalcat \arr \C$ as in \cref{rem:dispdense_coarseiso_factorization_explicit}, we obtain the finite-product clan on the same underlying category as $\C$ in the middle.
\end{example}

\begin{corollary}\label{cor:dispdense_stable_under_bipushout}
    Display dense clan morphisms are stable under bipushout in $\Clan$.
\end{corollary}
\begin{proof}
    This follows from \cref{lem:dispdense_coarseiso_OFS,lem:coarse-isomorphism-property,cor:leftclass_stable_under_bipushout}.
\end{proof}

\subsection{Simple clan morphisms}\label{subsec:simple_clan_morphisms}
\begin{definition}[Simple clan morphism]\label{def:simple-clan-morphism}
    A clan morphism $\Phi\colon \C \arr \D$ is said to be \emph{simple} if the clan structure of $\D$ is generated by display morphisms of the form $Y\darr \Phi(X)$ for $X\in\C$ and $Y\in\D$.
\end{definition}

\begin{example}\label{exa:trivial-case-non-dependence}
    The unique clan morphism $\C \arr \terminalcat$ is always simple.
    The canonical clan morphism $\terminalcat \arr \C$ is simple if and only if the clan $\C$ is finite-product because the morphisms of the form $Y \darr \Phi(X)$ are precisely the morphisms $Y \darr 1$ in this case, which generate the finite-product clan structure.
\end{example}

\begin{example}[Simple clan morphisms between syntactic clans]\label{exa:syntactic-simple}
    For a \ac{GAT} $\theory{T}$ with a subtheory $\theory{T}'$, the canonical inclusion of clans $\clfy{\theory{T}'} \arr[hook] \clfy{\theory{T}}$ is simple if for any sort symbol $\syn{A}$ in $\theory{T}$, the context $\syn{\Gamma}_{\syn{A}}$ in which the sort symbol is defined belongs to $\theory{T}'$.
    Intuitively, this means that the theory $\theory{T}$ is simple relative to $\theory{T}'$ as a simple type theory is relative to the empty theory.

    For instance, consider the morphism $\clfy{\theory{T}_\setg} \arr \clfy{\theory{T}_\catg}$ in \cref{ex:non-display-dense}.
    The clan structure of $\clfy{\theory{T}_\catg}$ is generated by the display morphism \cref{eq:dependent-morphism-in-CAT} because $\syn{Mor}$ is the only sort dependent on another sort.
    Since the codomain of the morphism is in the image of the functor, the morphism is simple.
    The same argument holds for the clan morphisms that satisfy the above condition.
\end{example}

\begin{lemma}
    \label{lem:clan-dopfib-pullback-non-dependency}
    Let $\Phi\colon \C \arr \D$ be a simple clan morphism and $\Psi\colon \D' \arr \D$ be a clan discrete opfibration between small clans.
    Then, as in \cref{note:Grothendieck-pullback-square}, we have the pullback square
    \begin{equation*}
        \begin{tikzcd}
            \C'
            \ar[r,"\Phi'"]
            \ar[d]
            \pullbackcorner
            &
            \D'
            \ar[d,"\Psi"]
            \\
            \C
            \ar[r,"\Phi"]
            &
            \D
        \end{tikzcd}
        \incat{\Cat}.
    \end{equation*}
    Then, the clan morphism $\Phi'\colon \C' \arr \D'$ is also simple.
\end{lemma}
\begin{proof}
    By assumption, morphisms in $\D$ of the form $Y \darr \Phi(X)$ for $X\in\C$ and $Y\in\D$ generate the clan structure of $\D$.
    By \cref{prop:properties-of-generation}\cref{prop:properties-of-generation-preimage}, the inverse image of the class by $\Psi$ generates the clan structure of $\D'$.
    Since morphisms in this class have the form $Y' \darr \Phi'(X')$ for $X'\in\C'$ and $Y'\in\D'$ because of the pullback square, we conclude that the clan structure of $\D'$ is generated by morphisms of the form $Y' \darr \Phi'(X')$.
    This means that $\Phi'$ is simple.
\end{proof}

\begin{proposition}\label{prop:preservation-of-simple}
\
\begin{enumerate}
    \item\label{prop:preservation-of-simple-1} If the composite $\Psi\circ\Phi$ is simple for clan morphisms $\Phi\colon \C \arr \D$ and $\Psi\colon \D \arr \E$, then so is $\Psi$.
    \item\label{prop:preservation-of-simple-2} If $\Phi\colon \C \arr \D$ is simple and $\Psi\colon \D \arr \E$ is a display dense morphism, then $\Psi\circ\Phi$ is also simple.
\end{enumerate}
\end{proposition}
\begin{proof}\
\begin{enumerate}
    \item The class of morphisms of the form $ Z \darr \Psi(\Phi(X))$ for $X\in\C$ is a subclass of that of morphisms of the form $Z \darr \Psi(Y)$ for $Y\in\D$, hence the latter class generates the clan structure of $\E$ as the former class does.
    \item By \cref{prop:properties-of-generation}\cref{prop:properties-of-generation-image}, the images of morphisms $Y\darr\Phi(X)$ for $X\in\C$ by $\Psi$ generate the clan structure of $\E$.
    However, this image is a subclass of the class of morphisms of the form $Z \darr \Psi(\Phi(X))$ for $X\in\C$, which therefore also generates the clan structure of $\E$.
    \qedhere
\end{enumerate}
\end{proof}

\begin{corollary}\label{cor:non-dependency-clan-dense}
    Consider a bipushout square
    \begin{equation*}
        \begin{tikzcd}
            \C
            \ar[r,"F"]
            \ar[d," \Phi"']
            \ar[rd, "\rotatebox{45}{$\cong$}", phantom]
            &
            \C'
            \ar[d," \Phi'"]
            \\
            \D
            \ar[r," G"']
            &
            \D'
            \ar[lu,phantom, "\ulcorner", very near start]
        \end{tikzcd}\incat{\Clan}
    \end{equation*}
    where $F$ is display dense and $\Phi$ is simple.
    Then the morphism $\Phi'$ is also simple.
\end{corollary}
\begin{proof}
    By \cref{cor:dispdense_stable_under_bipushout}, the morphism $G$ is display dense, and thus, by \cref{prop:preservation-of-simple}\cref{prop:preservation-of-simple-2}, the composite $G\circ \Phi$ is simple.
    Since a clan morphism isomorphic to a simple clan morphism is also simple, $\Phi'\circ F$ is simple.
    By \cref{prop:preservation-of-simple}\cref{prop:preservation-of-simple-1}, the morphism $\Phi'$ is simple.
\end{proof}

\begin{lemma}\label{lem:non-dependent-reflective-subclan-fix}
    Let $\C'$ be a reflective full subclan of a small clan $\C$ whose inclusion $I\colon\C' \arr[hook] \C$ is simple. Then, the clan $\Fix{\C}{M}$ constructed in \cref{const:fixclan} is a finite-product clan for any model $M\colon \C' \arr \Set$.
\end{lemma}
\begin{proof}
    When the inclusion $I \colon \C' \arr[hook] \C$ is simple, so is the morphism $T$ as in \cref{const:fixclan} because of \cref{lem:clan-dopfib-pullback-non-dependency}.
    Since the unique clan morphism $!\colon\El M\arr \terminalcat$ is always display dense, applying \cref{cor:non-dependency-clan-dense} to the bipushout square, we conclude that the canonical morphism $\terminalcat\arr\Fix{\C}{M}$ is simple, which means that $\Fix{\C}{M}$ is a finite-product clan as described in \cref{exa:trivial-case-non-dependence}.
\end{proof}

\begin{theorem}[Comparison formula II]\label{thm:difference-of-decomposition}
    Let $\C'$ be a reflective full subclan of a small clan $\C$ whose inclusion $\C' \arr[hook] \C$ is simple.
    Then, $\decnum{\ClanMod\C} \le \decnum{\ClanMod\C'} + 1$.
\end{theorem}
\begin{proof}
    By \cref{lem:non-dependent-reflective-subclan-fix} and \cref{prop:finite-product-clan-model}\cref{prop:finite-product-clan-model-finprod}, the global regular-decomposition number of $\ClanMod(\Fix{\C}{M})$ is at most $2$.
    As $\ClanMod\C$ and $\ClanMod\C'$ are locally finitely presentable, their global regular-decomposition numbers are at most $\omega+1$ by \cref{cor:upperbound_for_global_regdecnum_of_lp}.
    When $\decnum{\ClanMod\C'}$ is either $\omega$ or $\omega+1$, the desired inequality holds trivially.
    When it is a finite number, we can apply \cref{thm:bound-of-decomposition-length} to obtain the desired inequality.
\end{proof}

\begin{example}
    We continue discussing with the example $\clfy{\theory{T}_\setg}$ and $\clfy{\theory{T}_\catg}$ in \cref{ex:non-display-dense,exa:syntactic-simple}.
    The full inclusion $\clfy{\theory{T}_\setg} \arr[hook] \clfy{\theory{T}_\catg}$ is simple by \cref{exa:syntactic-simple}, and has a left adjoint, which sends each context by cutting down the $\syn{Mor}$-part.
    The unit component at each context is a morphism that forgets the variables of sort $\syn{Mor}$, which is a display morphism.
    Applying \cref{thm:difference-of-decomposition}, we obtain $\decnum{\Cat}\le \decnum{\Set}+1=3$.
\end{example}

We now provide a sufficient condition under which a \ac{GAT} has a category of models whose global regular-decomposition number is finite.
The conventions and notations we adopt here are summarized in \cref{sec:GAT}.
For a type judgment $\syn{\Gamma}\vdash \syn{T}\type$ in a \ac{GAT} $\theory{T}$, there is a unique sort symbol $\syn{A}$ such that $\syn{T}$ is obtained by substituting terms into $\syn{A}(\syn{x}_1,\dots,\syn{x}_n)$.
We call this sort symbol the \emph{associated sort} of the type judgment.
Since this does not depend on the context $\syn{\Gamma}$, we also say that $\syn{A}$ is the associated sort of the type $\syn{T}$ in a judgment $\syn{\Gamma}\vdash \syn{T}\type$.

\begin{definition}[Dependency rank]
    For a \ac{GAT} $\theory{T}$, we inductively assign to each sort symbol $\syn{A}$ in $\theory{T}$ a finite number, which we call \emph{dependency rank} and write $\depran{\syn{A}}$,  as follows:
    \begin{equation*}
    \depran{\syn{A}} = \max \left\{\,\depran{\syn{B}}\mid \syn{B} \text{ appears in }\syn{\Gamma}_{\syn{A}}\,\right\} +1
    \end{equation*}
    where we set $\max\emptyset = 0$.
    This is well-defined because there is no circular dependency between sort symbols in a \ac{GAT}.

    For a type judgment $\syn{\Gamma}\vdash \syn{T}\type$, we define the \emph{dependency rank of the type $\syn{T}$}, denoted by $\depran{\syn{T}}$, to be that of its associated sort.
    For a context $\syn{\Gamma}$, we denote by $\depran{\syn{\Gamma}}$ the maximum dependency rank of any sort symbol appearing in the context $\syn{\Gamma}$ and call it the \emph{dependency rank of the context $\syn{\Gamma}$}.
\end{definition}

\begin{definition}[Non-descending \acp{GAT}]
    A \ac{GAT} $\theory{T}$ is said to be \emph{non-descending} if 
    \begin{enumerate}
        \item for each operator symbol $\syn{f}$ in $\theory{T}$, the dependency rank of the type of its value $\syn{T}_\syn{f}$ is larger than or equal to the dependency rank of the context $\syn{\Gamma}_\syn{f}$, and
        \begin{equation*}
            \syn{\Gamma}_{\syn{f}}\vdash \syn{f}(\syn{x}_1,\dots,\syn{x}_n) : \syn{T}_\syn{f}
            \quad \in \theory{T}
            \quad \Longrightarrow \quad
            \depran{\syn{\Gamma}_\syn{f}} \le \depran{\syn{T}_\syn{f}} 
        \end{equation*}
        \item for each axiom of $\theory{T}$, the dependency rank of the type $\syn{T}$ where we take the equality into account is larger than or equal to that of the context $\syn{\Gamma}$.
        \begin{equation*}
            \syn{\Gamma} \vdash \syn{t_1} = \syn{t_2} : \syn{T}
            \quad \in \theory{T}
            \quad \Longrightarrow \quad
            \depran{\syn{\Gamma}} \le \depran{\syn{T}}\qedhere
        \end{equation*}
    \end{enumerate}
\end{definition}

A theory being non-descending intuitively means that fixing a natural number $n$ and cutting out the sorts of dependency rank greater than $n$ and other things defined using these sorts does not affect the remaining theory.
When this is the case for $\theory{T}$, we can stratify contexts in $\theory{T}$ according to their dependency ranks, by decomposing the canonical clan morphism $\terminalcat\arr\clfy{\theory{T}}$ into a series of clan reflective simple inclusions, as shown below.
\begin{definition}
    For a non-descending \ac{GAT} $\theory{T}$ and a natural number $n$, we define by ${\clfy{\theory{T}}}_{\le n}$ the full subcategory of $\clfy{\theory{T}}$ spanned by the contexts of dependency rank at most $n$.
    The clan structure on ${\clfy{\theory{T}}}_{\le n}$ is given by the same display morphisms as in $\clfy{\theory{T}}$ between these contexts.
    This is indeed a clan structure because pullbacks along display morphisms between contexts of dependency rank at most $n$ again have dependency rank at most $n$, so this subcategory is closed under pullbacks along display morphisms.
\end{definition}

\begin{lemma}\label{lem:reflective-subclan-dependency-rank-syntactic}
    Let $\theory{T}$ be a non-descending \ac{GAT}.
    Then, for each natural number $n$, the full subclan ${\clfy{\theory{T}}}_{\le n}$ of $\clfy{\theory{T}}$ is reflective.
\end{lemma}
\begin{proof}
    First, we define a precontext $\syn{\Gamma}\discard{n}$ for each context $\syn{\Gamma}$ in $\theory{T}$ by removing all pairs $\syn{x}:\syn{T}$ such that $\depran{\syn{T}}>n$.
    We will show the following five statements concerning judgments of different forms:
    \begin{enumerate}
        \item 
            If $\theory{T} \drv \vdash \syn{\Gamma} \ctx$, then $\theory{T} \drv \vdash \syn{\Gamma}\discard{n} \ctx$.
        \item
            If $\theory{T} \drv \syn{\Gamma} \vdash \syn{T} \type$ and the associated sort of $\syn{T}$ has a dependency rank at most $n$, then $\theory{T} \drv \syn{\Gamma}\discard{n} \vdash \syn{T}\type$.
        \item
            If $\theory{T} \drv \syn{\Gamma} \vdash \syn{t} : \syn{T}$ and the associated sort of $\syn{T}$ has a dependency rank at most $n$, then $\theory{T} \drv \syn{\Gamma}\discard{n} \vdash \syn{t} : \syn{T}$.
        \item
            If $\theory{T} \drv \syn{\Gamma} \vdash \syn{T_1} = \syn{T_2} \type$ and the associated sorts of $\syn{T}_1,\syn{T}_2$ have dependency ranks at most $n$, then $\theory{T} \drv \syn{\Gamma}\discard{n} \vdash \syn{T_1} = \syn{T_2} \type$.
        \item\label{item:reflectivity_fifth_statement}
            If $\theory{T} \drv \syn{\Gamma} \vdash \syn{t_1} = \syn{t_2} : \syn{T}$ and the associated sort of $\syn{T}$ has a dependency rank at most $n$, then $\theory{T} \drv \syn{\Gamma}\discard{n} \vdash \syn{t_1} = \syn{t_2} : \syn{T}$.
    \end{enumerate}

    We prove these altogether by induction on the derivation rule of $\theory{T} \drv \syn{J}$.
    This means that for each derivation rule, we assume that the five statements above on a judgment $\syn{J}$ hold for all the premises of the rule and show that they also hold for the conclusion.
    See \cref{def:derivation-rules-of-GAT} for the list of the derivation rules.

    The first eleven rules, \cref{rule:refl-ty,rule:refl-tm,rule:sym-ty,rule:sym-tm,rule:trans-ty,rule:trans-tm,rule:conv-tm,rule:conv-tm-eq,rule:alpha-ctx,rule:alpha,rule:empty}, are straightforward.

    For \cref{rule:ctx-extend}, the context $(\syn{\Gamma}, \syn{x}:\syn{T})\discard{n}$ is 
    $(\syn{\Gamma}\discard{n}, \syn{x}:\syn{T})$ when $\syn{T}$ has a dependency rank at most $n$, and $\syn{\Gamma}\discard{n}$ when $\syn{T}$ has a dependency rank higher than $n$.
    In both cases, the desired judgment is derived from $\theory{T}$ by the induction hypothesis.

    \cref{rule:var} is only applied when $\syn{T}$ has a dependency rank at most $n$, in which case $(\syn{\Gamma}, \syn{x}:\syn{T})\discard{n}=\syn{\Gamma}, \syn{x}:\syn{T}$, so there is nothing to prove.

    For \cref{rule:weak}, we can suppose that $\syn{K}$ only has sorts with a dependency rank at most $n$, because otherwise the assumption in the statement does not hold.
    If $\syn{T}$ has a dependency rank at most $n$, then $(\syn{\Gamma}, \syn{x}:\syn{T}, \syn{\Delta})\discard{n}=(\syn{\Gamma}\discard{n}, \syn{x}:\syn{T}, \syn{\Delta}\discard{n})$.
    By the induction hypothesis, we have $\theory{T} \drv \syn{\Gamma}\discard{n}\vdash \syn{T} \type$ in this case.
    Thus, we obtain the desired judgment by the same rule.
    If $\syn{T}$ has a dependency rank higher than $n$, then $(\syn{\Gamma}, \syn{x}:\syn{T}, \syn{\Delta})\discard{n}$ is $(\syn{\Gamma}\discard{n}, \syn{\Delta}\discard{n})$.
    By the induction hypothesis for the second premise, we have the desired judgment.

    For \cref{rule:subst}, we can suppose that $\syn{K}$ only has sorts with a dependency rank at most $n$ for the same reason as above.
    If $\syn{T}$ has a dependency rank at most $n$, then we have $\theory{T} \drv \syn{\Gamma}\discard{n} \vdash \syn{t} : \syn{T}$ by the induction hypothesis for the first premise.
    Thus, with the induction hypothesis for the second premise, we obtain the desired judgment by the same rule.
    If $\syn{T}$ has a dependency rank higher than $n$, then by the induction hypothesis for the second premise, we have $\theory{T} \drv \syn{\Gamma}\discard{n}, \syn{\Delta}\discard{n} \vdash \syn{K}$.
    This means that $\syn{x}$ does not appear either in $\syn{\Delta}$ or in $\syn{K}$.
    Thus, we have the desired judgment by the induction hypothesis for the second premise.

    \cref{rule:eq-subst-ty} and \cref{rule:eq-subst-tm} are proved in the same way as \cref{rule:subst}.

    For \cref{rule:sort}, by the assumption, $\syn{A}$ has a dependency rank at most $n$, which by definition means that all the sorts appearing in $\syn{\Gamma}_\syn{A}$ have a dependency rank at most $n-1$.
    Therefore, $\syn{\Gamma}_\syn{A}\discard{n}=\syn{\Gamma}_\syn{A}$, and we have nothing to prove.

    For \cref{rule:oper}, again by the assumption, $\syn{T}_\syn{f}$ has a dependency rank at most $n$.
    Since $\theory{T}$ is non-descending, all the sorts appearing in $\syn{\Gamma}_\syn{f}$ have a dependency rank at most $\depran{\syn{\Gamma}_\syn{f}} \le \depran{\syn{T}_\syn{f}} \le n$, so again we have nothing to prove.

    \cref{rule:eq-axiom} is similarly proved by the assumption that $\theory{T}$ is non-descending.

    Given a morphism $\phi\colon \syn{\Gamma} \arr \syn{\Delta}$ in $\clfy{\theory{T}}$ with $\depran{\syn{\Delta}}\le n$, we can define a morphism $\phi'\colon \syn{\Gamma}\discard{n} \arr \syn{\Delta}$ by taking the same terms as $\phi$, which is well-defined by the third item above.
    This morphism $\phi'$ composed with the canonical projection $\syn{\Gamma} \arr \syn{\Gamma}\discard{n}$ gives $\phi$.
    The uniqueness of such a morphism $\phi'$ also follows from \cref{item:reflectivity_fifth_statement}.
    Thus, the inclusion ${\clfy{\theory{T}}}_{\le n} \arr[hook] \clfy{\theory{T}}$ has a left adjoint, where the unit component on each context $\syn{\Gamma}$ is given by the canonical projection $\syn{\Gamma} \arr \syn{\Gamma}\discard{n}$, which is a display morphism.
\end{proof}

\begin{lemma}\label{lem:simple-subclan-dependency-rank-syntactic}
    Let $\theory{T}$ be a non-descending \ac{GAT}.
    Then, the full inclusion ${\clfy{\theory{T}}}_{\le n} \arr[hook] {\clfy{\theory{T}}}_{\le n+1}$ is a simple clan morphism for each natural number $n$.
\end{lemma}
\begin{proof}
    Note that the syntactic clan structure on the syntactic category ${\clfy{\theory{T}}}$ is generated by the display morphisms of the form $\syn{\Gamma}_{\syn{A}},\syn{y}:\syn{A}(\syn{x}_1,\dots,\syn{x}_n)\darr \syn{\Gamma}_{\syn{A}}$ for each sort symbol $\syn{A}$ in $\theory{T}$.
    Thus, the clan structure on ${\clfy{\theory{T}}}_{\le n+1}$ is generated by the display morphisms of the form $\syn{\Gamma}_{\syn{A}},\syn{y}:\syn{A}(\syn{x}_1,\dots,\syn{x}_n)\darr \syn{\Gamma}_{\syn{A}}$ for sort symbols $\syn{A}$ with $\depran{\syn{A}}\le n+1$.
    By the non-descending condition, the context $\syn{\Gamma}_{\syn{A}}$ has a dependency rank at most $n$, which belongs to the image of the inclusion ${\clfy{\theory{T}}}_{\le n} \arr[hook] {\clfy{\theory{T}}}_{\le n+1}$.
    Therefore, the inclusion is simple.
\end{proof}

\begin{theorem}[Main theorem in the \acs{GAT} approach]\label{thm:global-decomposition-by-clan-approach}
    Let $\theory{T}$ be a non-descending \ac{GAT} with the maximum dependency rank $n$ of its sort symbols.
    Then, $\decnum{\ClanMod(\clfy{\theory{T}})} \le n+2$.
\end{theorem}
\begin{proof}
    By \cref{ex:trivial-clan-reflective-subclan,lem:reflective-subclan-dependency-rank-syntactic}, we have a sequence of reflective full subclans
    \begin{equation*}
        \terminalcat \arr[hook] {\clfy{\theory{T}}}_{\le 0} \arr[hook] {\clfy{\theory{T}}}_{\le 1} \arr[hook] \cdots \arr[hook] {\clfy{\theory{T}}}_{\le n} = \clfy{\theory{T}},
    \end{equation*}
    where each inclusion is simple by \cref{exa:trivial-case-non-dependence,lem:simple-subclan-dependency-rank-syntactic}.
    Note that the category $\ClanMod\terminalcat$ is equivalent to the terminal category, and hence $\decnum{\ClanMod\terminalcat}=1$.
    Applying \cref{thm:difference-of-decomposition} repeatedly to this sequence, we have
    \begin{align*}
        \decnum{\ClanMod(\clfy{\theory{T}})} &= \decnum{\ClanMod({\clfy{\theory{T}}}_{\le n})} \\
        &\le \decnum{\ClanMod({\clfy{\theory{T}}}_{\le n-1})} + 1 \\
        &\le \cdots \le \decnum{\ClanMod\terminalcat} + (n+1) = n + 2.
        \qedhere
    \end{align*}
\end{proof}

\begin{example}[Global regular-decomposition numbers of categories of categorical structures]\quad\label{eg:global_decnum_by_clanapproach}
    \begin{enumerate}
        \item\label{eg:global_decnum_by_clanapproach-nCat} 
            Let $\theory{T}_\ncatg$ be the \ac{GAT} for (small) strict $n$-categories for $n\ge 1$, which has sort symbols for each $k$-cells ($0\le k \le n$) and operator symbols for compositions and identities.
            One can verify that this theory is non-descending and the maximum dependency rank of its sort symbols is $n$.
            Therefore, by \cref{thm:global-decomposition-by-clan-approach}, the global regular-decomposition number of the category $\nCat\simeq \ClanMod\theory{T}_\ncatg$ of its models is at most $n+2$.
            Note that the sequence of reflective full subclans obtained in \cref{lem:reflective-subclan-dependency-rank-syntactic} is identified with the following sequence:
            \begin{equation*}
                \theory{T}_\setg \subseteq \theory{T}_\catg \subseteq \theory{T}_\twocatg \subseteq \cdots \subseteq \theory{T}_\ncatg.
            \end{equation*}
            This partially recovers the result in \cref{thm:global_decnum_of_nCat}, although the lower bound is not shown here.
        \item
            The \ac{GAT} for (small) strict monoidal categories $\theory{T}_\moncatg$ is defined by adding operator symbols for tensor products and unit objects to that for (small) categories $\theory{T}_\catg$ with appropriate axioms.
            One can verify that this theory is still non-descending and the maximum dependency rank of its sort symbols is unchanged.
            Therefore, by \cref{thm:global-decomposition-by-clan-approach}, the global regular-decomposition number of the category of small strict monoidal categories and strict monoidal functors $\StrMonCat\simeq \ClanMod\theory{T}_\moncatg$ is at most $3$.

            The free-strict-monoidal-category functor $\Cat\arr\StrMonCat$ preserves regular epimorphisms because it is a left adjoint, preserves kernel pairs by direct verification,\footnote{%
                According to \cite[Example 4.1.15]{Leinster2004higher}, the free-strict-monoidal-category monad on $\Cat$ is cartesian, which implies that the functor we consider here preserves pullbacks because monadic functors create limits.%
            } and it also preserves monomorphisms.
            Take a morphism in $\Cat$ whose regular-decomposition number is $2$ and send it by the functor to $\StrMonCat$.
            Then, the canonical regular decomposition is preserved because of the above properties.
            Thus, $\StrMonCat$ has a morphism whose regular-decomposition number is $2$, which shows $\decnum{\StrMonCat}=3$.
        \item
            The \ac{GAT} for (small) multicategories $\theory{T}_\multicatg$ is defined by adding sort symbols for $n$-ary morphisms for $n\neq 1$ and operator symbols for their compositions to that for categories $\theory{T}_\catg$ with appropriate axioms.
            This has countably many operator symbols but is still non-descending, and the maximum dependency rank of its sort symbols is still $1$.
            Therefore, by \cref{thm:global-decomposition-by-clan-approach}, the global regular-decomposition number of the category of small multicategories $\MultiCat\simeq \ClanMod\theory{T}_\multicatg$ is at most $3$.

            A category is seen as a multicategory with only unary morphisms, which gives a functor $\Cat\arr\MultiCat$.
            Similarly to the case of strict monoidal categories, this functor preserves regular epimorphisms, kernel pairs, and monomorphisms.
            Thus, the same argument shows that $\decnum{\MultiCat}=3$.
        \item
            The \ac{GAT} for (small) strict double categories $\theory{T}_\dblcatg$ is defined as a \ac{GAT} with sort symbols for objects, horizontal morphisms, vertical morphisms, and cells, and operator symbols for their compositions and identities with appropriate axioms.
            It is again a simple verification to see that this theory is non-descending and the maximum dependency rank of its sort symbols is $2$.
            Therefore, by \cref{thm:global-decomposition-by-clan-approach}, the global regular-decomposition number of the category of small strict double categories $\DblCat\simeq \ClanMod\theory{T}_\dblcatg$ is at most $4$.
            
            A 2-category can be seen as a strict double category whose horizontal morphisms are all identities, which leads to a functor $\twoCat\arr\DblCat$.
            This functor again has the same properties we used above, and thus the same argument together with \cref{eg:nfunctor_whose_decnum_is_nplus1} shows $\decnum{\DblCat}=4$.\qedhere
    \end{enumerate}
\end{example}

\begin{example}
    It is easy to construct a \ac{GAT} that is not non-descending.
    Let $\syn{X}$ be a sort symbol defined in the empty context, and $\syn{Y}$ be another sort symbol defined in the context $\syn{x}:\syn{X}$.
    When we add the following axioms, the resulting theory is not non-descending.
    Here, $\syn{Y}$ represents a ``partial constant'' valued in $\syn{X}$.
    \begin{equation}\label{eq:non-non-descending}
        \syn{x}:\syn{X}, \syn{y}:\syn{Y}(\syn{x}), \syn{y'}:\syn{Y}(\syn{x})\vdash \syn{y} = \syn{y'} : \syn{Y}(\syn{x}),
        \quad
        \syn{x}:\syn{X}, \syn{x'}:\syn{X}, \syn{y}:\syn{Y}(\syn{x}), \syn{y'}:\syn{Y}(\syn{x'}) \vdash \syn{x} = \syn{x'} : \syn{X}.
    \end{equation}
    This theory fails to be non-descending because of the second axiom.

    In general, a \ac{GAT} that is not non-descending may have a syntactic category whose clan structure can not be stratified as in the proof of \cref{thm:global-decomposition-by-clan-approach}, and thus the global regular-decomposition number of its category of models may not be bounded.
    We can even construct a \ac{GAT} whose category of models has an infinite global regular-decomposition number, based on \cref{eg:global_decnum_is_omegaplusone}.
    Take a sort $\syn{X}$ defined in the empty context, for each natural number $n$, take another sort $\syn{C}_n$ defined in the context $\syn{x}:\syn{X}$, and another sort $\syn{B}$ defined in the same context.
    We ask the same axioms as \cref{eq:non-non-descending} for each $\syn{C}_n$ and $\syn{B}$.
    We also add operator symbols $\syn{f}_n, \syn{g}_n$ for each natural number $n$ as follows:
    \begin{equation*}
        \begin{aligned}
            &\syn{x}:\syn{X}, \syn{c}:\syn{C}_n(\syn{x}), \syn{b}:\syn{B}(\syn{x}) \vdash \syn{f}_n(\syn{x},\syn{c},\syn{b}):\syn{X}, \\
            &\syn{x}:\syn{X}, \syn{c}:\syn{C}_n(\syn{x}), \syn{b}:\syn{B}(\syn{x}) \vdash \syn{g}_n(\syn{x},\syn{c},\syn{b}):\syn{C}_{n+1}(\syn{f}_n(\syn{x},\syn{c},\syn{b})). \\
        \end{aligned}
    \end{equation*}
    This theory has the same models as that in \cref{eg:global_decnum_is_omegaplusone}, and thus its category of models has an infinite global regular-decomposition number.
\end{example}
\addtocontents{toc}{\protect\setcounter{tocdepth}{0}}
\section*{Acknowledgments}
    We would like to thank Nathanael Arkor, Keisuke Hoshino, V\'it Jel\'inek, and Yuki Maehara for helpful discussions.
    We are also grateful to the workshop ``Categories in Tokyo'' for the opportunity to discuss the decomposition number for monotone quotient maps.
    We would like to thank those who joined the discussion; in particular, we would like to express our deep gratitude to Soichiro Fujii and Junnosuke Koizumi for providing their idea, and to Junnosuke Koizumi for his permission to include the construction of the example in \cref{eg:monotone_example_by_koizumi} in this paper.
    We are sincerely thankful to our supervisor, Masahito Hasegawa, and the second author's supervisor, Dorette Pronk, for their continuous support and encouragement.
    The first author is supported by JSPS KAKENHI Grant Number JP24KJ1462, and the second author was supported by JSPS KAKENHI Grant Number JP25KJ1474.
\addtocontents{toc}{\protect\setcounter{tocdepth}{1}}
\begin{appendix}
\section{Orthogonal and locally orthogonal factorizations}\label[appendix]{sec:local_orthogonality}
We recall the notion of \textit{locally orthogonal factorizations} from \cite{MacdonaldTholen1982decomposition,Tholen1983factorizations}.
\begin{definition}
    Let $\C$ be a category.
    \begin{enumerate}
        \item
            Morphisms $X\arr(p)[][1]Y$ and $A\arr(m)[][1]B$ in $\C$ are called \emph{locally orthogonal with respect to $A'\arr(e)[][1]A$} (written $p\orth[e] m$) if, for any morphisms $X\arr(u)[][1]A'$ and $Y\arr(v)[][1]B$ such that $meu=vp$, there uniquely exists a morphism $Y\arr(h)[][1]A$ making the following diagram commute.\vspace{-0.5em}
            \begin{equation*}
                \begin{tikzcd}
                    X\ar[d,"p"']\ar[r,"u"] & A'\ar[r,"e"] & A\ar[d,"m"] \\
                    Y\ar[rr,"v"']\ar[urr,"h",dotted] && B
                \end{tikzcd}\incat{\C}
            \end{equation*}
        \item
            Morphisms $X\arr(p)[][1]Y$ and $A\arr(m)[][1]B$ in $\C$ are called \emph{orthogonal} (written $p\orth m$) if they are locally orthogonal with respect to the identity on $A$.\qedhere
    \end{enumerate}
\end{definition}

\begin{notation}\label{note:local_orthogonality}
    For a class $\mathbf{E}$ of morphisms, we write $\mathbf{E}\orth[e] m$ if for every $p\in\mathbf{E}$, $p\orth[e] m$ holds.
    We also use a notation $\mathbf{E}\orth m$ similarly.
\end{notation}

\begin{definition}\label{def:locally_orthogonal_factorization}
    Let $\bfE$ be an iso-closed class of morphisms in a category $\C$.
    \begin{enumerate}
        \item
            We say \emph{$\C$ admits locally orthogonal $\bfE$-factorizations} if every morphism $f$ in $\C$ has a factorization $f=me$ such that $e\in\bfE$ and $\bfE\orth[e] m$.
        \item
            We say \emph{$\C$ admits orthogonal $\bfE$-factorizations} if every morphism $f$ in $\C$ has a factorization $f=me$ such that $e\in\bfE$ and $\bfE\orth m$.
    \end{enumerate}
    Each pair $(e,m)$ above is called a \emph{(locally) orthogonal $\bfE$-factorization of $f$}.
\end{definition}

We now recall the connection with orthogonal factorization systems.
\begin{definition}[Factorization systems] Let $\C$ be a category.
    \begin{enumerate}
        \item
        Given a class of morphisms $\mathbf{\Lambda}$, denote by $\lorth(\mathbf{\Lambda})$ and $\rorth(\mathbf{\Lambda})$ the classes
        \begin{align*}
            \lorth(\mathbf{\Lambda})&\coloneq\{ e \mid e\perp m\text{ for all }m\in\mathbf{\Lambda}\}, \\
            \rorth(\mathbf{\Lambda})&\coloneq\{ m \mid e\perp m\text{ for all }e\in\mathbf{\Lambda}\}.
        \end{align*}
        \item
        An \emph{orthogonal factorization system} on $\C$ is a pair $(\bfE,\bfM)$ of classes of morphisms in $\C$ that satisfies the following conditions:
        \begin{itemize}
            \item $\bfE$ and $\bfM$ are closed under composition and contain all isomorphisms in $\C$;
            \item Every morphism $f$ in $\C$ has a factorization $f=me$ with $e\in\bfE$ and $m\in\bfM$;
            \item $\bfE\perp\bfM$ holds, i.e., for any $e\in\bfE$ and $m\in\bfM$, $e\perp m$ holds.
        \end{itemize}
        Given an orthogonal factorization system $(\bfE,\bfM)$, we have $\bfE=\lorth(\bfM)$ and $\bfM=\rorth(\bfE)$.
        These can be verified straightforwardly.\qedhere
    \end{enumerate}
\end{definition}

\begin{proposition}
    Let $\bfE$ be a class of morphisms in a category $\C$.
    Then, the following are equivalent:
    \begin{enumerate}
        \item
            $(\bfE,\rorth(\bfE))$ becomes an orthogonal factorization system on $\C$.
        \item
            $\bfE$ is iso-closed, and $\C$ admits orthogonal $\bfE$-factorizations.
    \end{enumerate}
\end{proposition}

\begin{proposition}[{\cite[2.1.~Proposition]{Tholen1983factorizations}}]
    Let $\bfE$ be an iso-closed class of morphisms in a category $\C$.
    Then, the following are equivalent:
    \begin{enumerate}
        \item
            $\C$ admits orthogonal $\bfE$-factorizations.
        \item
            $\C$ admits locally orthogonal $\bfE$-factorizations, and $\bfE$ is closed under composition.
    \end{enumerate}
\end{proposition}

As the following statements are used in the paper, we include them here:
\begin{lemma}\label{lem:leftclass_epic}
    For a class $\bfE$ of morphisms in a category $\C$, the implications \cref{lem:leftclass_epic-1}$\implies$\cref{lem:leftclass_epic-2}$\implies$\cref{lem:leftclass_epic-3}$\implies$\cref{lem:leftclass_epic-4} hold.
    \begin{enumerate}
        \item\label{lem:leftclass_epic-1}
            $\lorth\rorth(\bfE) \subseteq \Epi[\C]$.
        \item\label{lem:leftclass_epic-2}
            $\bfE \subseteq \Epi[\C]$.
        \item\label{lem:leftclass_epic-3}
            $\rorth(\bfE)$ has the \textit{strong cancellation property}: $gf\in\rorth(\bfE)$ implies $f\in\rorth(\bfE)$.
        \item\label{lem:leftclass_epic-4}
            Every split monomorphism belongs to $\rorth(\bfE)$.
    \end{enumerate}
    Moreover, if the category $\C$ has binary powers, then all the conditions above are equivalent.
\end{lemma}
\begin{proof}
    The implications \cref{lem:leftclass_epic-1}$\implies$\cref{lem:leftclass_epic-2}$\implies$\cref{lem:leftclass_epic-3}$\implies$\cref{lem:leftclass_epic-4} follow straightforwardly.
    If $\C$ has binary powers, the implication \cref{lem:leftclass_epic-4}$\implies$\cref{lem:leftclass_epic-1} additionally follows by considering the diagonal split monomorphisms $\Delta_X\colon X\arr X\times X$.
\end{proof}

\begin{lemma}\label{lem:existence_proper_OFS}
    Let $\bfM$ be a class of monomorphisms in a category $\C$ satisfying the following conditions:
    \begin{itemize}
        \item
            $\bfM$ is closed under composition and contains all isomorphisms.
        \item
            $\C$ has a pullback of any morphism in $\bfM$ along any morphism, and $\bfM$ is stable under such pullbacks.
        \item
            $\C$ has a wide pullback of any (not necessarily small) family of morphisms in $\bfM$, and $\bfM$ is stable under such wide pullbacks.
    \end{itemize}
    Let $\bfE$ be the class of all morphisms $f$ such that any factorization $f=m\circ g$ with $m\in\bfM$ implies $m$ is an isomorphism.
    Then, the pair $(\bfE,\bfM)$ becomes an orthogonal factorization system.
\end{lemma}
\begin{proof}
    Combine \cite[Lemma 3.1]{CassidyHebertKelly1985reflective} with the dual of the equivalence between (1) and (7) in \cite[{}14.11]{AdamekHerrlichStrecker2006joy}.
\end{proof}\pagebreak
\section{Partial Horn theories}\label[appendix]{sec:partial_horn_theories}
For convenience, we summarize the notation and conventions for \textit{partial Horn theories}, following \cite{Kawase2026relativized} rather than the original ones \cite{PalmgrenVickers2007partial}.
\subsection{The syntax of partial Horn logic}\label[appendix]{subsec:the_syntax_of_partial_horn_logic}
\begin{definition}[Signatures]
    Let $\kappa$ be a small regular cardinal, and let $S$ be a small set.
    A \emph{($\kappa$-ary) $S$-sorted signature} $\Sigma$ consists of:
    \begin{itemize}
        \item
            a small set $\Sigma_\mathrm{f}$, whose elements are called \emph{function symbols};
        \item
            a small set $\Sigma_\mathrm{r}$, whose elements are called \emph{relation symbols};
        \item
            for each $\syn{f}\in\Sigma_\mathrm{f}$, an element $\syn{s}\in S$ and a family $(\syn{s}_i \in S)_{i<\alpha}$ indexed by a small ordinal $\alpha<\kappa$, which are often denoted by $\syn{f}\colon \sqcap_{i<\alpha}\syn{s}_i \to \syn{s}$;
        \item
            for each $\syn{R}\in\Sigma_\mathrm{r}$, a family $(\syn{s}_i \in S)_{i<\alpha}$ indexed by a small ordinal $\alpha<\kappa$, which is often denoted by $\syn{R}\colon \sqcap_{i<\alpha} \syn{s}_i$.
    \end{itemize}
    Elements of the small set $S$ are called \emph{sorts}.
\end{definition}


\begin{definition}[Variables]
    Given the small set $S$ of sorts, we fix an $S$-sorted set $\Var=(\Var_\syn{s})_{\syn{s}\in S}$ such that $\Var_\syn{s}$ has cardinal $\kappa$ for each $\syn{s}\in S$.
    We assume $\Var_\syn{s}\cap\Var_{\syn{s}'}=\emptyset$ if $\syn{s}\neq \syn{s}'$.
    An element $\vx\in\Var_\syn{s}$ is called a \emph{variable of sort $\syn{s}$}.
    The notation $\vx\ofsort\syn{s}$ means that $\vx$ is a variable of sort $\syn{s}$.
\end{definition}

\begin{definition}
    Let $\Sigma$ be an $S$-sorted $\kappa$-ary signature.
    \begin{enumerate}
        \item
        \emph{Raw terms} $\ttau$ (over $\Sigma$) and their \emph{sorts} $\syn{s}$, written $\ttau\ofsort\syn{s}$, are defined by the following inductive rules:
        \begin{itemize}
            \item
                Given a variable $\vx\ofsort\syn{s}$, $\vx$ is a raw term of sort $\syn{s}$;
            \item
                Given a function symbol $\syn{f}\colon\sqcap_{i<\alpha}\syn{s}_i\to\syn{s}$ in $\Sigma$ and raw terms $\ttau_i\ofsort\syn{s}_i$ $(i<\alpha)$,
                then the expression $\syn{f}(\ttau_i)_{i<\alpha}$ is a raw term of sort $\syn{s}$.
        \end{itemize}
        \item
        \emph{Raw ($\kappa$-ary) Horn formulas} (over $\Sigma$) are defined by the following inductive rules:
        \begin{itemize}
            \item
                Given a relation symbol $\syn{R}\colon\sqcap_{i<\alpha}\syn{s}_i$ in $\Sigma$ and raw terms $\ttau_i\ofsort\syn{s}_i$ $(i<\alpha)$, then the expression $\syn{R}(\ttau_i)_{i<\alpha}$ is a raw $\kappa$-ary Horn formula;
            \item
                Given two raw terms $\ttau$ and $\ttau'$ of the same sort $\syn{s}$,
            then $\ttau=\ttau'$ is a raw $\kappa$-ary Horn formula;
            \item
                The truth constant $\syn{\top}$ is a raw $\kappa$-ary Horn formula;
            \item
                Given raw $\kappa$-ary Horn formulas $\syn{\phi}_i$ $(i<\alpha)$ with $\alpha<\kappa$, the expression $\bigwedge_{i<\alpha}\syn{\phi}_i$ is a raw $\kappa$-ary Horn formula.
            When $\alpha=0$, $\bigwedge_{i<\alpha}\syn{\phi}_i$ expresses $\syn{\top}$.
        \end{itemize}
        \item
        A \emph{($\kappa$-ary) context} is a tuple $\tup{\vx}=(\syn{x}_i\ofsort\syn{s}_i)_{i<\alpha}$ of distinct variables with $\alpha<\kappa$.
        The empty context is denoted by the symbol $()$.
        \item
        A \emph{($\kappa$-ary) term} (over $\Sigma$) is a pair of a $\kappa$-ary context $\tup{\vx}$ and a raw term $\ttau$ (over $\Sigma$), written as $\tup{\vx}.\ttau$, where all variables appearing in $\ttau$ occur in $\tup{\vx}$.
        The \emph{sort} of $\tup{\vx}.\ttau$ is defined as that of $\ttau$.
        \item
        A \emph{($\kappa$-ary) Horn formula} (over $\Sigma$) is a pair of a $\kappa$-ary context $\tup{\vx}$ and a raw $\kappa$-ary Horn formula $\syn{\phi}$ (over $\Sigma$), written as $\tup{\vx}.\syn{\phi}$, where all variables appearing in $\syn{\phi}$ occur in $\tup{\vx}$.
        \item
        A \emph{($\kappa$-ary) Horn sequent} (over $\Sigma$) is a pair of two $\kappa$-ary Horn formulas $\tup{\vx}.\syn{\phi}$ and $\tup{\vx}.\syn{\psi}$ (over $\Sigma$) with the same context, written as
        \begin{equation*}
            \syn{\phi} \seq{\tup{\vx}} \syn{\psi}.
        \end{equation*}
        \item
        A \emph{$\kappa$-ary \acf{PHT}} $\theory{T}$ (over $\Sigma$) is a set of $\kappa$-ary Horn sequents (over $\Sigma$).\qedhere
    \end{enumerate}
\end{definition}

\begin{remark}
    Our terminologies raw terms and raw Horn formulas are simply called \textit{terms} and \textit{Horn formulas} in usual.
    Our terminologies terms and Horn formulas are ordinarily called \textit{terms-in-context} and \textit{Horn formulas-in-context} in the literature.
\end{remark}

\begin{remark}[The set of free variables]
    Let $\Sigma$ be a $\kappa$-ary $S$-sorted signature.
    For a raw term $\ttau$ over $\Sigma$, we write $\fv(\ttau)$ for the set of all free variables appears in $\ttau$.
    We also use a similar notation $\fv(\fphi)$ for a raw $\kappa$-ary Horn formula $\fphi$.
\end{remark}

\begin{remark}
    Note that we do not consider the equal sign ``$=$'' to be a relation symbol.
    We informally use the abbreviation $\syn{\phi}\biseq{\tup{\vx}}\syn{\psi}$ for ``$(\syn{\phi}\seq{\tup{\vx}}\syn{\psi})$ and $(\syn{\psi}\seq{\tup{\vx}}\syn{\phi})$,'' and $\ttau\defined$ for $\ttau=\ttau$.
\end{remark}

\begin{notation}[Substitution]
    If we are given a term $\tup{\vx}.\ttau\ofsort\syn{s}$ with $\tup{\vx_i}=(\vx\ofsort\syn{s}_i)_{i<\alpha}$ and a family of terms $\tup{\vy}.\tsigma_i\ofsort\syn{s}_i$ $(i<\alpha)$ with a common context $\tup{\vy}$, simultaneous substitution of $(\tsigma_i)_i$ into $(\vx_i)_i$ yields a new term of sort $\syn{s}$ in the context $\tup{\vy}$, which is denoted by $\tup{\vy}.\ttau\subst{\tsigma_i}{\vx_i}_{i<\alpha}\ofsort\syn{s}$.
\end{notation}

\subsection{The semantics of partial Horn logic}\label[appendix]{subsec:The_semantics_of_partial_Horn_logic}
\begin{definition}
    Let $\Sigma$ be an $S$-sorted $\kappa$-ary signature.
    A \emph{partial $\Sigma$-structure} $M$ consists of:
    \begin{itemize}
        \item
            a set $M_\syn{s}$ for each sort $\syn{s}\in S$,
        \item
            a partial map
        \[
        \intpn{\syn{f}}{M}~(\text{or}~\intpn{\tup{\vx}.\syn{f}(\tup{\vx})}{M})\colon \prod_{i<\alpha}M_{\syn{s}_i}\pto M_\syn{s}
        \]
        for each function symbol $\syn{f}\colon \sqcap_{i<\alpha}\syn{s}_i\to\syn{s}$ in $\Sigma$,
        \item
            a subset $\intpn{\syn{R}}{M}$ (or $\intpn{\tup{\vx}.\syn{R}(\tup{\vx})}{M}$) $\subseteq \prod_{i<\alpha}M_{\syn{s}_i}$ for each relation symbol $\syn{R}\colon \sqcap_{i<\alpha}\syn{s}_i$ in $\Sigma$.\qedhere
    \end{itemize}
\end{definition}

We can extend the above definitions of $\intpn{\tup{\vx}.\syn{f}(\tup{\vx})}{M}$ and $\intpn{\tup{\vx}.\syn{R}(\tup{\vx})}{M}$ to arbitrary terms and Horn formulas:
\begin{definition}[Partial $\Sigma$-structures]
    Let $\Sigma$ be an $S$-sorted $\kappa$-ary signature, and let $M$ be a partial $\Sigma$-structure.
    Fix a context $\tup{\vx}=(\syn{x}_i\ofsort\syn{s}_i)_{i<\alpha}$.
    \begin{enumerate}
        \item
        For an arbitrary $\kappa$-ary term $\tup{\vx}.\ttau\ofsort\syn{s}$ over $\Sigma$, we define a partial map
        \begin{equation*}
            \intpn{\tup{\vx}.\ttau}{M}\colon \prod_{i<\alpha}M_{\syn{s}_i}\pto M_\syn{s}
        \end{equation*}
        as follows:
        \begin{itemize}
            \item
            For each $i<\alpha$, $\intpn{\tup{\vx}.\syn{x}_i}{M}\colon\prod_{i<\alpha}M_{\syn{s}_i}\to M_{\syn{s}_i}$ is the $i$-th projection;
            \item
            For a function symbol $\syn{f}\colon\sqcap_{j<\beta}\syn{s}_j\to\syn{s}$ in $\Sigma$ and terms $\tup{\vx}.\ttau_j\ofsort\syn{s}_j$, $\intpn{\tup{\vx}.\syn{f}(\ttau_j)_{j<\beta}}{M}(\tup{m})$ is defined if and only if all $\intpn{\tup{\vx}.\ttau_j}{M}(\tup{m})$ are defined and $\intpn{\syn{f}}{M}(\intpn{\tup{\vx}.\ttau_j}{M}(\tup{m}))_{j<\beta}$ is also defined,
            and then $\intpn{\tup{\vx}.\syn{f}(\ttau_j)_{j<\beta}}{M}(\tup{m})\coloneq\intpn{\syn{f}}{M}(\intpn{\tup{\vx}.\ttau_j}{M}(\tup{m}))_{j<\beta}.$ 
        \end{itemize}
        \item
        For an arbitrary $\kappa$-ary Horn formula $\tup{\vx}.\syn{\phi}$ over $\Sigma$, we define a subset
        \begin{equation*}
            \intpn{\tup{\vx}.\syn{\phi}}{M}\subseteq\prod_{i<\alpha}M_{\syn{s}_i}
        \end{equation*}
        as follows:
        \begin{itemize}
            \item
            For a relation symbol $\syn{R}\colon\sqcap_{j<\beta}\syn{s}_j$ in $\Sigma$ and terms $\tup{\vx}.\ttau_j\ofsort\syn{s}_j$,
            $\tup{m}$ belongs to $\intpn{\tup{\vx}.\syn{R}(\ttau_j)_{j<\beta}}{M}$ if and only if all $\intpn{\tup{\vx}.\ttau_j}{M}(\tup{m})$ are defined and $(\intpn{\tup{\vx}.\ttau_j}{M}(\tup{m}))_{j<\beta}$ belongs to $\intpn{\syn{R}}{M}$;
            \item
            For two terms $\tup{\vx}.\ttau$ and $\tup{\vx}.\ttau'$ of the same sort,
            $\tup{m}$ belongs to $\intpn{\tup{\vx}.\ttau=\ttau'}{M}$ if and only if both $\intpn{\tup{\vx}.\ttau}{M}(\tup{m})$ and $\intpn{\tup{\vx}.\ttau'}{M}(\tup{m})$ are defined and equal to each other;
            \item
            $\intpn{\tup{\vx}.\syn{\top}}{M}\coloneq\prod_{i<\alpha} M_{\syn{s}_i}$;
            \item
            For Horn formulas $(\tup{\vx}.\syn{\phi}_j)_{j<\beta}$,
            $\intpn{\tup{\vx}.\bigwedge_{j<\beta}\syn{\phi}_j}{M}\coloneq\bigcap_{j<\beta}\intpn{\tup{\vx}.\syn{\phi}_j}{M}$.\qedhere
        \end{itemize}
    \end{enumerate}
\end{definition}\pagebreak

\begin{definition}[Interpretations]
    Let $\Sigma$ be an $S$-sorted $\kappa$-ary signature, and let $M$ be a partial $\Sigma$-structure.
    \begin{enumerate}
        \item
            We say that a $\kappa$-ary Horn sequent $\syn{\phi}\seq{\tup{\vx}}\syn{\psi}$ over $\Sigma$ is \emph{valid} in $M$ and write
            \[
                M \vDash (\syn{\phi}\seq{\tup{\vx}}\syn{\psi})
            \]
            if $\intpn{\tup{\vx}.\syn{\phi}}{M}\subseteq\intpn{\tup{\vx}.\syn{\psi}}{M}$.
        \item
            $M$ is called a \emph{(partial) $\theory{T}$-model} for a $\kappa$-ary \ac{PHT} $\theory{T}$ over $\Sigma$ if all Horn sequents in $\theory{T}$ are valid in $M$.\qedhere
    \end{enumerate}
\end{definition}

\begin{definition}[Validity of a sequent]
    Let $\Sigma$ be an $S$-sorted $\kappa$-ary signature, and let $\theory{T}$ be a $\kappa$-ary \ac{PHT} over $\Sigma$.
    We say that a $\kappa$-ary Horn sequent $\syn{\phi}\seq{\tup{\vx}}\syn{\psi}$ over $\Sigma$ is a \emph{\ac{PHL}-theorem} of $\theory{T}$, or \emph{derivable} from $\theory{T}$, and write
    \[
        \theory{T} \vDash (\syn{\phi}\seq{\tup{\vx}}\syn{\psi})
    \]
    if it is valid in every partial $\theory{T}$-model.
\end{definition}

\begin{remark}
    An alternative definition of \ac{PHL}-theorem can also be given syntactically using axioms and inference rules.
    The above definition is equivalent to this by the completeness theorem.
    See \cite{PalmgrenVickers2007partial,Kawase2026relativized} for details.
\end{remark}

\begin{definition}[Morphisms of partial $\Sigma$-structures]
    Let $\Sigma$ be an $S$-sorted $\kappa$-ary signature.
    A \emph{$\Sigma$-homomorphism} $h\colon M\to N$ between partial $\Sigma$-structures is a family of a total map $h_\syn{s}\colon M_\syn{s}\to N_\syn{s}$ indexed by sorts $\syn{s}\in S$ such that for each function symbol $\syn{f}\colon \sqcap_{i<\alpha}\syn{s}_i\to\syn{s}$ in $\Sigma$ and relation symbol $\syn{R}\colon \sqcap_{j<\beta}\syn{s}_j$ in $\Sigma$, there exist (necessarily unique) total maps (denoted by dashed arrows) making the following diagrams commute:
    \begin{equation*}
        \begin{tikzcd}[large]
            \prod_{i<\alpha}M_{\syn{s}_i}\arrow[d,"\prod_{i<\alpha}h_{\syn{s}_i}"'] &[-10pt] \mathrm{Dom}(\intpn{\syn{f}}{M})\arrow[d,"\exists"',dashed]\arrow[l,hook']\arrow[r,"\intpn{\syn{f}}{M}"] &[20pt] M_\syn{s}\arrow[d,"h_\syn{s}"] \\
            \prod_{i<\alpha}N_{\syn{s}_i} & \mathrm{Dom}(\intpn{\syn{f}}{N})\arrow[l,hook']\arrow[r,"\intpn{\syn{f}}{N}"'] & N_\syn{s}
        \end{tikzcd}
        \quad
        \begin{tikzcd}[large]
            \prod_{j<\beta}M_{\syn{s}_j}\arrow[d,"{\prod_{j<\beta}h_{\syn{s}_j}}"'] &[-10pt] \intpn{\syn{R}}{M}\arrow[d,"\exists"',dashed]\arrow[l,hook'] \\
            \prod_{j<\beta}N_{\syn{s}_j} & {\intpn{\syn{R}}{N}} \arrow[l,hook']
        \end{tikzcd}
        \incat{\Set}.
    \end{equation*}
\end{definition}

\begin{notation}
    Let $\theory{T}$ be a $\kappa$-ary \ac{PHT} over an $S$-sorted $\kappa$-ary signature $\Sigma$.
    We will denote by $\PStr\Sigma$ the category of partial $\Sigma$-structures and $\Sigma$-homomorphisms and by $\PMod\theory{T}$ the full subcategory of $\PStr\Sigma$ consisting of all partial $\theory{T}$-models.
\end{notation}

\begin{theorem}[Characterization theorem]
    For a category $\C$, the following are equivalent:
    \begin{enumerate}
        \item
            $\C$ is locally $\kappa$-presentable.
        \item
            There exists a $\kappa$-ary \ac{PHT} $\theory{T}$ such that $\C\simeq\PMod\theory{T}$.
    \end{enumerate}
\end{theorem}

\begin{theorem}[Limits and colimits of $\theory{T}$-models]
    Let $\theory{T}$ be a $\kappa$-ary \ac{PHT} over an $S$-sorted $\kappa$-ary signature $\Sigma$.
    Then, the forgetful functor $\PMod\theory{T}\arr \Set^S$, defined by $M\mapsto (M_\syn{s})_{\syn{s}\in S}$, creates any (small) limits and any $\kappa$-filtered (small) colimits.
\end{theorem}

\begin{remark}[Replacement of arity]
    If we are given small regular cardinals $\kappa\le\mu<\inacc$, then $\kappa$-ary signatures, $\kappa$-ary contexts, $\kappa$-ary Horn formulas, $\kappa$-ary partial Horn theories, and so on can also be regarded as $\mu$-ary ones.
    Thus, for a partial model $M$ of a $\kappa$-ary \ac{PHT} $\theory{T}$, we can ask whether a $\mu$-ary Horn sequent is valid in $M$.
    We can likewise ask whether a $\mu$-ary Horn sequent is a \ac{PHL}-theorem of $\theory{T}$.
\end{remark}

\subsection{The representing models}\label[appendix]{subsec:the_representing_models}
\newcommand{\equivclass}[1]{\underline{#1}}
\begin{construction}[The representing models]
    Let $\kappa\le\mu$ be small regular cardinals.
    Let $\theory{T}$ be a $\kappa$-ary \ac{PHT} over an $S$-sorted $\kappa$-ary signature $\Sigma$.
    Let $\tup{\vx}.\fphi$ be a $\mu$-ary Horn formula over $\Sigma$.
    \begin{enumerate}
        \item
            A $\mu$-ary term $\tup{\vx}.\ttau$ over $\Sigma$ is called a \emph{$\theory{T}$-term generated by $\tup{\vx}.\fphi$} if $\fphi\seq{\tup{\vx}}\ttau = \ttau$ is a \ac{PHL}-theorem of $\theory{T}$.
            We write $\bT\-\Term(\tup{\vx}.\fphi)$ for the $S$-sorted set of all $\mu$-ary $\theory{T}$-terms generated by $\tup{\vx}.\fphi$.
        \item
            Define an equivalence relation $\sim_\theory{T}$ on $\theory{T}\-\Term(\tup{\vx}.\fphi)$ as follows:
            $\tup{\vx}.\ttau\sim_\theory{T} \tup{\vx}.\ttau'$ if and only if $\fphi\seq{\tup{\vx}}\ttau =\ttau'$ is a \ac{PHL}-theorem of $\theory{T}$.
            We write $\equivclass{\tup{\vx}.\ttau}_\theory{T}$ for the equivalence class of $\tup{\vx}.\ttau$ under $\sim_\theory{T}$.
        \item
            The quotient $S$-sorted set $\theory{T}\-\Term(\tup{\vx}.\fphi)/{\sim_\theory{T}}$ yields a (partial) $\theory{T}$-model, which is called the \emph{representing $\theory{T}$-model for $\tup{\vx}.\fphi$} and is denoted by $\repn{\tup{\vx}.\fphi}_\theory{T}$.
            Its $\Sigma$-structure is uniquely determined as follows:
            \begin{itemize}
                \newcommand{\repmodel}{\repn{\tup{\vx}.\fphi}_\theory{T}}
                \item
                    For a term $\tup{\vy}.\tsigma$ with $\tup{\vy}=(\vy_j)_{j<\beta}$ over $\Sigma$, the partial map $\intpn{\tup{\vy}.\tsigma}{\repmodel}$ sends a tuple $( \equivclass{\tup{\vx}.\ttau_j}_\theory{T} )_{j<\beta}$ to $\equivclass{\tup{\vx}.\tsigma\subst{\ttau_j}{\vy_j}_{j<\beta}}_\theory{T}$ if and only if
                    \[
                        \theory{T}
                        \vDash
                        \left(
                            \fphi
                                \seq{\tup{\vx}}
                                    \tsigma\subst{\ttau_j}{\vy_j}_{j<\beta} = \tsigma\subst{\ttau_j}{\vy_j}_{j<\beta}
                        \right).
                    \]
                \item
                    For a Horn formula $\tup{\vy}.\fpsi$ with $\tup{\vy}=(\vy_j)_{j<\beta}$ over $\Sigma$, a tuple $(\equivclass{\tup{\vx}.\ttau_j}_\theory{T})_{j<\beta}$ belongs to $\intpn{\tup{\vy}.\fpsi}{\repmodel}$ if and only if
                    \[
                        \theory{T}
                        \vDash
                        \left(
                            \fphi
                                \seq{\tup{\vx}}
                                    \fpsi\subst{\ttau_j}{\vy_j}_{j<\beta}
                        \right).
                    \]\qedhere
            \end{itemize}
    \end{enumerate}
\end{construction}

\begin{theorem}
    Let $\theory{T}$ be a $\kappa$-ary \ac{PHT}.
    For a $\mu(\ge\kappa)$-ary Horn formula $\tup{\vx}.\fphi$, the $\theory{T}$-model $\repn{\tup{\vx}.\fphi}_\theory{T}$ represents the interpretation functor $\intpn{\tup{\vx}.\fphi}{\bullet}\colon\PMod\theory{T}\ni M\mapsto\intpn{\tup{\vx}.\fphi}{M}\in\Set$,
    i.e., for every $\bT$-model $M$, we have the following natural isomorphism:
    \begin{equation*}
        \intpn{\tup{\vx}.\fphi}{M} \cong \PMod\theory{T}(\repn{\tup{\vx}.\fphi}_\theory{T},M).
    \end{equation*}
\end{theorem}

From the construction of the representing models, the theorem above leads to the following corollary:
\begin{corollary}\label{cor:morphism_between_repn}
    Let $\theory{T}$ be a $\kappa$-ary \ac{PHT} over an $S$-sorted signature.
    Let $\tup{\vx}.\fphi$ and $\tup{\vy}.\fpsi$ be $\mu(\ge\kappa)$-ary Horn formulas over $\Sigma$ with $\tup{\vx}=(\vx_i)_{i<\alpha}$.
    Then, the following data bijectively correspond to each other:
    \begin{enumerate}
        \item
            A $\Sigma$-homomorphism $h\colon\repn{\tup{\vx}.\fphi}_\theory{T}\arr \repn{\tup{\vy}.\fpsi}_\theory{T}$;
        \item
            A family of equivalence classes $(\equivclass{\tup{\vy}.\ttau_i}_\theory{T})_{i<\alpha}$ of $\theory{T}$-terms generated by $\tup{\vy}.\fpsi$ such that
            \[
                \theory{T}
                \vDash
                \left(
                    \fpsi
                        \seq{\tup{\vy}}
                            \fphi\subst{\ttau_i}{\vx_i}_{i<\alpha}
                \right).
            \]
    \end{enumerate}
\end{corollary}

\begin{notation}[Morphisms between representing models]
    We write
    \[
        \repn{\tup{\tau}}_\bT\colon \repn{\tup{x}.\phi}_\bT\to\repn{\tup{y}.\psi}_\bT
    \]
    for the morphism corresponding to $\theory{T}$-terms $(\tup{\vy}.\ttau_i)_{i<\alpha}$ by \cref{cor:morphism_between_repn}.
\end{notation}

\begin{theorem}
    Let $\kappa\le\mu$ be small regular cardinals.
    Let $\theory{T}$ be a $\kappa$-ary \ac{PHT}.
    Then, for every $M\in\PMod\theory{T}$, the following are equivalent:
    \begin{enumerate}
        \item
            $M$ is $\mu$-presentable in $\PMod\theory{T}$.
        \item
            There exists a $\mu$-ary Horn formula $\tup{\vx}.\fphi$ such that $M\cong\repn{\tup{\vx}.\fphi}_\theory{T}$ in $\PMod\theory{T}$.
    \end{enumerate}
\end{theorem}

\subsection{Reduction of partial terms}\label[appendix]{subsec:reduction_of_partialterm}
We define a preorder $\tsigma \termle{} \ttau$ on the set of terms, which informally means that if $\tsigma$ is defined, then $\ttau$ is defined and equals to $\tsigma$.
\begin{definition}
    Let $\theory{T}=(S,\Sigma,\theory{T})$ be a $\kappa$-ary \ac{PHT}.
    Let $\tup{\vx}.\tsigma,\tup{\vx}.\ttau$ be terms over $\Sigma$ in the same context.
    We write $\tsigma\termle{\tup{\vx}}[\theory{T}]\ttau$ if the following Horn sequent is derivable from $\theory{T}$:
    \begin{equation*}
        \tsigma\defined
            \seq{\tup{\vx}}
                \tsigma=\ttau
    \end{equation*}
    We write $\tsigma\termeq{\tup{\vx}}[\theory{T}]\ttau$ if $\tsigma\termle{\tup{\vx}}[\theory{T}]\ttau$ and $\ttau\termle{\tup{\vx}}[\theory{T}]\tsigma$ hold.
    It is easy to see that $\termle{\tup{\vx}}[\theory{T}]$ becomes a preorder, and $\termeq{\tup{\vx}}[\theory{T}]$ becomes an equivalence relation.
\end{definition}

\begin{lemma}
    Let $\theory{T}=(S,\Sigma,\theory{T})$ be a $\kappa$-ary \ac{PHT}.
    Let $\tup{\vx}\subseteq\tup{\vy}$ be contexts, and let $\tup{\vx}.\tsigma,\tup{\vx}.\ttau$ be terms.
    Then,
    \begin{equation*}
        \tsigma \termle{\tup{\vx}}[\theory{T}] \ttau
        \implies
        \tsigma \termle{\tup{\vy}}[\theory{T}] \ttau
    \end{equation*}
\end{lemma}

\begin{lemma}\label{lem:substitution_into_reduction}
    Let $\theory{T}=(S,\Sigma,\theory{T})$ be a $\kappa$-ary \ac{PHT}.
    Let $\tup{\vx}.\tsigma,\tup{\vx}.\ttau$ be terms.
    Let $\tup{\vy}.\tup{\syn{\rho}}=(\tup{\vy}.\syn{\rho}_i)_i$ be terms, sort compatible with the context $\tup{\vx}$.
    Then,
    \begin{equation*}
        \tsigma\termle{\tup{\vx}}[\theory{T}] \ttau,
        \quad
        \fv(\tsigma)\supseteq\fv(\ttau)
        \implies
        \tsigma\subst{\tup{\syn{\rho}}}{\tup{\vx}}\termle{\tup{\vy}}[\theory{T}] \ttau\subst{\tup{\syn{\rho}}}{\tup{\vx}}.
    \end{equation*}
\end{lemma}

\begin{lemma}
    Let $\theory{T}=(S,\Sigma,\theory{T})$ be a $\kappa$-ary \ac{PHT}.
    Let $\tup{\tsigma}=(\tsigma_i)_i,\tup{\ttau}=(\ttau_i)_i$ be raw terms, sort compatible with the context $\tup{\vx}=(\vx_i)_i$, and whose free variables are among $\tup{\vy}$.
    If $\tsigma_i\termle{\tup{\vy}}[\theory{T}] \ttau_i$ for all $i$, the following hold:
    \begin{enumerate}
        \item
            For any term $\tup{\vx}.\syn{\rho}$, we have
            $
            \syn{\rho}\subst{\tup{\tsigma}}{\tup{\vx}}
                \termle{\tup{\vy}}[\theory{T}]
                    \syn{\rho}\subst{\tup{\ttau}}{\tup{\vx}}
            $.
        \item
            For any Horn formula $\tup{\vx}.\syn{\phi}$,
            $
            \syn{\phi}\subst{\tup{\tsigma}}{\tup{\vx}}
                \seq{\tup{\vy}}
                    \syn{\phi}\subst{\tup{\ttau}}{\tup{\vx}}
            $
            is derivable from $\theory{T}$.
    \end{enumerate}
\end{lemma}
\section{Generalized algebraic theories}\label[appendix]{sec:GAT}
For the convenience of the reader and the clarity of notation, we recall the syntax of \acfp{GAT}.
The original definition is given in \cite{Cartmell1986generalised}, but we follow more recent presentations, \cite{Garner2015combinatorial}, with some modifications.
\subsection{Syntax of generalized algebraic theories}
In the following, we fix a small set $\Var$ of variables. 
\begin{definition}
    \quad
    \begin{enumerate}
        \item The small set $S^*$ of \emph{expressions} over a small set $S$ is defined inductively as follows:
        \begin{itemize}
            \item For each $\syn{x} \in \Var$, $\syn{x} \in S^*$.
            \item For each $n \in \bN$, $\syn{s} \in S$, and $\syn{e}_1, \ldots, \syn{e}_n \in S^*$, we have $\syn{s}(\syn{e}_1, \ldots, \syn{e}_n) \in S^*$.
        \end{itemize}
        \item The small set of \emph{free variables} of an expression $\syn{e} \in S^*$, denoted by $\fv(\syn{e})$, is defined inductively as follows:
        \begin{itemize}
            \item $\fv(\syn{x}) \coloneq \{\syn{x}\}$.
            \item $\fv(\syn{s}(\syn{e}_1, \ldots, \syn{e}_n)) \coloneq \fv(\syn{e}_1) \cup \cdots \cup \fv(\syn{e}_n)$.
        \end{itemize}
        \item For an expression $\syn{e} \in S^*$, its free variable $\syn{x} \in \fv(\syn{e})$, and another expression $\syn{t} \in S^*$, we define the \emph{substitution} of $\syn{t}$ for $\syn{x}$ in $\syn{e}$, denoted by $\syn{e}[\syn{t}/\syn{x}]$, as follows:
        \begin{itemize}
            \item If $\syn{e} = \syn{x}$, then $\syn{e}[\syn{t}/\syn{x}] \colonequiv \syn{t}$.
            \item If $\syn{e} = \syn{y}$ for $\syn{y} \neq \syn{x}$, then $\syn{e}[\syn{t}/\syn{x}] \colonequiv \syn{y}$.
            \item If $\syn{e} = \syn{s}(\syn{e}_1, \ldots, \syn{e}_n)$, then $\syn{e}[\syn{t}/\syn{x}] \colonequiv \syn{s}(\syn{e}_1[\syn{t}/\syn{x}], \ldots, \syn{e}_n[\syn{t}/\syn{x}])$.\qedhere
        \end{itemize}
    \end{enumerate}
\end{definition}
\begin{definition}
    \quad
    \begin{enumerate}
        \item A \emph{signature} $\Sigma$ is a pair $(\Sigma\sort,  \Sigma\oper)$ of small sets.
        Elements of $\Sigma\sort$ are called \emph{sort symbols}, and elements of $\Sigma\oper$ are called \emph{operator symbols}.
        We mean by an \emph{expression over a signature $\Sigma$} an expression $\syn{e} \in (\Sigma\sort \sqcup \Sigma\oper)^*$.
        \item A \emph{precontext} over a signature $\Sigma$ is a finite sequence of pairs $(\syn{x}, \syn{e})$, where $\syn{x} \in \Var$ is a variable and $\syn{e}$ is an expression over $\Sigma$, such that the first components of all pairs (i.e., the variables) are distinct.
        Precontexts are denoted by $\syn{\Gamma} = (\syn{x}_1 : \syn{e}_1, \ldots, \syn{x}_n : \syn{e}_n)$, for example.
        In particular, the empty sequence is a precontext, which is denoted by $()$.
        If clear, we drop the parentheses and write $\syn{x}_1 : \syn{e}_1, \ldots, \syn{x}_n : \syn{e}_n$ for $\syn{\Gamma}$.
        The set of all variables appearing as the first components of pairs in a precontext $\syn{\Gamma}$ is denoted by $\fv(\syn{\Gamma})$.
        \item A \emph{prejudgment} over a signature $\Sigma$ is a tuple of one of the following forms:
        \begin{itemize}
            \item (\textbf{Context prejudgment}) A precontext $\syn{\Gamma}$. To specify it as a prejudgment, we write $\vdash \syn{\Gamma} \ctx  $ for it;
            \item (\textbf{Type prejudgment}) A pair $(\syn{\Gamma}, \syn{T})$ for a precontext $\syn{\Gamma}$ and an expression $\syn{T}$ over $\Sigma$, which will be denoted by $\syn{\Gamma} \vdash \syn{T} \type$;
            \item (\textbf{Term prejudgment}) A triple $(\syn{\Gamma}, \syn{t}, \syn{T})$ for a precontext $\syn{\Gamma}$, and expressions $t$ and $\syn{T}$ over $\Sigma$, which will be denoted by $\syn{\Gamma} \vdash \syn{t} : \syn{T}$;
            \item (\textbf{Type equality prejudgment}) A triple $(\syn{\Gamma}, \syn{T}_1, \syn{T}_2)$ for a precontext $\syn{\Gamma}$, and expressions $\syn{T}_1$ and $\syn{T}_2$ over $\Sigma$, which will be denoted by $\syn{\Gamma} \vdash \syn{T}_1 = \syn{T}_2 \type$;
            \item (\textbf{Term equality prejudgment}) A quadruple $(\syn{\Gamma}, \syn{t}_1, \syn{t}_2, \syn{T})$ for a precontext $\syn{\Gamma}$, and expressions $\syn{t}_1$, $\syn{t}_2$, and $\syn{T}$ over $\Sigma$, which will be denoted by $\syn{\Gamma} \vdash \syn{t}_1 = \syn{t}_2 : \syn{T}$.
        \end{itemize}
        The empty precontext in a prejudgment is often omitted.
        By abuse of terminology, we also call the first component $\syn{\Gamma}$ of a prejudgment the \emph{context} of the prejudgment.
        \item A \emph{(generalized algebraic) pretheory} $\theory{T}$ over a signature $\Sigma$ consists of the following data:
        \begin{itemize}
            \item
                For each sort symbol $\syn{A} \in \Sigma\sort$, a type prejudgment 
                \begin{equation}\label{eq:sortsymbol}
                    \syn{x}_1:\syn{T}_1,\dots, \syn{x}_n:\syn{T}_n \vdash \syn{A}(\syn{x}_1,\dots,\syn{x}_n) \type.
                \end{equation}
                The context in this prejudgment is denoted by $\syn{\Gamma}_\syn{A}$.
            \item
                For each operator symbol $\syn{f} \in \Sigma\oper$, a term prejudgment
                \begin{equation}\label{eq:operatorsymbol}
                    \syn{x}_1:\syn{T}_1,\dots, \syn{x}_n:\syn{T}_n \vdash \syn{f}(\syn{x}_1,\dots,\syn{x}_n) : \syn{T}.
                \end{equation}
                The context $\syn{x}_1:\syn{T}_1,\dots,\syn{x}_n:\syn{T}_n$ in this prejudgment and the expression $\syn{T}$ are denoted by $\syn{\Gamma}_\syn{f}$ and $\syn{T}_\syn{f}$, respectively.
            \item
                A small set of term equality prejudgments over the signature $\Sigma$, whose elements are called \emph{axioms of $\theory{T}$}.
        \end{itemize}
        We write $\syn{J}\in \theory{T}$ for a prejudgment $\syn{J}$ if it is one of the forms \Cref{eq:sortsymbol,eq:operatorsymbol} for some $\syn{A}\in \Sigma\sort$ or $\syn{f}\in \Sigma\oper$, or $\syn{J}$ is an axiom of $\theory{T}$.\qedhere
    \end{enumerate}
\end{definition}

\begin{remark}
    There are some subtleties in our definition of \acp{GAT} compared to the standard one.
    The first thing is that we adopt contexts as one kind of judgments, similarly to the formulation in \cite{Pitts2000categorical}.
    The second thing is that we do not allow type equality judgments as axioms.
    This will certainly impair the expressive ability of the type theory.
    Nevertheless, most of our target examples are captured, and also the class of (the equivalence classes of) the categories of models does not diminish since one can always replace a type equality axiom with operator symbols that express the mutual inverses between the expectedly equal types and the axioms stating that those are indeed mutually inverses.
\end{remark}

\begin{definition}
    \label{def:derivation-rules-of-GAT}
    We define when a prejudgment $\syn{J}$ is \emph{derived from} a pretheory $\theory{T}$, for which we write $\theory{T}\drv \syn{J}$.
    This is defined inductively by the following derivation rules.
    \begin{mathparpagebreakable}
            \derule{rule:refl-ty}
            {\theory{T} \drv \syn{\Gamma} \vdash \syn{T} \type}
            {\theory{T} \drv \syn{\Gamma} \vdash \syn{T} = \syn{T} \type}
            
            \derule{rule:refl-tm}
                    {\theory{T} \drv \syn{\Gamma} \vdash \syn{t} : \syn{T}}
                    {\theory{T} \drv \syn{\Gamma} \vdash \syn{t} = \syn{t} : \syn{T}}
            
            \derule{rule:sym-ty}
                    {\theory{T} \drv \syn{\Gamma} \vdash \syn{T}_1 = \syn{T}_2 \type}
                    {\theory{T} \drv \syn{\Gamma} \vdash \syn{T}_2 = \syn{T}_1 \type}
            
            \derule{rule:sym-tm}
                    {\theory{T} \drv \syn{\Gamma} \vdash \syn{t}_1=\syn{t}_2 : \syn{T}}
                    {\theory{T} \drv \syn{\Gamma} \vdash \syn{t}_2=\syn{t}_1 : \syn{T}}
            
            \derule{rule:trans-ty}
                    {\theory{T} \drv \syn{\Gamma} \vdash \syn{T}_1=\syn{T}_2 \type \\ \theory{T} \drv \syn{\Gamma} \vdash \syn{T}_2=\syn{T}_3 \type}
                    {\theory{T} \drv \syn{\Gamma} \vdash \syn{T}_1=\syn{T}_3 \type}
            
            \derule{rule:trans-tm}
                    {\theory{T} \drv \syn{\Gamma} \vdash \syn{t}_1=\syn{t}_2 : \syn{T} \\ \theory{T} \drv \syn{\Gamma} \vdash \syn{t}_2=\syn{t}_3 : \syn{T}}
                    {\theory{T} \drv \syn{\Gamma} \vdash \syn{t}_1=\syn{t}_3 : \syn{T}}
            
            \derule{rule:conv-tm}
                    {\theory{T} \drv \syn{\Gamma} \vdash \syn{T}_1 = \syn{T}_2 \\ \theory{T} \drv \syn{\Gamma} \vdash \syn{t} : \syn{T}_1}
                    {\theory{T} \drv \syn{\Gamma} \vdash \syn{t} : \syn{T}_2}
            
            \derule{rule:conv-tm-eq}
                    {\theory{T} \drv \syn{\Gamma} \vdash \syn{T}_1 = \syn{T}_2 \\ \theory{T} \drv \syn{\Gamma} \vdash \syn{t} = \syn{t}' : \syn{T}_1}
                    {\theory{T} \drv \syn{\Gamma} \vdash \syn{t} = \syn{t}' : \syn{T}_2}
            
            \derule{rule:alpha-ctx}
                    {\theory{T} \drv \vdash \syn{\Gamma} \ctx \\ \sigma\colon \fv(\syn{\Gamma}) \arr(\cong) \fv(\syn{\Gamma})}
                    {\theory{T} \drv \vdash \sigma\cdot \syn{\Gamma} \ctx}
            
            \derule{rule:alpha}
                    {\theory{T} \drv \syn{\Gamma} \vdash \syn{K} \\ \sigma \colon \fv(\syn{\Gamma}) \arr(\cong) \fv(\syn{\Gamma})}
                    {\theory{T} \drv \sigma\cdot \syn{\Gamma} \vdash \sigma\cdot \syn{K}}
            
            \derule{rule:empty}
                    { }
                    {\theory{T} \drv \vdash () \ctx}
            
            \derule{rule:ctx-extend}
                    {\theory{T} \drv \syn{\Gamma} \vdash \syn{T} \type \\ \syn{x} \notin \fv(\syn{\Gamma})}
                    {\theory{T} \drv \vdash \syn{\Gamma}, \syn{x}: \syn{T} \ctx}
            
            \derule{rule:var}
                    {\theory{T} \drv \syn{\Gamma} \vdash \syn{T} \type \\ \syn{x} \notin \fv(\syn{\Gamma})}
                    {\theory{T} \drv \syn{\Gamma}, \syn{x} : \syn{T} \vdash \syn{x} : \syn{T}}
            
            \derule{rule:weak}
                    {\theory{T} \drv \syn{\Gamma} \vdash \syn{T} \type \\ \theory{T} \drv \syn{\Gamma}, \syn{\Delta} \vdash \syn{K} \\ \syn{x} \notin \fv(\syn{\Gamma})\cup \fv(\syn{\Delta})}
                    {\theory{T} \drv \syn{\Gamma}, \syn{x} : \syn{T}, \syn{\Delta} \vdash \syn{K}}
            
            \derule{rule:subst}
                    {\theory{T} \drv \syn{\Gamma} \vdash \syn{t} : \syn{T} \\ \theory{T} \drv \syn{\Gamma}, \syn{x} : \syn{T}, \syn{\Delta} \vdash \syn{K}}
                    {\theory{T} \drv \syn{\Gamma}, \syn{\Delta}[\syn{t}/\syn{x}] \vdash \syn{K}[\syn{t}/\syn{x}]}
            
            \derule{rule:eq-subst-ty}
                    {\theory{T} \drv \syn{\Gamma} \vdash \syn{t}_1 = \syn{t}_2 : \syn{T} \\ \theory{T} \drv \syn{\Gamma}, \syn{x} : \syn{T}, \syn{\Delta} \vdash \syn{T}'}
                    {\theory{T} \drv \syn{\Gamma}, \syn{\Delta}[\syn{t}_1/\syn{x}] \vdash \syn{T}'[\syn{t}_1/\syn{x}]=\syn{T}'[\syn{t}_2/\syn{x}]}
            
            \derule{rule:eq-subst-tm}
                    {\theory{T} \drv \syn{\Gamma} \vdash \syn{t}_1 = \syn{t}_2 : \syn{T} \\ \theory{T} \drv \syn{\Gamma}, \syn{x} : \syn{T}, \syn{\Delta} \vdash \syn{t}' : \syn{T}'}
                    {\theory{T} \drv \syn{\Gamma}, \syn{\Delta}[\syn{t}_1/\syn{x}] \vdash \syn{t}'[\syn{t}_1/\syn{x}]=\syn{t}'[\syn{t}_2/\syn{x}] : \syn{T}'}
            
            \derule{rule:sort}
                    {\syn{A} \in \Sigma\sort \\ \theory{T} \drv \vdash \syn{\Gamma}_\syn{A} \ctx}
                    {\theory{T} \drv \syn{\Gamma}_\syn{A} \vdash \syn{A}(\syn{x}_1, \ldots, \syn{x}_n) \type}
            
            \derule{rule:oper}
                    {\syn{f} \in \Sigma\oper \\ \theory{T} \drv \syn{\Gamma}_\syn{f} \vdash \syn{T}_\syn{f} \type}
                    {\theory{T} \drv \syn{\Gamma}_\syn{f} \vdash \syn{f}(\syn{x}_1, \ldots, \syn{x}_n) : \syn{T}_\syn{f}}
            
            \derule{rule:eq-axiom}
                    {\syn{\Gamma} \vdash \syn{t} = \syn{t}' : \syn{T} \in \theory{T} \\ \theory{T} \drv \syn{\Gamma} \vdash \syn{t} : \syn{T} \\ \theory{T} \drv \syn{\Gamma} \vdash \syn{t}' : \syn{T}}
                    {\theory{T} \drv \syn{\Gamma} \vdash \syn{t} = \syn{t}' : \syn{T}}
    \end{mathparpagebreakable}
    Here, $\syn{\Gamma}\vdash\syn{K}$ is one of the last for forms of prejudgments, that is, $\syn{K}$ is either $\syn{T} \type$, $\syn{t}:\syn{T}$, $\syn{T}_1=\syn{T}_2 \type$, or $\syn{t}_1=\syn{t}_2 : \syn{T}$.
    The substitution into $\syn{K}$ is defined entry-wise.
    $\sigma \cdot \syn{\Gamma}$ and $\sigma \cdot \syn{K}$ in the \cref{rule:alpha,rule:alpha-ctx} rule denotes the result of renaming (i.e., applying $\alpha$-conversion to) the variables in the prejudgment $\syn{K}$ according to the bijection $\sigma$.
    
    \cref{rule:refl-ty,rule:refl-tm,rule:sym-ty,rule:sym-tm,rule:trans-ty,rule:trans-tm} are the rules for equality being an equivalence relation, \cref{rule:conv-tm,rule:conv-tm-eq} are the type conversion rules, \cref{rule:alpha,rule:alpha-ctx} are the $\alpha$-conversion rules, \cref{rule:empty,rule:ctx-extend} are the rules for forming contexts, \cref{rule:var} is the variable rule, \cref{rule:weak} is the weakening rule, \cref{rule:subst,rule:eq-subst-ty,rule:eq-subst-tm} are the substitution rules, \cref{rule:sort,rule:oper} are the rules for sort and operator symbols, and \cref{rule:eq-axiom} is the axiom rule.
    \qedhere
\end{definition}

Note that, in general, we do not always have a prejudgment in $\theory{T}$ derived from $\theory{T}$, as the last three rules show.
\begin{definition}
    A \emph{\acf{GAT}} over a signature $\Sigma$ is a pretheory $\theory{T}$ such that all the prejudgments in $\theory{T}$ are derived from $\theory{T}$, i.e., for each $\syn{J} \in \theory{T}$, we have $\theory{T} \drv \syn{J}$.
    For a \ac{GAT} $\theory{T}$, we call a prejudgment derived from $\theory{T}$ a \emph{judgment of $\theory{T}$}.
    In particular, when we have a judgment $\vdash \syn{\Gamma} \ctx$ of $\theory{T}$, we say that $\syn{\Gamma}$ is a \emph{context of $\theory{T}$}.
    Similarly, when we have a judgment $\syn{\Gamma} \vdash \syn{T} \type$ (resp.\ $\syn{\Gamma} \vdash \syn{t} : \syn{T}$) of $\theory{T}$, we say that $\syn{T}$ is a \emph{type in the context $\syn{\Gamma}$ of $\theory{T}$} (resp.\ $\syn{t}$ is a \emph{term of type $\syn{T}$ in the context $\syn{\Gamma}$ of $\theory{T}$}).
\end{definition}

\begin{example}\label{eg:set-theory}
    The \ac{GAT} for (small) sets $\theory{T}_{\syn{Set}}$ consists of one sort symbol $\syn{El}$ and no operation symbols, with no axioms.
\end{example}

\begin{example}\label{eg:category-theory}
    The \ac{GAT} $\theory{T}_{\syn{Cat}}$ for (small) categories is given as follows.
    The signature $\Sigma_{\syn{Cat}}$ consists of two sort symbols $\syn{Ob}$ and $\syn{Mor}$, and two operator symbols $\syn{id}$ and $\circ$.
    \begin{equation*}
    \begin{aligned}
        &\vdash \syn{Ob}\ \type \\
        \syn{x} : \syn{Ob}, \syn{y} : \syn{Ob} & \vdash \syn{Mor}(\syn{x},\syn{y})\ \type \\
        \syn{x} : \syn{Ob} & \vdash \syn{id}(\syn{x}) : \syn{Mor}(\syn{x},\syn{x}) \\
        \syn{x} : \syn{Ob}, \syn{y} : \syn{Ob}, \syn{z} : \syn{Ob}, \syn{f} : \syn{Mor}(\syn{x},\syn{y}), \syn{g} : \syn{Mor}(\syn{y},\syn{z}) & \vdash \mathop{\circ}(\syn{x}, \syn{y}, \syn{z}, \syn{f},\syn{g}) : \syn{Mor}(\syn{x},\syn{z}) \\
        \syn{x} : \syn{Ob}, \syn{y} : \syn{Ob}, \syn{f} : \syn{Mor}(\syn{x},\syn{y}) & \vdash \mathop{\circ}(\syn{x}, \syn{y}, \syn{y}, \syn{f}, \syn{id}(\syn{y})) = \syn{f} : \syn{Mor}(\syn{x},\syn{y})\\
        \syn{x} : \syn{Ob}, \syn{y} : \syn{Ob}, \syn{f} : \syn{Mor}(\syn{x},\syn{y}) & \vdash \mathop{\circ}(\syn{x}, \syn{x}, \syn{y}, \syn{id}(\syn{x}), \syn{f}) = \syn{f} : \syn{Mor}(\syn{x},\syn{y})\\
        \begin{multlined}
        \syn{w} : \syn{Ob}, \syn{x} : \syn{Ob}, \syn{y} : \syn{Ob}, \syn{z} : \syn{Ob}, \\
        \syn{f} : \syn{Mor}(\syn{w},\syn{x}), \syn{g} : \syn{Mor}(\syn{x},\syn{y}), \syn{h} : \syn{Mor}(\syn{y},\syn{z})
        \end{multlined}
        & \vdash 
        \begin{multlined}
        \mathop{\circ}(\syn{w}, \syn{x}, \syn{z}, \syn{f}, \mathop{\circ}(\syn{x}, \syn{y}, \syn{z}, \syn{g}, \syn{h})) \\
        = \mathop{\circ}(\syn{w}, \syn{y}, \syn{z}, \mathop{\circ}(\syn{w}, \syn{x}, \syn{y}, \syn{f}, \syn{g}), \syn{h})
        \end{multlined}
        : \syn{Mor}(\syn{w},\syn{z})
    \end{aligned}
\end{equation*}
\end{example}

Morphisms between \acp{GAT} can be defined so that they form a category as defined in \cite{Cartmell1986generalised}, but here we confine ourselves to the case of full subtheories.
\begin{definition}
    A \emph{full sub-\ac{GAT}} $\theory{T}'$ of a \ac{GAT} $\theory{T}$ is a \ac{GAT} whose signature of $\theory{T}'$ consists of a subset of sort symbols and operation symbols of $\theory{T}$, and a prejudgment over the signature is derived from $\theory{T}'$ if and only if it is derived from $\theory{T}$ when we identify it as a prejudgment over the signature of $\theory{T}$ in the obvious way.
\end{definition}

\subsection{Syntactic categories of generalized algebraic theories}
\begin{definition}[{\cite[Section 14]{Cartmell1986generalised}}]
    \label{def:syntactic-category}
    The \emph{syntactic category} $\clfy{\theory{T}}$ of a \ac{GAT} $\theory{T}$ is the category defined as follows:
    \begin{itemize}
        \item Its objects are the contexts of $\theory{T}$.
        \item For contexts $\syn{\Gamma}$ and $\syn{\Delta} = (\syn{x}_1:\syn{T}_1, \ldots, \syn{x}_n:\syn{T}_n)$ of $\theory{T}$, a morphism from $\syn{\Gamma}$ to $\syn{\Delta}$ is given by an equivalence class of sequences of term judgments
        \begin{equation*}
            \begin{aligned}
            \syn{\Gamma} &\vdash \syn{t}_1 : \syn{T}_1,\\
            \syn{\Gamma} &\vdash \syn{t}_2 : \syn{T}_2[\syn{t}_1/\syn{x}_1],\\
            &\vdots\\
            \syn{\Gamma} &\vdash \syn{t}_n : \syn{T}_n[\syn{t}_1/\syn{x}_1, \ldots, \syn{t}_{n-1}/\syn{x}_{n-1}]
        \end{aligned}
        \end{equation*}
        of $\theory{T}$, where two such sequences $(\syn{t}_1, \ldots, \syn{t}_n)$ and $(\syn{t}'_1, \ldots, \syn{t}'_n)$ are considered equivalent if we have the following term equality judgment of $\theory{T}$:
        \begin{equation*}
            \syn{\Gamma} \vdash \syn{t}_i = \syn{t}'_i : \syn{T}_i[\syn{t}_1/\syn{x}_1, \ldots, \syn{t}_{i-1}/\syn{x}_{i-1}]
            \quad \text{for all } 1 \leq i \leq n.
        \end{equation*}
        Note that in the substitution of types above, we can safely replace $\syn{t}_j$ with $\syn{t}'_j$, due to the derivation rules for substitution and equality.
        We write the morphism represented by such a sequence $(\syn{t}_1, \ldots, \syn{t}_n)$ as $[\syn{t}_1, \ldots, \syn{t}_n]$.
        \item The composition and identities are defined in the obvious way using substitution.\qedhere
    \end{itemize}
\end{definition}

\begin{remark}
    In \cite{Cartmell1986generalised}, objects of the syntactic category are taken as $\alpha$-equivalence classes.
    However, the category defined here is equivalent to the original definition, primarily because $\alpha$-equivalent contexts are isomorphic in this category.
    Although the description of objects gets simpler in our definition, the drawback is that the category fails to be a contextual category in the sense of \cite[Section 14]{Cartmell1986generalised} in general.

    The originally defined syntactic category can also be obtained by taking the pushout of the inclusion of the wide subcategory consisting of ``$\alpha$-isomorphisms'' along the functor from this subcategory to the discrete category on the set of connected components of it.
    Since this subcategory is not just a groupoid but has contractible connected components, the functor into the discrete category is an equivalence, and the resulting category is equivalent to the one defined here because the inclusion is injective on objects, see \cite[Corollary 3]{JoyalStreet1993pullbacks}.
\end{remark}

\begin{proposition}
    \label{prop:syntactic-category-is-clan}
    The syntactic category $\clfy{\theory{T}}$ of a \ac{GAT} $\theory{T}$ has a clan structure by taking the set of display morphisms to consist of morphisms of the form
    \begin{equation}\label{eq:display-morphisms-in-syntactic-clan}
        \syn{\Delta}\arr(\cong)
        (\syn{x}_1 : \syn{T}_1, \ldots, \syn{x}_n : \syn{T}_n)
        \arr({[\syn{x}_1, \ldots, \syn{x}_k]})[][4]
        (\syn{x}_1 : \syn{T}_1, \ldots, \syn{x}_k : \syn{T}_k)
        \quad\text{with }0\le k\le n.
    \end{equation}
    We will denote this clan simply by $\clfy{\theory{T}}$, and call it the \emph{syntactic clan} of $\theory{T}$.
\end{proposition}
\begin{proof}
    This follows from the contextual-category structure of the originally defined syntactic category and the clan structure induced by it as in \cite[Proposition B.3]{Frey2025duality}, transferred along the equivalence between the two syntactic categories mentioned in the previous remark.

    We outline a direct proof here for completeness.
    First of all, the following square is a pullback in $\clfy{\theory{T}}$ for each morphism $[\syn{t}_1, \ldots, \syn{t}_k]$ and an extended context of its codomain:
    \begin{equation*}
        \begin{tikzpicture}
        \node (A) at (0,0) {
        $\left(
        \begin{tabular}{c}
        $\syn{y}_1 : \syn{S}_1, \ldots, \syn{y}_m : \syn{S}_m$,\\
        $\syn{x}_{k+1} : \syn{T}_{k+1}[\tup{\syn{t}}/\tup{\syn{x}}], \ldots, \syn{x}_n : \syn{T}_n[\tup{\syn{t}}/\tup{\syn{x}}]$
        \end{tabular}
        \right)$
        };
        \node (B) at (11,0) {
        $\left(
        \begin{tabular}{c}
        $\syn{x}_1 : \syn{T}_1, \ldots, \syn{x}_k : \syn{T}_k$,\\
        $\syn{x}_{k+1} : \syn{T}_{k+1}, \ldots, \syn{x}_n : \syn{T}_n$
        \end{tabular}
        \right)$
        };
        \node (C) at (0,-2) {$(\syn{y}_1 : \syn{S}_1, \ldots, \syn{y}_m : \syn{S}_m)$};
        \node (D) at (11,-2) {$(\syn{x}_1 : \syn{T}_1, \ldots, \syn{x}_k : \syn{T}_k)$}; 
        \draw[->] (A) to node[above] {%
        ${[\syn{t}_1, \ldots, \syn{t}_k, \syn{x}_{k+1}, \ldots, \syn{x}_n]}$%
        } (B);
        \draw[->] (A) to node[left] {${[\syn{y}_1, \ldots, \syn{y}_m]}$} (C);
        \draw[->] (B) to node[right] {${[\syn{x}_1, \ldots, \syn{x}_k]}$} (D);
        \draw[->] (C) to node[below] {${[\tup{\syn{t}}]\coloneq[\syn{t}_1, \ldots, \syn{t}_k]}$} (D);
        \draw[draw=none] (A) to node[very near start] {$\lrcorner$} (D);  
        \end{tikzpicture}
    \end{equation*}
    Here, $[\tup{\syn{t}}/\tup{\syn{x}}]$ denotes the substitution of the sequence of terms $\syn{t}_1, \ldots, \syn{t}_k$ for the sequence of variables $\syn{x}_1, \ldots, \syn{x}_k$.
    In particular, from the case where $[\tup{\syn{t}}]$ is an isomorphism, one can show that the given class of morphisms is closed under composition.
    The empty context is a terminal object since for each context $\syn{\Gamma}$, there exists a unique morphism from $\syn{\Gamma}$ to the empty context given by the empty sequence of terms, which belongs to the specified class of morphisms.
    The stability of the class of morphisms under pullback also follows from the above presentation of pullback squares.
\end{proof}

\begin{proposition}
    The syntactic clan $\clfy{\theory{T}}$ of a \ac{GAT} $\theory{T}$ is generated, in the sense of \cref{def:generated-clan-structure}, by the display morphisms of the form $\syn{\Gamma}_\syn{A}, \syn{y}:\syn{A}(\syn{x}_1,\dots,\syn{x}_n)\arr\syn{\Gamma}_\syn{A}$ for each sort symbol $\syn{A}\in \Sigma\sort$.
\end{proposition}
\begin{proof}
    Since these morphisms are clearly display morphisms in $\clfy{\theory{T}}$, we show that any display morphism in $\clfy{\theory{T}}$ necessarily belongs to the clan structure generated by them.
    As a morphism \cref{eq:display-morphisms-in-syntactic-clan} is a composite of an isomorphism and projections dropping only variables, it suffices to show for the case when $n=k+1$.
    This follows from the fact that any type judgment $\syn{\Gamma} \vdash \syn{T} \type$ in $\theory{T}$ is obtained by substitution of terms into some sort symbol $\syn{A}\in \Sigma\sort$, thus the morphism is a pullback of the display morphism $\syn{\Gamma}_\syn{A}, \syn{y}:\syn{A}(\syn{x}_1,\dots,\syn{x}_n)\arr\syn{\Gamma}_\syn{A}$ along the morphism representing the substitution.
\end{proof}

\begin{proposition}
    \label{prop:syntactic-clan-is-initial}
    A full sub-\ac{GAT} $\theory{T}'$ of a \ac{GAT} $\theory{T}$ induces a clan morphism $\clfy{\theory{T}'} \arr \clfy{\theory{T}}$, which is a full inclusion of clans in the sense of \cref{def:full_inclusion_clan}.
\end{proposition}
\begin{proof}
    The objects in $\clfy{\theory{T}'}$ are contexts of $\theory{T}'$, which are naturally identified with contexts of $\theory{T}$.
    Since the derivability of prejudgments in $\theory{T}'$ coincides with that in $\theory{T}$, morphisms between contexts in $\theory{T}'$ bijectively correspond to those in $\theory{T}$.
    It is straightforward to see that this indeed gives a full inclusion of clans $\clfy{\theory{T}'} \arr \clfy{\theory{T}}$.
\end{proof}
\section{Bicocompleteness of $\Clan$}\label[appendix]{sec:bicocompletenessofClan}
We prove that the 2-category $\Clan$ of clans is bicocomplete, following the argument in \cite{Jelinek2024typetheory}.
The idea is to reconstruct the 2-category $\Clan$ by some operations under which weak local presentability is preserved.
For the general theory of weakly locally presentable 2-categories, the reader is referred to \cite{Bourke2021accessible}.

First, we define some 2-categories related to $\Clan$.
\begin{definition}\label{def:marked_category}
    A \emph{marked category} is a pair $(\C,\Disp)$ of a category $\C$ and a class of morphisms $\Disp$ in $\C$ that is closed under composition and contains all identity morphisms, i.e., a pair satisfying \cref{def:clan-2} in \cref{def:clan}.
    A \emph{marked functor} between marked categories $(\C,\Disp_\C)$ and $(\D,\Disp_\D)$ is a functor $F\colon \C\arr \D$ preserving the marked morphisms, i.e., $F(\Disp_\C)\subseteq \Disp_\D$.
    We denote by $\MCat$ the 2-category of small marked categories, marked functors, and natural transformations between the underlying functors.
\end{definition}

\begin{notation}
    For a marked category $(\C,\Disp)$, we call the morphisms in $\Disp$ \emph{marked morphisms} and draw them with $\darr$.
\end{notation}

Here is the definition adopted in \cite{Jelinek2024typetheory}.
\begin{definition}[{\cite[Definition 3.1.3, Remark 3.1.5]{Jelinek2024typetheory}, \cite[Proposition 3.6 and the definition preceding it]{Bourke2021accessible}}]
    A 2-category is \emph{accessible} if it is locally small and has powers by the walking morphism category $\walkmorcat$, and for some small (infinite) regular cardinal $\kappa$, it underlies a $\kappa$-accessible (1-)category, and (conical) $\kappa$-filtered small 2-colimits commute with powers by $\walkmorcat$.
    The 2-category $\WLP$ consists of the following 0-cells, 1-cells, and 2-cells.
    \begin{itemize}
        \item
            A 0-cell is an accessible (large) 2-category with all flexible small 2-limits and (conical) filtered small 2-colimits.
        \item
            A 1-cell is a 2-functor preserving flexible small 2-limits and (conical) filtered small 2-colimits.
        \item
            A 2-cell is a 2-natural transformation between the underlying 2-functors of 1-cells.\qedhere
    \end{itemize}
\end{definition}

$\Cattwo$ is the 2-category of small categories, functors, and natural transformations, which is known to be in $\WLP$.
The following theorem guarantees some constructions inside $\WLP$.
Let $\Psd(\mathcal{K},\mathcal{L})$ denote the 2-category of 2-functors, pseudonatural transformations, and modifications between 2-categories $\mathcal{K}$ and $\mathcal{L}$.
Let $\arrowtwocat, \Comp, \Cosp, \freeAdj$ denote the walking 1-cell, the walking up-to-iso commutative triangle, the walking cospan of 1-cells, and the walking adjunction, respectively.
\begin{table}[h]
    \centering
    \begin{tabular}{|c|c|c|c|}
        \hline
        $\arrowtwocat$ & $\Comp$ & $\Cosp$ & $\freeAdj$
        \\
        \hline
        $\begin{tikzcd}[ampersand replacement=\&]
            0 \ar[r] \& 1
        \end{tikzcd}$
        &
        $\begin{tikzcd}[ampersand replacement=\&, small]
            0 \ar[rr] \ar[dr] \&{} \& 2 \\
            \& 1 \ar[ur]
            \dtwocell{1-2}{2-2}["{\rotatebox[origin=c]{-90}{$\scriptstyle\cong$}}"{right},pos=0.4]
        \end{tikzcd}$
        &
        $\begin{tikzcd}[ampersand replacement=\&, small]
            \! \& 0 \ar[d] \\
            1 \ar[r] \& 2
        \end{tikzcd}$
        &
        $\begin{tikzcd}[ampersand replacement=\&]
            0 \ar[r, bend left] \ar[r, phantom, "\perp"]
            \& 1 \ar[l, bend left]   
        \end{tikzcd}$
        \\
        \hline
    \end{tabular}
\end{table}

\begin{fact}[{\cite[Theorem 3.1.7]{Jelinek2024typetheory}}]\label{fact:WLP_constructions}
    \
    \begin{enumerate}
        \item\label{fact:WLP_constructions-1}
            $\WLP$ is closed under large bilimits in the 2-category $\twoCATtwo$ of large 2-categories, 2-functors, and 2-natural transformations.
        \item\label{fact:WLP_constructions-2}
            $\Psd(\arrowtwocat,\Cattwo)$, $\Psd(\Comp,\Cattwo)$, $\Psd(\Cosp,\Cattwo)$, and $\Psd(\freeAdj,\Cattwo)$ are in $\WLP$, in which flexible small 2-limits and filtered small 2-colimits are preserved and jointly reflected by the evaluation 2-functors at each 0-cell of the domain 2-categories.
        \item\label{fact:WLP_constructions-3}
            The precomposition 2-functor $F^*\colon \Psd(\mathcal{L},\Cattwo)\arr \Psd(\mathcal{K},\Cattwo)$ induced by a 2-functor $F\colon \mathcal{K}\arr \mathcal{L}$ between any of $\arrowtwocat,\Comp,\Cosp,\freeAdj$ that is injective on 0-cells and faithful on 1-cells is an isofibration.\footnote{%
                A 2-functor is called an \emph{isofibration} if so is the underlying (1-)functor.%
            }\,\footnote{%
                In \cite{Jelinek2024typetheory}, $\Comp$ is taken to be the walking commutative triangle, which does not have a non-trivial 2-cell, but it is informed by the author that in order to make the precomposition an isofibration, $\Comp$ should be taken to be the walking up-to-iso commutative triangle as above.%
            }\qedhere
    \end{enumerate}
\end{fact}

\begin{notation}
    For a marked category $\C$, we write $\arrowcat{\C}$ for the arrow category of the underlying category, and write $\marrowcat{\C}\subseteq\arrowcat{\C}$ for the full subcategory spanned by all marked morphisms.
    These assignments induce representable 2-functors $\arrowcat{(-)}\colon \MCat\arr \Cattwo$ and $\marrowcat{(-)}\colon \MCat\arr \Cattwo$.
\end{notation}

\begin{lemma}[{\cite[Section 3.2.1, 3.2.2]{Jelinek2024typetheory}}]\label{lemma:MCat_bicocomplete}
    Let $\MCatIso$ be the full sub-2-category of $\MCat$ spanned by small marked categories $(\C,\Disp)$ such that $\Disp$ contains all isomorphisms in $\C$.
    \begin{enumerate}
        \item\label{lemma:MCat_bicocomplete-1}
            $\MCat$, $\MCatIso$, and the forgetful 2-functors $\MCatIso\arr \MCat, \MCat \arr \Cattwo$ are in $\WLP$.
        \item\label{lemma:MCat_bicocomplete-2}
            The forgetful 2-functor $\MCatIso\arr \MCat$ is an isofibration.
        \item\label{lemma:MCat_bicocomplete-3}
            The 2-functors $\arrowcat{(-)}$ and $\marrowcat{(-)}$ from $\MCat$ to $\Cattwo$ are in $\WLP$.
    \end{enumerate}
\end{lemma}
\begin{proof}
    The statements can be found in Sections 3.2.1 and 3.2.2 of \cite{Jelinek2024typetheory},\footnote{%
        $\MCat$ is denoted by $\mathbf{Cat}_{m}$ and $\MCatIso$ by $\mathbf{Cat}_{m^+}$ in \cite{Jelinek2024typetheory}. 
        The author informed us that the proof of the fact that $\MCatIso$ is in $\WLP$ has a gap in the thesis, but it is directly verified by a similar argument as the one for $\MCat$ in the thesis, which will be included in the published version of the thesis.%
} except for the second statement, which follows from the fact that the property that all isomorphisms are marked is preserved by isomorphisms in $\MCat$.
\end{proof}

Let $\mcosp$ be the walking cospan whose one leg is marked, i.e.,
\begin{equation*}
    \mcosp \coloneq 
    \begin{tikzcd}[column sep=small, row sep = small, ampersand replacement=\&]
        \! \& 0 \ar[d, disp] \\
        1 \ar[r] \& 2
    \end{tikzcd}
\end{equation*}

\begin{lemma}\label{lemma:mcosp_in_WLP}
    The representable 2-functor $\MCat(\mcosp,-)\colon \MCat\arr \Cattwo$ is in $\WLP$.
\end{lemma}
\begin{proof}
    Since $\MCat$ has powers by $\walkmorcat$, it suffices to show that $\mcosp$ is finitely presentable in the underlying category of $\MCat$.
    However, as $\mcosp$ is obtained as a pushout of diagrams consisting of the walking object, morphism, and marked morphism in the underlying category of $\MCat$, which are finitely presentable by \cref{lemma:MCat_bicocomplete}\cref{lemma:MCat_bicocomplete-1,lemma:MCat_bicocomplete-3}, $\mcosp$ is also finitely presentable.
\end{proof}

\begin{lemma}\label{prop:MCatPbTer_in_WLP}
    Take the following 2-pullback diagram 
    \begin{equation*}
        \begin{tikzcd}
            \MCatPbTer\chosen 
            \ar[r]
            \ar[d]
            \ar[dr, phantom, "\lrcorner", very near start]
            & \Psd(\freeAdj,\Cattwo)^2
            \ar[d, "{(I^*)^2}"] \\
            \MCat
            \ar[r, "T"]
            & \Psd(\arrowtwocat,\Cattwo)^2
        \end{tikzcd}
        \incat{\twoCATtwo}.
    \end{equation*}
    Here, $I$ is the inclusion $\arrowtwocat\arr \freeAdj$ sending the non-trivial 1-cell to the left adjoint 1-cell in $\freeAdj$, and the 2-functor $T=(T_1,T_2)$ is defined by
    \begin{equation*}
        T(\C) \coloneq \left(T_1(\C), T_2(\C)\right)\coloneq\left( \C\arr(D_\C)[][1.5]\MCat(\mcosp,\C),\quad \C \arr(!) \terminalcat \right),
    \end{equation*}
    where $D_\C$ is the diagonal functor.
    $\MCatPbTer\chosen$ is the 2-category of small marked categories $\C$ together with the choice of pullback squares of any morphism along any marked morphism and a terminal object in $\C$, marked functors preserving these structures up to isomorphism, and natural transformations between the underlying functors.

    Moreover, $T$ and $(I^*)^2$ are in $\WLP$ and $ (I^*)^2$ is an isofibration,
    hence, $\MCatPbTer\chosen$ and the projections from it are in $\WLP$.
    The vertical 2-functor on the left is also an isofibration.
\end{lemma}
\begin{proof}[Sketch of proof]
    Choosing a right adjoint of $D_\C$ with the data of adjunction is equivalent to choosing, for each morphism $f\colon X\arr Y$ in $\C$ and each marked morphism $g\colon Z\darr Y$, a pullback square 
    \begin{equation*}
        \begin{tikzcd}
            {P(f,g)} \ar[r, "{p_2(f,g)}"] \ar[d, "{p_1(f,g)}"'] \ar[dr, phantom, "\lrcorner", very near start] & Z \dar[d, "g"] \\
            X \ar[r, "f"'] & Y
        \end{tikzcd}
        \incat{\C}.
    \end{equation*} 
    This can be strengthened to the statement that the pullback of $I^*$ along $T_1$ is the 2-category of marked categories with the choice of such pullbacks, marked functors preserving them up to isomorphism, and natural transformations between the underlying functors.
    We can similarly characterize the pullback of $I^*$ along $T_2$ in terms of choices of a terminal object $T$.
    Thus, the 2-category $\MCatPbTer\chosen$ given by the 2-pullback is described as in the statement.

    It follows from \cref{fact:WLP_constructions}\cref{fact:WLP_constructions-2,fact:WLP_constructions-3} that $I^*$ is an isofibration and in $\WLP$.
    As for $T$, the flexible limits and filtered colimits in $\Psd(\arrowtwocat,\Cattwo)$ are computed componentwise, so by \cref{lemma:MCat_bicocomplete}\cref{lemma:MCat_bicocomplete-1} and \cref{lemma:mcosp_in_WLP}, $T$ preserves them.
\end{proof}

From now on, we represent an object in $\MCatPbTer\chosen$ by $(\C, \Disp, P, T)$, where $P$ and $T$ are the choices of pullbacks and a terminal object, respectively.
\begin{proposition}\label{prop:PreClan_in_WLP}
    Take the following 2-pullback diagram 
    \begin{equation*}
        \begin{tikzcd}
            \PreClan\chosen
            \ar[r]
            \ar[d]
            \ar[dr, phantom, "\lrcorner", very near start]
            & \Psd(\Comp,\Cattwo)^2
            \ar[d, "{(J^*)^2}"] \\
            \MCatPbTer\chosen
            \ar[r, "S"]
            & \Psd(\Cosp,\Cattwo)^2
        \end{tikzcd}
        \incat{\twoCATtwo}.
    \end{equation*}
    Here, $J$ is the inclusion $\Cosp\arr \Comp$ characterized by the cospan consisting of $0\arr 2$ and $1\arr 2$, and $S$ is the 2-functor defined by
    \begin{equation*}
        S(\C) \coloneq
        \left(
        \begin{tikzcd}[column sep=small, ampersand replacement=\&]
            \&
            \MCat(\mcosp,\C)
            \ar[d, "\mathrm{pb}"] \\
            \marrowcat{\C}
            \ar[r, hookrightarrow]
            \& 
            \arrowcat{\C}
        \end{tikzcd}
        \raisebox{-26pt}{,\qquad}
        \begin{tikzcd}[column sep=small, ampersand replacement=\&]
            \& \C
            \ar[d, "\mathrm{unq}"] \\
            \marrowcat{\C}
            \ar[r, hookrightarrow]
            \&
            \arrowcat{\C}
        \end{tikzcd}
        \right)
    \end{equation*}
    where $\mathrm{pb}$ is a functor assigning, to each cospan with one leg marked, the chosen pullback of the marked morphism along the morphism, and $\mathrm{unq}$ is a functor assigning to each object the unique morphism from it to the chosen terminal object.
    The 2-category $\PreClan\chosen$ is described as follows:
    \begin{itemize}
        \item
            its objects are objects $(\C,\Disp,P,T)$ in $\MCatPbTer\chosen$ equipped with the choice of marked morphisms $p_1'(f,g)$ and an isomorphism $p_1'(f,g)\cong p_1(f,g)$ in $\arrowcat{\C}$ for each pair $(f,g)$ as above, and the choice of marked morphisms $t_X$ and an isomorphism $t_X \cong {!_X}$ for each object $X$ in $\C$ where $!_X$ is the unique morphism from $X$ to the chosen terminal object,
        \item
            its 1-cells are morphisms between the underlying marked categories in $\MCatPbTer\chosen$ equipped with the choice of isomorphisms in $\arrowcat{\C}$ between the images of $p_1'(f,g)$ and $t_X$ and the corresponding morphisms in the codomain, and
        \item
            its 2-cells are natural transformations between the underlying functors.
    \end{itemize}
    Moreover, $S$ and $(J^*)^2$ are in $\WLP$ and $ (J^*)^2$ is an isofibration, hence, $\PreClan\chosen$ and the projections from it are in $\WLP$.
    The vertical 2-functor on the left is also an isofibration.
\end{proposition}
\begin{proof}[Sketch of proof]
    By the definition of $\mathrm{pb}$ and $\mathrm{unq}$, the choice of filling the cospans with a 1-cell and an iso-2-cell along $J$ is equivalent to the choice of marked morphisms $p_1'(f,g)$ and $t_X$ as above together with the choice of isomorphisms in $\arrowcat{\C}$ between them.
    The correspondence between the 1-cells and 2-cells is similar.
    Since $J^*$ is an isofibration by \cref{fact:WLP_constructions}\cref{fact:WLP_constructions-3}, it is also a bipullback.
    To see the diagram belongs to $\WLP$, $J^*$ is in $\WLP$ by \cref{fact:WLP_constructions}\cref{fact:WLP_constructions-2}, and $S$ preserves flexible limits and filtered colimits by the same argument as in the proof of \cref{prop:MCatPbTer_in_WLP}.
\end{proof}

\begin{lemma}\label{lemma:Clan_chosen_in_WLP}
    Take the following 2-pullback diagram
    \begin{equation*}
        \begin{tikzcd}
            \Clan\chosen
            \ar[r]
            \ar[d]
            \ar[dr, phantom, "\lrcorner", very near start]
            & \PreClan\chosen
            \ar[d]
            \\
            \MCatIso
            \ar[r]
            & \MCat
        \end{tikzcd}
        \incat{\twoCATtwo}.
    \end{equation*}
    This 2-category $\Clan\chosen$ is given by the full sub-2-category of $\PreClan\chosen$ spanned by small marked categories $(\C,\Disp)$ such that $\Disp$ contains all isomorphisms in $\C$.
    This is also a bipullback diagram and $\Clan\chosen$ is in $\WLP$.
\end{lemma}
\begin{proof}[Sketch of proof]
    Since the 2-functors that form the cospan are in $\WLP$ and isofibrations by \cref{lemma:MCat_bicocomplete,prop:MCatPbTer_in_WLP,prop:PreClan_in_WLP}, the 2-pullback of the cospan is in $\WLP$ and also a bipullback.
    The description of the 2-category $\Clan\chosen$ is straightforward from the definition of the pullback.
    Thus, the 2-category $\Clan\chosen$ is in $\WLP$ by \cref{fact:WLP_constructions}\cref{fact:WLP_constructions-1}.
\end{proof}

\begin{lemma}\label{lemma:Clan_equivalent_to_Clan_chosen}
    The 2-category $\Clan$ is 2-equivalent to $\Clan\chosen$.
\end{lemma}
\begin{proof}[Sketch of proof]
    By choosing pullbacks and a terminal object in each clan, and choosing the isomorphisms in $\arrowcat{\C}$ in \cref{prop:PreClan_in_WLP} to be the identities, we can construct a 2-functor $\Clan\arr \Clan\chosen$.
    This 2-functor is clearly fully faithful on 2-cells, and it is so on 1-cells because the choice of isomorphisms involved in the definition of 1-cells in $\Clan\chosen$ is uniquely determined by the universal properties when the 1-cells are between the images of $\Clan$.
    It is essentially surjective on objects up to isomorphism because for an object in $\PreClan\chosen$ with all isomorphisms marked, the chosen pullbacks of marked morphisms $p_1(f,g)$ and the unique morphisms $!_X$ to the chosen terminal object themselves are marked, so the underlying marked category can be seen as an object in $\Clan$.
    Thus, there is an isomorphism in $\Clan\chosen$ between the image of this object and the original object in $\Clan\chosen$, which replaces the choice of pullbacks, terminal objects, the morphisms $p_1'(f,g)$ and $t_X$ in the unique way.
    Therefore, the 2-functor $\Clan\arr \Clan\chosen$ is a 2-equivalence.
\end{proof}

\begin{corollary}
    The 2-category $\Clan$ is in $\WLP$ and thus bicocomplete.
\end{corollary}
\begin{proof}
    By \cref{lemma:Clan_chosen_in_WLP,lemma:Clan_equivalent_to_Clan_chosen}, $\Clan$ is 2-equivalent to $\Clan\chosen$, which is in $\WLP$.
    An accessible 2-category with flexible limits \cite[Theorem 9.4]{LackRosicky2012enriched} is bicocomplete, see also \cite[Section 9.3]{BourkeLackVokrinek2023adjoint}.
\end{proof}
\end{appendix}
\addtocontents{toc}{\protect\setcounter{tocdepth}{2}}

\printbibliography

\end{document}